\documentclass[11pt]{article}

\usepackage[
  shownumpages,               
  bgcolor={245,245,250},      
  braincolor={190,32,240},    
  citingstyle=authoryear,     
  bibliostyle=plainnat,       
  bibfile=refs          
]{brainlab}

\setbrainmeta{
  title={Exploring New Frontiers in Vertical Federated Learning: the Role of Saddle Point Reformulation},
  authors={
    Aleksandr Beznosikov\textsuperscript{1,2,3}, Georgiy Kormakov \textsuperscript{2}, 
    Alexander Grigorievskiy \textsuperscript{4},
    Mikhail Rudakov\textsuperscript{3},
    Ruslan Nazykov\textsuperscript{5},
    Alexander Rogozin\textsuperscript{3}, 
    Anton Vakhrushev\textsuperscript{4}, 
    Andrey Savchenko\textsuperscript{4,6}, 
    Martin Takáč\textsuperscript{7}, 
    Alexander Gasnikov\textsuperscript{3}
  },
  affiliations={
    \textsuperscript{1}Basic Research of Artificial Intelligence Laboratory (BRAIn Lab) \\
    \textsuperscript{2}Federated Learning Problems Laboratory \\
    \textsuperscript{3} Innopolis University \\
    \textsuperscript{4} Sber AI Lab\\
    \textsuperscript{5} Technical University of Munich\\
    \textsuperscript{6} HSE University\\
    \textsuperscript{7} Mohamed bin Zayed University of Artificial Intelligence\\
  },
  abstract={
    The objective of Vertical Federated Learning (VFL) is to collectively train a model using features available on different devices while sharing the same users. This paper focuses on the saddle point reformulation of the VFL problem via the classical Lagrangian function. We first demonstrate how this formulation can be solved using deterministic methods. More importantly, we explore various stochastic modifications to adapt to practical scenarios, such as employing compression techniques for efficient information transmission, enabling partial participation for asynchronous communication, and utilizing coordinate selection for faster local computation. We show that the saddle point reformulation plays a key role and opens up possibilities to use mentioned extension that seem to be impossible in the standard minimization formulation. Convergence estimates are provided for each algorithm, demonstrating their effectiveness in addressing the VFL problem. Additionally, alternative reformulations are investigated, and numerical experiments are conducted to validate performance and effectiveness of the proposed approach.
  }
}

\newcommand{\Exp}{\mathbb{E}}
\newcommand{\PP}{\mathbb{P}}
\newcommand{\E}[1]{{\mathbb{E}\left[#1\right] }}    

\newcommand{\R}{\mathbb{R}}

\usepackage{color}
\usepackage{graphicx}
\usepackage{enumitem}
\usepackage{nicefrac}
\usepackage{wrapfig}
 \allowdisplaybreaks

\usepackage{tikz}
\usepackage{threeparttable}
\usepackage{makecell}
\usepackage{lscape}
\usepackage{multirow}

\newcommand{\prox}{\text{prox}}

\newcommand{\cO}{{\cal O}}

\newcommand{\cX}{{\cal X}}
\newcommand{\cY}{{\cal Y}}

\newcommand{\cZ}{{\cal Z}}

\makeatletter
\newcommand\fs@nocaptionruled{
  \let\@fs@capt\relax
  \def\@fs@pre{}
  \def\@fs@post{\kern2pt\hrule\relax}%
  \def\@fs@mid{\kern2pt\hrule\kern2pt}%
  \let\@fs@iftopcapt\iftrue}
\makeatother

\begin{document}
\begin{mainpart}
\section{Introduction}
\label{sec:intro}

Federated Learning is an emergent paradigm that involves training a model on private data from several devices. 
It can be divided into two types: horizontal (HFL)\cite{konevcny2016federated,mcmahan2017communication}, where data samples are distributed across clients, and vertical (VFL)~\cite{liu2022vertical,yang2023survey,vfl_review_2,vfl_review_3} with {feature-wise partitioning}. In contrast to HFL, VFL divides the features of the same samples across clients. 
In this paper, we focus on the VFL problem, which appears in various fields from scoring problems~\cite{vfl_finsector_2021} to healthcare~\cite{vfl_med_2019} and smart manufacturing~\cite{vfl_smartmanuf_2021}.

Since we deal with a distributed environment in both horizontal and vertical data partitioning, the organization of the communication process plays a crucial role in developing learning algorithms. Due to the difference in the formulations, unique characteristics and issues can arise. The HFL problem statement is very similar to the classical distributed cluster learning \cite{verbraeken2020survey}, therefore, the study of various kinds of specialized HFL algorithms that take into account different aspects ranging from communication efficiency to personalization is quite extensive and comprehensive~\cite{kairouz2021advances}. It is a natural idea to transfer most of the techniques and useful stories from the horizontal scenario to the vertical one. And there are such results -- see e.g. \cite[Table~3]{liu2022vertical}, but not many at the moment. This can be due to the fact that the VFL problem is more ambiguous and complex from a formal optimization point of view, {than} it is not easy to use the theory from HFL.

Formally, VFL can be viewed as a classical minimization problem, with specifics in calculating the loss function, its gradient, or possibly higher-order derivatives.
But there is another way to look at. In particular, the VFL problem can be rewritten as an augmented Lagrangian \cite[Section 8]{boyd2011distributed}, which can be solved using the \texttt{ADMM} method \cite{Glowinski1975SurLP,GABAY197617}. 
Recent works argue that such a view of VFL is more private \cite{hu2019learning,xie2022improving}. 
The augmented Lagrangian reformulation combined with the \texttt{ADMM} algorithm is a powerful tool for solving many practical optimization problems (not just VFL) \cite{5594963,JMLR:v11:forero10a,WAHLBERG201283,wang2013online,NIPS2014_8fb5f8be,NIPS2014_1ff1de77}.
It provides privacy and an efficient solution for various scenarios, offering superior performance compared to other methods, making it a viable choice. 

In spite of this, the already mentioned modifications from \cite[Table~3]{liu2022vertical} focus on the basic minimization formulation. 
Notably, many of these results are often empirical and lack a theoretical foundation for convergence. Moreover, current results around the Lagrangian statement have also shortcomings and are weakly studied. In particular, the widely-used approach to VFL based on  
\texttt{ADMM} is costly, as two additional minimization subproblems must be solved on each iteration. Thus, we propose to expand the theory around  saddle point reformulations in this paper since the augmented Lagrangian reformulation is a good alternative to the classical minimization formulation. 
In particular, we address three research questions:
\begin{enumerate}[nosep]
\it
    \item 
    Is there any other way to rewrite the VFL problem which can provide advantages to the standard minimization formulation?
    \item What basic method should be used to solve the new VFL problem reformulation?
    \item Is it possible to modify the basic method for practical use? 
\end{enumerate}

\subsection{Our contribution} \label{sec:contr}

$\bullet$ \textbf{New look at VFL.} We first consider the VFL reformulation via classical Lagrangian and show that if the original VFL problem is convex, then the reformulation is convex-concave Saddle Point Problem (SPP), hence methods for SPP, such as \texttt{ExtraGradient} \cite{korpelevich1976extragradient,nemirovski2004prox}, can be applied to it.

$\bullet$ \textbf{New basic method for VFL.}
The classical Lagrangian is a convex-concave SPP that can be solved using optimal methods. 
We introduce the basic deterministic algorithm and its efficient stochastic modifications for VFL and prove that they significantly outperform existing techniques, e.g., \texttt{ADMM}, in terms of iteration cost (Table~\ref{tab:comparison0} in Appendix~\ref{app:algos}).

$\bullet$ \textbf{Family of practical modifications.} 
We present various modifications of the basic version of the algorithm to address practical needs and to make the basic algorithm more robust, including 
i)  introducing
compression operators to reduce the amount of transmitted information and solve the communication bottleneck \cite{alistarh2017qsgd,Alistarh-EF2018};
ii) allowing partial participation  for asynchronous device communication \cite{ribero2020communication};
iii) a coordinate modification to reduce the cost of local computing \cite{doi:10.1137/100802001};
{iv) adding noise and encryption for more privacy \cite{abadi2016deep}.}
Moreover, we show that the saddle reformulation allows to fully reveal the possibilities of these modifications.


$\bullet$ \textbf{New solver for augmented Lagrangian problem.}
Although the focus of this paper is primarily on the classical Lagrangian, we also consider the augmented version, present an algorithm for it and prove convergence estimates. 
The convergence estimates of the method for the augmented Lagrangian are no better (or even worse if the augmentation parameter is high) than those of the method for the classical Lagrangian.  
That is why we focus on the non-augmented formulation and put the augmented one in Appendix~\ref{app:aug}. 

$\bullet$ \textbf{More VFL reformulations.} 
We consider additional saddle point reformulations of the VFL problem, which have advantages, such as easier stepsize estimation, but require extra memory or existence of dual function of the loss. 

$\bullet$  
\textbf{Extension to non-convex problems.}
We show how our approach can be generalized to handle non-convex learning problems. It is worth noting that all modifications are easily transferable.

$\bullet$ \textbf{Numerical experiments.} We show empirically that our approach can outperform existing VFL solutions in the standard minimization formulation and the saddle problem reformulation.

\subsection{Technical preliminaries}

We use $\langle a,b \rangle = \sum_{i=1}^d [a]_i [b]_i$ to denote the standard inner product of $a,b\in\R^d$ where $[a]_i$ corresponds to the $i$-th component of $a$ in the standard basis in $\R^d$. It induces $\ell_2$-norm in $\R^d$ in the following way $\|x\|_2 = \sqrt{\langle x, x \rangle}$. To denote maximal eigenvalue of positive semidefinite matrix $M\in\R^{d\times d}$ we use $\lambda_{\max}(M)$. Operator $\E{\cdot}$ denotes mathematical expectation, and operator $\mathbb{E}_\xi[\cdot]$ express conditional mathematical expectation w.r.t. all randomness coming from random variable $\xi$.

We also need two classical definitions for the function $f$. 

\begin{definition}
The function $f: \R^d \to \R$, is $L$-\textit{smooth}, if there exists a constant $L > 0$ such that $\forall x,y \in \R^d$
$\| \nabla f(x) - \nabla f(y)\| \leq L \|x-y \|.$ 
\end{definition}
\begin{definition}
The function $f: \R^d \to \R$, is \textit{convex}, if $f(x) \geq f(y) + \langle \nabla f(y) , x - y \rangle$ ~for all $x,y \in \R^d$.
\end{definition}

\section{Saddle Point Reformulation and Extragradient} \label{sec:ref}

\textbf{Reformulation.}
The most common problem in machine learning, known as empirical risk minimization \cite{shalev2014understanding}, can be formulated as follows:
\begin{equation}
    \label{eq:main_problem}
        \min_{x \in \R^d} ~~ \left[ f(x) := \ell\left(Ax,b\right) + r(x)\right],
\end{equation}
where $x$ is a vector of model parameters, $A \in \R^{s \times d}$ is a data matrix, $b \in \R^{s}$ is a vector of labels, $\ell: \R^s \times \R^s \to \R$ is a loss function, $r: \R^d \to \R$ is a separable regularizer, $s$ is a number of data samples and $d$ is a size of the model. This paper considers a VFL setting where data is stored across $n$ different devices. Here, the matrix $A$ is divided by columns, and each device gets different features of each of the $s$ data points (for simplicity, we assume that the matrix $A$ contains no missing data). Thus, we can rewrite (\ref{eq:main_problem}) in the form of the VFL problem~\cite{liu2024vertical}:
\begin{equation}
    \label{eq:main_problem_vfl}
        \min_{x \in \R^d}   \left[\ell\left(\sum_{i=1}^{n} A_ix_i,b\right) + \sum_{i=1}^n r_i(x_i) \right],
\end{equation}
where $A_i \in \R^{s \times d_i}$ is a local data matrix on the $i$-th device, $x_i$ is a part of the parameters corresponding to the features of $i$-th device. It is natural to assume that $x_i$ lies on $i$th device. 
We additionally assume that labels $b$ contain private information and are stored in the first device. 
The problem (\ref{eq:main_problem_vfl}) can be rewritten as a constrained problem with additional variable $z \in \R^s$:
\begin{equation}
\label{eq:vfl_constr_1}
\begin{split}
        \min_{x \in \R^d} 
        \min_{z \in \R^s} ~~ \left[\ell\left(z,b\right) + \sum_{i=1}^n r_i(x_i)\right],
        ~~
        \text{s.t.} ~~ \sum_{i=1}^n A_i x_i = z.
\end{split}
\end{equation}

In turn, the problem with constraints can be rewritten as a saddle point problem, where the target function is the Lagrangian function 
\begin{equation}
    \label{eq:vfl_lin_spp_1}
\begin{split}
    \min_{(x, z) \in \R^{d + s}}\max_{y \in \R^s} \left[L (x,z,y) := \ell\left(z,b\right)
    + \sum_{i=1}^n r_i(x_i) + y^T \left(\sum_{i=1}^n A_i x_i - z\right)\right].
\end{split}
\end{equation}

The formulation (\ref{eq:vfl_lin_spp_1}) is the focus of our paper. 
Meanwhile, as we mentioned earlier, approaches to VFL based on \texttt{ADMM} also consider the Lagrangian functions with a regularizer $(\rho/2) \| \sum_{i=1}^n A_i x_i - z \|^2$ (we consider this case in Appendix \ref{app:aug}). 
For both reformulations, we propose a method that guarantees its convergence.

\textbf{Why saddle point?} 
Let us try to motivate the use of the saddle point reformulation (\ref{eq:vfl_lin_spp_1}) instead of the classical minimization the problem (\ref{eq:main_problem}) with the following example. 

If we consider the classical formulation (\ref{eq:main_problem}), which is valid for both vertical and horizontal cases, the main difference between these two types of data partitioning is the nature of the gradient computation, in particular concerning the communication process. In the horizontal case of (\ref{eq:main_problem}), all workers have the same parameter vectors but different training samples: $\ell (Ax,b) = \sum_{i=1}^n \ell (\hat A_i x, b_i)$, where $\hat A_i \in \R^{s_i \times d}$, $b_i \in \R^{s_i}$. To compute the gradient, we simply accumulate $\nabla_x \ell (\hat A_j x^k, \hat b_j)$ from all the workers: $\nabla f(x) = \sum_{j=1}^n \nabla_x \ell (\hat A_j x^k, \hat b_j)$. In the vertical case, to calculate the gradient for the parameters $x_i$ stored on the $i$th device, it is necessary to obtain $A_j x_j$ from all the devices: $\nabla_{x_i} f(x) = A_i^T \nabla_z \ell (z, b)$ with $z = \sum_{j=1}^n A_j x_j$.

In modern applications, various kinds of stochasticity arise in communication: compression to speed up information transfer or random noise for privacy \cite{abadi2016deep}. Let us consider the simplest model in which the stochasticity of communication is additive to the package on which it acts: package + noise $\xi$. As we discussed, we send different things in the horizontal and vertical cases. More specifically, the randomness we introduce has the following effect on the true gradients:
\begin{itemize}
    \item in the horizontal case
    \begin{equation*}
\nabla f(x)  ~\to~ \sum_{j=1}^n [\nabla_x \ell (\hat A_j x^k, \hat b_j) + \hat \xi_j] = \sum_{j=1}^n \nabla_x \ell (\hat A_j x^k, \hat b_j) + \sum_{j=1}^n \hat \xi_j,
\end{equation*}
\item in the vertical one
\begin{equation*}
\nabla_{x_i} f(x)  ~\to~  A_i^T \nabla_z \ell (z, b), \quad \text{where} \quad  z = \sum_{j=1}^n A_j x_j + \xi_j.
\end{equation*}
\end{itemize}

A key detail can be seen here: the simplest additive stochasticity in the horizontal case remains additive, but in the vertical case, the influence of randomness dips much more firmly into the gradient structure.
Let us look at how this kind of stochasticity affects the saddle point reformulation (\ref{eq:vfl_lin_spp_1}). One can note that it is also necessary to collect $A_j x_j$ during gradient computing. In more details, $\nabla_y L(x,z,y) = \sum_{j=1}^n A_j x_j - z$. With communication stochasticity this transfers to $\sum_{j=1}^n [A_j x_j + \xi_j] - z = \sum_{j=1}^n A_j x_j + \sum_{j=1}^n \xi_j - z$. The impact of randomness is additive. Because the saddle point reformulation "separates" the loss function $\ell$ and the data matrix $A$, the influence of stochasticity becomes more straightforward compared to (\ref{eq:main_problem}).

Before moving to stochastic methods, we must learn a deterministic algorithm as a base for constructing.

\textbf{Basic method.} 
The most straightforward idea is to solve the saddle problem using the gradient descent-ascent method: $x^{k+1}= x^k - \gamma \nabla_x L (x^k, z^k, y^k)$ (the same for $z$), $y^{k+1}= y^k + \gamma \nabla_y L (x^k, z^k, y^k)$. Gradient descent-ascent is not the best solution (\ref{eq:vfl_lin_spp_1}). Indeed, it gives relatively poor convergence guarantees for strongly convex -- strongly concave problems \cite{browder1966existence,rockafellar1969convex,sibony1970methodes}, and may diverge for convex-concave problems altogether \cite[Sections 7.2 and 8.2]{goodfellow2016nips}. Therefore, it is suggested to take the \texttt{ExtraGradient}/\texttt{Mirror Prox} method \cite{korpelevich1976extragradient,nemirovski2004prox}. 
The essence of this method is the use of an additional extrapolation step: $x^{k+1/2}= x^k - \gamma \nabla_x L (x^k, y^k)$, $x^{k+1}= x^k - \gamma \nabla_x L (x^{k+1/2}, y^{k+1/2})$ (the same for $y$). It can be explained by the simplest example of a two-dimensional saddle point problem $\min_{x \in \R} \max_{y \in \R} g(x,y)=xy$. For the first-order optimality condition, it has the unique saddle point with $(x^*, y^*)=(0, 0)$. In any point $(x^k, y^k)$, the step direction of gradient descent-ascent $(-\nabla_x L (x^k, y^k), 
\nabla_y L (x^k, y^k))$ is orthogonal to $(x^k - x^*, y^k - y^*)$; thus the iteration of gradient descent-ascent enlarges the distance to the saddle point. However, if we make the step of \texttt{ExtraGradient}, the direction $(-\nabla_x L (x^{k+1/2}, y^{k+1/2})$, 
$\nabla_y L (x^{k+1/2}, y^{k+1/2}))$ attracts to the saddle point. Furthermore, \texttt{ExtraGradient} is optimal for convex-concave saddle point problems \cite{zhang2021lower}. 
Iteration of the \texttt{ExtraGradient} method for our problem (\ref{eq:vfl_lin_spp_1}) is given in Algorithm~\ref{alg:EG}, and convergence is proved in Theorem~\ref{th:EG_basic_1}. The proof is postponed to Appendix \ref{app:th:EG_basic_1}.

\begin{assumption} \label{as:convexity_smothness}
The function $\ell: \R^s \to \R$, is $L_{\ell}$-smooth and convex.
Each function $r_i: \R^{d_i} \to \R$, is $L_{r}$-smooth and convex.
\end{assumption}

\begin{theorem} \label{th:EG_basic_1}
Let Assumption \ref{as:convexity_smothness} hold. Let the problem (\ref{eq:vfl_lin_spp_1})
be solved by Algorithm~\ref{alg:EG}. \\ Then for 
$
\gamma = \tfrac{1}{2} \min \left\{ 1; \tfrac{1}{\sqrt{\lambda_{\max}(A^T A)}}; \tfrac{1}{L_r}; \tfrac{1}{L_{\ell}} \right\},
$ 
it holds that
$$
\text{gap}(\bar x^K, \bar z^K,\bar y^K) = \mathcal{O} \left(  \frac{ ( 1 + \sqrt{\lambda_{\max}(A^T A)} + L_{\ell} + L_r  ) D^2}{K}  \right),
$$ 
where $\bar x^K := \tfrac{1}{K}\sum_{k=0}^{K-1} x^{k+1/2}$, $\bar z^K := \tfrac{1}{K}\sum_{k=0}^{K-1} z^{k+1/2}$, $\bar y^K := \tfrac{1}{K}\sum_{k=0}^{K-1} y^{k+1/2}$ and \\ $D^2 := \max_{x, z, y \in \cX, \cZ, \cY} \left[ \|x^0 - x \|^2 + \|z^0 - z \|^2 + \|y^0 - y \|^2 \right]$.
\end{theorem}

\begin{algorithm}{\texttt{EGVFL} for (\ref{eq:vfl_lin_spp_1})}
   \label{alg:EG}
\begin{algorithmic}[1]
   \State {\bfseries Input:} starting point $(x^0, z^0, y^0) \in \R^{d+2s}$, stepsize $\gamma > 0$, number of steps $K$
   \For{$k=0$ {\bfseries to} $K-1$}
   \State \hspace{-0.4cm} 1st device sends $y^k$ to other devices  \label{lin_alg:EG_liny1_send}
   \State \hspace{-0.4cm} All send $A_i x^k_i$ to 1st device \label{lin_alg:EG_linx1_send}
   \State \hspace{-0.4cm} All update: $x^{k+1/2}_i = x^k_i - \gamma \left(A_i^T y^k + \nabla r_i (x^k_i) \right)$ \label{lin_alg:EG_linx1}
   \State \hspace{-0.4cm} 1st updates: $z^{k+1/2} = z^k - \gamma (\nabla \ell(z^k, b) - y^k)$ \label{lin_alg:EG_linz1}
   \State \hspace{-0.4cm} 1st updates: $y^{k+1/2} = y^k + \gamma (\sum_{i=1}^n A_i x^k_i - z^k)$ \label{lin_alg:EG_liny1}
   \State \hspace{-0.4cm} 1st device sends $y^{k+1/2}$ to other devices \label{lin_alg:EG_liny2_send}
   \State \hspace{-0.4cm} All devices send $ A_i x^{k+1/2}_i$ to 1st device  \label{lin_alg:EG_linx2_send}
   \State \hspace{-0.4cm} All update: $x^{k+1}_i = x^k_i - \gamma (A_i^T y^{k+1/2} + \nabla r_i(x^{k+1/2}_i))$ \label{lin_alg:EG_linx2}
   \State \hspace{-0.4cm} 1st updates: $z^{k+1} = z^k - \gamma (\nabla \ell(z^{k+1/2}, b) - y^{k+1/2})$ \label{lin_alg:EG_linz2}
   \State \hspace{-0.4cm} 1st updates: $y^{k+1} = y^k + \gamma (\sum_{i=1}^n A_i x^{k+1/2}_i - z^{k+1/2})$\label{lin_alg:EG_liny2}
   \EndFor
\end{algorithmic}
\end{algorithm}

In Theorem \ref{th:EG_basic_1}, we use the convergence criterion for convex-concave saddle point problems
\begin{equation}
\label{eq:duality_gap}
    \begin{split} 
    \text{gap}(x,z,y)
    := \max_{\tilde y \in \cY} L(x, z, \tilde y) - \min_{\tilde x, \tilde z \in \cX, \cZ} L(\tilde x, \tilde z, y).
    \end{split}
\end{equation}

It is important that in the formulation of Theorem \ref{th:EG_basic_1} and in the definition (\ref{eq:duality_gap}), we use some bounded sets $\cX$, $\cZ$, $\cY$ although the original problem (\ref{eq:main_problem}) is unbounded. Such an assumption is standard for the analysis of methods for convex-concave problems. The criterion (\ref{eq:duality_gap}) can also be used for unconstrained/unbounded problems. To do this, one can use the trick from \cite{nesterov2007dual} and introduce bounded sets $\cX$, $\cZ$, $\cY$ artificially as compact subsets of $\R^d, \R^s, \R^s$. 
This trick is valid if some solution $x^*, y^*, z^*$ lies in $\cX$, $\cZ$, $\cY$. 
Moreover, following \cite[Theorems 3.59, 3.60]{beck2017first}, one can show that in Theorem \ref{th:EG_basic_1} we can use the criterion: $\ell (\bar z^K, b) - \ell (z^*, b) + \| A \bar x^{K} - \bar z^{K} \|$, instead of (\ref{eq:duality_gap}). It is more natural and means that $A \bar x^{K}  \to \bar z^{K}$ and $\ell (\bar z^K, b) \to \ell (A x^*, b)$, which is what is required in the original problem (\ref{eq:main_problem}) (see Section~\ref{sec:gap} for more details). One can find a simplified version of gap {in} \cite{xu2017accelerated}. If we assume the existence of some solution $(x^*, z^* , y^*)$, it can be used as follows: $\text{gap}^*(x,z,y) = L(x, z, y^*) - L(x^*, z^*, y)$. Theorem \ref{th:EG_basic_1} can be rewritten with $\text{gap}^*(x,z,y)$: there is no maximum in the right-hand side of the estimate, simply $\|x^0 - x^* \|^2 + \|z^0 - z^* \|^2 + \|y^0 - y^* \|^2$. But still, the use of (\ref{eq:duality_gap}) is preferable, e.g., for the already mentioned problem $\min_{x \in \R} \max_{y \in \R} g(x,y)=xy$ with the solution $(0,0)$, but then $\text{gap}^*(x,z,y)$ is always exactly $0$.
The following corollary can be easily derived from Theorem~\ref{th:EG_basic_1}.

\begin{corollary} \label{cor:EG_basic_1}
Under the conditions of Theorem \ref{th:EG_basic_1}, to achieve $\varepsilon$-solution we need
$
\mathcal{O} \left( \tfrac{ ( 1+ \sqrt{\lambda_{\max}(A^T A)} + L_{\ell} + L_r  ) D^2}{\varepsilon}  \right) \text{iterations.}
$
\end{corollary}

An intriguing feature of the saddle point reformulation is that the expression  $\ell(Ax, b)$ can be equivalently rewritten as $\tilde \ell ( \tilde A x, b)$ with $\tilde \ell (y, b) = \ell (y / \beta, b)$ and $\tilde A = \beta A$. We can select   $\beta$ such that in Theorem \ref{th:EG_basic_1} we have $\textstyle{\sqrt{\lambda_{\max}(\tilde A^T \tilde A)} = L_{\tilde \ell}}$. 
It becomes evident that the appropriate choice for 
 $\beta = L_{\ell}^{1/3} / \lambda^{1/6}_{\max} (A^T A)$. {Hence, in Corollary \ref{cor:EG_basic_1}, one can achieve $\mathcal{O} \left( \tfrac{ ( 1+ \sqrt[3]{\lambda_{\max}(A^T A) \cdot L_{\ell}} + L_r  ) D^2}{\varepsilon}  \right)$ iteration complexity (further details can be found in Section~\ref{sec:beta}). One can note that \texttt{Gradient Descent} (\texttt{GD}) and \texttt{Accelerated Gradient Descent} (\texttt{AGD}) -- classical deterministic methods for (\ref{eq:main_problem}) have the following convergence estimates \cite{nesterov2003introductory}:} $\mathcal{O} \left( \tfrac{ ( \lambda_{\max}(A^T A) \cdot L_{\ell} + L_r  ) D^2}{\varepsilon}  \right)$ and $\mathcal{O} \left( \tfrac{ \sqrt{\lambda_{\max}(A^T A) \cdot L_{\ell}+ L_r}  D}{\sqrt{\varepsilon}}  \right)$, respectively. It can be seen that the obtained result is better than \texttt{Gradient Descent}, and better than \texttt{Accelerated Gradient Descent} in terms of $\lambda_{\max}(A^T A)$ viewpoint (but worse in $\varepsilon$). As we mentioned in Section \ref{sec:contr}, there are approaches for the saddle point reformulation, e.g., \texttt{ADMM} \cite{vfl_admm_xie2022}. We compare the results in Table~\ref{tab:comparison0} (Appendix~ \ref{app:algos}).

This is possible because the loss function $\ell$ and the data matrix $A$ are ``separated''. {As mentioned before, ``separation'' can also be beneficial for stochastic algorithms.} We explore them in the next section, but now let us note that Algorithm \ref{alg:EG} presents several drawbacks. One notable limitation is its deterministic nature. In the subsequent section, we underscore the disadvantages of this characteristic and suggest alterations to enhance the foundational version of our general approach.
Another significant drawback of Algorithm \ref{alg:EG} is its reliance on the knowledge of $\lambda_{\max}(A^T A)$.
Given that parts of matrix $A$ are dispersed across different devices, determining $\lambda_{\max}(A^T A)$ is challenging. However, an estimation can be made using $\lambda_{\max}(A^T_i A_i)$, as illustrated in Lemma \ref{lem:matrix}. Alternatively, we can contemplate a reformulation that negates the need for $\lambda_{\max}(A^T A)$ entirely, as discussed in Section \ref{sec:reform}. Furthermore, incorporating augmentation, as outlined in \cite{boyd2011distributed}, can be beneficial and straightforward for implementation. It is crucial to highlight that any variations derived from Section~\ref{sec:reform}, as well as any adaptations of Algorithm \ref{alg:EG} from Section \ref{sec:mod}, can be seamlessly integrated. 

\section{Family of Modifications} \label{sec:mod}

This section presents the different modifications of Algorithm \ref{alg:EG}. Most of them are stochastic modifications. These stochastic modifications are one of the main reasons for using saddle point reformulation. In any distributed optimization, including federated learning, both in its vertical and horizontal setting, the issue of communication organization is crucial. In particular, a lot of research is related to the efficiency to spend less time on communications \cite{FEDLEARN,cocoa-2018-JMLR,Ghosh2020,MARINA}, since they are from some point of view a waste of time \cite{kairouz2021advances}.

But we start with a non-stochastic modification.

\subsection{Proximal modification for computational friendly losses/regularizers and constrained setting} \label{app:prox}

Algorithm \ref{alg:EG} assumes we can calculate the gradient of the function $\ell$ and the function $r$. But not all functions even allow this. For example, one can choose the $\ell_1$ regularizer as the function $r$. Or if we want to solve the constrained version of (\ref{eq:main_problem}), we can take $r$ as an indicator function of some set $\cX$. 
Here we consider the case of non-smooth, but computing-friendly $\ell$ and $r$. One can modify lines \ref{lin_alg:EG_linx1}, \ref{lin_alg:EG_linz1}, \ref{lin_alg:EG_linx2}, \ref{lin_alg:EG_linz2} in Algorithm \ref{alg:EG} as follows. 

\begin{algorithm}{\texttt{EGVFL} for (\ref{eq:vfl_lin_spp_1}) with proximal friendly functions}
   \label{alg:prox_EG}
\begin{algorithmic}[1]
   \State {\bfseries Input:} starting point $(x^0, z^0, y^0) \in \R^{d+2s}$, stepsize $\gamma > 0$, number of steps $K$
   \For{$k=0$ {\bfseries to} $K-1$}
   \State First device sends $y^k$ to other devices 
   \State All devices in parallel send $A_i x^k_i$ to first device
   \State All devices in parallel update: $x^{k+1/2}_i = \text{prox}_{\gamma r_i}(x^k_i - \gamma A_i^T y^k)$ \label{lin_alg:EG_prox_linx1}
   \State First device updates: $z^{k+1/2} = \text{prox}_{\gamma \ell}(z^k + \gamma y^k)$
   \State First device updates: $y^{k+1/2} = y^k + (\sum_{i=1}^n \gamma A_i x^k_i - \gamma z^k)$
   \State First device sends $y^{k+1/2}$ to other devices 
   \State All devices send $\gamma A_i x^{k+1/2}_i$ to first device
   \State All devices in parallel update: $x^{k+1}_i = \text{prox}_{\gamma r_i}(x^k_i - \gamma A_i^T y^{k+1/2})$ \label{lin_alg:EG_prox_linx2}
   \State First device: $z^{k+1} = \text{prox}_{\gamma \ell}(z^k + \gamma y^{k+1/2})$
   \State First device: $y^{k+1} = y^k + (\sum_{i=1}^n \gamma A_i x^{k+1/2}_i - \gamma z^{k+1/2})$
   \EndFor
\end{algorithmic}
\end{algorithm}

Here $\text{prox}_{\gamma f}$ is a proximal operator \cite{parikh2014proximal}: $\text{prox}_{\gamma f}(x) = \arg \min_{y \in \R^d} (\gamma f(y) + \tfrac{1}{2}\| x - y\|^2)$. In the general case, solving an additional minimization problem to calculate such an operator is necessary. But, in the case of simple, proximal-friendly functions $\ell$ and $r$, the proximal operator has a closed-form solution and can be computed exactly and sometimes for free. 
Theorem \ref{th:EG_prox} gives the convergence, and proof can be found in Appendix \ref{app:th:EG_prox}.
{
\begin{theorem} \label{th:EG_prox}
Let $\ell$ and $r$ be proximal-friendly and convex functions. Let problem (\ref{eq:vfl_lin_spp_1})
be solved by Algorithm \ref{alg:prox_EG}. Then for 
$
\gamma = \tfrac{1}{\sqrt{2}} \cdot \min \bigg\{ 1;$ $ \tfrac{1}{\sqrt{\lambda_{\max}(A^T A)}} \bigg\},
$
it holds that
$$
\text{gap}(\bar x^K, \bar z^K,\bar y^K) = \mathcal{O} \left( \frac{(1 + \sqrt{\lambda_{\max}(A^T A)} ) D^2}{K}   \right),
$$
where $\bar x^K$, $\bar z^K$, $\bar y^K$, $D^2$ are defined in Theorem \ref{th:EG_basic_1}.
\end{theorem}
}







\subsection{Modification with quantization for effective communications}

Let us take a look at one of one of the main techniques in the fight for communication efficiency -- compression \cite{1bit,alistarh2017qsgd}. The following definition can formally describe the compression of communicated vectors.
{
\begin{definition} \label{def:quantization}
Operator $Q:\R^d\to \R^d$ is called unbiased compressor if there exists a constant $\omega \geq 1$ such that for all $x \in \R^d$ it holds 
$$
    \Exp[Q(x)] = x, \quad \Exp[\| Q(x)\|^2] \leq \omega \| x\|^2.
$$
\end{definition}
}

Operator $Q$ can be e.g., random coordinate choice or randomized rounding \cite{beznosikov2020biased}.



Methods with compression in horizontal distributed learning are studied for quite a long time \cite{1bit,alistarh2017qsgd}. Variance reduction methods provide a breakthrough here, initially proposed to solve non-distributed stochastic finite-sum problems \cite{schmidt2017minimizing,defazio2014saga,johnson2013accelerating,nguyen2017sarah}. Several 
papers have shown that the variance reduction technique can be transferred to the distributed case, where stochasticity appears not from the random choice of the batch number but from the compression \cite{DIANA,MARINA,qian2020error}. 
For our algorithm with unbiased compression (Algorithm \ref{alg:EG_quantization}), we take the variance reduction method for saddle point problems from \cite{alacaoglu2021stochastic}. 
We introduce an additional sequence of points $w^k_i$ (reference points for $x^k_i$) and $u^k$ (reference points for $y^k$). In contrast to the classical variance reduction technique, we do not introduce reference points for the $z$ variables since we do not communicate them. 
To update all $w^k_i$ and $u^k$ synchronously, we need to generate $b_k\in\{0,1\}$, 
one can set the same random seed for generating $b^k$ on all devices to avoid additional communication.  
Next, we have to send full vectors $Aw^k_i$ and $u^k$ to the first device and all the others, respectively. The key is that the reference points are updated rarely, namely with low probability $p$, only when $b_k = 1$ (see lines \ref{lin_alg:unbiased_linws} --
\ref{lin_alg:unbiased_linwf}). When $w^k_i$ and $u^k$ are not updated, we only send compressed vectors $Q(y^{k+1/2} - u^k)$ and $Q(A_i x^{k+1/2}_i - A_i w^{k}_i)$ (lines \ref{lin_alg:unbiased_q1}, \ref{lin_alg:unbiased_q2}). Sending compressed information, rarely forwarding full packages, is the main point of Algorithm \ref{alg:EG_quantization}. 
Theorem \ref{th:EG_compressed_unbiased} gives the convergence; its proof can be found in Appendix \ref{app:th:EG_compressed_unbiased}.

{
\begin{theorem} \label{th:EG_compressed_unbiased}
Let Assumption \ref{as:convexity_smothness} hold. Let problem (\ref{eq:vfl_lin_spp_1})
be solved by Algorithm~\ref{alg:EG_quantization} with operator $Q$ that satisfies Definition \ref{def:quantization}. Then for 
$
\gamma = \frac{1}{4}\min\bigg\{ 1;\frac{1}{L_r};$ $\frac{1}{L_{\ell}};\sqrt{\frac{1-\tau}{\omega \lambda_{\max}(A A^T)}}; \sqrt{\frac{1-\tau}{\omega\lambda_{\max}(A A^T)}} \bigg\}$, $\tau = 1 - p$
it holds that
\begin{align*}
    \E{\text{gap}(\bar x^K, \bar z^K, \bar y^K)}
    = 
    \mathcal{O}  \left(  \left[ 1 + \sqrt{\frac{\omega}{p}}  \sqrt{\lambda_{\max}(A A^T) }  + L_{\ell} + L_r \right] \frac{D^2}{K}  \right),
\end{align*}
where $\bar x^K$, $\bar z^K$, $\bar y^K$, $D^2$ are defined in Theorem \ref{th:EG_basic_1}.
\end{theorem}
}

\begin{algorithm}{\texttt{EGVFL} with unbiased compression for (\ref{eq:vfl_lin_spp_1})}   \label{alg:EG_quantization}
\begin{algorithmic}[1]
   \State {\bfseries Input:}  
   initial point $(x^0, z^0, y^0) \in \R^{d+2s}$, $(w^0, u^0) \in \R^{d+s}$, stepsize $\gamma > 0$, number of steps $K$
   \For{$k=0$ {\bfseries to} $K-1$}
   \State \hspace{-0.4cm} All update: $x^{k+1/2}_i = \tau x^{k}_i + (1 - \tau) w^k_i - \gamma \left(A_i^T u^k + \nabla r_i (x^k_i) \right)$
   \State \hspace{-0.4cm} 1st updates: $z^{k+1/2} = z^k - \gamma (\nabla \ell(z^k, b) - y^k)$
   \State \hspace{-0.4cm} 1st updates: $y^{k+1/2} = \tau y^k + (1 - \tau) u^k + \gamma (\sum_{i=1}^n A_i w^k_i - z^k)$
   \State \hspace{-0.4cm} 1st sends $Q(y^{k+1/2} - u^k)$ to other devices \label{lin_alg:unbiased_q1}
   \State \hspace{-0.4cm} All send $Q(A_i x^{k+1/2}_i - A_i w^{k}_i)$ to 1st \label{lin_alg:unbiased_q2}
   \State \hspace{-0.4cm} All update: $x^{k+1}_i = \tau x^{k}_i + (1 - \tau) w^k_i - \gamma (A_i^T  [Q(y^{k+1/2} - u^k) + u^k] + \nabla r_i(x^{k+1/2}_i))$
   \State \hspace{-0.4cm} 1st update: $z^{k+1} = z^k - \gamma (\nabla \ell(z^{k+1/2}, b) - y^{k+1/2})$
   \State \hspace{-0.4cm} 1st update: $y^{k+1} = \tau y^k + (1 - \tau) u^k + \gamma (\sum_{i=1}^n [Q(A_i x^{k+1/2}_i - A_i w^{k}_i) + A_i w^{k}_i ] - z^{k+1/2})$
   \State \hspace{-0.4cm} Flip a coin $b_k \in \{0, 1\}$ where $\PP\{ b_k = 1 \} = p$ \label{lin_alg:unbiased_linws}
   \If{$b_k = 1$} 
   \State \hspace{-0.4cm} All update: $w^{k+1}_i = x^{k}_i$
   \State \hspace{-0.4cm} 1st updates: $u^{k+1} = y^{k}$
   \State \hspace{-0.4cm} All send uncompressed $A_i w^{k+1}_i$ to 1st
   \State \hspace{-0.4cm} 1st sends uncompressed $u^{k+1}$ to other devices
   \Else
   \State \hspace{-0.4cm} All update: $w^{k+1}_i = w^{k}_i$
   \State \hspace{-0.4cm} 1st updates: $u^{k+1} = u^{k}$ \label{lin_alg:unbiased_linwf}
   \EndIf
   \EndFor
\end{algorithmic}
\end{algorithm}

In Algorithm ~\ref{alg:EG_quantization}, one mandatory communication round with compression occurs and possibly one more (without compression) with probability $p$. If $Q$ compress a package by a factor of $\beta$, then each iteration requires $\mathcal{O}\left( \beta^{-1} + p \right)$ data transfers on average. If $p$ is close to 1, Theorem \ref{th:EG_compressed_unbiased} gives faster convergence, but more data transfer is needed. If $p$ tends to 0, the transmitted information complexity per iteration decreases but the iterative convergence rate drops. The optimal choice of $p$ is $\beta^{-1}$. 
For Theorem \ref{th:EG_compressed_unbiased}, one can obtain an analogue of Corollary \ref{cor:EG_basic_1}, which states that if $p = \beta^{-1}$, then the iterative complexity of Algorithm \ref{alg:EG_quantization} is $\sqrt{w \beta}$ times higher than for Algorithm \ref{alg:EG}. But the estimated amount of information transferred for Algorithm  \ref{alg:EG_quantization}  is $\beta$ times less than the iterative complexity.
For Algorithm \ref{alg:EG}, the complexity of the transmitted information matches the iterative one.  Moreover, for most practical operators $\beta \geq w$. Hence, in the view of full information transferred, Algorithm \ref{alg:EG_quantization} may be better than Algorithm~\ref{alg:EG}.

The use of compression was investigated for the VFL problem, but not in the saddle point formulation. The papers \cite{chen2021secureboost,xu2021efficient,Cai_2022,sun2023communicationefficient} do not provide theoretical guarantees at all. The work \cite{castiglia2023compressedvfl} investigates only special cases of compression operators. Only the authors of paper \cite{stanko2024accelerated} give guarantees only for the quadratic loss $\ell$: $\mathcal{O}  \left( \omega^2 \frac{\lambda^3_{\max}(A A^T)}{\lambda^2_{\min}(A A^T)}\tfrac{ D^2 }{K^2}  \right)$. This is much worse than our guarantee estimates. 

\subsection{Modification with biased compression for more effective communications}

Using unbiased compression operators is more straightforward in theory, but the most popular compression operators in practice are biased (deterministic rounding \cite{horvath2019natural}, greedy coordinate selection \cite{Alistarh-EF2018}, vector decomposition \cite{powersgd})
and can be described as follows.

{
\begin{definition}\label{def:biased}
Operator $C:\R^d\to \R^d$ (possibly randomized) is called a biased compressor if there exists a constant $\delta \geq 1$ such that for all $x \in \R^d$ it holds 
$$
    \E{\| C(x) - x \|^2} \leq \left( 1 - \frac{1}{\delta}\right) \| x\|^2.
$$
\end{definition}
}


Using biased compressors is a complex issue. It can cause divergence even for quadratic problems \cite{beznosikov2020biased}. To fix this, an error compensation technique \cite{stich2018sparsified,error_feedback,stich2019error} can be applied. This approach accumulates non-transmitted information ($\{e_k\}$, $\{ e_i^k\}$) and adds it to a new package at the next iteration of Algorithm \ref{alg:EG_quantization}.

 Theorem \ref{th:EG_compressed_biased} gives the convergence, and proof can be found in Appendix \ref{app:th:EG_compressed_biased}. Note that the proof techniques of Theorems \ref{th:EG_compressed_unbiased} and Theorem~ \ref{th:EG_compressed_biased} differ considerably, just as the proofs of convergence of distributed GD with unbiased and biased compression~\cite{DIANA,stich2019error}. 
{
\begin{theorem} \label{th:EG_compressed_biased}
Let Assumption \ref{as:convexity_smothness} hold. Let the problem (\ref{eq:vfl_lin_spp_1})
be solved by Algorithm~\ref{alg:EG_biased} with operators and $C$, which satisfy Definition \ref{def:biased}. Then for
$\tau = 1 - p$ and
$
\textstyle{
\gamma = \frac{1}{4}\min\bigg\{ 1;\frac{1}{L_r}; \frac{1}{L_{\ell}};\sqrt{\frac{1-\tau}{\delta^2[\lambda_{\max}(A A^T) + n \cdot \max_{i}\{\lambda_{\max}(A_i A^T_i)\}]}};}$ $\textstyle{\sqrt{\frac{1-\tau}{\omega\lambda_{\max}(A A^T)}}\bigg\},
}
$
it holds that
\begin{align*}
\E{\text{gap}(\bar x^K, \bar z^K, \bar y^K)} = \mathcal{O} \left(
     \left[\frac{\delta}{\sqrt{p}}  \left(\sqrt{\lambda_{\max}(A A^T) } +  n \cdot \max\limits_{i = 1, \ldots, n} \{\sqrt{\lambda_{\max}(A_i A^T_i) } \}\right) + L_{\ell} + L_r \right] \frac{D^2}{K}\right),
\end{align*}
where $\bar x^K$, $\bar z^K$, $\bar y^K$, $D^2$ are defined in Theorem \ref{th:EG_basic_1}.
\end{theorem}
}

The choice of optimal $p$ is the same as Algorithm \ref{alg:EG_quantization}. It is enough to take $p = \beta^{-1}$, where $\beta$ is the compression power of $C$.
The estimate from Theorem \ref{th:EG_compressed_biased} shows the central theoretical problem with biased compressors. If $\delta \sim w$, the results in Theorem \ref{th:EG_compressed_biased} are worse than in Theorem \ref{th:EG_compressed_unbiased}. Unfortunately, this kind of problem is inherent in all work around biased compressions -- one cannot fully theoretically justify that biased compressors perform better \cite{MARINA,stich2019error,EF21}. The only thing we can fight for is more or less acceptable convergence. Meanwhile, intuition and practical results show that biased operators are superior to unbiased ones \cite{beznosikov2020biased,EF21}.


\begin{algorithm}{\texttt{EGVFL} with biased compression for (\ref{eq:vfl_lin_spp_1})}
   \label{alg:EG_biased}
\begin{algorithmic}[1]
   \State {\bfseries Input:} starting point $(x^0, z^0, y^0) \in \R^{d+2s}$, $(w^0, u^0) \in \R^{d+s}$, stepsize $\gamma > 0$, number of steps $K$
   \For{$k=0$ {\bfseries to} $K-1$}
   \State All devices in parallel update: $x^{k+1/2}_i = \tau x^{k}_i + (1 - \tau) w^k_i - \gamma \left(A_i^T u^k + \nabla r_i (x^k_i) \right)$
   \State First device updates: $z^{k+1/2} = z^k - \gamma (\nabla \ell(z^k, b) - y^k)$ \label{lin_alg:biased_linz1}
   \State First device updates: $y^{k+1/2} = \tau y^k + (1 - \tau) u^k + \gamma (\sum_{i=1}^n A_i w^k_i - z^k)$ \label{lin_alg:biased_liny1}
   \State First device compresses $C(y^{k+1/2} - u^k + e^k)$ and sends to other devices
   \State First device updates: $e^{k+1} = y^{k+1/2} - u^k + e^k - C(y^{k+1/2} - u^k + e^k)$ and sends to other devices 
   \State All devices in parallel compress $C(A_i x^{k+1/2}_i - A_i w^{k}_i + e^{k}_i)$ and send to first device
   \State All devices in parallel update: $e^{k+1}_i = A_i x^{k+1/2}_i - A_i w^{k}_i + e^{k}_i - C(A_i x^{k+1/2}_i - A_i w^{k}_i + e^{k}_i)$
   \State All devices update: $x^{k+1}_i = \tau x^{k}_i + (1 - \tau) w^k_i - \gamma (A_i^T  [C(y^{k+1/2} - u^k + e^k) + u^k] + \nabla r_i(x^{k+1/2}_i))$
   \State First device update: $z^{k+1} = z^k - \gamma (\nabla \ell(z^{k+1/2}, b) - y^{k+1/2})$ \label{lin_alg:biased_linz2}
   \State First device update: $y^{k+1} = \tau y^k + (1 - \tau) u^k + \gamma (\sum_{i=1}^n [C(A_i x^{k+1/2}_i - A_i w^{k}_i + e^k_i) + A_i w^{k}_i ] - z^{k+1/2})$ \label{lin_alg:biased_liny2}
   \State Flip a coin $b_k \in \{0, 1\}$ where $\PP\{ b_k = 1 \} = p$ \label{lin_alg:biased_linws}
   \If{$b_k = 1$} 
   \State All devices in parallel update: $w^{k+1}_i = x^{k}_i$
   \State First device updates: $u^{k+1} = y^{k}$
   \State All devices send uncompressed $A_i w^{k+1}_i$ to first device
   \State First device sends uncompressed $u^{k+1}$ to other devices
   \Else
   \State All devices in parallel update: $w^{k+1}_i = w^{k}_i$
   \State First device updates: $u^{k+1} = u^{k}$ \label{lin_alg:biased_linwf}
   \EndIf
   \EndFor
\end{algorithmic}
\end{algorithm}

\subsection{Partial participation for asynchronous client connection}

Algorithm \ref{alg:EG} requires that at each iteration all devices communicate (send and receive messages). It is possible that some devices may drop out of the learning process. In this subsection, we consider a modification of Algorithm \ref{alg:EG}, where only 1 randomly selected device communicates at each iteration \cite{ribero2020communication,chen2020optimal,cho2020client,273723}. We take Algorithm \ref{alg:EG_quantization} as a base, but instead of compression, we use random client selection and only send information from this client to the first device.

Even though the first device sends $y^{k+1/2}$ to all devices, this does not mean that all devices need to receive the message at the exact same moment. They can get a set of several messages with $y$ at once when they contact the first device. In that case they just do several sequential  updates of $x_i$.
{
\begin{theorem}\label{th:pp}
Let Assumption \ref{as:convexity_smothness} hold. Let the problem (\ref{eq:vfl_lin_spp_1})
be solved by Algorithm~\ref{alg:EG_pp} (Appendix \ref{app:algos}). Then for $\tau = 1 - p$ and
$ 
\gamma = \frac{1}{4}\min \bigg\{ 1;\frac{1}{L_r}; \frac{1}{L_{\ell}};$ $\sqrt{\frac{1-\tau}{\lambda_{\max}(A A^T) + n \cdot \max_{i}\{\lambda_{\max}(A_i A^T_i)\}}} \bigg\}
$ 
it holds that
\begin{align*}
    \E{\text{gap}(\bar x^K, \bar z^K, \bar y^K)} = \mathcal{O}\left(
     \left[ \frac{1}{\sqrt{p}}\left(\sqrt{\lambda_{\max}(A A^T) } +  n \cdot \max\limits_{i = 1, \ldots, n}\{\sqrt{\lambda_{\max}(A_i A^T_i) }\}\right) + L_{\ell} + L_r \right] \frac{D^2}{K}\right),
\end{align*}
where $\bar x^K$, $\bar z^K$, $\bar y^K$, $D^2$ are defined in Theorem \ref{th:EG_basic_1}.
\end{theorem}
}

\begin{algorithm}{\texttt{EGVFL} with partial participation for (\ref{eq:vfl_lin_spp_1})}
   \label{alg:EG_pp}
\begin{algorithmic}[1]
   \State {\bfseries Input:} starting point $(x^0, z^0, y^0) \in \R^{d+2s}$, $(w^0, u^0) \in \R^{d+s}$, stepsize $\gamma > 0$, number of steps $K$
   \For{$k=0$ {\bfseries to} $K-1$}
   \State All devices in parallel update: $x^{k+1/2}_i = \tau x^{k}_i + (1 - \tau) w^k_i - \gamma \left(A_i^T u^k + \nabla r_i (x^k_i) \right)$
   \State First device updates: $z^{k+1/2} = z^k - \gamma (\nabla \ell(z^k, b) - y^k)$
   \State First device updates: $y^{k+1/2} = \tau y^k + (1 - \tau) u^k + \gamma (\sum_{i=1}^n A_i w^k_i - z^k)$
   \State First device sends $y^{k+1/2}$ to other devices 
   \State Random device $i_k$ sends $(A_{i_k} x^{k+1/2}_{i_k} - A_{i_k} w^{k}_{i_k})$ to first device
   \State All devices update: $x^{k+1}_i = \tau x^{k}_i + (1 - \tau) w^k_i - \gamma (A_i^T  y^{k+1/2} + \nabla r_i(x^{k+1/2}_i))$
   \State First device update: $z^{k+1} = z^k - \gamma (\nabla \ell(z^{k+1/2}, b) - y^{k+1/2})$
   \State First device update: $y^{k+1} = \tau y^k + (1 - \tau) u^k + \gamma ( n \cdot [A_{i_k} x^{k+1/2}_{i_k} - A_{i_k} w^{k}_{i_k}] + \sum_{i=1}^n  A_i w^{k}_i - z^{k+1/2})$
   \State Flip a coin $b_k \in \{0, 1\}$ where $\PP\{ b_k = 1 \} = p$ 
   \If{$b_k = 1$} 
   \State All devices in parallel update: $w^{k+1}_i = x^{k}_i$
   \State First device updates: $u^{k+1} = y^{k}$
   \State All devices send uncompressed $A_i w^{k+1}_i$ to first device
   \State First device sends uncompressed $u^{k+1}$ to other devices
   \Else
   \State All devices in parallel update: $w^{k+1}_i = w^{k}_i$
   \State First device updates: $u^{k+1} = u^{k}$
   \EndIf
   \EndFor
\end{algorithmic}
\end{algorithm}

Using the same reasonings as after Theorem \ref{th:EG_compressed_unbiased}, one can find the optimal choice of $p$. In Algorithm ~\ref{alg:EG_pp}, one mandatory communication round with only 1 client occurs and possibly one more (with all clients) with probability $p$. Then each iteration requires $\mathcal{O}\left( n^{-1} + p \right)$ data transfers on average.
The optimal choice of $p$ is $\beta^{-1}$.





\subsection{Coordinate modification for low-cost local computing}

The last modification is related to cheapening the cost of local computation in Algorithm \ref{alg:EG}. The most expensive local operations are matrix vector multiplications: $A_i x^k_i$ and $A_i^T y^k$. To make them cheaper, we can apply to the idea of coordinate descent \cite{doi:10.1137/100802001,doi:10.1137/16M1060182,richtarik2013optimal,doi:10.1080/10556788.2016.1190360} and compute not all coordinates for the resulting vectors $A_i x^k_i$ and $A_i^T y^k$ but only $1$, then instead of multiplying matrix by vector, we just compute the scalar product of two vectors. This is implemented in the following modification of Algorithm~\ref{alg:EG_quantization}.  Theorem \ref{th:EG_coord} gives the convergence, proof can be found in Appendix \ref{app:th:EG_coord}.

\begin{algorithm}{\texttt{EGVFL} with coordinate choice for (\ref{eq:vfl_lin_spp_1})}
   \label{alg:EG_coord}
\begin{algorithmic}[1]
   \State {\bfseries Input:} starting point $(x^0, z^0, y^0) \in \R^{d+2s}$, $(w^0, u^0) \in \R^{d+s}$, stepsize $\gamma > 0$, number of steps $K$
   \For{$k=0$ {\bfseries to} $K-1$}
   \State All devices in parallel update: $x^{k+1/2}_i = \tau x^{k}_i + (1 - \tau) w^k_i - \gamma \left(A_i^T u^k + \nabla r_i (x^k_i) \right)$
   \State First device updates: $z^{k+1/2} = z^k - \gamma (\nabla \ell(z^k, b) - y^k)$
   \State First device updates: $y^{k+1/2} = \tau y^k + (1 - \tau) u^k + \gamma (\sum_{i=1}^n A_i w^k_i - z^k)$
   \State First device sends  $y^{k+1/2}$ sends to other devices 
   \State All devices in parallel choice coordinate(s) $c^k_i$, computes $\langle A_i(x^{k+1/2}_i - w^{k}_i), e_{c^k_i} \rangle e_{c^k_i}$ and send to first device
   \State All devices choice coordinate(s) $j^k_i$ and update: $x^{k+1}_i = \tau x^{k}_i + (1 - \tau) w^k_i - \gamma ( d_i \cdot \langle A_i^T (y^{k+1/2} - u^k), e_{j^k_i} \rangle e_{j^k_i} + A_i^T u^k + \nabla r_i(x^{k+1/2}_i))$
   \State First device update: $z^{k+1} = z^k - \gamma (\nabla \ell(z^{k+1/2}, b) - y^{k+1/2})$
   \State First device update: $y^{k+1} = \tau y^k + (1 - \tau) u^k + \gamma (\sum_{i=1}^n [s \cdot \langle A_i(x^{k+1/2}_i - w^{k}_i), e_{c^k_i} \rangle e_{c^k_i} + A_i w^{k}_i ] - z^{k+1/2})$
   \State Flip a coin $b_k \in \{0, 1\}$ where $\PP\{ b_k = 1 \} = p$
   \If{$b_k = 1$} \label{lin_alg:coord_linws}
   \State All devices in parallel update: $w^{k+1}_i = x^{k}_i$
   \State First device updates: $u^{k+1} = y^{k}$
   \State All devices send uncompressed $A_i w^{k+1}_i$ to first device
   \State First device sends uncompressed $u^{k+1}$ to other devices
   \Else
   \State All devices in parallel update: $w^{k+1}_i = w^{k}_i$
   \State First device updates: $u^{k+1} = u^{k}$ \label{lin_alg:coord_linwf}
   \EndIf
   \EndFor
\end{algorithmic}
\end{algorithm}

{
\begin{theorem} \label{th:EG_coord}
Let Assumption \ref{as:convexity_smothness} hold. Let the problem (\ref{eq:vfl_lin_spp_1})
be solved by Algorithm~\ref{alg:EG_coord}. \\ Then for 
$
\gamma = \tfrac{1}{4}\min
 \bigg\{ 1;\frac{1}{L_r}; \frac{1}{L_{\ell}};\sqrt{\frac{1-\tau}{s  \lambda_{\max} (A^T A) }}; \sqrt{\frac{1-\tau}{d \max_{i} \{\lambda_{\max} (A_i^T A_i)\}}}\bigg\}
$
and 
$\tau = 1 - p$,
it holds that
{
\begin{align*}
\E{\text{gap}(\bar x^K, \bar z^K, \bar y^K)}
=  
    \mathcal{O}\left(
     \left[  \frac{s}{\sqrt{p}} \sqrt{\lambda_{\max} (A^T A)}  
     +
     \frac{d}{\sqrt{p}} \cdot \max\limits_{i = 1, \ldots, n} \{\sqrt{\lambda_{\max}(A_i A^T_i) } \}  + L_{\ell} + L_r  \right] \frac{D^2}{K}\right),
\end{align*}}where $\bar x^K$, $\bar z^K$, $\bar y^K$, $D^2$ are defined in Theorem \ref{th:EG_basic_1}.
\end{theorem}
}

Using the same reasonings as after Theorem ~\ref{th:EG_compressed_unbiased}, one can find the optimal choice of $p$. 
In Algorithm~\ref{alg:EG_coord}, two mandatory computing of scalar products (instead of matrix vector multiplication) take places and possibly two matrix vector multiplications with probability $p$. Then each iteration requires $\mathcal{O}\left( n + s + p \cdot ns\right)$ local computations on average.
The optimal choice of $p$ is $(n+s)/ (ns)$.

{In addition to computing efficiency aspects, privacy protection algorithms are of particular importance in federated learning. We address this issue in the next two sections.}

{

\subsection{Modification with additive noise for privacy protection}
\label{sec:new_app_0}

If we want to add extra privacy to basic Algorithm \ref{alg:EG}, we can first use the well-known and classic mechanism of adding noise to packages \cite{abadi2016deep}. Following this idea, we add noise to $y^k$, $A_i x^k_i$, $y^{k+1/2}$ and $A_i x^{k+1/2}_i$ in Algorithm \ref{alg:EG}. Namely, in lines \ref{lin_alg:EG_liny1_send}, \ref{lin_alg:EG_linx1_send}, \ref{lin_alg:EG_linx1}, \ref{lin_alg:EG_liny1}, \ref{lin_alg:EG_liny2_send}, \ref{lin_alg:EG_linx2_send}, \ref{lin_alg:EG_linx2} and \ref{lin_alg:EG_liny2} we change: $y^k \to y^k + \xi^k$, $A_i x^k_i \to A_i x^k_i + \xi_i^k$, $y^{k+1/2} \to y^{k+1/2}  + \xi^{k+1/2}$ and $A_i x^{k+1/2}_i \to A_i x^{k+1/2}_i + \xi_i^{k+1/2}$. 

\begin{algorithm}{\texttt{EGVFL} with additive noise for (\ref{eq:vfl_lin_spp_1})}
   \label{alg:EG_noise}
\begin{algorithmic}[1]
   \State {\bfseries Input:} starting point $(x^0, z^0, y^0) \in \R^{d+2s}$, stepsize $\gamma > 0$, number of steps $K$
   \For{$k=0$ {\bfseries to} $K-1$}
   \State First device generates noise $\xi^k$ and sends $y^k + \xi^k$ to other devices
   \State All devices in parallel generate noises $\xi^k_i$ and send $A_i x^k_i + \xi^k_i$ to first device
   \State All devices in parallel update: $x^{k+1/2}_i = x^k_i - \gamma \left(A_i^T (y^k + \xi^k )+ \nabla r_i (x^k_i) \right)$
   \State First device updates: $z^{k+1/2} = z^k - \gamma (\nabla \ell(z^k, b) - y^k)$
   \State First device updates: $y^{k+1/2} = y^k + \gamma (\sum_{i=1}^n (A_i x^k_i + \xi^k_i) - z^k)$
   \State First device generates noise $\xi^{k+1/2}$ and sends $y^{k+1/2} + \xi^{k+1/2}$ to other devices 
   \State All devices in parallel generate noises $\xi^{k+1/2}_i$ and send $A_i x^{k+1/2}_i + \xi^{k+1/2}_i$ to first device
   \State All devices in parallel update: $x^{k+1}_i = x^k_i - \gamma (A_i^T (y^{k+1/2} + \xi^{k+1/2})+ \nabla r_i(x^{k+1/2}_i))$
   \State First device: $z^{k+1} = z^k - \gamma (\nabla \ell(z^{k+1/2}, b) - y^{k+1/2})$
   \State First device: $y^{k+1} = y^k + \gamma (\sum_{i=1}^n (A_i x^{k+1/2}_i + \xi^{k+1/2}_i) - z^{k+1/2})$
   \EndFor
\end{algorithmic}
\end{algorithm}

Theorem \ref{th:EG_add_noise} gives the convergence, proof can be found in Appendix \ref{app:th:EG_add_noise}.
{
\begin{theorem} \label{th:EG_add_noise}
Let Assumption \ref{as:convexity_smothness} hold. Let $\mathbb{E}[\xi^k] = 0$, $\mathbb{E}[\xi^k_i] = 0$, $\mathbb{E}[\|\xi^k\|^2] \leq \sigma^2$ and $\mathbb{E}[\|\xi^k_i\|^2] \leq \sigma^2$ (the same for $\xi^{k+1/2}$ and $\xi^{k+1/2}_i$). Let the problem (\ref{eq:vfl_lin_spp_1})
be solved by Algorithm~\ref{alg:EG_noise}. Then for 
$
\gamma = \min \bigg\{ \tfrac{1}{2}; \tfrac{1}{\sqrt{8\lambda_{\max}(A^T A)}}; \tfrac{1}{2L_r}; $ $ \tfrac{1}{2L_{\ell}}; \sqrt{\tfrac{D^2}{8(\lambda_{\max}(A A^T) + n)\sigma^2 K}}  \bigg\},
$ 
it holds that
\begin{align*}
    \mathbb{E} \big[ \text{gap} (\bar x^K, \bar z^K, \bar y^K)\big]
    =
    \cO \left(\frac{ (1 + \sqrt{\lambda_{\max}(A^T A)} + L_r + L_{\ell} ) D^2}{ K}
    \!+\!
    \sqrt{\frac{(\lambda_{\max}(A A^T) + n)\sigma^2 D^2}{K}} \right),
\end{align*}
where $\bar x^K$, $\bar z^K$, $\bar y^K$, $D^2$ are defined in Theorem \ref{th:EG_basic_1}.
\end{theorem}
}

The pros and cons of this approach are known. Noise makes it difficult to use the forwarded information for adversarial purposes, but the convergence of the method also suffers - the higher the noise, the worse the solution we find, or the slower the convergence we have. In particular, in Theorem \ref{th:EG_add_noise} we have the rate $\mathcal{O} (  {1}/{\sqrt{K}} )$ instead of $\mathcal{O} (  {1}/{K} )$ from Theorem \ref{th:EG_basic_1}.

\subsection{Homomorphic encryption modification for coded communication}
\label{sec:new_app_1}

If even more privacy is required compared to adding noise, then encryption techniques can be applied.
{
\begin{definition} \label{def:hom_encr}
Pair of encoding $E$ and decoding $D$ operators is called a homomorphic encryption if for any $n \in \mathbb{N}$, $x_i \in \R$, $\alpha_i \in \R$ it holds
$$
    D \left( \sum_{i=1}^n \alpha_i E(x_i)\right) = \sum_{i=1}^n \alpha_i x_i.
$$
\end{definition}
}

\begin{algorithm}{EG with homomorphic encrypt. for (\ref{eq:vfl_lin_spp_1})}
   \label{alg:EG_code}
\begin{algorithmic}[1]
   \State {\bfseries Input:} starting point $(x^0, z^0, y^0) \in \R^{d+2s}$, stepsize $\gamma > 0$, number of steps $K$
   \For{$k=0$ {\bfseries to} $K-1$}
   \State 1st uses encryption $E(y^{k})$, $E(x_i^k)$ and sends to other devices
   \State All send $A_i^T E(y^{k})$, $A_i E(x_i^k)$ to first device
   \State 1st decodes: $D(A_i^T E(y^{k}))= A_i^T y^{k}$, $D(A_i E(x^{k}_i))= A_i x^{k}_i$
   \State 1st updates: $x^{k+1/2}_i = x^k_i - \gamma ( A^T_i y^k + \nabla r_i (x^k_i)) $
   \State 1st updates: $z^{k+1/2} = z^k - \gamma (\nabla \ell(z^k, b) - y^k)$
   \State 1st updates: $y^{k+1/2} = y^k + \gamma (\sum_{i=1}^n A_i x^k_i - z^k)$
   \State 1st uses encryption $E(y^{k+1/2})$, $E(x_i^{k+1/2})$ and sends to other devices
   \State All send $A_i^T E(y^{k+1/2})$, $A_i E(x_i^{k+1/2})$ to 1st
   \State 1st decodes: $D(A_i^T E(y^{k+1/2}))= A_i^T y^{k+1/2}$, $D(A_i E(x^{k+1/2}_i))= A_i x^{k+1/2}_i$
   \State 1st updates: {\small $x^{k+1}_i = x^k_i - \gamma ( A^T_i y^{k+1/2} + \nabla r_i (x^{k+1/2}_i)) $}
   \State 1st updates: {\small $z^{k+1} = z^k - \gamma (\nabla \ell(z^{k+1/2}, b) - y^{k+1/2})$}
   \State 1st updates: {\small $y^{k+1} = y^k + \gamma (\sum_{i=1}^n A_i x^{k+1/2}_i - z^{k+1/2})$}
   \EndFor
\end{algorithmic}
\end{algorithm}

This definition is not consistent with the classical definition of homomorphic encryption \cite{acar2018survey}. But most classical encryption operators satisfy Definition \ref{def:hom_encr}. In our modification we focus on protecting the labels that are stored on the first device. Only the first device has the decryption key. Therefore, we move the updates $x_i$ from the local device to the first device and encrypt everything sent from the first device, and all other devices only perform multiplication on matrices $A_i$. 
This operation is valid for working with encoded information; it produces a linear combination message. According to Definition \ref{def:hom_encr}, this type of information can be decoded.
This scheme is implemented in Algorithm~\ref{alg:EG_code}. In terms of $x$, $z$ and $y$ updates, Algorithm~\ref{alg:EG_code} is the same as Algorithm~\ref{alg:EG}, thus Theorem~\ref{th:EG_basic_1} is valid. The main strength is obviously security, but using encryption greatly slows down the learning process because encoding and decoding is computationally expensive.

}

\section{Family of Reformulations}\label{sec:reform}

\subsection{Reformulation with additional variables}

Let us discuss other reformulations beyond (\ref{eq:vfl_lin_spp_1}), e.g., a reformulation with additional variables. 

In the formulation (\ref{eq:vfl_constr_1}), instead of $Ax = z$, we can introduce constraints in a different way  with variables $z_i \in \R^s,
$  for $i \in \{1,2,\dots,n\}$ as follows
\begin{align*}
        \min_{x \in \R^d}
        \min_{z \in \R^s} \left[\ell\left(\sum_{i=1}^n z_i,b\right) + \sum_{i=1}^n r_i(x_i) \right],~~
        \text{s.t.} ~~ A_i x_i = z_i ~~ \text{for}~~ i = 1, \ldots, n.
\end{align*}

The expression in the form of a Lagrangian function is
\begin{align}
    \label{eq:vfl_lin_spp_2}
    \min_{(x, z) \in \R^{d + sn}} \max_{y \in \R^{sn}}  \left[ \tilde L (x,z,y) := \ell \left(\sum_{i=1}^n z_i,b \right)
    + \sum_{i=1}^n r_i(x_i) +  \sum_{i=1}^n y^T_i  ( A_i x_i - z_i ) \right].
\end{align}

This saddle can be also solved using \texttt{ExtraGradient}.

\begin{algorithm}{\texttt{EGVFL} for (\ref{eq:vfl_lin_spp_2})}
   \label{alg:EG_basic_2}
\begin{algorithmic}[1]
   \State {\bfseries Input:} starting point $(x^0, z^0, y^0) \in \R^{d+2s}$, stepsize $\gamma > 0$, number of steps $K$
   \For{$k=0$ {\bfseries to} $K-1$}
   \State First device sends $y^k_i$ to other devices 
   \State All devices in parallel send $A_i x^k_i$ to first device
   \State All devices in parallel update: $x^{k+1/2}_i = x^k_i - \gamma (A_i^T y^k_i + \nabla r_i (x_i^k))$
   \State First device updates: $z^{k+1/2}_i = z^k_i - \gamma (\nabla \ell(\sum_{i=1}^n z^k_i, b) - y^k_i)$
   \State First device updates: $y^{k+1/2}_i = y^k_i + \gamma (A_i x^k_i - z^k_i)$
   \State First device sends $y^{k+1/2}_i$ to other devices 
   \State All devices in parallel send $A_i x^{k+1/2}_i$ to first device
   \State All devices in parallel update: $x^{k+1}_i = x^k_i - \gamma (A_i^T y^{k+1/2}_i + \nabla r_i (x_i^{k+1/2}))$
   \State First device updates: $z^{k+1}_i = z^k_i - \gamma (\nabla \ell(\sum_{i=1}^n z^{k+1/2}_i, b) - y^{k+1/2}_i)$
   \State First device updates: $y^{k+1}_i = y^k_i + \gamma (A_i x^{k+1/2}_i - z^{k+1/2}_i)$
   \EndFor
\end{algorithmic}
\end{algorithm}

{
\begin{theorem}\label{th:EG_basic_2}
Let Assumption \ref{as:convexity_smothness} hold. Let the problem (\ref{eq:vfl_lin_spp_2})
be solved by Algorithm~\ref{alg:EG_basic_2}.\\ Then for 
$
\gamma = \frac{1}{2} \min\left\{ 1; \frac{1}{\sqrt{\max_i\{ \lambda_{\max}(A_i^T A_i) \}}}; \frac{1}{L_r}; \frac{1}{n L_{\ell}}\right\},
$ 
it holds that 
\begin{align*} 
    \text{gap}_1(\bar x^{K}, \bar z^K, \bar y^K)
    = \mathcal{O}
     \left(\frac{ \left( 1+ \sqrt{\max_{i = 1, \ldots, n}\{ \lambda_{\max}(A_i^T A_i) \}} + nL_{\ell} + L_r \right)\tilde D^2}{K}  \right),
\end{align*}
where $\text{gap}_1(x,y,z) := \max_{\tilde y_i \in \mathcal{ \tilde Y}} \tilde L(x,z, \tilde y) - \min_{\tilde x, \tilde z \in \cX, \mathcal{\tilde Z}} \tilde L(\tilde x, \tilde z, y)$, $\bar x^K := \tfrac{1}{K}\sum_{k=0}^{K-1} x^{k+1/2}$, \\ 
$\bar z^K_i := \tfrac{1}{K}\sum_{k=0}^{K-1} z^{k+1/2}_i$, $\bar y^K_i := \tfrac{1}{K}\sum_{k=0}^{K-1} y^{k+1/2}_i$ and \\ $\tilde D^2 := \max_{x, z, y \in \cX, \cZ^n, \cY^n} \left[ \|x^0 - x \|^2 + \|z^0 - z \|^2 + \|y^0 - y \|^2 \right]$.
\end{theorem}
}

An important detail to note it is that the step $\gamma$ in Theorem \ref{th:EG_basic_2} depends on $\lambda_{\max} (A_i A_i^T)$. 
Previous algorithms assumed knowledge of the estimate for $\lambda_{\max} (A A^T)$ which can be disadvantageous because we cannot collect $A$ on a single device, and estimating $\lambda_{\max} (A A^T)$ through $\lambda_{\max} (A_i A_i^T)$ can give deplorable results.

{We compare these results with Theorem \ref{th:EG_basic_1}. Note that, on the one hand, according to Lemma \ref{lem:matrix}: $\lambda_{\max} (A A^T) \leq n \cdot \max_{i} \lambda_{\max} (A_i A_i^T)$, and this estimate achieves equality. But on the other hand, the vectors $y$ and $z$ in the formulation \eqref{eq:vfl_lin_spp_2} are of size $sn$, and in the problem \eqref{eq:vfl_lin_spp_1} only $s$. This means that $\tilde D$ in Theorem \ref{th:EG_basic_2} can be significantly larger than the diameter $D$ in Theorem \ref{th:EG_basic_1}. It turns out that the results of Theorem \ref{th:EG_basic_2} cannot be said to be superior to Theorem \ref{th:EG_basic_1}, but \eqref{eq:vfl_lin_spp_2} and Algorithm \ref{alg:EG_basic_2} can be considered as an alternative to the formulation \eqref{eq:vfl_lin_spp_1} and Algorithm \ref{alg:EG} with tuning according to Theorem \ref{th:EG_basic_1}.}


\subsection{Reformulation with augmentation} \label{app:aug}

Let us consider the augmented version of (\ref{eq:vfl_lin_spp_1}):
\begin{equation}
    \begin{split}
    \label{eq:vfl_lin_spp_aug}
    \min\limits_{(x, z) \in \R^{d + s}}~ \max\limits_{y \in \R^s} \Bigg[L_{\text{aug}} (x,z,y) :=  \ell\left(z,b\right) + \sum_{i=1}^n r_i(x_i) + y^T \left(\sum_{i=1}^n A_i x_i - z\right) 
    + \frac{\rho}{2} \| \sum_{i=1}^n A_i x_i - z\|^2\Bigg],
    \end{split}
\end{equation}
where $\rho \geq 0$. The statement (\ref{eq:vfl_lin_spp_aug}) is classical and is considered in \cite{boyd2011distributed}. The saddle point problem (\ref{eq:vfl_lin_spp_aug}) can also be solved using the \texttt{ExtraGradient} technique. 

\begin{algorithm}{\texttt{EGVFL} for (\ref{eq:vfl_lin_spp_aug})}
   \label{alg:EG_aug}
\begin{algorithmic}[1]
   \State {\bfseries Input:} starting point $(x^0, z^0, y^0) \in \R^{d+2s}$, stepsize $\gamma > 0$, regularizer $\rho$, number of steps $K$
   \For{$k=0$ {\bfseries to} $K-1$}
   \State All devices in parallel send $A_i x^k_i$ to first device \label{lin_alg:EG_aug_linx1_send}
   \State First device sends $y^k$ and $\sum_{i=1}^n A_i x^{k}_i - z^k$ to other devices  \label{lin_alg:EG_aug_liny1_send}
   \State All devices in parallel update: \label{lin_alg:EG_aug_linx1}
   \Statex \hspace{0.45cm} $x^{k+1/2}_i = x^k_i - \gamma \left(A_i^T y^k + \nabla r_i (x^k_i) + \rho A_i^T (\sum_{i=1}^n A_i x^{k}_i - z^k)\right)$ 
   \State First device updates: $z^{k+1/2} = z^k - \gamma \left(\nabla \ell(z^k, b) - y^k + \rho (z^k - \sum_{i=1}^n A_i x^{k}_i)\right)$ \label{lin_alg:EG_aug_linz1}
   \State First device updates: $y^{k+1/2} = y^k + \gamma (\sum_{i=1}^n A_i x^k_i - z^k)$ \label{lin_alg:EG_aug_liny1}
   \State All devices send $ A_i x^{k+1/2}_i$ to first device  \label{lin_alg:EG_aug_linx2_send}
   \State First device sends $y^{k+1/2}$ and $\sum_{i=1}^n A_i x^{k+1/2}_i - z^{k+1/2}$ to other devices \label{lin_alg:EG_aug_liny2_send}
   \State All devices in parallel update:  \label{lin_alg:EG_aug_linx2}
   \Statex \hspace{0.45cm} $x^{k+1}_i = x^k_i - \gamma (A_i^T y^{k+1/2} + \nabla r_i(x^{k+1/2}_i) + \rho A_i^T (\sum_{i=1}^n A_i x^{k}_i - z^k) )$
   \State First device: $z^{k+1} = z^k - \gamma (\nabla \ell(z^{k+1/2}, b) - y^{k+1/2} + \rho (z^k - \sum_{i=1}^n A_i x^{k}_i))$ \label{lin_alg:EG_aug_linz2}
   \State First device: $y^{k+1} = y^k + \gamma (\sum_{i=1}^n A_i x^{k+1/2}_i - z^{k+1/2})$ \label{lin_alg:EG_aug_liny2}
   \EndFor
\end{algorithmic}
\end{algorithm}

{
\begin{theorem}\label{th:EG_aug}
Let Assumption \ref{as:convexity_smothness} hold. Let the problem (\ref{eq:vfl_lin_spp_aug}) be solved by Algorithm \ref{alg:EG_aug}. Then for 
$$
\gamma = \frac{1}{4} \cdot \min\left\{ 1; \frac{1}{\rho}; \frac{1}{\sqrt{\lambda_{\max}(A^T A)}}; \frac{1}{\sqrt{\rho\lambda_{\max}(A^T A)}}; \frac{1}{\rho\lambda_{\max}(A^T A)}; \frac{1}{L_r}; \frac{1}{L_{\ell}}\right\},
$$
it holds that
\begin{align*}
    \text{gap}_{\text{aug}} (\bar x^{K}, \bar z^K, \bar y^K)
    =
    \mathcal{O} \left(\frac{ \left( 1 + \rho + \sqrt{(1 + \rho)\lambda_{\max}(A^T A)} + \rho \lambda_{\max}(A^T A) + L_{\ell} + L_r\right) D^2}{K} \right),
\end{align*}
where $\text{gap}_{\text{aug}}(x,z,y) := \max_{\tilde y \in \cY} L_{\text{aug}}(x, z, \tilde y) - \min_{\tilde x, \tilde z \in \cX, \cZ} L_{\text{aug}}(\tilde x, \tilde z, y)$ and \\ $\bar x^K := \tfrac{1}{K}\sum_{k=0}^{K-1} x^{k+1/2}$, $\bar z^K := \tfrac{1}{K}\sum_{k=0}^{K-1} z^{k+1/2}$, $\bar y^K := \tfrac{1}{K}\sum_{k=0}^{K-1} y^{k+1/2}$ and $D^2 := \max_{x, z, y \in \cX, \cZ, \cY} [ \|x^0 - x \|^2 +  \|z^0 - z \|^2 + \|y^0 - y \|^2 ]$.
\end{theorem}
}

The proof is postponed to Appendix \ref{app:th:EG_aug}. The results of Theorem \ref{th:EG_aug} are no better than Theorem \ref{th:EG_basic_1}, and in the case of large $\rho$ are even worse. Based on these guarantees (and they seem reasonable to us) the use of augmentation with \texttt{ExtraGradeint} in the theory does not give bonuses.

\subsection{Reformulation with dual loss}

The definition of the dual function gives $\ell^* (y, b) = \max_{x \in \R^s} \{ \langle z , y\rangle - \ell (z, b) \}$. With small reformulation and $z = Ax$, we get that $\ell (Ax, b) = \max_{y \in \R^s} \{ \langle y , Ax \rangle -  \ell^* (y, b)\}$. Then, one can rewrite (\ref{eq:main_problem}) as follows, 
\begin{equation}
    \label{eq:vfl_lin_spp_3}
    \textstyle{
    \min\limits_{x \in \R^d} \max\limits_{y \in \R^s} \hat L(x,y) := \left[ -l^*\left(y,b\right) + \sum\limits_{i=1}^n r_i(x_i) + y^T \left(\sum_{i=1}^n A_i x_i \right)\right].
    }
\end{equation}

The statement (\ref{eq:vfl_lin_spp_3}) is simpler than (\ref{eq:vfl_lin_spp_1}), since it does not contain additional variables $z$, but it requires the existence of a dual function for $\ell$.
The saddle point problem (\ref{eq:vfl_lin_spp_3}) can also be solved using the \texttt{ExtraGradient} technique. 

\begin{algorithm}{\texttt{EGVFL} for (\ref{eq:vfl_lin_spp_3})}
   \label{alg:EG_basic_3}
\begin{algorithmic}[1]
   \State {\bfseries Input:} starting point $(x^0, y^0) \in \R^{d+s}$, stepsize $\gamma > 0$, number of steps $K$
   \For{$k=0$ {\bfseries to} $K-1$}
   \State First device sends $y^k$ to other devices
   \State All devices in parallel send $A_i x^k_i$ to first device
   \State All devices in parallel update: $x^{k+1/2}_i = x^k_i - \gamma (A_i^T y^k + \nabla r_i (x^k_i))$
   \State First device updates: $y^{k+1/2} = y^k - \gamma ( \nabla \ell^*(y^k, b) - \sum_{i=1}^n A_i x^k_i )$
   \State First device sends $y^{k+1/2}$ to other devices 
   \State All devices in parallel send $A_i x^{k+1/2}_i$ to first device
   \State All devices in parallel update: $x^{k+1}_i = x^k_i - \gamma (A_i^T y^{k+1/2} + \nabla r_i (x^{k+1/2}_i))$
   \State First device updates: $y^{k+1} = y^k - \gamma ( \nabla \ell^*(y^{k+1/2}, b) - \sum_{i=1}^n A_i x^{k+1/2}_i)$
   \EndFor
\end{algorithmic}
\end{algorithm}

{
\begin{theorem}\label{th:EG_basic_3}
Let $l^*$ be $L_{\ell^*}$-smooth and convex, $r$ be $L_r$-smooth and convex. Let the problem (\ref{eq:vfl_lin_spp_3}) be solved by Algorithm \ref{alg:EG_basic_3}. Then for 
$$
\gamma = \frac{1}{2} \cdot \min\left\{ 1; \frac{1}{\sqrt{\lambda_{\max}(A^T A)}}; \frac{1}{L_r}; \frac{1}{L_{\ell^*}}\right\},
$$
it holds that
\begin{align*}
    \text{gap}_2(\bar x^{K}, \bar y^K)
    =
    \mathcal{O} \left(\frac{ \left( 1 + \sqrt{\lambda_{\max}(A^T A)} + L_{\ell^*} + L_r\right) \hat D^2}{K} \right),
\end{align*}
where $\text{gap}_2(x,y) := \max_{\tilde y \in \cY} \hat L(x, \tilde y) - \min_{\tilde x \in \cX} \hat L(\tilde x, y)$ and $\bar x^K := \tfrac{1}{K}\sum_{k=0}^{K-1} x^{k+1/2}$, $\bar y^K := \tfrac{1}{K}\sum_{k=0}^{K-1} y^{k+1/2}$ and $\hat D^2 := \max_{x, y \in \cX, \cY} \left[ \|x^0 - x \|^2 + \|y^0 - y \|^2 \right]$.
\end{theorem}
}

The proof is postponed to Appendix \ref{app:th:EG_basic_3}.

\subsection{Reformulation with dual loss and regularizer}

If we introduce dual functions for both $\ell$ and $r$, then \eqref{eq:vfl_lin_spp_1} can be rewritten as follows 
\begin{align}
\label{eq:dual_problem_separated}
\textstyle{
	\max_{y \in \R^s} ~ \left[-\sum_{i=1}^n r_i^*(-A_i^T y) - \ell^*(y, b)\right].
 }
\end{align}

To prove it, we start from (\ref{eq:vfl_lin_spp_1})
\begin{align*}
	\min_{(x, z) \in \R^{d + s}} \max_{y \in \R^s} \quad &  \ell \left(z,b\right) + r(x) + y^T \left(Ax - z\right)
 \\
	&= \max_{y \in \R^s} \left[\min_{(x, z) \in \R^{d + s}} \left[\left(- \langle z, y \rangle + \ell(z, b)\right) + \left( \langle  Ax , y \rangle + r(x)  \right)\right]\right] \\
	&= \max_{y \in \R^s} \left[-\max_{z \in \R^{s}}\left(\langle z, y \rangle  - \ell(z, b)\right) - \max_{x \in \R^{d}} \left(\langle -A^T y, x \rangle - r(x)\right)\right].
\end{align*}

Definitions of dual functions: $\ell^* (y, b) = \max_{z \in \R^s} \{ \langle z , y\rangle - \ell (z, b) \}$ and 
\\
$r^* (-A^T y) = \max_{x \in \R^s} \{ \langle -A^T y , x\rangle - r (x) \}$, give
\begin{align*}
	\max_{y \in \R^s} ~ \left[ -\ell^*(y, b) - r^*(-A^T y) \right].
\end{align*}

Due to the separability of $r$, its conjugate is also separable. Hence, we have (\ref{eq:dual_problem_separated}).

In the fact (\ref{eq:dual_problem_separated}) is the maximization of a concave function, which is very close to the original formulation (\ref{eq:main_problem}). This problem can be solved by distributed variants of \texttt{GD} and not only.

\section{Extension to Non-Convex Models}\label{sec:extension}

Let us consider a more general formulation where we can use arbitrary functions/models $g_i (A_i, w_i): \R^{d_{w_i}} \to \R^{s \times d_i}$ with weights/tuning variables $w_i \in \R^{d_{w_i}}$ instead of fixed data matrices $A_i$ (\ref{eq:main_problem_vfl}):
$$ 
        \min_{(x, w) \in \R^{d + d_w}}   \left[\ell\left(\sum_{i=1}^{n} g_i (A_i, w_i) x_i,b\right) + \sum_{i=1}^n r_i(x_i) \right],
$$

Here, the analogue of the Lagrangian function (\ref{eq:vfl_lin_spp_1}) can be written as follows:
\begin{equation}
    \label{eq:vfl_extension_spp_1}
\begin{split}
    \min_{(x, w, z) \in \R^{d + d_w + s} }\max_{y \in \R^s} \left[ \ell\left(z,b\right)
    + \sum_{i=1}^n r_i(x_i) + y^T \left(\sum_{i=1}^n g_i(A_i, w_i) x_i - z\right)\right].
\end{split}
\end{equation}

This SPP is generally not convex-concave, but can be solved by the modified version of Algorithm \ref{alg:EG}. 

\begin{algorithm}{\texttt{EGVFL} for (\ref{eq:vfl_extension_spp_1})}
   \label{alg:EG_extension}
\begin{algorithmic}[1]
   \State {\bfseries Input:} starting point $(x^0, w^0, z^0, y^0) \in \R^{d + d_w +2s}$, stepsize $\gamma > 0$, number of steps $K$
   \For{$k=0$ {\bfseries to} $K-1$}
   \State First device sends $y^k$ to other devices  
   \State All devices in parallel send $g_i (A_i, w^k_i) x^k_i$ to first device 
   \State All devices in parallel update: $x^{k+1/2}_i = x^k_i - \gamma \left(g_i^T (A_i, w^k_i) y^k + \nabla r_i (x^k_i) \right)$ 
   \State All devices in parallel update: 
   $w^{k+1/2}_i = w^{k}_i - \gamma \left( (y^k)^T \nabla g_i (A_i, w^k_i) x^k_i \right)$
   \State First device updates: $z^{k+1/2} = z^k - \gamma (\nabla \ell(z^k, b) - y^k)$ 
   \State First device updates: $y^{k+1/2} = y^k + \gamma \left(\sum_{i=1}^n g_i(A_i, w^k_i) x^k_i - z^k\right)$ 
   \State First device sends $y^{k+1/2}$ to other devices
   \State All devices in parallel send $g_i (A_i, w^{k+1/2}_i) x^{k+1/2}_i$ to first device 
   \State All devices in parallel update: $x^{k+1}_i = x^k_i - \gamma \left(g_i^T (A_i, w^{k+1/2}_i) y^{k+1/2} + \nabla r_i (x^{k+1/2}_i) \right)$ 
   \State All devices in parallel update: 
   $w^{k+1}_i = w^{k}_i - \gamma \left( (y^{k+1/2})^T \nabla g_i (A_i, w^{k+1/2}_i) x^{k+1/2}_i \right)$
   \State First device updates: $z^{k+1} = z^k - \gamma (\nabla \ell(z^{k+1/2}, b) - y^{k+1/2})$
   \State First device updates: $y^{k+1} = y^k + \gamma \left(\sum_{i=1}^n g_i(A_i, w^{k+1/2}_i) x^{k+1/2}_i - z^{k+1/2}\right)$ 
   \EndFor
\end{algorithmic}
\end{algorithm}

\section{Experiments}\label{sec:exp}

\textbf{Regression.} We conduct experiments on the linear regression problem: 
$$
\min_{x \in \R^d} \left[f(x) = \frac{1}{2}\|Ax - b \|^2 + \lambda \| x \|^2\right].
$$

Here, the smoothness constant of gradients is $L = \lambda_{\max}(AA^T) + \lambda$ with $\lambda = \lambda_{\max}(AA^T) / 10^3$. 
Other smoothness constants, which we use in theory for our method, are $L_{\ell} = 1$, $L_{r} = \lambda$. 
We take \texttt{mushrooms}, \texttt{a9a}, \texttt{w8a} and \texttt{MNIST} datasets from LibSVM library \cite{chang2011libsvm}. We vertically (by features) uniformly divide the dataset between 5 devices. 

This experiment uses different formulations to compare deterministic methods for solving the VFL problem. Here, we're not focusing on the distributed nature of the problems; instead, we aim to show that the saddle point reformulation using the classical Lagrangian function has merit (we investigate modifications in Appendix \ref{app:exp})
and 
methods for solving it can compete effectively with other approaches. 

Since there are two formulations of VFL, classical minimization and saddle point, we choose several methods for each formulation. For the minimization formulation, we take \texttt{GD} as the most popular method, and \texttt{AGD} \cite{nesterov2003introductory} as the theoretically unimprovable first-order method for smooth convex problems. For the saddle point formulation, we consider \texttt{ADMM} and Algorithm~\ref{alg:EG}. 
The methods are tuned according to the corresponding theory. For \texttt{GD} we choose step as $\tfrac{1}{L}$ \cite{book}, for \texttt{Nesterov} --  step as $\tfrac{1}{L}$ and momentum as $\tfrac{\sqrt{L} - \sqrt{\lambda}}{\sqrt{L} + \sqrt{\lambda}}$ \cite{nesterov2003introductory}, for \texttt{ADMM} we take regularizer parameter equal to $\tfrac{1}{\sqrt{\lambda_{\max}(AA^T)}}$ \cite{lu2023unified}. Algorithm \ref{alg:EG} is tuned according to Theorem~\ref{th:EG_basic_1} with and without $\beta$-trick (see disscusion after Corollary \ref{cor:EG_basic_1}). The application of the $\beta$-trick can also be considered for other methods. However, in the case of \texttt{GD} and \texttt{Nesterov}, it does not alter the method since the data matrix $A$ and the loss function $\ell$ are not split.
All methods start from zero. 

\begin{figure}
\begin{minipage}{0.49\textwidth}
  \centering
\includegraphics[width =  \textwidth ]{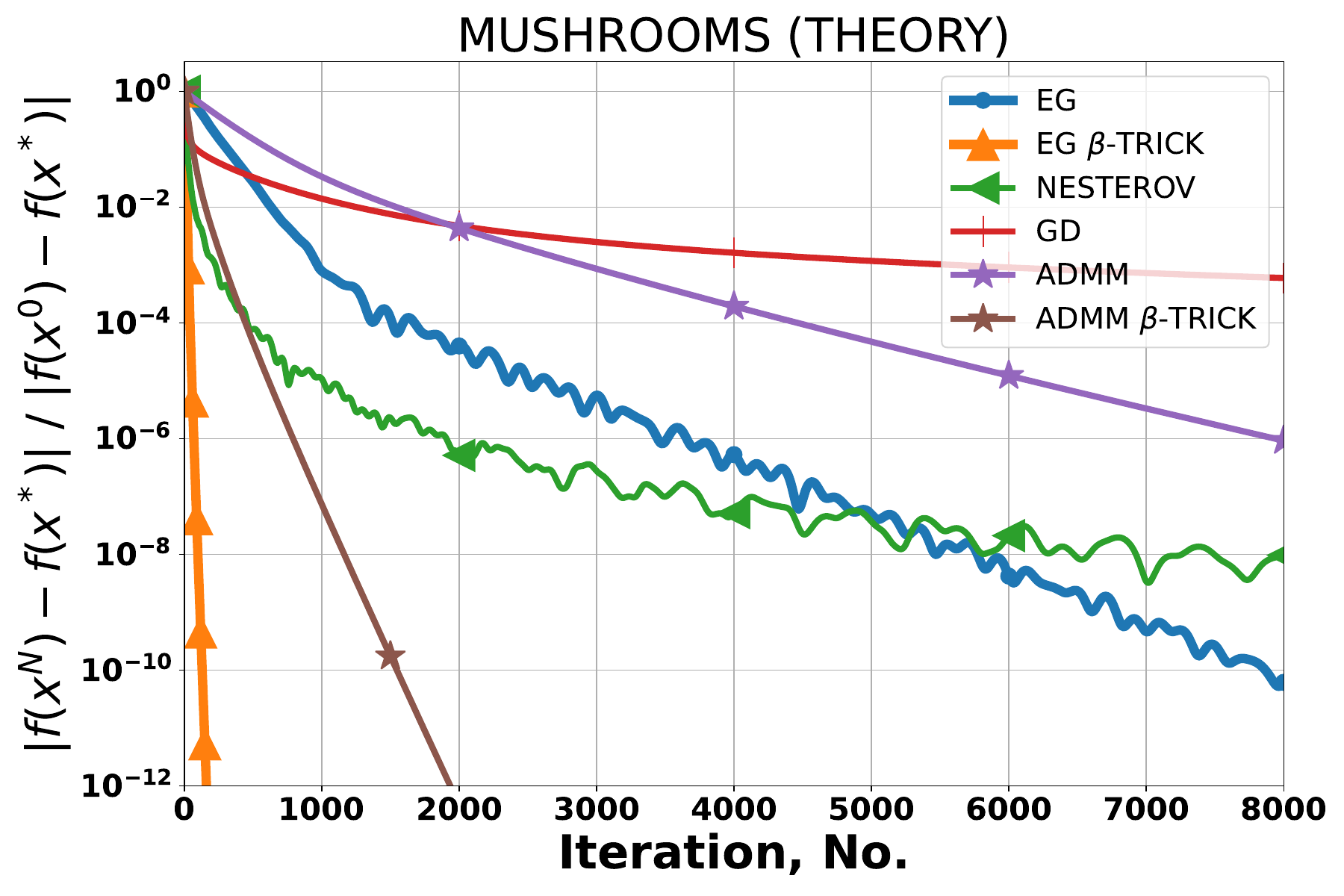}
\end{minipage}%
\begin{minipage}{0.49\textwidth}
  \centering
\includegraphics[width =  \textwidth ]{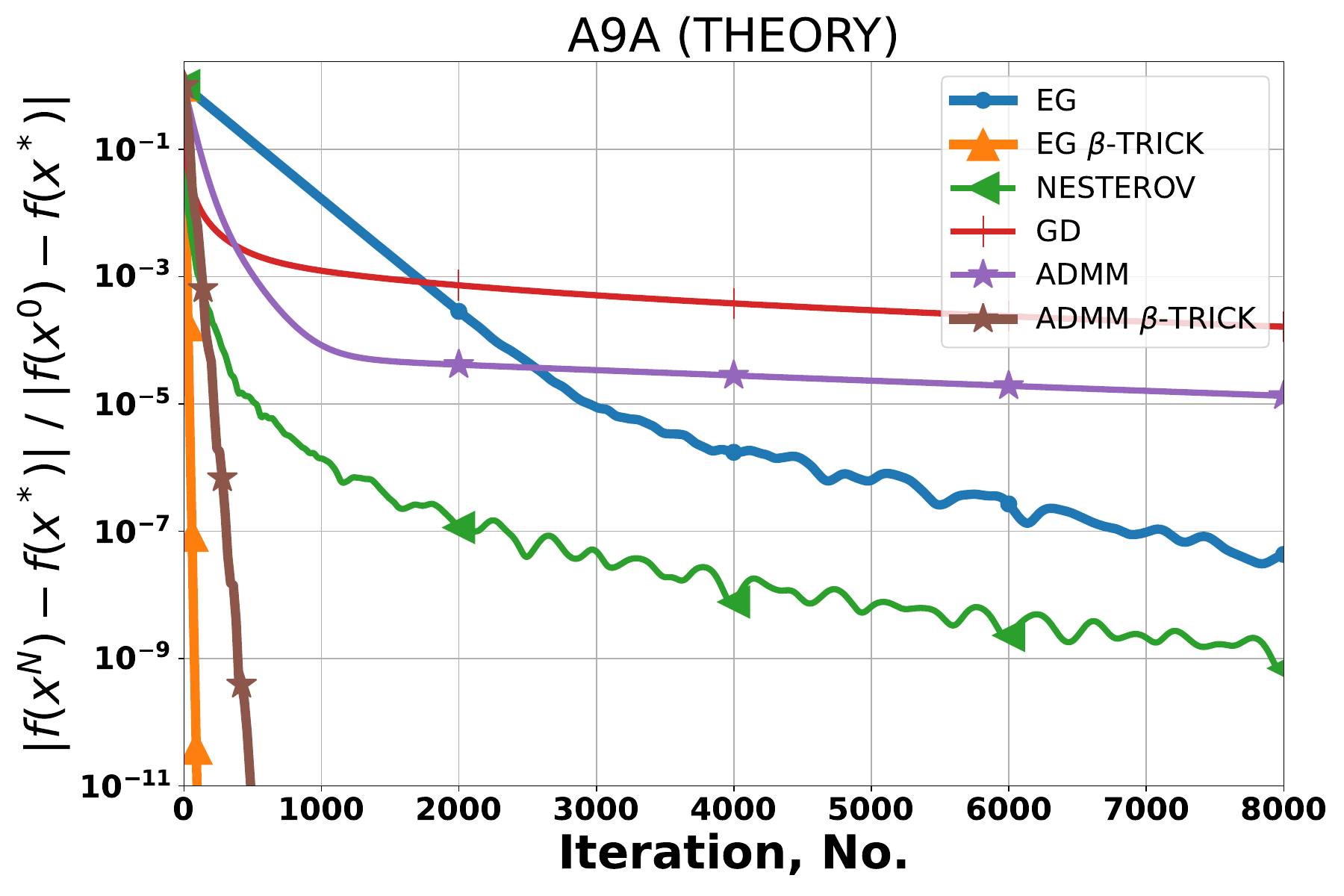}
\end{minipage}%
\\
\begin{minipage}{0.01\textwidth}
\quad
\end{minipage}%
\begin{minipage}{0.49\textwidth}
  \centering
(a) \texttt{mushrooms}
\end{minipage}%
\begin{minipage}{0.49\textwidth}
\centering
 (b) \texttt{a9a}
\end{minipage}%
\\
\begin{minipage}{0.49\textwidth}
  \centering
\includegraphics[width =  \textwidth ]{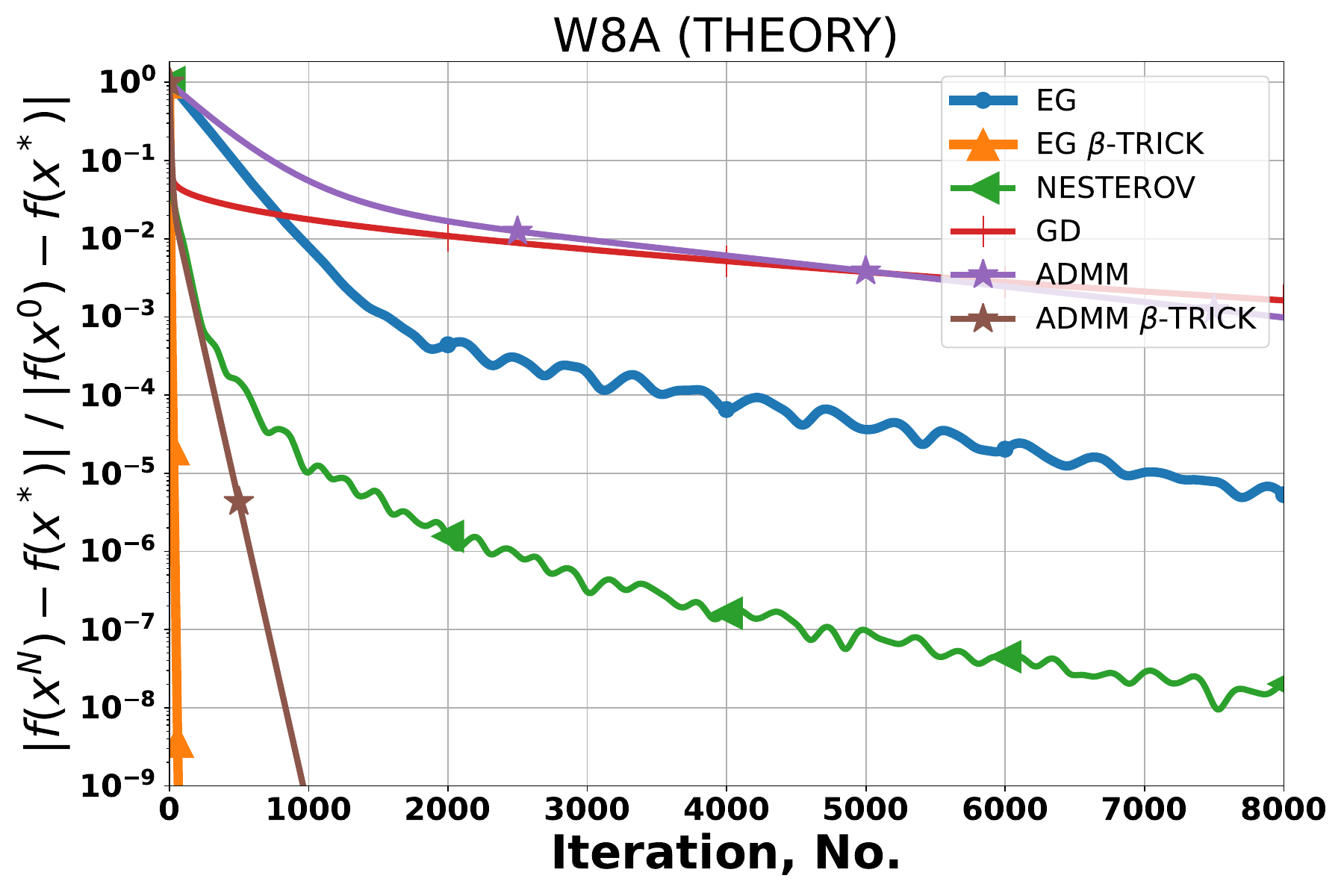}
\end{minipage}%
\begin{minipage}{0.49\textwidth}
  \centering
\includegraphics[width =  \textwidth ]{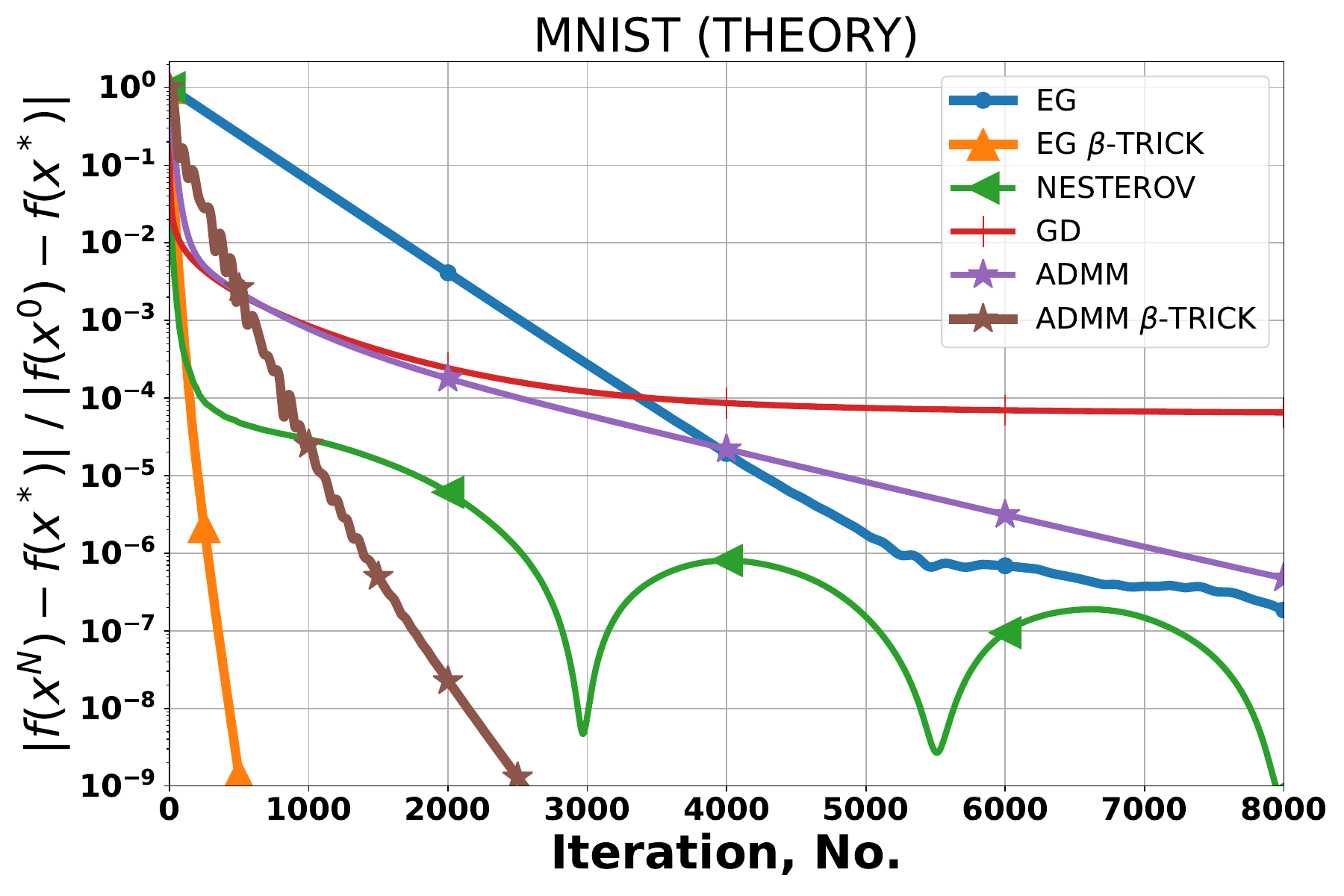}
\end{minipage}%
\\
\begin{minipage}{0.49\textwidth}
\centering
  (c) \texttt{w8a}
\end{minipage}%
\begin{minipage}{0.49\textwidth}
\centering
  (d) \texttt{MNIST}
\end{minipage}%
\caption{
Comparison of methods for solving the VFL problem in different formulations: minimization (\texttt{GD}, \texttt{Nesterov}) and saddle point (\texttt{ADMM}, \texttt{ExtraGradient}/Algorithm \ref{alg:EG}). The comparison is made on LibSVM datasets \texttt{mushrooms}, \texttt{a9a}, \texttt{w8a} and \texttt{MNIST}.}
    \label{fig:comparison1}
\end{figure}

The results, illustrated in Figure \ref{fig:comparison1}, show that Algorithm~\ref{alg:EG} with the choice $\beta$ convergence dramatically faster, the basic version Algorithm~\ref{alg:EG} initially lags behind \texttt{GD} and \texttt{Nesterov}, but in terms of steady-state convergence Algorithm~\ref{alg:EG} converges faster and eventually surpasses both \texttt{GD} and sometimes \texttt{Nesterov}. 

{We also note that our methods exhibit linear convergence, although the theory provides only sublinear guarantees. Linear convergence can be proven. For example, this is demonstrated in \cite{salim2022optimal}, where a problem similar to \eqref{eq:vfl_constr_1} is considered, and it is also proposed to solve it through a saddle point reformulation.}

Furthermore, as previously discussed in Section \ref{sec:intro}, the saddle reformulation offers advantages in terms of privacy. Significantly, we surpass our competitor in solving SPP – \texttt{ADMM}, with \texttt{ADMM} also exhibiting notably costlier iterations. The same experiments but with grid-search tuning of parameters for all methods is presented in Appendix \ref{app:exp}. In this setting, Algorithm~\ref{alg:EG} is even more faster than competitors.

\begin{wrapfigure}[19]{r}{9.5cm}
\vskip-15pt
  \centering
\begin{minipage}{0.49\textwidth}
  \centering
\includegraphics[width =  \textwidth ]{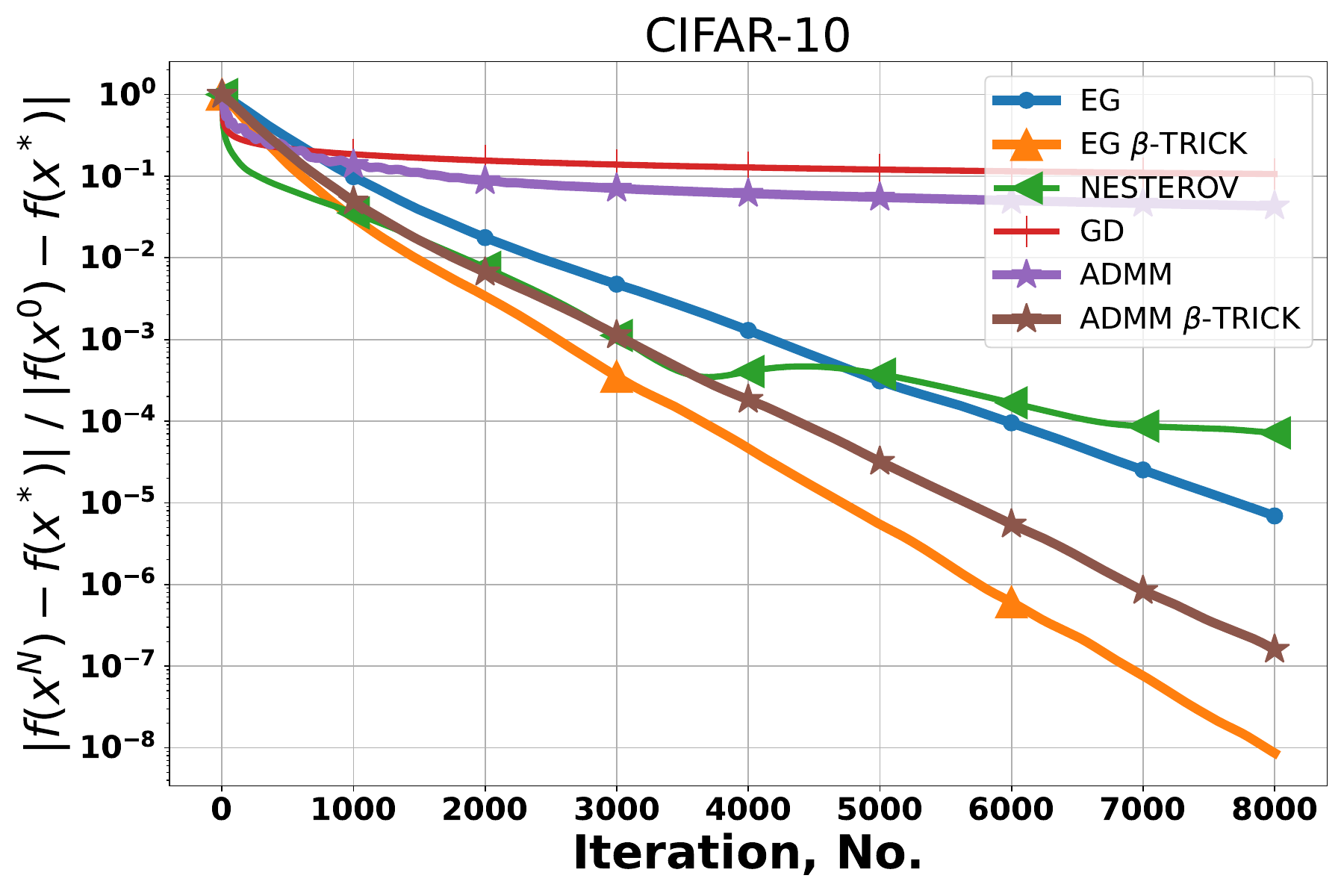}
\end{minipage}%
\caption{
Comparison of methods for solving the VFL problem in different formulations: minimization (\texttt{GD}, \texttt{Nesterov}) and saddle point (\texttt{ADMM}, \texttt{ExtraGradient}/Algorithm \ref{alg:EG}). The comparison is made on \texttt{CIFAR-10} dataset.}
    \label{fig:comparison2}
\end{wrapfigure}

\textbf{Fine-tuning of neural network.} We consider the pre-trained ResNet18 model on the ImageNet dataset. Our goal is to fine tune it on the CIFAR-10 dataset. As in the previous experiments, we take 5 clients, each client gets all the images, but only parts of them (about 1/5 of the whole image for each client). Then, each client passes its image portions through the pre-trained ResNet without the last linear layer, adjusting for the square size. As a result, each client receives embeddings corresponding to its sliced images. A new linear layer with the cross-entropy loss is trained on the embeddings of all clients, which means that the partitioning of the data is also vertical in this case. As in the previous paragraph, we use \texttt{GD}, \texttt{AGD}, \texttt{ADMM} and Algorithm~\ref{alg:EG} as methods for comparison. The methods are tuned as in the corresponding theory, since for this problem we can also estimate $L$. In the case of the ~\ref{alg:EG} and \texttt{ADMM} algorithms, we also use the $\beta$-trick.  

The results reflected in Figure \ref{fig:comparison2} show the superiority of Algorithm~\ref{alg:EG} over competitors. When the $\beta$-trick is used, \texttt{ExtraGradient} significantly outperforms other methods, but even without the $\beta$-trick Algorithm~\ref{alg:EG}  converges slightly worse than \texttt{AGD}, but later overtakes it as well.

\section*{Acknowledgments}

The work was done in the Laboratory of Federated Learning Problems of the ISP RAS (Supported by Grant App. No. 2 to Agreement No. 075-03-2024-214).

\end{mainpart}

\begin{appendixpart}

\addcontentsline{toc}{section}{Appendix} 


\section{Missing Table} \label{app:algos}



\renewcommand{\arraystretch}{2}
\begin{table*}[h]
\vspace{-0.3cm}
    \centering
\captionof{table}{Comparison of different saddle point reformulations of the VFL problem (\ref{eq:main_problem_vfl}) and deterministic methods for solving these reformulations.}
\vspace{-0.2cm}
    \label{tab:comparison0}   
    \scriptsize
    \resizebox{\linewidth}{!}{
  \begin{threeparttable}
    \begin{tabular}{|c|c|c|c|c|c|c|}
    \cline{3-7}
    \multicolumn{2}{c|}{}
     & \textbf{Method} & \textbf{Iteration complexity} & \textbf{Local device cost of iteration}  & \textbf{Local server cost of iteration} & \textbf{Communication cost of iteration}\\
    \hline
    \multirow{5}{*}{\rotatebox[origin=c]{90}{\textbf{Formulation}}} & (\ref{eq:vfl_lin_spp_1})
    & Algorithm \ref{alg:EG} & $\mathcal{O} \left( \frac{ 1+ \sqrt{\lambda_{\max}(A^T A)} + L_{\ell} + L_r  }{\varepsilon}  \right)$ & $\mathcal{O}\left( ds \cdot \text{Cost}(\text{a.o.}) + \text{Cost}(\nabla r)\right)$
    & $\mathcal{O}\left(  s \cdot \text{Cost}(\text{a.o.}) + \text{Cost}(\nabla \ell)\right)$ & $\mathcal{O}\left(  s \cdot \text{Send}(\text{1 coord.})  \right)$
    \\ \cline{2-7}
    & (\ref{eq:vfl_lin_spp_2})
    & Algorithm \ref{alg:EG_basic_2} & $\mathcal{O}\left(\frac{ 1+ \sqrt{\max_{i}\{ \lambda_{\max}(A_i^T A_i) \}} + nL_{\ell} + L_r}{\varepsilon} \right)$ & $\mathcal{O}\left( ds \cdot \text{Cost}(\text{a.o.}) + \text{Cost}(\nabla r)\right)$
    & $\mathcal{O}\left(  s \cdot \text{Cost}(\text{a.o.}) + \text{Cost}(\nabla \ell)\right)$ & $\mathcal{O}\left(  s \cdot \text{Send}(\text{1 coord.})  \right)$
    \\ \cline{2-7}
    & \multirow{2}{*}{(\ref{eq:vfl_lin_spp_aug})} & Algorithm \ref{alg:EG_aug} & $\mathcal{O} \left( \frac{  1+ \sqrt{\lambda_{\max}(A^T A)} + L_{\ell} + L_r  }{\varepsilon}  \right)$ \tnote{{\color{blue}(1)}}
    & $\mathcal{O}\left( ds \cdot \text{Cost}(\text{a.o.}) + \text{Cost}(\nabla r)\right)$
    & $\mathcal{O}\left(  s \cdot \text{Cost}(\text{a.o.}) + \text{Cost}(\nabla \ell)\right)$ & $\mathcal{O}\left(  s \cdot \text{Send}(\text{1 coord.})  \right)$
    \\ \cline{3-7} 
    && \texttt{ADMM} \cite{boyd2011distributed,lu2023unified} & $\mathcal{O} \left( \frac{  1+ \sqrt{\lambda_{\max}(A^T A)}}{\varepsilon}  \right)$ \tnote{{\color{blue}(2)}} & $\mathcal{O} \left( \sqrt{\frac{  1+ \sqrt{\lambda_{\max}(A^T A)} + L_r  }{e}}  \right) \cdot \mathcal{O}\left( d^2 \cdot \text{Cost}(\text{a.o.}) + \text{Cost}(\nabla r)\right) + \mathcal{O}\left( ds \cdot \text{Cost}(\text{a.o.})\right)^\text{{\color{blue}(3)}}$ &
    $\mathcal{O} \left( \sqrt{\frac{n L_{\ell}}{e}} \right) \cdot \mathcal{O}\left(  s \cdot \text{Cost}(\text{a.o.}) + \text{Cost}(\nabla \ell)\right)^\text{{\color{blue}(3)}}$ & $\mathcal{O}\left(  s \cdot \text{Send}(\text{1 coord.})  \right)$
    \\ \cline{2-7}
    & (\ref{eq:vfl_lin_spp_3})
    & Algorithm \ref{alg:EG_basic_3} & $\mathcal{O} \left(\frac{ 1 + \sqrt{\lambda_{\max}(A^T A)} + L_{\ell^*} + L_r}{\varepsilon} \right)$ 
    & $\mathcal{O}\left( ds \cdot \text{Cost}(\text{a.o.}) + \text{Cost}(\nabla r)\right)$
    & $\mathcal{O}\left(  s \cdot \text{Cost}(\text{a.o.}) + \text{Cost}(\nabla \ell^*)\right)$ & $\mathcal{O}\left(  s \cdot \text{Send}(\text{1 coord.})  \right)$
    \\\hline
    \end{tabular}   
    \begin{tablenotes}
    {\small   
    \item [] \tnote{{\color{blue}(1)}} detailed estimate has a dependence on the parameter $\rho$, here we substitute $\rho=1/\sqrt{\lambda_{\max}(A^T A)}$; \tnote{{\color{blue}(2)}} detailed estimate has a dependence on the parameter $\rho$ (see Corollary 3 from \cite{lu2023unified}), in particular $\rho \lambda_{\max}(A^T A) + 1/\rho$, then we substitute $\rho=1/\sqrt{\lambda_{\max}(A^T A)}$; \tnote{{\color{blue}(3)}} at each iteration of \texttt{ADMM} (see Section 8.3 from \cite{boyd2011distributed}) we need to solve subproblems on the first device/server and on all other devices, we assume that these subproblems are solved to precision $e$ using an optimal 1st order method -- Nesterov's accelerated method \cite{nesterov2003introductory}; {in some special cases, the complexity can be reduced to the complexities of other algorithms from this table}
 \item [] {\em Columns:} Iteration complexity = number of iterations to achieve $\varepsilon$-solutio, Local device cost of iteration = computational cost of operations on devices per one iteration,  Local server cost of iteration = computational cost of operations on server per one iteration,  Communication cost of iteration = communication spending during one iteration
 \item [] {\em Notation:} $L_{\ell}$ = smoothness constant of $\ell$, $L_r$ = smoothness constant of $r$, $\varepsilon$ = precision of the solution, $\text{Cost}(\text{a.o.})$ = cost of calculations of atomic operations: addition, multiplication of two numbers, $\text{Cost}(\nabla r)$, $\text{Cost}(\nabla \ell)$ = cost of computations of gradients for $r$ and $\ell$, $\text{Send}(\text{1 coord.})$ = cost of sending one coordinate/one number.
    }
\end{tablenotes}    
    \end{threeparttable}
    }
\end{table*}

\clearpage


\section{Missing Proofs}

\subsection{Proof of Theorem \ref{th:EG_basic_1}}\label{app:th:EG_basic_1}

{
\begin{theorem}[Theorem \ref{th:EG_basic_1}]
Let Assumption \ref{as:convexity_smothness} hold. Let the problem \eqref{eq:vfl_lin_spp_1}
be solved by Algorithm~\ref{alg:EG}. Then for 
$$
\gamma = \frac{1}{2} \cdot \min \left\{ 1; \frac{1}{\sqrt{\lambda_{\max}(A^T A)}}; \frac{1}{L_r}; \frac{1}{L_{\ell}} \right\},
$$
it holds that
$$
\text{gap}(\bar x^K, \bar z^K,\bar y^K) = \mathcal{O} \left(  \frac{ ( 1 + \sqrt{\lambda_{\max}(A^T A)} + L_{\ell} + L_r  ) D^2}{K}  \right),
$$
where $\bar x^K := \tfrac{1}{K}\sum_{k=0}^{K-1} x^{k+1/2}$, $\bar z^K := \tfrac{1}{K}\sum_{k=0}^{K-1} z^{k+1/2}$, $\bar y^K := \tfrac{1}{K}\sum_{k=0}^{K-1} y^{k+1/2}$ and \\ $D^2 := \max_{x, z, y \in \cX, \cZ, \cY} \left[ \|x^0 - x \|^2 + \|z^0 - z \|^2 + \|y^0 - y \|^2 \right]$.
\end{theorem}
}

To prove the convergence it is sufficient to show that the problem is convex--concave (Lemma \ref{lem:tech1}), to estimate the Lipschitz constant of gradients and use the general results from \cite{nemirovski2004prox}. But since proofs of the other algorithms is somewhat similar to proof of the basic algorithm, we provide the proof of Theorem \ref{th:EG_basic_1} to complete the picture and to move from basic proofs to more complex ones.

\begin{proof}
We start the proof with the following equations on the variables $x^{k+1}_i$, $x^{k+1/2}_i$, $x^k_i$ and any $x_i \in \R^{d_i}$:
\begin{align*}
    \|x^{k+1}_i - x_i \|^2
    &=
    \|x^{k}_i - x_i \|^2 + 2 \langle  x^{k+1}_i -  x^{k}_i,  x^{k+1}_i - x_i\rangle - \|  x^{k+1}_i -  x^{k}_i\|^2,
    \\
    \|x^{k+1/2}_i - x^{k+1}_i \|^2
    &=
    \| x^{k}_i - x^{k+1}_i \|^2 + 2 \langle x^{k+1/2}_i -   x^{k}_i, x^{k+1/2}_i - x^{k+1}_i\rangle - \|  x^{k+1/2}_i - x^k_i\|^2.
\end{align*}
Summing up two previous equations and making small rearrangements, we get
\begin{align*}
    \|x^{k+1}_i - x_i \|^2
    =&
    \|x^{k}_i - x_i \|^2 - \|  x^{k+1/2}_i - x^k_i\|^2 - \|x^{k+1/2}_i - x^{k+1}_i \|^2
    \\
    &+
    2 \langle  x^{k+1}_i -  x^{k}_i,  x^{k+1}_i - x_i\rangle + 2 \langle x^{k+1/2}_i -   x^{k}_i, x^{k+1/2}_i - x^{k+1}_i\rangle.
\end{align*}
Using that $x^{k+1}_i -  x^{k}_i = - \gamma (A_i^T y^{k+1/2} + \nabla r_i(x^{k+1/2}_i))$ and $x^{k+1/2}_i -   x^{k}_i = - \gamma (A_i^T y^k + \nabla r_i (x^k_i))$ (see lines \ref{lin_alg:EG_linx1} and \ref{lin_alg:EG_linx2} of Algorithm \ref{alg:EG}), we obtain
\begin{align}
    \label{eq:basic_proof_temp0}
    \|x^{k+1}_i - x_i \|^2
    =&
    \|x^{k}_i - x_i \|^2 - \|  x^{k+1/2}_i - x^k_i\|^2 - \|x^{k+1/2}_i - x^{k+1}_i \|^2
    \notag\\
    &-
    2\gamma \langle  A_i^T y^{k+1/2} + \nabla r_i(x^{k+1/2}_i),  x^{k+1}_i - x_i\rangle 
    \notag\\
    &- 2 \gamma \langle A_i^T y^k + \nabla r_i (x^k_i), x^{k+1/2}_i - x^{k+1}_i\rangle
    \notag\\
    =&
    \|x^{k}_i - x_i \|^2 - \|  x^{k+1/2}_i - x^k_i\|^2 - \|x^{k+1/2}_i - x^{k+1}_i \|^2
    \notag\\
    &-
    2\gamma \langle  A_i^T y^{k+1/2} + \nabla r_i(x^{k+1/2}_i),  x^{k+1/2}_i - x_i\rangle
    \notag\\
    &-
    2\gamma \langle  A_i^T (y^{k+1/2} - y^k) + \nabla r_i(x^{k+1/2}_i) - \nabla r_i (x^k_i),  x^{k+1}_i - x^{k+1/2}_i\rangle
    \notag\\
    =&
    \|x^{k}_i - x_i \|^2 - \|  x^{k+1/2}_i - x^k_i\|^2 - \|x^{k+1/2}_i - x^{k+1}_i \|^2
    \notag\\
    &-
    2\gamma \langle  A_i (x^{k+1/2}_i - x_i) ,  y^{k+1/2}\rangle
    -
    2\gamma \langle  \nabla r_i(x^{k+1/2}_i),  x^{k+1/2}_i - x_i\rangle
    \notag\\
    &-
    2\gamma \langle  A_i (x^{k+1}_i - x^{k+1/2}_i) ,  y^{k+1/2} - y^k\rangle
    \notag\\
    &-
    2\gamma \langle  \nabla r_i(x^{k+1/2}_i) - \nabla r_i (x^k_i),  x^{k+1}_i - x^{k+1/2}_i\rangle.
\end{align}
Summing over all $i$ from $1$ to $n$, we deduce
\begin{align*}
    \sum\limits_{i=1}^n \|x^{k+1}_i - x_i \|^2
    =&
    \sum\limits_{i=1}^n \|x^{k}_i - x_i \|^2 - \sum\limits_{i=1}^n\|  x^{k+1/2}_i - x^k_i\|^2 - \sum\limits_{i=1}^n \|x^{k+1/2}_i - x^{k+1}_i \|^2
    \\
    &-
    2\gamma \langle  \sum\limits_{i=1}^n A_i (x^{k+1/2}_i - x_i) ,  y^{k+1/2}\rangle
    -
    2\gamma \sum\limits_{i=1}^n \langle  \nabla r_i(x^{k+1/2}_i),  x^{k+1/2}_i - x_i\rangle
    \\
    &-
    2\gamma \langle  \sum\limits_{i=1}^n A_i (x^{k+1}_i - x^{k+1/2}_i) ,  y^{k+1/2} - y^k\rangle
    \\
    &-
    2\gamma \sum\limits_{i=1}^n \langle  \nabla r_i(x^{k+1/2}_i) - \nabla r_i (x^k_i),  x^{k+1}_i - x^{k+1/2}_i\rangle.
\end{align*}
With notation of $A = [A_1, \ldots, A_i, \ldots, A_n] $ and notation of $x = [x_1^T, \ldots,x_i^T, \ldots,x_n^T]^T$ from \eqref{eq:main_problem} and \eqref{eq:main_problem_vfl}, one can obtain that $\sum_{i=1}^n A_i x_i = A x$:
\begin{align*}
    \|x^{k+1} - x \|^2
    =&
    \|x^{k} - x \|^2 - \|  x^{k+1/2} - x^k\|^2 - \|x^{k+1/2} - x^{k+1} \|^2
    \\
    &-
    2\gamma \langle  A (x^{k+1/2} - x) ,  y^{k+1/2}\rangle
    -
    2\gamma \sum\limits_{i=1}^n \langle  \nabla r_i(x^{k+1/2}_i),  x^{k+1/2}_i - x_i\rangle
    \\
    &-
    2\gamma \langle  A (x^{k+1} - x^{k+1/2}) ,  y^{k+1/2} - y^k\rangle
    \\
    &-
    2\gamma \sum\limits_{i=1}^n \langle  \nabla r_i(x^{k+1/2}_i) - \nabla r_i (x^k_i),  x^{k+1}_i - x^{k+1/2}_i\rangle
    \\
    =&
    \|x^{k} - x \|^2 - \|  x^{k+1/2} - x^k\|^2 - \|x^{k+1/2} - x^{k+1} \|^2
    \\
    &-
    2\gamma \langle  A (x^{k+1/2} - x) ,  y^{k+1/2}\rangle
    -
    2\gamma \sum\limits_{i=1}^n \langle  \nabla r_i(x^{k+1/2}_i),  x^{k+1/2}_i - x_i\rangle
    \\
    &-
    2\gamma \langle  A^T (y^{k+1/2} - y^k) , x^{k+1} - x^{k+1/2} \rangle
    \\
    &-
    2\gamma \sum\limits_{i=1}^n \langle  \nabla r_i(x^{k+1/2}_i) - \nabla r_i (x^k_i),  x^{k+1}_i - x^{k+1/2}_i\rangle.
\end{align*}
By Cauchy Schwartz inequality: $2 \langle a, b \rangle \leq \eta \| a \|^2 + \tfrac{1}{\eta}\| b \|^2$ with $a = A^T ( y^{k+1/2} - y^{k})$, $b = x^{k+1/2} - x^{k+1}$, $\eta = 2\gamma$ and $a = \nabla r_i(x^{k+1/2}_i) - \nabla r_i (x^k_i)$, $b = x^{k+1/2}_i - x^{k+1}_i$, $\eta = 2\gamma$, we get
\begin{align}
    \label{eq:basic_proof_temp1}
    \|x^{k+1} - x \|^2
    \leq&
    \|x^{k} - x \|^2 - \|  x^{k+1/2} - x^k\|^2 - \|x^{k+1/2} - x^{k+1} \|^2
    \notag\\
    &-
    2\gamma \langle  A (x^{k+1/2} - x) ,  y^{k+1/2}\rangle
    -
    2\gamma \sum\limits_{i=1}^n \langle  \nabla r_i(x^{k+1/2}_i),  x^{k+1/2}_i - x_i\rangle
    \notag\\
    &+
    2\gamma^2 \|A^T (y^{k+1/2} - y^k)\|^2 + \frac{1}{2}\|x^{k+1} - x^{k+1/2} \|^2
    \notag\\
    &+
    2\gamma^2 \sum\limits_{i=1}^n \| \nabla r_i(x^{k+1/2}_i) - \nabla r_i (x^k_i)\|^2 + \frac{1}{2}  \sum\limits_{i=1}^n \| x^{k+1}_i - x^{k+1/2}_i\|^2
    \notag\\
    =&
    \|x^{k} - x \|^2 - \|  x^{k+1/2} - x^k\|^2 
    \notag\\
    &-
    2\gamma \langle  A (x^{k+1/2} - x) ,  y^{k+1/2}\rangle
    -
    2\gamma \sum\limits_{i=1}^n \langle  \nabla r_i(x^{k+1/2}_i),  x^{k+1/2}_i - x_i\rangle
    \notag\\
    &+
    2\gamma^2 \|A^T (y^{k+1/2} - y^k)\|^2 +
    2\gamma^2 \sum\limits_{i=1}^n \| \nabla r_i(x^{k+1/2}_i) - \nabla r_i (x^k_i)\|^2.
\end{align}
Using the same steps, one can obtain for $z \in \R^s$, 
\begin{align}
    \label{eq:basic_proof_temp2}
    \|z^{k+1} - z \|^2
    \leq&
    \|z^{k} - z \|^2 - \|  z^{k+1/2} - z^k\|^2 
    \notag\\
    &+
    2\gamma \langle  y^{k+1/2} ,  z^{k+1/2} - z\rangle
    -
    2\gamma \langle  \nabla \ell(z^{k+1/2}, b),  z^{k+1/2} - z\rangle
    \notag\\
    &+
    2\gamma^2 \|y^{k+1/2} - y^k\|^2 +
    2\gamma^2 \| \nabla \ell(z^{k+1/2}, b) -\nabla \ell(z^k, b)\|^2.
\end{align}
and for all $y \in \R^s$,
\begin{align}
    \label{eq:basic_proof_temp3}
    \|y^{k+1} - y \|^2
    \leq&
    \|y^{k} - y \|^2 - \|  y^{k+1/2} - y^k\|^2 
    \notag\\
    &-
    2\gamma \langle  z^{k+1/2} ,  y^{k+1/2} - y\rangle
    +
    2\gamma \langle  \sum_{i=1}^n A_i x^{k+1/2}_i,  y^{k+1/2} - y\rangle
    \notag\\
    &+
    2\gamma^2 \|z^{k+1/2} - z^k\|^2 +
    2\gamma^2 \left\| \sum_{i=1}^n A_i (x^{k+1/2}_i - x^k_i) \right\|^2
    \notag\\
    =&
    \|y^{k} - y \|^2 - \|  y^{k+1/2} - y^k\|^2 
    \notag\\
    &-
    2\gamma \langle  z^{k+1/2} ,  y^{k+1/2} - y\rangle
    +
    2\gamma \langle  A x^{k+1/2},  y^{k+1/2} - y\rangle
    \notag\\
    &+
    2\gamma^2 \|z^{k+1/2} - z^k\|^2 +
    2\gamma^2 \| A (x^{k+1/2} - x^k) \|^2.
\end{align}
Here we also use notation of $A$ and $x$. 
Summing up (\ref{eq:basic_proof_temp1}), (\ref{eq:basic_proof_temp2}) and (\ref{eq:basic_proof_temp3}), we obtain
\begin{align*}
    \|x^{k+1} - x \|^2 + \|z^{k+1} - z \|^2 + \|y^{k+1} - y \|^2
    \leq&
    \|x^{k} - x \|^2 + \|z^{k} - z \|^2 + \|y^{k} - y \|^2
    \notag\\
    &
    - \|  x^{k+1/2} - x^k\|^2 - \|  z^{k+1/2} - z^k\|^2 - \|  y^{k+1/2} - y^k\|^2 
    \notag\\
    &-
    2\gamma \langle  A (x^{k+1/2} - x) ,  y^{k+1/2}\rangle
    +
    2\gamma \langle  y^{k+1/2} ,  z^{k+1/2} - z\rangle
    \notag\\
    &-
    2\gamma \langle  z^{k+1/2} ,  y^{k+1/2} - y\rangle
    +
    2\gamma \langle  A x^{k+1/2},  y^{k+1/2} - y\rangle
    \notag\\
    &-
    2\gamma \sum\limits_{i=1}^n \langle  \nabla r_i(x^{k+1/2}_i),  x^{k+1/2}_i - x_i\rangle
    -
    2\gamma \langle  \nabla \ell(z^{k+1/2}, b),  z^{k+1/2} - z\rangle
    \notag\\
    &+
    2\gamma^2 \|A^T (y^{k+1/2} - y^k)\|^2 +
    2\gamma^2 \sum\limits_{i=1}^n \| \nabla r_i(x^{k+1/2}_i) - \nabla r_i (x^k_i)\|^2
    \notag\\
    &+
    2\gamma^2 \|y^{k+1/2} - y^k\|^2 +
    2\gamma^2 \| \nabla \ell(z^{k+1/2}, b) -\nabla \ell(z^k, b)\|^2
    \notag\\
    &+
    2\gamma^2 \|z^{k+1/2} - z^k\|^2 +
    2\gamma^2 \| A (x^{k+1/2} - x^k) \|^2.
\end{align*}
Using convexity and $L_r$-smoothness of the function $r_i$ with convexity and $L_{\ell}$-smoothness of the function $\ell$ (Assumption \ref{as:convexity_smothness}), we have
\begin{align*}
    \|x^{k+1} - x \|^2 + \|z^{k+1} - z \|^2 + \|y^{k+1} - y \|^2
    \leq&
    \|x^{k} - x \|^2 + \|z^{k} - z \|^2 + \|y^{k} - y \|^2
    \notag\\
    &
    - \|  x^{k+1/2} - x^k\|^2 - \|  z^{k+1/2} - z^k\|^2 - \|  y^{k+1/2} - y^k\|^2 
    \notag\\
    &-
    2\gamma \langle  A (x^{k+1/2} - x) ,  y^{k+1/2}\rangle
    +
    2\gamma \langle  y^{k+1/2} ,  z^{k+1/2} - z\rangle
    \notag\\
    &-
    2\gamma \langle  z^{k+1/2} ,  y^{k+1/2} - y\rangle
    +
    2\gamma \langle  A x^{k+1/2},  y^{k+1/2} - y\rangle
    \notag\\
    &-
    2\gamma \sum\limits_{i=1}^n \left(r_i(x^{k+1/2}_i) - r_i(x_i) \right)
    -
    2\gamma \left(  \ell(z^{k+1/2}, b)  - \ell(z, b)\right)
    \notag\\
    &+
    2\gamma^2 \|A^T (y^{k+1/2} - y^k)\| +
    2\gamma^2 L_r^2 \sum\limits_{i=1}^n  \| x^{k+1/2}_i - x^k_i\|^2
    \notag\\
    &+
    2\gamma^2 \|y^{k+1/2} - y^k\|^2 +
    2\gamma^2 L_{\ell}^2\| z^{k+1/2} - z^k\|^2
    \notag\\
    &+
    2\gamma^2 \|z^{k+1/2} - z^k\|^2 +
    2\gamma^2 \| A (x^{k+1/2} - x^k) \|^2.
\end{align*}
Using the definition of $\lambda_{\max}(\cdot)$ as a maximum eigenvalue, we get
\begin{align*}
    \|x^{k+1} - x \|^2 + \|z^{k+1} - z \|^2 + \|y^{k+1} - y \|^2
    \leq&
    \|x^{k} - x \|^2 + \|z^{k} - z \|^2 + \|y^{k} - y \|^2
    \notag\\
    &
    - \|  x^{k+1/2} - x^k\|^2 - \|  z^{k+1/2} - z^k\|^2 - \|  y^{k+1/2} - y^k\|^2 
    \notag\\
    &-
    2\gamma \langle  A (x^{k+1/2} - x) ,  y^{k+1/2}\rangle
    +
    2\gamma \langle  y^{k+1/2} ,  z^{k+1/2} - z\rangle
    \notag\\
    &-
    2\gamma \langle  z^{k+1/2} ,  y^{k+1/2} - y\rangle
    +
    2\gamma \langle  A x^{k+1/2},  y^{k+1/2} - y\rangle
    \notag\\
    &-
    2\gamma \sum\limits_{i=1}^n \left(r_i(x^{k+1/2}_i) - r_i(x_i) \right)
    -
    2\gamma \left(  \ell(z^{k+1/2}, b)  - \ell(z, b)\right)
    \notag\\
    &+
    2\gamma^2 \lambda_{\max} (AA^T) \|y^{k+1/2} - y^k\| +
    2\gamma^2 L_r^2 \| x^{k+1/2} - x^k\|^2
    \notag\\
    &+
    2\gamma^2 \|y^{k+1/2} - y^k\|^2 +
    2\gamma^2 L_{\ell}^2\| z^{k+1/2} - z^k\|^2
    \notag\\
    &+
    2\gamma^2 \|z^{k+1/2} - z^k\|^2 +
    2\gamma^2 \lambda_{\max} (A^T A)\| x^{k+1/2} - x^k \|^2.
\end{align*}
With the choice of $\gamma \leq \tfrac{1}{2} \cdot \min\left\{ 1; \tfrac{1}{\sqrt{\lambda_{\max}(A^T A)}}; \tfrac{1}{L_r}; \tfrac{1}{L_{\ell}}\right\}$, we get
\begin{align*}
    \|x^{k+1} - x \|^2 + \|z^{k+1} - z \|^2 + \|y^{k+1} - y \|^2
    \leq&
    \|x^{k} - x \|^2 + \|z^{k} - z \|^2 + \|y^{k} - y \|^2
    \notag\\
    &-
    2\gamma \langle  A (x^{k+1/2} - x) ,  y^{k+1/2}\rangle
    +
    2\gamma \langle  y^{k+1/2} ,  z^{k+1/2} - z\rangle
    \notag\\
    &-
    2\gamma \langle  z^{k+1/2} ,  y^{k+1/2} - y\rangle
    +
    2\gamma \langle  A x^{k+1/2},  y^{k+1/2} - y\rangle
    \notag\\
    &-
    2\gamma \sum\limits_{i=1}^n \left(r_i(x^{k+1/2}_i) - r_i(x_i) \right)
    -
    2\gamma \left(  \ell(z^{k+1/2}, b)  - \ell(z, b)\right)
    \\
    =&
    \|x^{k} - x \|^2 + \|z^{k} - z \|^2 + \|y^{k} - y \|^2
    \notag\\
    &+
    2\gamma \langle  A x - z,  y^{k+1/2}\rangle
    -
    2\gamma \langle  A x^{k+1/2} - z^{k+1/2}, y\rangle
    \notag\\
    &-
    2\gamma \sum\limits_{i=1}^n \left(r_i(x^{k+1/2}_i) - r_i(x_i) \right)
    -
    2\gamma \left(  \ell(z^{k+1/2}, b)  - \ell(z, b)\right).
\end{align*}
After small rearrangements, we obtain
\begin{align*}
    &\left( \ell (z^{k+1/2}, b) - \ell (z, b)\right) + \sum_{i=1}^n \left( r_i (x^{k+1/2}_i) - r_i (x_i)\right) 
    +
    \langle A x^{k+1/2} - z^{k+1/2}, y\rangle
    - \langle  A x - z, y^{k+1/2} \rangle
    \\
    &\hspace{8cm}\leq
    \frac{1}{2\gamma}\Big(\| x^k - x\|^2 + \| z^k - z\|^2 + \| y^k - y\|^2 
    \notag\\
    &\hspace{8.4cm} -\| x^{k+1} - x\|^2 - \| z^{k+1} - z \|^2 - \| y^{k+1} -  y \|^2\Big).
\end{align*}
Then we sum all over $k$ from $0$ to $K-1$, divide by $K$, and have
\begin{align*}
    &\frac{1}{K} \sum\limits_{k=0}^{K-1}\left( \ell (z^{k+1/2}, b) - \ell (z, b)\right) + \sum_{i=1}^n \frac{1}{K} \sum\limits_{k=0}^{K-1} \left( r_i (x^{k+1/2}_i) - r_i (x_i)\right) 
    \\
    &+
    \langle A \cdot \frac{1}{K} \sum\limits_{k=0}^{K-1} x^{k+1/2} - \frac{1}{K} \sum\limits_{k=0}^{K-1} z^{k+1/2}, y\rangle
    - \langle  A x - z, \frac{1}{K} \sum\limits_{k=0}^{K-1} y^{k+1/2} \rangle
    \\
    &\hspace{8cm}\leq
    \frac{1}{2\gamma K}\Big(\| x^0 - x\|^2 + \| z^0 - z\|^2 + \| y^0 - y\|^2 
    \notag\\
    &\hspace{8.4cm} -\| x^{K} - x\|^2 - \| z^{K} - z \|^2 - \| y^{K} -  y \|^2\Big)
    \\
    &\hspace{8cm}\leq
    \frac{1}{2\gamma K}(\| x^0 - x\|^2 + \| z^0 - z\|^2 + \| y^0 - y\|^2).
\end{align*}
With Jensen inequality for convex functions $\ell$ and $r_i$, one can note that
\begin{align*}
 \ell \left( \frac{1}{K} \sum\limits_{k=0}^{K-1} z^{k+1/2}, b \right) \leq \frac{1}{K} \sum\limits_{k=0}^{K-1} \ell (z^{k+1/2}, b), \\
 r_i \left( \frac{1}{K} \sum\limits_{k=0}^{K-1} x^{k+1/2}_i \right) \leq \frac{1}{K} \sum\limits_{k=0}^{K-1} r_i (x^{k+1/2}_i).
\end{align*}
Then, with notation $\bar x^K_i = \frac{1}{K} \sum\limits_{k=0}^{K-1} x^{k+1/2}_i$, $\bar z^K = \frac{1}{K} \sum\limits_{k=0}^{K-1} z^{k+1/2}$, $\bar y^K = \frac{1}{K} \sum\limits_{k=0}^{K-1} y^{k+1/2}$, we have
\begin{align*}
    &\ell (\bar z^K, b) - \ell (z, b) + \sum_{i=1}^n \left( r_i (\bar x^{K}_i) - r_i (x_i)\right) 
    +
    \langle A \bar x^{K} - \bar z^{K}, y\rangle
    - \langle  A x - z, \bar y^{K} \rangle
    \\
    &\hspace{10cm}\leq
    \frac{1}{2\gamma K}(\| x^0 - x\|^2 + \| z^0 - z\|^2 + \| y^0 - y\|^2).
\end{align*}
Following the definition \eqref{eq:duality_gap}, we only need to take the maximum in the variable $y \in \cY$ and the minimum in $x \in \cX$ and $z \in \cZ$.
\begin{align}
    \label{eq:basic_crit1}
    \text{gap}(\bar x^{K}, \bar z^K, \bar y^K)
    &=
    \max_{y \in \cY}L(\bar x^{K}, \bar z^K, y) - \min_{x,z \in \cX, \cZ} L( x, z, \bar y^K)
    \notag\\
    &=\max_{y \in \cY} \left[\ell (\bar z^K, b) + \sum_{i=1}^n r_i (\bar x^{K}_i) +
    \langle A \bar x^{K} - \bar z^{K}, y\rangle \right]
    - \min_{x,z \in \cX, \cZ} \left[ \ell (z, b) +  \sum_{i=1}^n r_i (x_i) + \langle  A x - z, \bar y^{K} \rangle\right]
    \notag\\
    &=\max_{y \in \cY} \max_{x,z \in \cX, \cZ} \left[\ell (\bar z^K, b) - \ell (z, b) + \sum_{i=1}^n \left( r_i (\bar x^{K}_i) - r_i (x_i)\right) 
    +
    \langle A \bar x^{K} - \bar z^{K}, y\rangle
    - \langle  A x - z, \bar y^{K} \rangle \right]
    \notag\\
    &\leq
    \frac{1}{2\gamma K}(\max_{x \in \cX} \| x^0 - x\|^2 + \max_{z \in \cZ} \| z^0 - z\|^2 + \max_{y \in \cY} \| y^0 - y\|^2).
\end{align}
To complete the proof in the cases \eqref{eq:basic_crit1} 
, it remains to put $\gamma = \tfrac{1}{2} \cdot \min\left\{ 1; \tfrac{1}{\sqrt{\lambda_{\max}(A^T A)}}; \tfrac{1}{L_r}; \tfrac{1}{L_{\ell}}\right\}$.
\end{proof}

\subsection{Proof of Theorem \ref{th:EG_compressed_unbiased}}
\label{sec:proof_EG_compressed_unbiased} 
\label{app:th:EG_compressed_unbiased}

{
\begin{theorem}[Theorem \ref{th:EG_compressed_unbiased}]
Let Assumption \ref{as:convexity_smothness} hold. Let the problem (\ref{eq:vfl_lin_spp_1})
be solved by Algorithm~\ref{alg:EG_quantization} with operators and $Q$, which satisfy Definition \ref{def:quantization}. Then for $\tau = 1 - p$ and
$$
\textstyle{
\gamma = \frac{1}{4}\min\left\{ 1;\frac{1}{L_r}; \frac{1}{L_{\ell}};\sqrt{\frac{1-\tau}{\omega[\lambda_{\max}(A A^T) + \mathbb{I}(\text{diff. seed}) \max_{i}\{\lambda_{\max}(A_i A^T_i)\}]}}; \sqrt{\frac{1-\tau}{\omega\lambda_{\max}(A A^T)}} \right\},
}
$$
it holds that
\begin{align*}
    \E{\text{gap}(\bar x^K, \bar z^K, \bar y^K)} &= 
    \mathcal{O} \Bigg(  [ 1 + \sqrt{\frac{\omega}{p}}  (\sqrt{\lambda_{\max}(A A^T) } ) + L_{\ell} + L_r ] \cdot \frac{D^2}{K}  
    \\
    & \qquad\qquad    +     \sqrt{\frac{\omega}{p}}  
    \mathbb{I}(\text{diff. seed}) \max_{i = 1, \ldots, n} \{\sqrt{\lambda_{\max}(A_i A^T_i) }\}   \cdot \frac{D^2}{K} \Bigg),
\end{align*}
where the indicator function $\mathbb{I}(\text{diff. seed})$ is responsible for whether the different or same random seed is used on all devices, $\bar x^K := \tfrac{1}{K}\sum_{k=0}^{K-1} x^{k+1/2}$, $\bar z^K := \tfrac{1}{K}\sum_{k=0}^{K-1} z^{k+1/2}$, $\bar y^K := \tfrac{1}{K}\sum_{k=0}^{K-1} y^{k+1/2}$ and \\ $D^2 := \max_{x, z, y \in \cX, \cZ, \cY} \left[ \|x^0 - x \|^2 + \|z^0 - z \|^2 + \|y^0 - y \|^2 \right]$.
\end{theorem}
}
\begin{proof}
We start the proof with the following equations on the variables $x^{k+1}_i$, $x^{k+1/2}_i$, $x^k_i$ and any $x_i \in \R^{d_i}$:
\begin{align*}
    \|x^{k+1}_i - x_i \|^2
    &=
    \|x^{k}_i - x_i \|^2 + 2 \langle  x^{k+1}_i -  x^{k}_i,  x^{k+1}_i - x_i\rangle - \|  x^{k+1}_i -  x^{k}_i\|^2,
    \\
    \|x^{k+1/2}_i - x^{k+1}_i \|^2
    &=
    \| x^{k}_i - x^{k+1}_i \|^2 + 2 \langle x^{k+1/2}_i -   x^{k}_i, x^{k+1/2}_i - x^{k+1}_i\rangle - \|  x^{k+1/2}_i - x^k_i\|^2.
\end{align*}
Summing up two previous equations and making small rearrangements, we get
\begin{align*}
    \|x^{k+1}_i - x_i \|^2
    =&
    \|x^{k}_i - x_i \|^2 - \|  x^{k+1/2}_i - x^k_i\|^2 - \|x^{k+1/2}_i - x^{k+1}_i \|^2
    \\
    &+
    2 \langle  x^{k+1}_i -  x^{k}_i,  x^{k+1}_i - x_i\rangle + 2 \langle x^{k+1/2}_i -   x^{k}_i, x^{k+1/2}_i - x^{k+1}_i\rangle.
\end{align*}
Using that $x^{k+1}_i -  x^{k}_i = (1 - \tau) (w^k_i -x^k_i) - \gamma (A_i^T  [Q(y^{k+1/2} - u^k) + u^k] + \nabla r_i(x^{k+1/2}_i))$ and $x^{k+1/2}_i -   x^{k}_i = (1 - \tau) (w^k_i - x^k_i) - \gamma (A_i^T u^k + \nabla r_i (x^k_i) )$, we obtain
\begin{align*}
    \|x^{k+1}_i - x_i \|^2
    =&
    \|x^{k}_i - x_i \|^2 - \|  x^{k+1/2}_i - x^k_i\|^2 - \|x^{k+1/2}_i - x^{k+1}_i \|^2
    \\
    &+
    2 (1 - \tau) \langle  w^k_i -x^k_i,  x^{k+1}_i - x_i\rangle 
    \\
    &
    -
    2\gamma \langle  A_i^T  [Q(y^{k+1/2} - u^k) + u^k] + \nabla r_i(x^{k+1/2}_i),  x^{k+1}_i - x_i\rangle
    \\
    &
    + 2 (1 - \tau) \langle w^k_i - x^k_i, x^{k+1/2}_i - x^{k+1}_i\rangle
    \\
    &
    - 2 \gamma \langle A_i^T u^k + \nabla r_i (x^k_i) , x^{k+1/2}_i - x^{k+1}_i\rangle
    \\
    =&
    \|x^{k}_i - x_i \|^2 - \|  x^{k+1/2}_i - x^k_i\|^2 - \|x^{k+1/2}_i - x^{k+1}_i \|^2
    \\
    &+
    2 (1 - \tau) \langle  w^k_i -x^k_i,  x^{k+1/2}_i - x_i\rangle 
    \\
    &
    -
    2\gamma \langle  A_i^T  [Q(y^{k+1/2} - u^k) + u^k] + \nabla r_i(x^{k+1/2}_i),  x^{k+1/2}_i - x_i\rangle
    \\
    &
    + 2 \gamma \langle A_i^T  [Q(y^{k+1/2} - u^k) ] + \nabla r_i(x^{k+1/2}_i) - \nabla r_i (x^k_i) , x^{k+1/2}_i - x^{k+1}_i\rangle
    \\
    =&
    \|x^{k}_i - x_i \|^2 - \|  x^{k+1/2}_i - x^k_i\|^2 - \|x^{k+1/2}_i - x^{k+1}_i \|^2
    \\
    &+
    2 (1 - \tau) \langle  w^k_i -x^{k+1/2}_i,  x^{k+1/2}_i - x_i\rangle 
    \\
    &+
    2 (1 - \tau) \langle  x^{k+1/2}_i -x^k_i,  x^{k+1/2}_i - x_i\rangle 
    \\
    &
    -
    2\gamma \langle  A_i^T  [Q(y^{k+1/2} - u^k) + u^k] + \nabla r_i(x^{k+1/2}_i),  x^{k+1/2}_i - x_i\rangle
    \\
    &
    + 2 \gamma \langle A_i^T  [Q(y^{k+1/2} - u^k) ] + \nabla r_i(x^{k+1/2}_i) - \nabla r_i (x^k_i) , x^{k+1/2}_i - x^{k+1}_i\rangle.
\end{align*}
For the second and third lines we use identity $2\langle a, b \rangle = \| a  + b\|^2 - \| a\|^2 - \| b\|^2$, and get
\begin{align}
\label{eq:unbiased_temp22}
    \|x^{k+1}_i - x_i \|^2
    =&
    \|x^{k}_i - x_i \|^2 - \|  x^{k+1/2}_i - x^k_i\|^2 - \|x^{k+1/2}_i - x^{k+1}_i \|^2
    \notag\\
    &+
    (1 - \tau) ( \| w^k_i - x_i \|^2 - \| w^k_i -x^{k+1/2}_i \|^2 - \| x^{k+1/2}_i - x_i \|^2 ) 
    \notag\\
    &+
    (1 - \tau) ( \| x^{k+1/2}_i -x^k_i \|^2 + \| x^{k+1/2}_i - x_i \|^2 - \| x^k_i - x_i\|^2 )
    \notag\\
    &
    -
    2\gamma \langle  A_i^T  [Q(y^{k+1/2} - u^k) + u^k] + \nabla r_i(x^{k+1/2}_i),  x^{k+1/2}_i - x_i\rangle
    \notag\\
    &
    + 2 \gamma \langle A_i^T  [Q(y^{k+1/2} - u^k) ] + \nabla r_i(x^{k+1/2}_i) - \nabla r_i (x^k_i) , x^{k+1/2}_i - x^{k+1}_i\rangle
    \notag\\
    =&
    \tau \|x^{k}_i - x_i \|^2 + (1 - \tau) \| w^k_i - x_i \|^2 
    \notag\\
    &
    - \tau \|  x^{k+1/2}_i - x^k_i\|^2 - (1 - \tau) \| w^k_i -x^{k+1/2}_i \|^2 - \|x^{k+1/2}_i - x^{k+1}_i \|^2
    \notag\\
    &
    -
    2\gamma \langle  A_i  (x^{k+1/2}_i - x_i),  y^{k+1/2}\rangle
    -
    2\gamma \langle \nabla r_i(x^{k+1/2}_i),  x^{k+1/2}_i - x_i\rangle
    \notag\\
    &
    -
    2\gamma \langle  A_i (x^{k+1/2}_i - x_i)  ,  Q(y^{k+1/2} - u^k) - y^{k+1/2} + u^k \rangle
    \notag\\
    &
    + 2 \gamma \langle A_i (x^{k+1/2}_i - x^{k+1}_i),  Q(y^{k+1/2} - u^k) \rangle
    \notag\\
    &
    + 2 \gamma \langle \nabla r_i(x^{k+1/2}_i) - \nabla r_i (x^k_i) , x^{k+1/2}_i - x^{k+1}_i\rangle.
\end{align}
Summing over all $i$ from $1$ to $n$ and using the notation of $A = [A_1, \ldots, A_i, \ldots, A_n] $, $x = [x_1^T, \ldots,x_i^T, \ldots,x_n^T]^T$, $w = [w_1^T, \ldots,w_i^T, \ldots,w_n^T]^T$, we deduce
\begin{align*}
    \|x^{k+1} - x \|^2
    =&
    \tau \|x^{k} - x \|^2 + (1 - \tau) \| w^k - x \|^2 
    \\
    &
    - \tau \|  x^{k+1/2} - x^k\|^2 - (1 - \tau) \| w^k -x^{k+1/2} \|^2 - \|x^{k+1/2} - x^{k+1} \|^2
    \\
    &
    -
    2\gamma \langle  A  (x^{k+1/2} - x),  y^{k+1/2}\rangle
    -
    2\gamma \sum\limits_{i=1}^n\langle \nabla r_i(x^{k+1/2}_i),  x^{k+1/2}_i - x_i\rangle
    \\
    &
    -
    2\gamma \langle  A (x^{k+1/2} - x)  ,  Q(y^{k+1/2} - u^k) - y^{k+1/2} + u^k \rangle
    \\
    &
    + 2 \gamma \langle A (x^{k+1/2} - x^{k+1}),  Q(y^{k+1/2} - u^k) \rangle
    \\
    &
    + 2 \gamma \sum\limits_{i=1}^n \langle \nabla r_i(x^{k+1/2}_i) - \nabla r_i (x^k_i) , x^{k+1/2}_i - x^{k+1}_i\rangle
    \\
    =&
    \tau \|x^{k} - x \|^2 + (1 - \tau) \| w^k - x \|^2 
    \\
    &
    - \tau \|  x^{k+1/2} - x^k\|^2 - (1 - \tau) \| w^k -x^{k+1/2} \|^2 - \|x^{k+1/2} - x^{k+1} \|^2
    \\
    &
    -
    2\gamma \langle  A  (x^{k+1/2} - x),  y^{k+1/2}\rangle
    -
    2\gamma \sum\limits_{i=1}^n\langle \nabla r_i(x^{k+1/2}_i),  x^{k+1/2}_i - x_i\rangle
    \\
    &
    -
    2\gamma \langle  A (x^{k+1/2} - x)  ,  Q(y^{k+1/2} - u^k) - y^{k+1/2} + u^k \rangle
    \\
    &
    + 2 \gamma \langle A^T Q(y^{k+1/2} - u^k),  x^{k+1/2} - x^{k+1} \rangle
    \\
    &
    + 2 \gamma \sum\limits_{i=1}^n \langle \nabla r_i(x^{k+1/2}_i) - \nabla r_i (x^k_i) , x^{k+1/2}_i - x^{k+1}_i\rangle.
\end{align*}
By simple fact: $2 \langle a, b \rangle \leq \eta \| a \|^2 + \tfrac{1}{\eta}\| b \|^2$ with $a = A^T Q( y^{k+1/2} - u^{k})$, $b = x^{k+1/2} - x^{k+1}$, $\eta = 2\gamma$ and $a = \nabla r_i(x^{k+1/2}_i) - \nabla r_i (x^k_i)$, $b = x^{k+1/2}_i - x^{k+1}_i$, $\eta = 2\gamma$, we get
\begin{align*}
    \|x^{k+1} - x \|^2
    \leq&
    \tau \|x^{k} - x \|^2 + (1 - \tau) \| w^k - x \|^2 
    \\
    &
    - \tau \|  x^{k+1/2} - x^k\|^2 - (1 - \tau) \| w^k -x^{k+1/2} \|^2 - \|x^{k+1/2} - x^{k+1} \|^2
    \\
    &
    -
    2\gamma \langle  A  (x^{k+1/2} - x),  y^{k+1/2}\rangle
    -
    2\gamma \sum\limits_{i=1}^n\langle \nabla r_i(x^{k+1/2}_i),  x^{k+1/2}_i - x_i\rangle
    \\
    &
    -
    2\gamma \langle  A (x^{k+1/2} - x)  ,  Q(y^{k+1/2} - u^k) - y^{k+1/2} + u^k \rangle
    \\
    &
    + 2 \gamma^2 \| A^T Q(y^{k+1/2} - u^k) \|^2 + \frac{1}{2}  \|x^{k+1/2} - x^{k+1} \|^2
    \\
    &
    + 2 \gamma^2 \sum\limits_{i=1}^n \| \nabla r_i(x^{k+1/2}_i) - \nabla r_i (x^k_i) \|^2 + \frac{1}{2}\| x^{k+1/2} - x^{k+1} \|^2.
\end{align*}
Adding to the both sides $\| w^{k+1} - x\|^2$, one can obtain
\begin{align}
\label{eq:unbiased_proof_temp1}
    \|x^{k+1} - x \|^2 + \|w^{k+1} - x \|^2
    \leq&
    \|x^{k} - x \|^2 + \| w^k - x \|^2 
    \notag\\
    &- (1 - \tau) \|x^{k} - x \|^2 - \tau \| w^k - x \|^2  + \|w^{k+1} - x \|^2
    \notag\\
    &
    - \tau \|  x^{k+1/2} - x^k\|^2 - (1 - \tau) \| w^k -x^{k+1/2} \|^2 
    \notag\\
    &
    -
    2\gamma \langle  A  (x^{k+1/2} - x),  y^{k+1/2}\rangle
    -
    2\gamma \sum\limits_{i=1}^n\langle \nabla r_i(x^{k+1/2}_i),  x^{k+1/2}_i - x_i\rangle
    \notag\\
    &
    -
    2\gamma \langle  A (x^{k+1/2} - x)  ,  Q(y^{k+1/2} - u^k) - y^{k+1/2} + u^k \rangle
    \notag\\
    &
    + 2 \gamma^2 \| A^T Q(y^{k+1/2} - u^k) \|^2
    \notag\\
    &
    + 2 \gamma^2 \sum\limits_{i=1}^n \| \nabla r_i(x^{k+1/2}_i) - \nabla r_i (w^k_i) \|^2
    \notag\\
    =&
    \|x^{k} - x \|^2 + \| w^k - x \|^2 
    \notag\\
    &
    - \tau \|  x^{k+1/2} - x^k\|^2 - (1 - \tau) \| w^k -x^{k+1/2} \|^2
    \notag\\
    &
    -
    2\gamma \langle  A  (x^{k+1/2} - x),  y^{k+1/2}\rangle
    -
    2\gamma \sum\limits_{i=1}^n\langle \nabla r_i(x^{k+1/2}_i),  x^{k+1/2}_i - x_i\rangle
    \notag\\
    &- (1 - \tau) \|x^{k}\|^2 - \tau \| w^k \|^2  + \|w^{k+1} \|^2
    \notag\\
    &+ 2\langle (1 - \tau) x^k + \tau w^k - w^{k+1}, x\rangle 
    \notag\\
    &
    -
    2\gamma \langle  A (x^{k+1/2} - x^0)  ,  Q(y^{k+1/2} - u^k) - y^{k+1/2} + u^k \rangle
    \notag\\
    &
    -
    2\gamma \langle  A (x^{0} - x)  ,  Q(y^{k+1/2} - u^k) - y^{k+1/2} + u^k \rangle
    \notag\\
    &
    + 2 \gamma^2 \| A^T Q(y^{k+1/2} - u^k) \|^2
    \notag\\
    &
    + 2 \gamma^2 \sum\limits_{i=1}^n \| \nabla r_i(x^{k+1/2}_i) - \nabla r_i (x^k_i) \|^2.
\end{align}
Using the same steps, one can obtain for $z \in \R^s$,
\begin{align}
    \label{eq:unbiased_proof_temp2}
    \|z^{k+1} - z \|^2
    \leq&
    \|z^{k} - z \|^2 - \|  z^{k+1/2} - z^k\|^2 
    \notag\\
    &+
    2\gamma \langle  y^{k+1/2} ,  z^{k+1/2} - z\rangle
    -
    2\gamma \langle  \nabla \ell(z^{k+1/2}, b),  z^{k+1/2} - z\rangle
    \notag\\
    &+
    2\gamma^2 \|y^{k+1/2} - y^k\|^2 +
    2\gamma^2 \| \nabla \ell(z^{k+1/2}, b) -\nabla \ell(z^k, b)\|^2.
\end{align}
and for all $y \in \R^s$,
\begin{align}
\label{eq:unbiased_proof_temp3}
    \| y^{k+1} - y \|^2 + \|u^{k+1} - y \|^2
    \leq&
    \|y^{k} - y \|^2 + \| u^k - y \|^2 
    \notag\\
    &
    - \tau \|  y^{k+1/2} - y^k\|^2 - (1 - \tau) \| u^k -y^{k+1/2} \|^2
    \notag\\
    &
    -
    2\gamma \langle  z^{k+1/2},  y^{k+1/2} - y\rangle
    +
    2\gamma \langle \sum_{i=1}^n A_i x^{k+1/2}_i,  y^{k+1/2} - y\rangle
    \notag\\
    &
    +
    2\gamma \langle \sum_{i=1}^n [Q(A_i x^{k+1/2}_i - A_i w^{k}_i) + A_i w^{k}_i - A_i x^{k+1/2}_i ] ,  y^{k+1/2} - y^0\rangle
    \notag\\
    &
    +
    2\gamma \langle \sum_{i=1}^n [Q(A_i x^{k+1/2}_i - A_i w^{k}_i) + A_i w^{k}_i - A_i x^{k+1/2}_i ] ,  y^{0} - y\rangle
    \notag\\
    &- (1 - \tau) \|y^{k}\|^2 - \tau \| u^k \|^2  + \|y^{k+1} \|^2
    \notag\\
    &+ 2\langle (1 - \tau) y^k + \tau u^k - u^{k+1}, y\rangle 
    \notag\\
    &
    + 2 \gamma^2 \| \sum_{i=1}^n Q(A_i x^{k+1/2}_i - A_i w^{k}_i) \|^2
    + 2 \gamma^2 \| z^{k+1/2} - z^k \|^2.
\end{align}
Summing up (\ref{eq:unbiased_proof_temp1}), (\ref{eq:unbiased_proof_temp2}) and (\ref{eq:unbiased_proof_temp3}), we obtain
\begin{align*}
    \|x^{k+1} - x \|^2 + \|w^{k+1}  - x \|^2 + \|z^{k+1} - z \|^2  + & \|y^{k+1} - y \|^2 + \| u^{k+1} - y \|^2   
    \notag\\
    \leq&
    \|x^{k} - x \|^2 + \| w^k - x \|^2 + \|z^{k} - z \|^2 + \|y^{k} - y \|^2 + \| u^k - y \|^2 
    \notag\\
    &
    - \tau \|  x^{k+1/2} - x^k\|^2 - (1 - \tau) \| w^k -x^{k+1/2} \|^2 - \|  z^{k+1/2} - z^k\|^2 
    \notag\\
    &
    - \tau \|  y^{k+1/2} - y^k\|^2 - (1 - \tau) \| u^k -y^{k+1/2} \|^2
    \notag\\
    &
    -
    2\gamma \langle  A  (x^{k+1/2} - x),  y^{k+1/2}\rangle
    +
    2\gamma \langle  y^{k+1/2} ,  z^{k+1/2} - z\rangle
    \\
    &-
    2\gamma \langle  z^{k+1/2},  y^{k+1/2} - y\rangle
    +
    2\gamma \langle \sum_{i=1}^n A_i x^{k+1/2}_i,  y^{k+1/2} - y\rangle
    \\
    &
    -
    2\gamma \sum\limits_{i=1}^n\langle \nabla r_i(x^{k+1/2}_i),  x^{k+1/2}_i - x_i\rangle
    -
    2\gamma \langle  \nabla \ell(z^{k+1/2}, b),  z^{k+1/2} - z\rangle
    \notag\\
    &- (1 - \tau) \|x^{k}\|^2 - \tau \| w^k \|^2  + \|w^{k+1} \|^2
    \notag\\
    &+ 2\langle (1 - \tau) x^k + \tau w^k - w^{k+1}, x\rangle 
    \notag\\
    &- (1 - \tau) \|y^{k}\|^2 - \tau \| u^k \|^2  + \|y^{k+1} \|^2
    \notag\\
    &+ 2\langle (1 - \tau) y^k + \tau u^k - u^{k+1}, y\rangle 
    \notag\\
    &
    -
    2\gamma \langle  A (x^{k+1/2} - x^0)  ,  Q(y^{k+1/2} - u^k) - y^{k+1/2} + u^k \rangle
    \notag\\
    &
    -
    2\gamma \langle  A (x^{0} - x)  ,  Q(y^{k+1/2} - u^k) - y^{k+1/2} + u^k \rangle
    \notag\\
    &
    +
    2\gamma \langle \sum_{i=1}^n [Q(A_i x^{k+1/2}_i - A_i w^{k}_i) + A_i w^{k}_i - A_i x^{k+1/2}_i ] ,  y^{k+1/2} - y^0\rangle
    \notag\\
    &
    +
    2\gamma \langle \sum_{i=1}^n [Q(A_i x^{k+1/2}_i - A_i w^{k}_i) + A_i w^{k}_i - A_i x^{k+1/2}_i ] ,  y^{0} - y\rangle
    \notag\\
    &
    + 2 \gamma^2 \| A^T Q(y^{k+1/2} - u^k) \|^2
    + 2 \gamma^2 \| \sum_{i=1}^n Q(A_i x^{k+1/2}_i - A_i w^{k}_i) \|^2
    \notag\\
    &
    + 2 \gamma^2 \sum\limits_{i=1}^n \| \nabla r_i(x^{k+1/2}_i) - \nabla r_i (x^k_i) \|^2
    + 2\gamma^2 \|y^{k+1/2} - y^k\|^2
    \notag\\
    &
    +
    2 \gamma^2 \| z^{k+1/2} - z^k \|^2 
    +
    2\gamma^2 \| \nabla \ell(z^{k+1/2}, b) -\nabla \ell(z^k, b)\|^2.
\end{align*}
After small rearrangements, we have      
\begin{align*}
    2\gamma \bigg[\langle  \nabla \ell (z^{k+1/2}, b), & z^{k+1/2} - z\rangle + \sum\limits_{i=1}^n\langle \nabla r_i(x^{k+1/2}_i),  x^{k+1/2}_i - x_i\rangle
    \\
    +
     \langle A  x^{k+1/2}  - & z^{k+1/2} ,  y\rangle 
    - \langle  A  x  - z,   y^{k+1/2}\rangle
    \bigg]
    \\
    &\hspace{1.4cm}\leq
    \|x^{k} - x \|^2 + \| w^k - x \|^2 + \|z^{k} - z \|^2 + \|y^{k} - y \|^2 + \| u^k - y \|^2
    \notag\\
    &\hspace{1.8cm}
    - \left(\|x^{k+1} - x \|^2 + \|w^{k+1} - x \|^2 + \|z^{k+1} - z \|^2  + \|y^{k+1} - y \|^2 + \| u^{k+1} - y \|^2 \right) 
    \notag\\
    &\hspace{1.8cm}
    - \tau \|  x^{k+1/2} - x^k\|^2 - (1 - \tau) \| w^k -x^{k+1/2} \|^2 - \|  z^{k+1/2} - z^k\|^2 
    \notag\\
    &\hspace{1.8cm}
    - \tau \|  y^{k+1/2} - y^k\|^2 - (1 - \tau) \| u^k -y^{k+1/2} \|^2
    \notag\\
    &\hspace{1.8cm}
    - (1 - \tau) \|x^{k}\|^2 - \tau \| w^k \|^2  + \|w^{k+1} \|^2
    \notag\\
    &\hspace{1.8cm}
    + 2\langle (1 - \tau) x^k + \tau w^k - w^{k+1}, x\rangle 
    \notag\\
    &\hspace{1.8cm}
    - (1 - \tau) \|y^{k}\|^2 - \tau \| u^k \|^2  + \|y^{k+1} \|^2
    \notag\\
    &\hspace{1.8cm}
    + 2\langle (1 - \tau) y^k + \tau u^k - u^{k+1}, y\rangle 
    \notag\\
    &\hspace{1.8cm}
    -
    2\gamma \langle  A (x^{k+1/2} - x^0)  ,  Q(y^{k+1/2} - u^k) - y^{k+1/2} + u^k \rangle
    \notag\\
    &\hspace{1.8cm}
    -
    2\gamma \langle  A (x^{0} - x)  ,  Q(y^{k+1/2} - u^k) - y^{k+1/2} + u^k \rangle
    \notag\\
    &\hspace{1.8cm}
    +
    2\gamma \langle \sum_{i=1}^n [Q(A_i x^{k+1/2}_i - A_i w^{k}_i) + A_i w^{k}_i - A_i x^{k+1/2}_i ] ,  y^{k+1/2} - y^0\rangle
    \notag\\
    &\hspace{1.8cm}
    +
    2\gamma \langle \sum_{i=1}^n [Q(A_i x^{k+1/2}_i - A_i w^{k}_i) + A_i w^{k}_i - A_i x^{k+1/2}_i ] ,  y^{0} - y\rangle
    \notag\\
    &\hspace{1.8cm}
    + 2 \gamma^2 \| A^T Q(y^{k+1/2} - u^k) \|^2
    + 2 \gamma^2 \| \sum_{i=1}^n Q(A_i x^{k+1/2}_i - A_i w^{k}_i) \|^2
    \notag\\
    &\hspace{1.8cm}
    + 2 \gamma^2 \sum\limits_{i=1}^n \| \nabla r_i(x^{k+1/2}_i) - \nabla r_i (x^k_i) \|^2
    + 2\gamma^2 \|y^{k+1/2} - y^k\|^2
    \notag\\
    &\hspace{1.8cm}
    +
    2 \gamma^2 \| z^{k+1/2} - z^k \|^2 
    +
    2\gamma^2 \| \nabla \ell(z^{k+1/2}, b) -\nabla \ell(z^k, b)\|^2.
\end{align*}
Using convexity and $L_r$-smoothness of the function $r_i$ with convexity and $L_{\ell}$-smoothness of the function $\ell$ (Assumption \ref{as:convexity_smothness}), we have
\begin{align*}
    2\gamma \bigg[\ell(z^{k+1/2}, b) - \ell(z, b) &+ \sum\limits_{i=1}^n r_i(x^{k+1/2}_i) - r_i(x_i)
    \\
    +
    \langle A  x^{k+1/2} - z^{k+1/2}, & y\rangle 
    - \langle  A  x - z,  y^{k+1/2}\rangle
    \bigg]
    \\
    &\hspace{1.cm}\leq
    \|x^{k} - x \|^2 + \| w^k - x \|^2 + \|z^{k} - z \|^2 + \|y^{k} - y \|^2 + \| u^k - y \|^2
    \notag\\
    &\hspace{1.4cm}
    - \left(\|x^{k+1} - x \|^2 + \|w^{k+1} - x \|^2 + \|z^{k+1} - z \|^2  + \|y^{k+1} - y \|^2 + \| u^{k+1} - y \|^2 \right) 
    \notag\\
    &\hspace{1.4cm}
    - \tau \|  x^{k+1/2} - x^k\|^2 - (1 - \tau) \| w^k -x^{k+1/2} \|^2 - \|  z^{k+1/2} - z^k\|^2 
    \notag\\
    &\hspace{1.4cm}
    - \tau \|  y^{k+1/2} - y^k\|^2 - (1 - \tau) \| u^k -y^{k+1/2} \|^2
    \notag\\
    &\hspace{1.4cm}
    - (1 - \tau) \|x^{k}\|^2 - \tau \| w^k \|^2  + \|w^{k+1} \|^2
    \notag\\
    &\hspace{1.4cm}
    + 2\langle (1 - \tau) x^k + \tau w^k - w^{k+1}, x\rangle 
    \notag\\
    &\hspace{1.4cm}
    - (1 - \tau) \|y^{k}\|^2 - \tau \| u^k \|^2  + \|y^{k+1} \|^2
    \notag\\
    &\hspace{1.4cm}
    + 2\langle (1 - \tau) y^k + \tau u^k - u^{k+1}, y\rangle 
    \notag\\
    &\hspace{1.4cm}
    -
    2\gamma \langle  A (x^{k+1/2} - x^0)  ,  Q(y^{k+1/2} - u^k) - y^{k+1/2} + u^k \rangle
    \notag\\
    &\hspace{1.4cm}
    -
    2\gamma \langle  A (x^{0} - x)  ,  Q(y^{k+1/2} - u^k) - y^{k+1/2} + u^k \rangle
    \notag\\
    &\hspace{1.4cm}
    +
    2\gamma \langle \sum_{i=1}^n [Q(A_i x^{k+1/2}_i - A_i w^{k}_i) + A_i w^{k}_i - A_i x^{k+1/2}_i ] ,  y^{k+1/2} - y^0\rangle
    \notag\\
    &\hspace{1.4cm}
    +
    2\gamma \langle \sum_{i=1}^n [Q(A_i x^{k+1/2}_i - A_i w^{k}_i) + A_i w^{k}_i - A_i x^{k+1/2}_i ] ,  y^{0} - y\rangle
    \notag\\
    &\hspace{1.4cm}
    + 2 \gamma^2 \| A^T Q(y^{k+1/2} - u^k) \|^2
    + 2 \gamma^2 \| \sum_{i=1}^n Q(A_i x^{k+1/2}_i - A_i w^{k}_i) \|^2
    \notag\\
    &\hspace{1.4cm}
    + 2 \gamma^2 L_r^2 \|x^{k+1/2} - x^k \|^2
    + 2\gamma^2 \|y^{k+1/2} - y^k\|^2
    \notag\\
    &\hspace{1.4cm}
    +
    2 \gamma^2 \| z^{k+1/2} - z^k \|^2 
    +
    2\gamma^2 L_{\ell}^2 \| z^{k+1/2} - z^k \|^2.
\end{align*}
Then we sum all over $k$ from $0$ to $K-1$, divide by $K$, use Jensen inequality for convex functions $\ell$ and $r_i$ with notation $\bar x^K_i = \frac{1}{K} \sum\limits_{k=0}^{K-1} x^{k+1/2}_i$, $\bar z^K = \frac{1}{K} \sum\limits_{k=0}^{K-1} z^{k+1/2}$, $\bar y^K = \frac{1}{K} \sum\limits_{k=0}^{K-1} y^{k+1/2}$, and have
\begin{align*}
    2\gamma \bigg[\ell (\bar z^K, b) - \ell (z, b) + \sum_{i=1}^n &\left( r_i (\bar x^{K}_i) - r_i (x_i)\right) 
    +
    \langle A  \bar x^{K} - \bar z^{K},  y\rangle 
    - \langle  A  x - z,  \bar y^{K}\rangle
    \bigg]
    \\
    &\hspace{1.4cm}\leq
    \frac{1}{K}\left(\|x^{0} - x \|^2 + \| w^0 - x \|^2 + \|z^{0} - z \|^2 + \|y^{0} - y \|^2 + \| u^0 - y \|^2 \right)
    \notag\\
    &\hspace{1.8cm}
    - \frac{1}{K} \left(\|x^{K} - x \|^2 + \|w^{K} - x \|^2 + \|z^{K} - z \|^2  + \|y^{K} - y \|^2 + \| u^{K} - y \|^2 \right) 
    \notag\\
    &\hspace{1.8cm}
    -  \frac{\tau}{K} \sum\limits_{k=0}^{K-1} \|  x^{k+1/2} - x^k\|^2 -  \frac{1 - \tau}{K} \sum\limits_{k=0}^{K-1} \| w^k -x^{k+1/2} \|^2
    \notag\\
    &\hspace{1.8cm}
    - \frac{1}{K} \sum\limits_{k=0}^{K-1} \|  z^{k+1/2} - z^k\|^2 
    \notag\\
    &\hspace{1.8cm}
    - \frac{\tau}{K} \sum\limits_{k=0}^{K-1} \|  y^{k+1/2} - y^k\|^2 - \frac{1 - \tau}{K} \sum\limits_{k=0}^{K-1} \| u^k -y^{k+1/2} \|^2
    \notag\\
    &\hspace{1.8cm}
    +\frac{1}{K} \sum\limits_{k=0}^{K-1} \left[ \|w^{k+1} \|^2 - (1 - \tau) \|x^{k}\|^2 - \tau \| w^k \|^2 \right] 
    \notag\\
    &\hspace{1.8cm}
    + \frac{2}{K} \sum\limits_{k=0}^{K-1}\langle (1 - \tau) x^k + \tau w^k - w^{k+1}, x\rangle 
    \notag\\
    &\hspace{1.8cm}
    +\frac{1}{K} \sum\limits_{k=0}^{K-1} \left[  \|y^{k+1} \|^2 - (1 - \tau) \|y^{k}\|^2 - \tau \| u^k \|^2 \right]
    \notag\\
    &\hspace{1.8cm}
    + \frac{2}{K} \sum\limits_{k=0}^{K-1} \langle (1 - \tau) y^k + \tau u^k - u^{k+1}, y\rangle 
    \notag\\
    &\hspace{1.8cm}
    -
    \frac{2\gamma}{K} \sum\limits_{k=0}^{K-1} \langle  A (x^{k+1/2} - x^0)  ,  Q(y^{k+1/2} - u^k) - y^{k+1/2} + u^k \rangle
    \notag\\
    &\hspace{1.8cm}
    -
    \frac{2\gamma}{K} \sum\limits_{k=0}^{K-1} \langle  A (x^{0} - x)  ,  Q(y^{k+1/2} - u^k) - y^{k+1/2} + u^k \rangle
    \notag\\
    &\hspace{1.8cm}
    +
    \frac{2\gamma}{K} \sum\limits_{k=0}^{K-1} \langle \sum_{i=1}^n [Q(A_i x^{k+1/2}_i - A_i w^{k}_i) + A_i w^{k}_i - A_i x^{k+1/2}_i ] ,  y^{k+1/2} - y^0\rangle
    \notag\\
    &\hspace{1.8cm}
    +
    \frac{2\gamma}{K} \sum\limits_{k=0}^{K-1} \langle \sum_{i=1}^n [Q(A_i x^{k+1/2}_i - A_i w^{k}_i) + A_i w^{k}_i - A_i x^{k+1/2}_i ] ,  y^{0} - y\rangle
    \notag\\
    &\hspace{1.8cm}
    + \frac{2\gamma^2}{K} \sum\limits_{k=0}^{K-1} \| A^T Q(y^{k+1/2} - u^k) \|^2
    + \frac{2\gamma^2}{K} \sum\limits_{k=0}^{K-1} \| \sum_{i=1}^n Q(A_i x^{k+1/2}_i - A_i w^{k}_i) \|^2
    \notag\\
    &\hspace{1.8cm}
    + \frac{2\gamma^2 L_r^2}{K} \sum\limits_{k=0}^{K-1} \|x^{k+1/2} - x^k \|^2
    + \frac{2\gamma^2}{K} \sum\limits_{k=0}^{K-1} \|y^{k+1/2} - y^k\|^2
    \notag\\
    &\hspace{1.8cm}
    +
    \frac{2\gamma^2}{K} \sum\limits_{k=0}^{K-1} \| z^{k+1/2} - z^k \|^2 
    +
    \frac{2\gamma^2 L_{\ell}^2}{K} \sum\limits_{k=0}^{K-1} \| z^{k+1/2} - z^k \|^2.
\end{align*}
As in (\ref{eq:basic_crit1}) we pass to the gap criterion by taking the maximum in $y \in \cY$ and the minimum in $x\in \cX$ and $z \in \cZ$. Additionally, we also take the mathematical expectation
\begin{align}
    \label{eq:unbiased_temp5}
    2\gamma \E{\text{gap}(\bar x^K, \bar z^K, \bar y^K)}
    &\leq
    \frac{1}{K}\bigg(\max_{x \in \cX}\|x^{0} - x \|^2 + \max_{x \in \cX}\| w^0 - x \|^2 + \max_{z \in \cZ}\|z^{0} - z \|^2
    \notag\\
    &\hspace{.4cm}
    + \max_{y \in \cY}\|y^{0} - y \|^2 + \max_{y \in \cY}\| u^0 - y \|^2 \bigg)
    \notag\\
    &\hspace{.4cm}
    -  \frac{\tau}{K} \sum\limits_{k=0}^{K-1} \E{\|  x^{k+1/2} - x^k\|^2} -  \frac{1 - \tau}{K} \sum\limits_{k=0}^{K-1} \E{\| w^k -x^{k+1/2} \|^2} 
    \notag\\
    &\hspace{.4cm}
    - \frac{1}{K} \sum\limits_{k=0}^{K-1} \E{\|  z^{k+1/2} - z^k\|^2}
    \notag\\
    &\hspace{.4cm}
    - \frac{\tau}{K} \sum\limits_{k=0}^{K-1} \E{\|  y^{k+1/2} - y^k\|^2} - \frac{1 - \tau}{K} \sum\limits_{k=0}^{K-1} \E{\| u^k -y^{k+1/2} \|^2}
    \notag\\
    &\hspace{.4cm}
    +\frac{1}{K} \sum\limits_{k=0}^{K-1} \E{ \|w^{k+1} \|^2 - (1 - \tau) \|x^{k}\|^2 - \tau \| w^k \|^2} 
    \notag\\
    &\hspace{.4cm}
    + \frac{2}{K} \E{\max_{x\in \cX} \sum\limits_{k=0}^{K-1} \langle (1 - \tau) x^k + \tau w^k - w^{k+1}, x\rangle} 
    \notag\\
    &\hspace{.4cm}
    +\frac{1}{K} \sum\limits_{k=0}^{K-1} \E{  \|y^{k+1} \|^2 - (1 - \tau) \|y^{k}\|^2 - \tau \| u^k \|^2 }
    \notag\\
    &\hspace{.4cm}
    + \frac{2}{K} \E{\max_{y \in \cY}\sum\limits_{k=0}^{K-1} \langle (1 - \tau) y^k + \tau u^k - u^{k+1}, y\rangle }
    \notag\\
    &\hspace{.4cm}
    -
    \frac{2\gamma}{K} \sum\limits_{k=0}^{K-1} \E{\langle  A (x^{k+1/2} - x^0)  ,  Q(y^{k+1/2} - u^k) - y^{k+1/2} + u^k \rangle}
    \notag\\
    &\hspace{.4cm}
    +
    \frac{2\gamma}{K} \cdot \E{ \max_{x \in \cX}\sum\limits_{k=0}^{K-1} \langle  A (x - x^0)  ,  Q(y^{k+1/2} - u^k) - y^{k+1/2} + u^k \rangle}
    \notag\\
    &\hspace{.4cm}
    +
    \frac{2\gamma}{K} \sum\limits_{k=0}^{K-1} \E{\langle \sum_{i=1}^n [Q(A_i x^{k+1/2}_i - A_i w^{k}_i) + A_i w^{k}_i - A_i x^{k+1/2}_i ] ,  y^{k+1/2} - y^0\rangle}
    \notag\\
    &\hspace{.4cm}
    +
    \frac{2\gamma}{K} \cdot \E{ \max_{y \in \cY}\sum\limits_{k=0}^{K-1} \langle \sum_{i=1}^n [Q(A_i x^{k+1/2}_i - A_i w^{k}_i) + A_i w^{k}_i - A_i x^{k+1/2}_i ] ,  y^{0} - y\rangle}
    \notag\\
    &\hspace{.4cm}
    + \frac{2\gamma^2}{K} \sum\limits_{k=0}^{K-1} \E{\| A^T Q(y^{k+1/2} - u^k) \|^2}
    \notag\\
    &\hspace{.4cm}
    + \frac{2\gamma^2}{K} \sum\limits_{k=0}^{K-1} \E{\| \sum_{i=1}^n Q(A_i x^{k+1/2}_i - A_i w^{k}_i) \|^2}
    \notag\\
    &\hspace{.4cm}
    + \frac{2\gamma^2 L_r^2}{K} \sum\limits_{k=0}^{K-1} \E{\|x^{k+1/2} - x^k \|^2}
    + \frac{2\gamma^2}{K} \sum\limits_{k=0}^{K-1} \E{\|y^{k+1/2} - y^k\|^2}
    \notag\\
    &\hspace{.4cm}
    +
    \frac{2\gamma^2}{K} \sum\limits_{k=0}^{K-1} \E{\| z^{k+1/2} - z^k \|^2 }
    +
    \frac{2\gamma^2 L_{\ell}^2}{K} \sum\limits_{k=0}^{K-1} \E{\| z^{k+1/2} - z^k \|^2}.
\end{align}
Next, we work with the terms of (\ref{eq:unbiased_temp5}) separately. Using that $1 - \tau = p$ and lines \ref{lin_alg:unbiased_linws} -- \ref{lin_alg:unbiased_linwf}, we get
\begin{align}
\label{eq:unbiased_temp6}
\E{ \|w^{k+1} \|^2 - (1 - \tau) \|x^{k}\|^2 - \tau \| w^k \|^2} 
&=
\E{ \mathbb{E}_{b_k} \left[\|w^{k+1} \|^2\right] - (1 - \tau) \|x^{k}\|^2 - \tau \| w^k \|^2 }
\notag
\\
&=
\E{  p \|x^{k}\|^2 + (1-p)\|w^{k}\|^2 - (1 - \tau) \|x^{k}\|^2 - \tau \| w^k \|^2 } = 0.
\end{align}
The same way we can obtain 
\begin{align}
\label{eq:unbiased_temp7}
\E{ \|u^{k+1} \|^2 - (1 - \tau) \|y^{k}\|^2 - \tau \| u^k \|^2}  = 0.
\end{align}
With $1 - \tau = p$, one can also note
\begin{align*}
\E{\max_{x\in \cX} \sum\limits_{k=0}^{K-1} \langle (1 - \tau) x^k + \tau w^k - w^{k+1}, x\rangle} 
=&
\E{\max_{x\in \cX} \sum\limits_{k=0}^{K-1} \langle (1 - \tau) x^k + \tau w^k - w^{k+1}, x \rangle} + 0
\\
=&
\E{\max_{x\in \cX} \sum\limits_{k=0}^{K-1} \langle (1 - \tau) x^k + \tau w^k - w^{k+1}, x\rangle}
\\
&+
\E{\sum\limits_{k=0}^{K-1} \langle (1 - \tau) x^k + \tau w^k - \mathbb{E}_{b_k}[w^{k+1}], - x^0\rangle}
\\
=&
\E{\max_{x\in \cX} \sum\limits_{k=0}^{K-1} \langle (1 - \tau) x^k + \tau w^k - w^{k+1}, x - x^0\rangle}.
\end{align*}
By Cauchy Schwartz  inequality: $2 \langle a, b \rangle \leq \eta \| a \|^2 + \tfrac{1}{\eta}\| b \|^2$ with $a = \sum_{k=0}^{K-1}[(1 - \tau) x^k + \tau w^k - w^{k+1}]$, $b = x - x^0$ and $\eta = \tfrac{1}{4}$, one can obtain
\begin{align}
\label{eq:unbiased_temp8}
&\E{\max_{x\in \cX} \sum\limits_{k=0}^{K-1} \langle (1 - \tau) x^k + \tau w^k - w^{k+1}, x\rangle} 
\notag\\
&\hspace{1cm}\leq
\E{\max_{x\in \cX} \left[ \frac{1}{8}\|\sum\limits_{k=0}^{K-1} [(1 - \tau) x^k + \tau w^k - w^{k+1}] \|^2 + 2 \|x - x^0\|^2 \right]}
\notag\\
&\hspace{1cm}=
\E{\max_{x\in \cX} 2 \|x - x^0\|^2} +  \E{ \frac{1}{8} \|\sum\limits_{k=0}^{K-1} [(1 - \tau) x^k + \tau w^k - w^{k+1}] \|^2}
\notag\\
&\hspace{1cm}=
\E{\max_{x\in \cX} 2 \|x - x^0\|^2} +  \frac{1}{8} \sum\limits_{k=0}^{K-1}  \E{ \| (1 - \tau) x^k + \tau w^k - w^{k+1} \|^2} 
\notag\\
&\hspace{1.4cm}+ \frac{1}{4}\sum\limits_{k_1 < k_2}  \E{ \langle (1 - \tau) x^{k_1}+ \tau w^{k_1} - w^{k_1+1}, (1 - \tau) x^{k_2}+ \tau w^{k_2} - w^{k_2+1}\rangle } 
\notag\\
&\hspace{1cm}=
\E{\max_{x\in \cX} 2 \|x - x^0\|^2} +  \frac{1}{8} \sum\limits_{k=0}^{K-1}  \E{ \| (1 - \tau) x^k + \tau w^k - w^{k+1} \|^2} 
\notag\\
&\hspace{1.4cm}+ \frac{1}{4}\sum\limits_{k_1 < k_2}  \E{ \langle (1 - \tau) x^{k_1}+ \tau w^{k_1} - w^{k_1+1}, \mathbb{E}_{b_{k_2}}[(1 - \tau) x^{k_2}+ \tau w^{k_2} - w^{k_2+1} ]\rangle } 
\notag\\
&\hspace{1cm}=
\E{\max_{x\in \cX} 2 \|x - x^0\|^2} +  \frac{1}{8} \sum\limits_{k=0}^{K-1}  \E{ \| (1 - \tau) x^k + \tau w^k - w^{k+1} \|^2}
\notag\\
&\hspace{1cm}=
\E{\max_{x\in \cX} 2 \|x - x^0\|^2} +  \frac{1}{8}\sum\limits_{k=0}^{K-1}  \E{ \| \mathbb{E}_{b_k}[w^{k+1}] - w^{k+1} \|^2}
\notag\\
&\hspace{1cm}=
\E{\max_{x\in \cX} 2 \|x - x^0\|^2} +  \frac{1}{8} \sum\limits_{k=0}^{K-1}  \E{ \mathbb{E}_{b_k}[\| w^{k+1} \|^2] - \| \mathbb{E}_{b_k}[w^{k+1}]\|^2}
\notag\\
&\hspace{1cm}=
\E{\max_{x\in \cX} 2 \|x - x^0\|^2} +  \frac{1}{8}\sum\limits_{k=0}^{K-1}  \E{ \mathbb{E}_{b_k}[\| w^{k+1} \|^2] - \| \mathbb{E}_{b_k}[w^{k+1}]\|^2}
\notag\\
&\hspace{1cm}=
\E{\max_{x\in \cX} 2 \|x - x^0\|^2} +  \frac{1}{8}\sum\limits_{k=0}^{K-1}  \E{ \tau\| w^{k} \|^2 + (1 - \tau) \| x^{k} \|^2 - \| (1-\tau) x^k + \tau w^k\|^2}
\notag\\
&\hspace{1cm}=
\E{\max_{x\in \cX} 2 \|x - x^0\|^2} +  \frac{1}{8}\sum\limits_{k=0}^{K-1}  \E{ \tau (1 - \tau) \| w^{k} - x^k \|^2}.
\end{align}
Making the same steps, one can get
\begin{align}
\label{eq:unbiased_temp9}
\E{\max_{y\in \cY} \sum\limits_{k=0}^{K-1} \langle (1 - \tau) y^k + \tau u^k - u^{k+1}, y\rangle} 
\leq
\E{\max_{y\in \cY} 2 \|y - y^0\|^2} +  \frac{1}{8}\sum\limits_{k=0}^{K-1}  \E{ \tau (1 - \tau) \| u^{k} - y^k \|^2}.
\end{align}
With unbiasedness of $Q$, we have
\begin{align}
\label{eq:unbiased_temp10}
&\E{\langle  A (x^{k+1/2} - x^0)  ,  Q(y^{k+1/2} - u^k) - y^{k+1/2} + u^k \rangle} 
\notag\\
&\hspace{4cm}=
\E{\langle  A (x^{k+1/2} - x^0)  ,  \mathbb{E}_Q[Q(y^{k+1/2} - u^k) ]- y^{k+1/2} + u^k\rangle} = 0.
\end{align}
And
\begin{align}
\label{eq:unbiased_temp11}
\E{\langle \sum_{i=1}^n [Q(A_i x^{k+1/2}_i - A_i w^{k}_i) + A_i w^{k}_i - A_i x^{k+1/2}_i ] ,  y^{k+1/2} - y^0\rangle} = 0.
\end{align}
Also with Cauchy Schwartz  inequality: $2 \langle a, b \rangle \leq \eta \| a \|^2 + \tfrac{1}{\eta}\| b \|^2$ with $a = \sum_{k=0}^{K-1}A^T[ Q(y^{k+1/2} - u^k) - y^{k+1/2} + u^k]$, $b = x - x^0$ and $\eta = \gamma$, one can obtain
\begin{align}
\label{eq:unbiased_temp12}
&\E{ \max_{x \in \cX}\sum\limits_{k=0}^{K-1} \langle  x - x^0 ,  A^T Q(y^{k+1/2} - u^k) - y^{k+1/2} + u^k \rangle}
\notag\\
&\hspace{0.5cm}\leq \E{\max_{x \in \cX} \frac{1}{2 \gamma} \| x^0 - x \|^2} + \E{\frac{\gamma}{2} \| \sum\limits_{k=0}^{K-1} A^T [Q(y^{k+1/2} - u^k) - y^{k+1/2} + u^k] \|^2}
\notag\\
&\hspace{0.5cm}= \E{\max_{x \in \cX} \frac{1}{2\gamma} \| x^0 - x \|^2} + \E{\frac{\gamma}{2} \sum\limits_{k=0}^{K-1}  \| A^T [Q(y^{k+1/2} - u^k) - y^{k+1/2} + u^k] \|^2}
\notag\\
&\hspace{0.9cm}+ \mathbb{E}\Bigg[\gamma\sum\limits_{k_1 < k_2}  \langle A^T [Q(y^{k_1+1/2} - u^{k_1}) - y^{k_1+1/2} + u^{k_1}],
\notag\\
&\hspace{5cm}
A^T [Q(y^{k_2+1/2} - u^{k_2}) - y^{k_2+1/2} + u^{k_2}] \rangle \Bigg]
\notag\\
&\hspace{0.5cm}= \E{\max_{x \in \cX} \frac{1}{2\gamma} \| x^0 - x \|^2} + \E{\frac{\gamma}{2} \sum\limits_{k=0}^{K-1}  \| A^T [Q(y^{k+1/2} - u^k) - y^{k+1/2} + u^k] \|^2}
\notag\\
&\hspace{0.9cm}+ \mathbb{E}\Bigg[\gamma\sum\limits_{k_1 < k_2}  \langle A^T [Q(y^{k_1+1/2} - u^{k_1}) - y^{k_1+1/2} + u^{k_1}], 
\notag\\
&\hspace{5cm}
A^T \mathbb{E}_{Q_{k_2}}[Q(y^{k_2+1/2} - u^{k_2}) - y^{k_2+1/2} + u^{k_2}] \rangle \Bigg]
\notag\\
&\hspace{0.5cm}= \E{\max_{x \in \cX} \frac{1}{2\gamma} \| x^0 - x \|^2} + \E{\frac{\gamma}{2} \sum\limits_{k=0}^{K-1}  \| A^T [Q(y^{k+1/2} - u^k) - y^{k+1/2} + u^k] \|^2}
\notag\\
&\hspace{0.5cm}= \E{\max_{x \in \cX} \frac{1}{2\gamma} \| x^0 - x \|^2}
\notag\\
&\hspace{0.9cm}
+ \E{\frac{\gamma}{2} \sum\limits_{k=0}^{K-1}  \mathbb{E}_{Q}\left[\| A^T [Q(y^{k+1/2} - u^k)] - \mathbb{E}_{Q}[A^T[Q(y^{k+1/2} - u^k)]] \|^2 \right]}
\notag\\
&\hspace{0.5cm}\leq \E{\max_{x \in \cX} \frac{1}{2\gamma} \| x^0 - x \|^2} + \E{\frac{\gamma}{2} \sum\limits_{k=0}^{K-1}  \| A^T [Q(y^{k+1/2} - u^k)] \|^2}.
\end{align}
The same way one can note that
\begin{align}
\label{eq:unbiased_temp13}
&\E{ \max_{y \in \cY}\sum\limits_{k=0}^{K-1} \langle \sum_{i=1}^n [Q(A_i x^{k+1/2}_i - A_i w^{k}_i) + A_i w^{k}_i - A_i x^{k+1/2}_i ] ,  y^{0} - y\rangle}
\notag\\
&\hspace{4.5cm}\leq \E{\max_{y \in \cY} \frac{1}{2\gamma} \| y^0 - y\|^2} + \E{\frac{\gamma}{2} \sum\limits_{k=0}^{K-1}  \| \sum_{i=1}^n Q(A_i x^{k+1/2}_i - A_i w^{k}_i) \|^2}.
\end{align}
Combining (\ref{eq:unbiased_temp5}) with (\ref{eq:unbiased_temp6}), (\ref{eq:unbiased_temp7}), (\ref{eq:unbiased_temp8}), (\ref{eq:unbiased_temp9}), (\ref{eq:unbiased_temp10}), (\ref{eq:unbiased_temp11}), (\ref{eq:unbiased_temp12}), (\ref{eq:unbiased_temp13}), we obtain
\begin{align*}
    2\gamma \E{\text{gap}(\bar x^K, \bar z^K, \bar y^K)}
    \leq &
    \frac{1}{K}\bigg(\max_{x \in \cX}\|x^{0} - x \|^2 + \max_{x \in \cX}\| w^0 - x \|^2 + \max_{z \in \cZ}\|z^{0} - z \|^2
    \notag\\
    &
    + \max_{y \in \cY}\|y^{0} - y \|^2 + \max_{y \in \cY}\| u^0 - y \|^2 \bigg)
    \notag\\
    &
    -  \frac{\tau}{K} \sum\limits_{k=0}^{K-1} \E{\|  x^{k+1/2} - x^k\|^2} -  \frac{1 - \tau}{K} \sum\limits_{k=0}^{K-1} \E{\| w^k -x^{k+1/2} \|^2} 
    \notag\\
    &
    - \frac{1}{K} \sum\limits_{k=0}^{K-1} \E{\|  z^{k+1/2} - z^k\|^2}
    \notag\\
    &
    - \frac{\tau}{K} \sum\limits_{k=0}^{K-1} \E{\|  y^{k+1/2} - y^k\|^2} - \frac{1 - \tau}{K} \sum\limits_{k=0}^{K-1} \E{\| u^k -y^{k+1/2} \|^2}
    \notag\\
    &
    + \frac{4}{K} \E{\max_{x\in \cX} \|x - x^0\|^2} +  \frac{1}{4K}\sum\limits_{k=0}^{K-1}  \E{ \tau (1 - \tau) \| w^{k} - x^k \|^2}
    \notag\\
    &
    + \frac{4}{K} \E{\max_{y\in \cY} \|y - y^0\|^2} +  \frac{1}{4K}\sum\limits_{k=0}^{K-1}  \E{ \tau (1 - \tau) \| u^{k} - y^k \|^2}
    \notag\\
    &
    +
    \frac{1}{K} \E{\max_{x \in \cX} \| x^0 - x \|^2} + \frac{\gamma^2}{K}\E{ \sum\limits_{k=0}^{K-1}  \| A^T [Q(y^{k+1/2} - u^k)] \|^2}
    \notag\\
    &
    +
    \frac{1}{K} \E{\max_{y \in \cY} \| y^0 - y\|^2} + \frac{\gamma^2}{K}\E{\sum\limits_{k=0}^{K-1}  \| \sum_{i=1}^n Q(A_i x^{k+1/2}_i - A_i w^{k}_i) \|^2}
    \notag\\
    &
    + \frac{3\gamma^2}{K} \sum\limits_{k=0}^{K-1} \E{\| A^T Q(y^{k+1/2} - u^k) \|^2}
    \notag\\
    &
    + \frac{3\gamma^2}{K} \sum\limits_{k=0}^{K-1} \E{\| \sum_{i=1}^n Q(A_i x^{k+1/2}_i - A_i w^{k}_i) \|^2}
    \notag\\
    &
    + \frac{2\gamma^2 L_r^2}{K} \sum\limits_{k=0}^{K-1} \E{\|x^{k+1/2} - x^k \|^2}
    + \frac{2\gamma^2}{K} \sum\limits_{k=0}^{K-1} \E{\|y^{k+1/2} - y^k\|^2}
    \notag\\
    &
    +
    \frac{2\gamma^2}{K} \sum\limits_{k=0}^{K-1} \E{\| z^{k+1/2} - z^k \|^2 }
    +
    \frac{2\gamma^2 L_{\ell}^2}{K} \sum\limits_{k=0}^{K-1} \E{\| z^{k+1/2} - z^k \|^2}
    \notag\\
    \leq &
    \frac{1}{K}\bigg(6\max_{x \in \cX}\|x^{0} - x \|^2 + \max_{x \in \cX}\| w^0 - x \|^2 + \max_{z \in \cZ}\|z^{0} - z \|^2
    \notag\\
    &
    + 6\max_{y \in \cY}\|y^{0} - y \|^2 + \max_{y \in \cY}\| u^0 - y \|^2 \bigg)
    \notag\\
    &
    -  \frac{\tau}{K} \sum\limits_{k=0}^{K-1} \E{\|  x^{k+1/2} - x^k\|^2} -  \frac{1 - \tau}{K} \sum\limits_{k=0}^{K-1} \E{\| w^k -x^{k+1/2} \|^2} 
    \notag\\
    &
    - \frac{1}{K} \sum\limits_{k=0}^{K-1} \E{\|  z^{k+1/2} - z^k\|^2}
    \notag\\
    &
    - \frac{\tau}{K} \sum\limits_{k=0}^{K-1} \E{\|  y^{k+1/2} - y^k\|^2} - \frac{1 - \tau}{K} \sum\limits_{k=0}^{K-1} \E{\| u^k -y^{k+1/2} \|^2}
    \notag\\
    &
     +  \frac{1}{4K}\sum\limits_{k=0}^{K-1}  \E{ \tau (1 - \tau) \| w^{k} - x^k \|^2}
    +  \frac{1}{4K}\sum\limits_{k=0}^{K-1}  \E{ \tau (1 - \tau) \| u^{k} - y^k \|^2}
    \notag\\
    &
    + \frac{3\gamma^2}{K} \sum\limits_{k=0}^{K-1} \E{\| A^T Q(y^{k+1/2} - u^k) \|^2}
    \notag\\
    &
    + \frac{3\gamma^2}{K} \sum\limits_{k=0}^{K-1} \E{\| \sum_{i=1}^n Q(A_i x^{k+1/2}_i - A_i w^{k}_i) \|^2}
    \notag\\
    &
    + \frac{2\gamma^2 L_r^2}{K} \sum\limits_{k=0}^{K-1} \E{\|x^{k+1/2} - x^k \|^2}
    + \frac{2\gamma^2}{K} \sum\limits_{k=0}^{K-1} \E{\|y^{k+1/2} - y^k\|^2}
    \notag\\
    &
    +
    \frac{2\gamma^2 (1 + L_{\ell}^2)}{K} \sum\limits_{k=0}^{K-1} \E{\| z^{k+1/2} - z^k \|^2 }.
\end{align*}
Applying Cauchy Schwartz  inequality and using that $\tau \leq 1$, we get
\begin{align}
\label{eq:unbiased_temp14}
    2\gamma \E{\text{gap}(\bar x^K, \bar z^K, \bar y^K)}
    \leq&
    \frac{1}{K}\bigg(6\max_{x \in \cX}\|x^{0} - x \|^2 + \max_{x \in \cX}\| w^0 - x \|^2 + \max_{z \in \cZ}\|z^{0} - z \|^2
    \notag\\
    &
    + 6\max_{y \in \cY}\|y^{0} - y \|^2 + \max_{y \in \cY}\| u^0 - y \|^2 \bigg)
    \notag\\
    &
    -  \frac{\tau}{K} \sum\limits_{k=0}^{K-1} \E{\|  x^{k+1/2} - x^k\|^2} -  \frac{1 - \tau}{K} \sum\limits_{k=0}^{K-1} \E{\| w^k -x^{k+1/2} \|^2} 
    \notag\\
    &
    - \frac{1}{K} \sum\limits_{k=0}^{K-1} \E{\|  z^{k+1/2} - z^k\|^2}
    \notag\\
    &
    - \frac{\tau}{K} \sum\limits_{k=0}^{K-1} \E{\|  y^{k+1/2} - y^k\|^2} - \frac{1 - \tau}{K} \sum\limits_{k=0}^{K-1} \E{\| u^k -y^{k+1/2} \|^2}
    \notag\\
    &
     +  \frac{1}{2K}\sum\limits_{k=0}^{K-1}  \E{  (1 - \tau) \| w^{k} - x^{k+1/2}\|^2} +  \frac{1}{2K}\sum\limits_{k=0}^{K-1}  \E{ (1 - \tau) \| x^{k+1/2} - x^k \|^2}
     \notag\\
    &
    +  \frac{1}{2K}\sum\limits_{k=0}^{K-1}  \E{  (1 - \tau) \| u^{k} - y^{k+1/2} \|^2}
    +  \frac{1}{2K}\sum\limits_{k=0}^{K-1}  \E{  (1 - \tau) \| y^{k+1/2} - y^k \|^2}
    \notag\\
    &
    + \frac{3\gamma^2}{K} \sum\limits_{k=0}^{K-1} \E{\| A^T Q(y^{k+1/2} - u^k) \|^2}
    \notag\\
    &
    + \frac{3\gamma^2}{K} \sum\limits_{k=0}^{K-1} \E{\| \sum_{i=1}^n Q(A_i x^{k+1/2}_i - A_i w^{k}_i) \|^2}
    \notag\\
    &
    + \frac{2\gamma^2 L_r^2}{K} \sum\limits_{k=0}^{K-1} \E{\|x^{k+1/2} - x^k \|^2}
    + \frac{2\gamma^2}{K} \sum\limits_{k=0}^{K-1} \E{\|y^{k+1/2} - y^k\|^2}
    \notag\\
    &
    +
    \frac{2\gamma^2 (1 + L_{\ell}^2)}{K} \sum\limits_{k=0}^{K-1} \E{\| z^{k+1/2} - z^k \|^2 }
    \notag\\
    =&
    \frac{1}{K}\bigg(6\max_{x \in \cX}\|x^{0} - x \|^2 + \max_{x \in \cX}\| w^0 - x \|^2 + \max_{z \in \cZ}\|z^{0} - z \|^2
    \notag\\
    &
    + 6\max_{y \in \cY}\|y^{0} - y \|^2 + \max_{y \in \cY}\| u^0 - y \|^2 \bigg)
    \notag\\
    &
    -  \left( \frac{3\tau - 1}{2} - 2\gamma^2 L_r^2 \right) \cdot \frac{1}{K} \sum\limits_{k=0}^{K-1} \E{\|  x^{k+1/2} - x^k\|^2} 
    \notag\\
    &
    -  \frac{1 - \tau}{2K} \sum\limits_{k=0}^{K-1} \E{\| w^k -x^{k+1/2} \|^2} 
    \notag\\
    &
    - \left( 1 - 2\gamma^2 (1 + L_{\ell}^2) \right)\frac{1}{K} \sum\limits_{k=0}^{K-1} \E{\|  z^{k+1/2} - z^k\|^2}
    \notag\\
    &
    - \left( \frac{3\tau - 1}{2} - 2 \gamma^2 \right)\frac{1}{K} \sum\limits_{k=0}^{K-1} \E{\|  y^{k+1/2} - y^k\|^2} 
    \notag\\
    &
    - \frac{1 - \tau}{2K} \sum\limits_{k=0}^{K-1} \E{\| u^k -y^{k+1/2} \|^2}
    \notag\\
    &
    + \frac{3\gamma^2}{K} \sum\limits_{k=0}^{K-1} \E{\| A^T Q(y^{k+1/2} - u^k) \|^2}
    \notag\\
    &
    + \frac{3\gamma^2}{K} \sum\limits_{k=0}^{K-1} \E{\| \sum_{i=1}^n Q(A_i x^{k+1/2}_i - A_i w^{k}_i) \|^2}.
\end{align}
Using the notation of $\lambda_{\max}(\cdot)$ as a maximum eigenvalue and the definition of unbiased compression, we get
\begin{align*}
\E{\| A^T Q(y^{k+1/2} - u^k) \|^2} 
\leq& 
\lambda_{\max}(A A^T)\E{\| Q(y^{k+1/2} - u^k) \|^2}
\\
\leq& 
\lambda_{\max}(A A^T) \omega \E{\|y^{k+1/2} - u^k \|^2}.
\end{align*}
For $\E{\| \sum_{i=1}^n Q(A_i x^{k+1/2}_i - A_i w^{k}_i) \|^2}$ we have two options. If $\sum_{i=1}^n Q(A_i x^{k+1/2}_i - A_i w^{k}_i) = Q(\sum_{i=1}^n [A_i x^{k+1/2}_i - A_i w^{k}_i]) = Q(A x^{k+1/2} - A w^{k})$, then
\begin{align*}
\E{\| \sum_{i=1}^n Q(A_i x^{k+1/2}_i - A_i w^{k}_i) \|^2}
=& 
\E{\| Q(A x^{k+1/2} - A w^{k}) \|^2}
\\
\leq& 
\omega \E{\| A (x^{k+1/2} - w^{k}) \|^2}
\\
\leq& \lambda_{\max}(A^T A) \omega \E{\| x^{k+1/2} - w^{k} \|^2}.
\end{align*}
If $\sum_{i=1}^n Q(A_i x^{k+1/2}_i - A_i w^{k}_i) \neq Q(\sum_{i=1}^n [A_i x^{k+1/2}_i - A_i w^{k}_i])$, but $Q$ are independent, then
\begin{align*}
\E{\| \sum_{i=1}^n Q(A_i x^{k+1/2}_i - A_i w^{k}_i) \|^2}
=& 
\sum_{i=1}^n \E{\| Q(A_i x^{k+1/2}_i - A_i w^{k}_i) \|^2}
\\
&+\sum_{i \neq j} \E{\langle Q(A_i x^{k+1/2}_i - A_i w^{k}_i), Q(A_j x^{k+1/2}_j - A_j w^{k}_j) \rangle}
\\
=& 
\sum_{i=1}^n \E{\| Q(A_i x^{k+1/2}_i - A_i w^{k}_i) \|^2}
\\
&+\sum_{i \neq j} \E{\langle \mathbb{E}_{Q_i}[Q(A_i x^{k+1/2}_i - A_i w^{k}_i)], \mathbb{E}_{Q_j}[Q(A_j x^{k+1/2}_j - A_j w^{k}_j)] \rangle}
\\
=&  
\sum_{i=1}^n \E{\| Q(A_i x^{k+1/2}_i - A_i w^{k}_i) \|^2}
\\
&+\sum_{i \neq j} \E{\langle A_i x^{k+1/2}_i - A_i w^{k}_i, A_j x^{k+1/2}_j - A_j w^{k}_j \rangle}
\\
=& \sum_{i=1}^n \E{\| Q(A_i x^{k+1/2}_i - A_i w^{k}_i) \|^2}
\\
&+ \E{\| \sum_{i=1}^n [A_i x^{k+1/2}_i - A_i w^{k}_i] \|^2}
- \sum_{i=1}^n \E{\| A_i x^{k+1/2}_i - A_i w^{k}_i \|^2}
\\
\leq& \omega \sum_{i=1}^n \E{\|A_i x^{k+1/2}_i - A_i w^{k}_i \|^2}
 + \E{\| A (x^{k+1/2} - w^{k}) \|^2}
 \\
\leq& \omega \sum_{i=1}^n \lambda_{\max} (A_i^T A_i) \E{\|x^{k+1/2}_i - w^{k}_i \|^2}
 + \lambda_{\max} (A^T A) \E{\| x^{k+1/2} - w^{k} \|^2}
 \\
\leq& \omega \max_{i} \left\{\lambda_{\max} (A_i^T A_i) \right\}\sum_{i=1}^n \E{\|x^{k+1/2}_i - w^{k}_i \|^2}
 \\
&
 + \lambda_{\max} (A^T A) \E{\| x^{k+1/2} - w^{k} \|^2}
  \\
=& \left(\omega \max_{i} \left\{\lambda_{\max} (A_i^T A_i) \right\}+ \lambda_{\max} (A^T A) \right) \E{\| x^{k+1/2} - w^{k} \|^2}.
\end{align*}
Let us introduce
$$
\chi_{\text{compress}} =
\begin{cases}
\omega \lambda_{\max} (A^T A),\\
\omega \max_{i} \left\{\lambda_{\max} (A_i^T A_i) \right\}+ \lambda_{\max} (A^T A),
\end{cases}
$$
depending on the case $Q$ we consider.
Let us return to (\ref{eq:unbiased_temp14}) and obtain
\begin{align*}
    2\gamma \E{\text{gap}(\bar x^K, \bar z^K, \bar y^K)}
    \leq&
    \frac{1}{K}\bigg(6\max_{x \in \cX}\|x^{0} - x \|^2 + \max_{x \in \cX}\| w^0 - x \|^2 + \max_{z \in \cZ}\|z^{0} - z \|^2
    \notag\\
    &
    + 6\max_{y \in \cY}\|y^{0} - y \|^2 + \max_{y \in \cY}\| u^0 - y \|^2 \bigg)
    \notag\\
    &
    -  \left( \frac{3\tau - 1}{2} - 2\gamma^2 L_r^2 \right) \cdot \frac{1}{K} \sum\limits_{k=0}^{K-1} \E{\|  x^{k+1/2} - x^k\|^2} 
    \notag\\
    &
    -  \left( \frac{1 - \tau}{2} - 3 \chi_{\text{compress}} \gamma^2 \right) \frac{1}{K} \sum\limits_{k=0}^{K-1} \E{\| w^k -x^{k+1/2} \|^2} 
    \notag\\
    &
    - \left( 1 - 2\gamma^2 (1 + L_{\ell}^2) \right)\frac{1}{K} \sum\limits_{k=0}^{K-1} \E{\|  z^{k+1/2} - z^k\|^2}
    \notag\\
    &
    - \left( \frac{3\tau - 1}{2} - 2 \gamma^2 \right)\frac{1}{K} \sum\limits_{k=0}^{K-1} \E{\|  y^{k+1/2} - y^k\|^2}
    \notag\\
    &
    - \left( \frac{1 - \tau}{2} - 3\lambda_{\max}(A A^T) \omega \gamma^2 \right)\frac{1}{K} \sum\limits_{k=0}^{K-1} \E{\| u^k -y^{k+1/2} \|^2}.
\end{align*}
If we choose $\tau \geq \tfrac{1}{2}$ and $\gamma$ as follows
$$
\gamma \leq \frac{1}{4}\min\left\{ 1;\frac{1}{L_r}; \frac{1}{L_{\ell}};\sqrt{\frac{1-\tau}{\chi_{\text{compress}}}}; \sqrt{\frac{1-\tau}{\omega\lambda_{\max}(A A^T)}}; \right\},
$$
then one can obtain
\begin{align*}
    &\E{\text{gap}(\bar x^K, \bar z^K, \bar y^K)}
    \leq
    \frac{1}{2\gamma K}\bigg(6\max_{x \in \cX}\|x^{0} - x \|^2 + \max_{x \in \cX}\| w^0 - x \|^2 + \max_{z \in \cZ}\|z^{0} - z \|^2
    \notag\\
    &\hspace{7.5cm}
    + 6\max_{y \in \cY}\|y^{0} - y \|^2 + \max_{y \in \cY}\| u^0 - y \|^2 \bigg).
\end{align*}
With $\gamma = \frac{1}{4}\min\left\{ 1;\frac{1}{L_r}; \frac{1}{L_{\ell}};\sqrt{\frac{1-\tau}{\chi_{\text{compress}}}}; \sqrt{\frac{1-\tau}{\omega\lambda_{\max}(A A^T)}}; \right\}$, we finish the proof.
\end{proof}

\subsection{Proof of Theorem \ref{th:EG_compressed_biased}} \label{sec:proof_EG_compressed_biased} \label{app:th:EG_compressed_biased}

{
\begin{theorem}[Theorem \ref{th:EG_compressed_biased}]
Let Assumption \ref{as:convexity_smothness} hold. Let the problem (\ref{eq:vfl_lin_spp_1})
be solved by Algorithm~\ref{alg:EG_biased} with operators and $C$, which satisfy Definition \ref{def:biased}. Then for
$\tau = 1 - p$ and
$$
\gamma = \frac{1}{4}\min\left\{ 1;\frac{1}{L_r}; \frac{1}{L_{\ell}};\sqrt{\frac{1-\tau}{\delta^2[\lambda_{\max}(A A^T) + n \cdot \max_{i}\{\lambda_{\max}(A_i A^T_i)\}]}}; \sqrt{\frac{1-\tau}{\omega\lambda_{\max}(A A^T)}} \right\},
$$
it holds that
\begin{align*}
&\E{\text{gap}(\bar x^K, \bar z^K, \bar y^K)}
\\
& = \mathcal{O} \left(
     \left[1 + \frac{\delta}{\sqrt{p}}  \left(\sqrt{\lambda_{\max}(A A^T) } +  n \cdot \max\limits_{i = 1, \ldots, n} \{\sqrt{\lambda_{\max}(A_i A^T_i) } \} \right) + L_{\ell} + L_r \right] \cdot \frac{D^2}{K} \right),
\end{align*}
where $\bar x^K := \tfrac{1}{K}\sum_{k=0}^{K-1} x^{k+1/2}$, $\bar z^K := \tfrac{1}{K}\sum_{k=0}^{K-1} z^{k+1/2}$, $\bar y^K := \tfrac{1}{K}\sum_{k=0}^{K-1} y^{k+1/2}$ and \\ $D^2 := \max_{x, z, y \in \cX, \cZ, \cY} \left[ \|x^0 - x \|^2 + \|z^0 - z \|^2 + \|y^0 - y \|^2 \right]$.
\end{theorem}
}
To begin with, let us introduce the useful notation for the further proof:
\begin{equation}
    \label{eq:hat_add}
    \begin{split}
    \hat x^{k}_i = x^k_i - \gamma A_i^T e^k, \quad \hat x^{k+1/2}_i = x^{k+1/2}_i - \gamma A_i^T e^k, \\
    \hat y^{k} = y^k - \gamma \sum\limits_{i=1}^n e^k_i, \quad \hat y^{k+1/2} = y^{k+1/2} - \gamma \sum\limits_{i=1}^n e^k_i.
    \end{split}
\end{equation}
It is easy to verify that such sequences have useful properties:
\begin{align}
    \label{eq:hat_add_seq}
    \hat x^{k+1}_i 
    =& 
    x^{k+1}_i - \gamma A_i^T e^{k+1}
    \notag\\
    =& 
    \tau x^{k}_i + (1 - \tau) w^k_i - \gamma (A_i^T  [C(y^{k+1/2} - u^k + e^k) + u^k] + \nabla r_i(x^{k+1/2}_i)) - 
    \gamma A_i^T e^k
    \notag\\
    &
    - 
    \gamma A_i^T \left(y^{k+1/2} - u^k + e^k - C(y^{k+1/2} - u^k)\right)
    \notag\\
    =&
    \tau x^{k}_i + (1 - \tau) w^k_i - \gamma (A_i^T  u^k + \nabla r_i(x^{k}_i)) 
    \notag\\
    &
    - 
    \gamma A_i^T \left(y^{k+1/2} - u^k + e^k\right)
    - \gamma ( \nabla r_i(x^{k+1/2}_i) - \nabla r_i(x^{k}_i))
    \notag\\
    =&
    \hat x^{k+1/2}_i 
    - 
    \gamma A_i^T \left(y^{k+1/2} - u^k\right)
    - \gamma ( \nabla r_i(x^{k+1/2}_i) - \nabla r_i(x^{k}_i)),
\end{align}
and 
\begin{align*}
    \hat y^{k+1}
    =& 
    y^{k+1} + \gamma \sum\limits_{i=1}^n e^{k+1}_i
    \\
    =& 
    \tau y^k + (1 - \tau) u^k + \gamma \left(\sum_{i=1}^n [C(A_i x^{k+1/2}_i - A_i w^{k}_i + e^k_i) + A_i w^{k}_i ] - z^{k+1/2}\right)
    \\
    &
    + 
    \gamma \sum\limits_{i=1}^n [A_i x^{k+1/2}_i - A_i w^{k}_i + e^k_i - C(A_i x^{k+1/2}_i - A_i w^{k}_i + e^k_i) ]
    \\
    =&
    \tau y^k + (1 - \tau) u^k + \gamma \left(\sum_{i=1}^n A_i w^{k}_i  - z^{k}\right) - \gamma \sum\limits_{i=1}^n e^k_i
    \\
    &
    + 
    \gamma \sum\limits_{i=1}^n [A_i x^{k+1/2}_i - A_i w^{k}_i] - \gamma (z^{k+1/2} - z^k)
    \\
    =&
    \hat y^{k+1/2} + 
    \gamma \sum\limits_{i=1}^n [A_i x^{k+1/2}_i - A_i w^{k}_i] - \gamma (z^{k+1/2} - z^k).
\end{align*}
Now we are ready to start the proof.
\begin{proof}
We start the proof with the following equations on the variables $\hat x^{k+1}_i$, $x^{k+1/2}_i$, $\hat x^k_i$ and any $x_i \in \R^{d_i}$:
\begin{align*}
    \|\hat x^{k+1}_i - x_i \|^2
    &=
    \|x^{k+1/2}_i - x_i \|^2 + 2 \langle  \hat x^{k+1}_i -  x^{k+1/2}_i,  x^{k+1/2}_i - x_i\rangle + \|  \hat x^{k+1}_i -  x^{k+1/2}_i\|^2,
    \\
    \|x^{k+1/2}_i - x_i \|^2
    &=
    \| \hat x^{k}_i - x_i \|^2 + 2 \langle x^{k+1/2}_i -   \hat x^{k}_i, x^{k+1/2}_i - x_i\rangle - \|  x^{k+1/2}_i - \hat x^k_i\|^2.
\end{align*}
Summing up two previous equations and making small rearrangements, we get
\begin{align}
    \label{eq:biased_proof_0}
    \|\hat x^{k+1}_i - x_i \|^2
    =& \| \hat x^{k}_i - x_i \|^2
    + 2 \langle  \hat x^{k+1}_i -  \hat x^{k}_i,  x^{k+1/2}_i - x_i\rangle 
    \notag\\
    &
    + \|  \hat x^{k+1}_i -  x^{k+1/2}_i\|^2 - \|  x^{k+1/2}_i - \hat x^k_i\|^2.
\end{align}
Using the definitions (\ref{eq:hat_add}) and (\ref{eq:hat_add_seq}), one can obtain
\begin{align}
    \label{eq:biased_proof_1}
    \|  \hat x^{k+1}_i -   x^{k+1/2}_i\|^2 
    \leq& 
    2\|  \hat x^{k+1}_i -   \hat x^{k+1/2}_i\|^2 + 2\|  \hat x^{k+1/2}_i -   x^{k+1/2}_i\|^2
    \notag\\
    =&
    2\gamma^2\| A_i^T (y^{k+1/2} - u^k) -  (\nabla r_i(x^{k+1/2}_i) - \nabla r_i(x^{k}_i))\|^2 + 2 \gamma^2 \|   A_i^T e^k \|^2
    \notag\\
    \leq& 
    4\gamma^2\| A_i^T (y^{k+1/2} - u^k)\|^2 + 4\gamma^2\| \nabla r_i(x^{k+1/2}_i) - \nabla r_i(x^{k}_i)\|^2 
    \notag\\
    &+ 2 \gamma^2 \|   A_i^T e^k \|^2.
\end{align}
With  (\ref{eq:hat_add}), (\ref{eq:hat_add_seq}) and the update for $x^{k+1/2}_i$, we have
\begin{align}
    \label{eq:biased_proof_2}
    \hat x^{k+1}_i -  \hat x^{k}_i 
    =& 
    \hat x^{k+1}_i - \hat x^{k+1/2}_i + \hat x^{k+1/2}_i -  \hat x^{k}_i
    \notag\\
    =&
    \hat x^{k+1}_i - \hat x^{k+1/2}_i + x^{k+1/2}_i -  x^{k}_i
    \notag\\
    =& - 
    \gamma A_i^T \left(y^{k+1/2} - u^k\right)
    - \gamma ( \nabla r_i(x^{k+1/2}_i) - \nabla r_i(x^{k}_i))
    \notag\\
    &+ (1 - \tau) (w^k_i - x^k_i) - \gamma \left(A_i^T u^k + \nabla r_i (x^k_i) \right)
    \notag\\
    =& -
    \gamma (A_i^T y^{k+1/2}
    + \nabla r_i(x^{k+1/2}_i))
    + (1 - \tau) (w^k_i - x^k_i).
\end{align}
Combining (\ref{eq:biased_proof_0}), (\ref{eq:biased_proof_1}), (\ref{eq:biased_proof_2}), we get
\begin{align*}
    \|\hat x^{k+1}_i - x_i \|^2
    \leq& \| \hat x^{k}_i - x_i \|^2
    - 2 \langle  \gamma (A_i^T y^{k+1/2}
    + \nabla r_i(x^{k+1/2}_i))
    - (1 - \tau) (w^k_i - x^k_i),  x^{k+1/2}_i - x_i\rangle 
    \\
    &
    + 4\gamma^2\| A_i^T (y^{k+1/2} - u^k)\|^2 + 4\gamma^2\| \nabla r_i(x^{k+1/2}_i) - \nabla r_i(x^{k}_i)\|^2 
    \notag\\
    &+ 2 \gamma^2 \|   A_i^T e^k \|^2
    - \|  x^{k+1/2}_i - \hat x^k_i\|^2
    \\
    \leq&
    \| \hat x^{k}_i - x_i \|^2
    - 2  \langle  \gamma (A_i^T y^{k+1/2}
    + \nabla r_i(x^{k+1/2}_i)),  x^{k+1/2}_i - x_i\rangle 
    \\
    &+ 2 (1 - \tau) \langle 
    w^k_i - x^{k+1/2}_i,  x^{k+1/2}_i - x_i\rangle 
    \\
    &+ 2 (1 - \tau) \langle 
    x^{k+1/2}_i - x^k_i,  x^{k+1/2}_i - x_i\rangle 
    \\
    &
    + 4\gamma^2\| A_i^T (y^{k+1/2} - u^k)\|^2 + 4\gamma^2\| \nabla r_i(x^{k+1/2}_i) - \nabla r_i(x^{k}_i)\|^2 
    \notag\\
    &+ 2 \gamma^2 \|   A_i^T e^k \|^2
    - \frac{1}{2}\|  x^{k+1/2}_i - x^k_i\|^2 + \|  \hat x^{k}_i - x^k_i\|^2
    \\
    =&
    \| \hat x^{k}_i - x_i \|^2
    - 2  \langle  \gamma (A_i^T y^{k+1/2}
    + \nabla r_i(x^{k+1/2}_i)),  x^{k+1/2}_i - x_i\rangle 
    \\
    &+ 2 (1 - \tau) \langle 
    w^k_i - x^{k+1/2}_i,  x^{k+1/2}_i - x_i\rangle 
    \\
    &+ 2 (1 - \tau) \langle 
    x^{k+1/2}_i - x^k_i,  x^{k+1/2}_i - x_i\rangle 
    \\
    &
    + 4\gamma^2\| A_i^T (y^{k+1/2} - u^k)\|^2 + 4\gamma^2\| \nabla r_i(x^{k+1/2}_i) - \nabla r_i(x^{k}_i)\|^2 
    \notag\\
    &+ 2 \gamma^2 \|   A_i^T e^k \|^2
    - \frac{1}{2}\|  x^{k+1/2}_i - x^k_i\|^2 + \gamma^2 \|  A^T_i e^k\|^2.
\end{align*}
In the last two steps we use (\ref{eq:hat_add}) and Cauchy Schwartz inequality in the form 
$-\| a\|^2 \leq - \tfrac{1}{2} \| a + b\|^2 + \| b\|^2$ with $a = x^{k+1/2}_i - \hat x^k_i$ and $b = \hat x^k_i - x^k_i$. For the second and third lines we use identity $2\langle a, b \rangle = \| a  + b\|^2 - \| a\|^2 - \| b\|^2$, and have
\begin{align*}
    \|\hat x^{k+1}_i - x_i \|^2
    \leq& 
    \| \hat x^{k}_i - x_i \|^2
    - 2 \gamma \langle  A_i^T y^{k+1/2}
    + \nabla r_i(x^{k+1/2}_i),  x^{k+1/2}_i - x_i\rangle 
    \notag\\
    &+
    (1 - \tau) ( \| w^k_i - x_i \|^2 - \| w^k_i -x^{k+1/2}_i \|^2 - \| x^{k+1/2}_i - x_i \|^2 ) 
    \notag\\
    &+
    (1 - \tau) ( \| x^{k+1/2}_i -x^k_i \|^2 + \| x^{k+1/2}_i - x_i \|^2 - \| x^k_i - x_i\|^2 )
    \notag\\
    &
    + 4\gamma^2\| A_i^T (y^{k+1/2} - u^k)\|^2 + 4\gamma^2\| \nabla r_i(x^{k+1/2}_i) - \nabla r_i(x^{k}_i)\|^2 
    \notag\\
    &+ 3 \gamma^2 \|   A_i^T e^k \|^2
    - \frac{1}{2}\|  x^{k+1/2}_i - x^k_i\|^2
    \notag\\
    =&
    \| \hat x^{k}_i - x_i \|^2 - (1 - \tau) \| x^k_i - x_i\|^2 + (1 - \tau) \| w^k_i - x_i \|^2
    \notag\\
    &
    - 2 \gamma \langle  A_i^T y^{k+1/2}
    + \nabla r_i(x^{k+1/2}_i),  x^{k+1/2}_i - x_i\rangle 
    \notag\\
    &
    + 4\gamma^2\| A_i^T (y^{k+1/2} - u^k)\|^2 + 4\gamma^2\| \nabla r_i(x^{k+1/2}_i) - \nabla r_i(x^{k}_i)\|^2 
    \notag\\
    &+ 3 \gamma^2 \|   A_i^T e^k \|^2
    - \left( \tau - \frac{1}{2}\right)\|  x^{k+1/2}_i - x^k_i\|^2 - (1 - \tau) \| w^k_i -x^{k+1/2}_i \|^2.
\end{align*}
Summing over all $i$ from $1$ to $n$, we deduce
\begin{align*}
    \sum\limits_{i=1}^n\|\hat x^{k+1}_i - x_i \|^2
    \leq& 
    \sum\limits_{i=1}^n \| \hat x^{k}_i - x_i \|^2 - (1 - \tau) \sum\limits_{i=1}^n \| x^k_i - x_i\|^2 + (1 - \tau) \sum\limits_{i=1}^n \| w^k_i - x_i \|^2
    \notag\\
    &
    - 2 \gamma \sum\limits_{i=1}^n \langle  A_i^T y^{k+1/2}
    + \nabla r_i(x^{k+1/2}_i),  x^{k+1/2}_i - x_i\rangle 
    \notag\\
    &
    + 4\gamma^2 \sum\limits_{i=1}^n \| A_i^T (y^{k+1/2} - u^k)\|^2 + 4\gamma^2 \sum\limits_{i=1}^n \| \nabla r_i(x^{k+1/2}_i) - \nabla r_i(x^{k}_i)\|^2 
    \notag\\
    &+ 3 \gamma^2 \sum\limits_{i=1}^n \|   A_i^T e^k \|^2
    - \left( \tau - \frac{1}{2}\right) \sum\limits_{i=1}^n \|  x^{k+1/2}_i - x^k_i\|^2 
    \notag\\
    &- (1 - \tau) \sum\limits_{i=1}^n \| w^k_i -x^{k+1/2}_i \|^2.
\end{align*}
With notation of $A = [A_1, \ldots, A_i, \ldots, A_n] $, $x = [x_1^T, \ldots,x_i^T, \ldots,x_n^T]^T$, $\hat x = [\hat x_1^T, \ldots,\hat x_i^T, \ldots,\hat x_n^T]^T$ and $w = [w_1^T, \ldots,w_i^T, \ldots,w_n^T]^T$, one can obtain that $\sum_{i=1}^n A_i x_i = A x$, $\sum_{i=1}^n\| A^T_i e^k \| = \|A^T e^k \|^2$ and $\sum_{i=1}^n \| A_i^T (y^{k+1/2} - u^k)\|^2 = \| A^T (y^{k+1/2} - u^k) \|^2$:
\begin{align}
    \label{eq:biased_proof_3}
    \|\hat x^{k+1} - x \|^2
    \leq& 
    \| \hat x^{k} - x \|^2 - (1 - \tau) \| x^k - x\|^2 + (1 - \tau) \| w^k - x \|^2
    \notag\\
    &
    - 2 \gamma \langle  y^{k+1/2},  A (x^{k+1/2} - x)\rangle 
    - 2 \gamma \sum\limits_{i=1}^n \langle  \nabla r_i(x^{k+1/2}_i),  x^{k+1/2}_i - x_i\rangle 
    \notag\\
    &
    + 4\gamma^2 \| A^T (y^{k+1/2} - u^k)\|^2 + 4\gamma^2 \sum\limits_{i=1}^n \| \nabla r_i(x^{k+1/2}_i) - \nabla r_i(x^{k}_i)\|^2 
    \notag\\
    &+ 3 \gamma^2 \|   A^T e^k \|^2
    - \left( \tau - \frac{1}{2}\right) \|  x^{k+1/2} - x^k\|^2 - (1 - \tau) \| w^k -x^{k+1/2} \|^2.
\end{align}
One can note that the updates for the variable $z$ from lines \ref{lin_alg:biased_linz1} and \ref{lin_alg:biased_linz2} of Algorithm \ref{alg:EG_biased} are the same as those from lines \ref{lin_alg:EG_linz1} and \ref{lin_alg:EG_linz2} of Algorithm \ref{alg:EG}. Therefore, we can simply use (\ref{eq:basic_proof_temp2}), i.e. for $z \in \R^s$ it holds
\begin{align}
    \label{eq:biased_proof_temp2}
    \|z^{k+1} - z \|^2
    \leq&
    \|z^{k} - z \|^2 - \|  z^{k+1/2} - z^k\|^2 
    \notag\\
    &+
    2\gamma \langle  y^{k+1/2} ,  z^{k+1/2} - z\rangle
    -
    2\gamma \langle  \nabla \ell(z^{k+1/2}, b),  z^{k+1/2} - z\rangle
    \notag\\
    &+
    2\gamma^2 \|y^{k+1/2} - y^k\|^2 +
    2\gamma^2 \| \nabla \ell(z^{k+1/2}, b) -\nabla \ell(z^k, b)\|^2.
\end{align}
For the updates of the variable $y$ from lines (\ref{lin_alg:biased_liny1}), (\ref{lin_alg:biased_liny2}) and from (\ref{eq:hat_add}), we can repeat the same steps as in obtaining (\ref{eq:biased_proof_3}). In particular, for all $y \in \R^s$, we get
\begin{align}
    \label{eq:biased_proof_4}
    \|\hat y^{k+1} - y \|^2
    \leq& 
    \| \hat y^{k} - y \|^2 - (1 - \tau) \| y^k - y\|^2 + (1 - \tau) \| u^k - y \|^2
    \notag\\
    &
    + 2 \gamma \langle  \sum_{i=1}^n A_i x^{k+1/2}_i - z^{k+1/2} ,  y^{k+1/2} - y\rangle 
    \notag\\
    &
    + 4\gamma^2 \left\| \sum_{i=1}^n [A_i x^{k+1/2}_i - A_i w^{k}_i] \right\|^2 + 4\gamma^2\| z^{k+1/2} - z^{k}\|^2 
    \notag\\
    &+ 3 \gamma^2 \left\|   \sum\limits_{i=1}^n e^k_i \right\|^2
    - \left( \tau - \frac{1}{2}\right)\|  y^{k+1/2} - y^k\|^2 - (1 - \tau) \| u^k -y^{k+1/2} \|^2
    \notag\\
    =& 
    \| \hat y^{k} - y \|^2 - (1 - \tau) \| y^k - y\|^2 + (1 - \tau) \| u^k - y \|^2
    \notag\\
    &
    + 2 \gamma \langle  A x^{k+1/2} - z^{k+1/2} ,  y^{k+1/2} - y\rangle 
    \notag\\
    &
    + 4\gamma^2 \left\| A x^{k+1/2} - A w^{k} \right\|^2 + 4\gamma^2\| z^{k+1/2} - z^{k}\|^2 
    \notag\\
    &+ 3 \gamma^2 \left\|   \sum\limits_{i=1}^n e^k_i \right\|^2
    - \left( \tau - \frac{1}{2}\right)\|  y^{k+1/2} - y^k\|^2 - (1 - \tau) \| u^k -y^{k+1/2} \|^2.
\end{align}
Here we also use the notation of $A$ and $x$. Summing up (\ref{eq:biased_proof_3}), (\ref{eq:biased_proof_temp2}) and (\ref{eq:biased_proof_4}), we obtain
\begin{align*}
    \|\hat x^{k+1} - x \|^2 + \|z^{k+1} - z \|^2 &+ \|\hat y^{k+1} - y \|^2
    \\
    &\hspace{1cm}\leq
    \| \hat x^{k} - x \|^2 + \|z^{k} - z \|^2  + \| \hat y^{k} - y \|^2
    \notag\\
    &\hspace{1.4cm}
    - (1 - \tau) \| x^k - x\|^2 + (1 - \tau) \| w^k - x \|^2
    - (1 - \tau) \| y^k - y\|^2 + (1 - \tau) \| u^k - y \|^2
    \notag\\
    &\hspace{1.4cm}
    - 2 \gamma \langle  y^{k+1/2},  A (x^{k+1/2} - x)\rangle 
    - 2 \gamma \sum\limits_{i=1}^n \langle  \nabla r_i(x^{k+1/2}_i),  x^{k+1/2}_i - x_i\rangle
    \notag\\
    &\hspace{1.4cm}
    +
    2\gamma \langle  y^{k+1/2} ,  z^{k+1/2} - z\rangle
    -
    2\gamma \langle  \nabla \ell(z^{k+1/2}, b),  z^{k+1/2} - z\rangle
    \notag\\
    &\hspace{1.4cm}
    + 2 \gamma \langle  A x^{k+1/2} - z^{k+1/2} ,  y^{k+1/2} - y\rangle 
    \notag\\
    &\hspace{1.4cm}
    - \left( \tau - \frac{1}{2}\right) \|  x^{k+1/2} - x^k\|^2 - (1 - \tau) \| w^k -x^{k+1/2} \|^2 - \|  z^{k+1/2} - z^k\|^2
    \notag\\
    &\hspace{1.4cm}
    - \left( \tau - \frac{1}{2}\right)\|  y^{k+1/2} - y^k\|^2 - (1 - \tau) \| u^k -y^{k+1/2} \|^2
    \notag\\
    &\hspace{1.4cm}
    + 4\gamma^2 \| A^T (y^{k+1/2} - u^k)\|^2 + 4\gamma^2 \sum\limits_{i=1}^n \| \nabla r_i(x^{k+1/2}_i) - \nabla r_i(x^{k}_i)\|^2 
    \notag\\
    &\hspace{1.4cm}
    +
    2\gamma^2 \|y^{k+1/2} - y^k\|^2 +
    2\gamma^2 \| \nabla \ell(z^{k+1/2}, b) -\nabla \ell(z^k, b)\|^2
    \notag\\
    &\hspace{1.4cm}
    + 4\gamma^2 \left\| A x^{k+1/2} - A w^{k} \right\|^2 + 4\gamma^2\| z^{k+1/2} - z^{k}\|^2
    \notag\\
    &\hspace{1.4cm}
    + 3 \gamma^2 \|   A^T e^k \|^2 + 3 \gamma^2 \left\|   \sum\limits_{i=1}^n e^k_i \right\|^2.
\end{align*}
Using convexity and $L_r$-smoothness of the function $r_i$ with convexity and $L_{\ell}$-smoothness of the function $\ell$, we have
\begin{align*}
    \|\hat x^{k+1} - x \|^2 + \|z^{k+1} - z \|^2 &+ \|\hat y^{k+1} - y \|^2
    \\
    &\hspace{1cm}\leq
    \| \hat x^{k} - x \|^2 + \|z^{k} - z \|^2  + \| \hat y^{k} - y \|^2
    \notag\\
    &\hspace{1.4cm}
    - (1 - \tau) \| x^k - x\|^2 + (1 - \tau) \| w^k - x \|^2
    - (1 - \tau) \| y^k - y\|^2 + (1 - \tau) \| u^k - y \|^2
    \notag\\
    &\hspace{1.4cm}
    + 2 \gamma \langle  y^{k+1/2},  Ax - z\rangle  - 2 \gamma \langle  A x^{k+1/2} - z^{k+1/2} ,  y\rangle 
    \notag\\
    &\hspace{1.4cm}
    -
    2\gamma (\ell(z^{k+1/2}, b) - \ell(z, b)) - 2 \gamma \sum\limits_{i=1}^n ( r_i(x^{k+1/2}_i)  - r_i(x_i))
    \notag\\
    &\hspace{1.4cm}
    - \left( \tau - \frac{1}{2}\right) \|  x^{k+1/2} - x^k\|^2 - (1 - \tau) \| w^k -x^{k+1/2} \|^2 - \|  z^{k+1/2} - z^k\|^2
    \notag\\
    &\hspace{1.4cm}
    - \left( \tau - \frac{1}{2}\right)\|  y^{k+1/2} - y^k\|^2 - (1 - \tau) \| u^k -y^{k+1/2} \|^2
    \notag\\
    &\hspace{1.4cm}
    + 4\gamma^2 \lambda_{\max} (A A^T) \| y^{k+1/2} - u^k\|^2 + 4\gamma^2 L_r^2 \| x^{k+1/2} - x^{k}\|^2 
    \notag\\
    &\hspace{1.4cm}
    +
    2\gamma^2 \|y^{k+1/2} - y^k\|^2 +
    2\gamma^2 L^2_{\ell} \| z^{k+1/2} - z^k\|^2
    \notag\\
    &\hspace{1.4cm}
    + 4\gamma^2 \lambda_{\max} (A^T A) \| x^{k+1/2} - w^{k} \|^2 + 4\gamma^2\| z^{k+1/2} - z^{k}\|^2
    \notag\\
    &\hspace{1.4cm}
    + 3 \gamma^2 \lambda_{\max} (A A^T) \|  e^k \|^2 + 3 \gamma^2 \left\|   \sum\limits_{i=1}^n e^k_i \right\|^2.
\end{align*}
Also here we apply the definition of $\lambda_{\max}(\cdot)$ as a maximum eigenvalue. With Cauchy Schwartz inequality for $n$ summands: $\| \sum_{i=1}^n e^k_i\|^2 \leq n \sum_{i=1}^n \| e^k_i\|^2$ and after small rearrangements, we obtain
\begin{align*}
    &2\gamma \bigg[\ell(z^{k+1/2}, b) - \ell(z, b)) + \sum\limits_{i=1}^n ( r_i(x^{k+1/2}_i)  - r_i(x_i))
    \\
    &
    + \langle  A x^{k+1/2} - z^{k+1/2} ,  y\rangle - \langle  Ax - z, y^{k+1/2}\rangle \bigg]
    \\
    &\hspace{3cm}\leq
    \| \hat x^{k} - x \|^2 + \|z^{k} - z \|^2  + \| \hat y^{k} - y \|^2
    \notag\\
    &\hspace{3.4cm}
    -\left(\|\hat x^{k+1} - x \|^2 + \|z^{k+1} - z \|^2 + \|\hat y^{k+1} - y \|^2\right)
    \notag\\
    &\hspace{3.4cm}
    - (1 - \tau) \| x^k - x\|^2 + (1 - \tau) \| w^k - x \|^2
    - (1 - \tau) \| y^k - y\|^2 + (1 - \tau) \| u^k - y \|^2
    \notag\\
    &\hspace{3.4cm}
    - \left( \tau - \frac{1}{2} - 4\gamma^2 L_r^2\right) \|  x^{k+1/2} - x^k\|^2 - (1 - \tau - 4\gamma^2 \lambda_{\max} (A^T A)) \| w^k -x^{k+1/2} \|^2 
    \notag\\
    &\hspace{3.4cm}
    - (1 - 4\gamma^2 - 2\gamma^2 L^2_{\ell})\|  z^{k+1/2} - z^k\|^2
    \notag\\
    &\hspace{3.4cm}
    - \left( \tau - \frac{1}{2} - 2\gamma^2\right)\|  y^{k+1/2} - y^k\|^2 - (1 - \tau - 4\gamma^2 \lambda_{\max} (A A^T)) \| u^k -y^{k+1/2} \|^2
    \notag\\
    &\hspace{3.4cm}
    + 3 \gamma^2 \lambda_{\max} (A A^T) \|  e^k \|^2 + 3 \gamma^2 n\sum\limits_{i=1}^n \|   e^k_i \|^2.
\end{align*}
Then we sum all over $k$ from $0$ to $K-1$, divide by $K$, use Jensen inequality for convex functions $\ell$ and $r_i$ with notation $\bar x^K_i = \frac{1}{K} \sum\limits_{k=0}^{K-1} x^{k+1/2}_i$, $\bar z^K = \frac{1}{K} \sum\limits_{k=0}^{K-1} z^{k+1/2}$, $\bar y^K = \frac{1}{K} \sum\limits_{k=0}^{K-1} y^{k+1/2}$, and have
\begin{align}
    \label{eq:biased_temp5}
    2\gamma \bigg[\ell (\bar z^K, b) - \ell (z, b) + \sum_{i=1}^n &\left( r_i (\bar x^{K}_i) - r_i (x_i)\right) 
    +
    \langle A  \bar x^{K} - \bar z^{K},  y\rangle 
    - \langle  A  x - z,  \bar y^{K}\rangle
    \bigg]
    \notag\\
    &\hspace{2cm}\leq
    \frac{1}{K}\left(\|\hat x^{0} - x \|^2 + \|z^{0} - z \|^2 + \|\hat y^{0} - y \|^2\right)
    \notag\\
    &\hspace{2.4cm}
    - (1 - \tau) \cdot \frac{1}{K} \sum\limits_{k=0}^{K-1} \| x^k - x\|^2 + (1 - \tau) \cdot \frac{1}{K} \sum\limits_{k=0}^{K-1} \| w^k - x \|^2
    \notag\\
    &\hspace{2.4cm}
    - (1 - \tau) \cdot \frac{1}{K} \sum\limits_{k=0}^{K-1} \| y^k - y\|^2 + (1 - \tau) \cdot \frac{1}{K} \sum\limits_{k=0}^{K-1} \| u^k - y \|^2
    \notag\\
    &\hspace{2.4cm}
    - \left( \tau - \frac{1}{2} - 4\gamma^2 L_r^2\right) \cdot \frac{1}{K} \sum\limits_{k=0}^{K-1} \|  x^{k+1/2} - x^k\|^2 
    \notag\\
    &\hspace{2.4cm}
    - (1 - \tau - 4\gamma^2 \lambda_{\max} (A^T A)) \cdot \frac{1}{K} \sum\limits_{k=0}^{K-1} \| w^k -x^{k+1/2} \|^2 
    \notag\\
    &\hspace{2.4cm}
    - (1 - 4\gamma^2 - 2\gamma^2 L^2_{\ell})\frac{1}{K} \sum\limits_{k=0}^{K-1}\|  z^{k+1/2} - z^k\|^2
    \notag\\
    &\hspace{2.4cm}
    - \left( \tau - \frac{1}{2} - 2\gamma^2\right) \cdot \frac{1}{K} \sum\limits_{k=0}^{K-1}\|  y^{k+1/2} - y^k\|^2 
    \notag\\
    &\hspace{2.4cm}
    - (1 - \tau - 4\gamma^2 \lambda_{\max} (A A^T)) \cdot \frac{1}{K} \sum\limits_{k=0}^{K-1} \| u^k -y^{k+1/2} \|^2
    \notag\\
    &\hspace{2.4cm}
    + 3 \gamma^2 \lambda_{\max} (A A^T) \cdot \frac{1}{K} \sum\limits_{k=0}^{K-1} \|  e^k \|^2 + 3 \gamma^2 n \cdot \frac{1}{K}\sum\limits_{i=1}^n \sum\limits_{k=0}^{K-1} \|   e^k_i \|^2.
\end{align}
Using small rearrangements, we can deduce
\begin{align}
\label{eq:biased_temp10}
&- (1 - \tau) \cdot \frac{1}{K} \sum\limits_{k=0}^{K-1} \| x^k - x\|^2 + (1 - \tau) \cdot \frac{1}{K} \sum\limits_{k=0}^{K-1} \| w^k - x \|^2
\notag\\
&\hspace{3cm}
=
\frac{1}{K} \sum\limits_{k=0}^{K-1} \| w^{k} - x \|^2 - \frac{1}{K} \sum\limits_{k=0}^{K-1} [(1 - \tau) \| x^k - x\|^2 + \tau \| w^k - x \|^2]
\notag\\
&\hspace{3cm}=
\frac{1}{K}\| w^0 - x \|^2 - \frac{1}{K}\| w^K - x \|^2
\notag\\
&\hspace{3.4cm}
+ \frac{1}{K} \sum\limits_{k=0}^{K-1} [\| w^{k+1} - x \|^2 - (1 - \tau) \| x^k - x\|^2 - \tau \| w^k - x \|^2]
\notag\\
&\hspace{3cm}=
\frac{1}{K}\| w^0 - x \|^2 - \frac{1}{K}\| w^K - x \|^2 + \frac{1}{K} \sum\limits_{k=0}^{K-1} [\| w^{k+1} \|^2 - (1 - \tau) \| x^k \|^2 - \tau \| w^k \|^2]
\notag\\
&\hspace{3.4cm}+ \frac{2}{K} \sum\limits_{k=0}^{K-1} \langle (1 - \tau) x^k + \tau w^k - w^{k+1}, x\rangle .
\end{align}
The same way we can make
\begin{align}
\label{eq:biased_temp11}
&- (1 - \tau) \cdot \frac{1}{K} \sum\limits_{k=0}^{K-1} \| y^k - y\|^2 + (1 - \tau) \cdot \frac{1}{K} \sum\limits_{k=0}^{K-1} \| u^k - y \|^2
\notag\\
&\hspace{4cm}
=
\frac{1}{K}\| u^0 - y \|^2 - \frac{1}{K}\| u^K - y \|^2 + \frac{1}{K} \sum\limits_{k=0}^{K-1} [\| u^{k+1} \|^2 - (1 - \tau) \| y^k \|^2 - \tau \| u^k \|^2]
\notag\\
&\hspace{4.4cm}+ \frac{2}{K} \sum\limits_{k=0}^{K-1} \langle (1 - \tau) y^k + \tau u^k - u^{k+1}, y\rangle .
\end{align}
Substituting (\ref{eq:biased_temp10}) and (\ref{eq:biased_temp11}) to (\ref{eq:biased_temp5}), we obtain
\begin{align*}
    2\gamma \bigg[\ell (\bar z^K, b) - \ell (z, b) + \sum_{i=1}^n &\left( r_i (\bar x^{K}_i) - r_i (x_i)\right) 
    +
    \langle A  \bar x^{K} - \bar z^{K},  y\rangle 
    - \langle  A  x - z,  \bar y^{K}\rangle
    \bigg]
    \notag\\
    &\hspace{2cm}\leq
    \frac{1}{K}\left(\|\hat x^{0} - x \|^2 + \|z^{0} - z \|^2 + \|\hat y^{0} - y \|^2\right)
    \notag\\
    &\hspace{2.4cm}
    +\frac{1}{K}\| w^0 - x \|^2 - \frac{1}{K}\| w^K - x \|^2 
    \notag\\
    &\hspace{2.4cm}
    + \frac{1}{K} \sum\limits_{k=0}^{K-1} [\| w^{k+1} \|^2 - (1 - \tau) \| x^k \|^2 - \tau \| w^k \|^2]
    \notag\\
    &\hspace{2.4cm}+ \frac{2}{K} \sum\limits_{k=0}^{K-1} \langle (1 - \tau) x^k + \tau w^k - w^{k+1}, x\rangle 
    \notag\\
    &\hspace{2.4cm}
    +\frac{1}{K}\| u^0 - y \|^2 - \frac{1}{K}\| u^K - y \|^2 
    \notag\\
    &\hspace{2.4cm}
    + \frac{1}{K} \sum\limits_{k=0}^{K-1} [\| u^{k+1} \|^2 - (1 - \tau) \| y^k \|^2 - \tau \| u^k \|^2]
    \notag\\
    &\hspace{2.4cm}+ \frac{2}{K} \sum\limits_{k=0}^{K-1} \langle (1 - \tau) y^k + \tau u^k - u^{k+1}, y\rangle 
    \notag\\
    &\hspace{2.4cm}
    - \left( \tau - \frac{1}{2} - 4\gamma^2 L_r^2\right) \cdot \frac{1}{K} \sum\limits_{k=0}^{K-1} \|  x^{k+1/2} - x^k\|^2 
    \notag\\
    &\hspace{2.4cm}
    - (1 - \tau - 4\gamma^2 \lambda_{\max} (A^T A)) \cdot \frac{1}{K} \sum\limits_{k=0}^{K-1} \| w^k -x^{k+1/2} \|^2 
    \notag\\
    &\hspace{2.4cm}
    - (1 - 4\gamma^2 - 2\gamma^2 L^2_{\ell})\frac{1}{K} \sum\limits_{k=0}^{K-1}\|  z^{k+1/2} - z^k\|^2
    \notag\\
    &\hspace{2.4cm}
    - \left( \tau - \frac{1}{2} - 2\gamma^2\right) \cdot \frac{1}{K} \sum\limits_{k=0}^{K-1}\|  y^{k+1/2} - y^k\|^2 
    \notag\\
    &\hspace{2.4cm}
    - (1 - \tau - 4\gamma^2 \lambda_{\max} (A A^T)) \cdot \frac{1}{K} \sum\limits_{k=0}^{K-1} \| u^k -y^{k+1/2} \|^2
    \notag\\
    &\hspace{2.4cm}
    + 3 \gamma^2 \lambda_{\max} (A A^T) \cdot \frac{1}{K} \sum\limits_{k=0}^{K-1} \|  e^k \|^2 + 3 \gamma^2 n \cdot \frac{1}{K}\sum\limits_{i=1}^n \sum\limits_{k=0}^{K-1} \|   e^k_i \|^2.
\end{align*}
As in (\ref{eq:basic_crit1}) we pass to the gap criterion by taking the maximum in $y \in \cY$ and the minimum in $x\in \cX$ and $z \in \cZ$. Additionally, we also take the mathematical expectation
\begin{align*}
    2\gamma \E{\text{gap}(\bar x^K, \bar z^K, \bar y^K)}
    \leq&
    \frac{1}{K}\bigg(\max_{x\in \cX}\|\hat x^{0} - x \|^2 + \max_{x\in \cX}\| w^0 - x \|^2 + \max_{z\in \cX}\|z^{0} - z \|^2 
    \notag\\
    &
    + \max_{y\in \cY}\|\hat y^{0} - y \|^2 + \max_{y\in \cY}\| u^0 - y \|^2\bigg)
    \notag\\
    &
    + \frac{1}{K} \sum\limits_{k=0}^{K-1} \E{\| w^{k+1} \|^2 - (1 - \tau) \| x^k \|^2 - \tau \| w^k \|^2}
    \notag\\
    &
    + \frac{2}{K} \E{\max_{x\in \cX}\sum\limits_{k=0}^{K-1} \langle (1 - \tau) x^k + \tau w^k - w^{k+1}, x\rangle }
    \notag\\
    &
    + \frac{1}{K} \sum\limits_{k=0}^{K-1} \E{\| u^{k+1} \|^2 - (1 - \tau) \| y^k \|^2 - \tau \| u^k \|^2}
    \notag\\
    &
    + \frac{2}{K} \E{\max_{y\in \cY}\sum\limits_{k=0}^{K-1} \langle (1 - \tau) y^k + \tau u^k - u^{k+1}, y\rangle }
    \notag\\
    &
    - \left( \tau - \frac{1}{2} - 4\gamma^2 L_r^2\right) \cdot \frac{1}{K} \sum\limits_{k=0}^{K-1} \E{\|  x^{k+1/2} - x^k\|^2 }
    \notag\\
    &
    - (1 - \tau - 4\gamma^2 \lambda_{\max} (A^T A)) \cdot \frac{1}{K} \sum\limits_{k=0}^{K-1} \E{\| w^k -x^{k+1/2} \|^2 }
    \notag\\
    &
    - (1 - 4\gamma^2 - 2\gamma^2 L^2_{\ell})\frac{1}{K} \sum\limits_{k=0}^{K-1} \E{\|  z^{k+1/2} - z^k\|^2}
    \notag\\
    &
    - \left( \tau - \frac{1}{2} - 2\gamma^2\right) \cdot \frac{1}{K} \sum\limits_{k=0}^{K-1} \E{\|  y^{k+1/2} - y^k\|^2 }
    \notag\\
    &
    - (1 - \tau - 4\gamma^2 \lambda_{\max} (A A^T)) \cdot \frac{1}{K} \sum\limits_{k=0}^{K-1} \E{\| u^k -y^{k+1/2} \|^2}
    \notag\\
    &
    + 3 \gamma^2 \lambda_{\max} (A A^T) \cdot \frac{1}{K} \sum\limits_{k=0}^{K-1} \E{\|  e^k \|^2} + 3 \gamma^2 n \cdot \frac{1}{K}\sum\limits_{i=1}^n \sum\limits_{k=0}^{K-1} \E{\|   e^k_i \|^2}.
\end{align*}
Since lines \ref{lin_alg:biased_linws}--\ref{lin_alg:biased_linwf} of Algorithm \ref{alg:EG_biased} are equivalent to lines \ref{lin_alg:unbiased_linws}--\ref{lin_alg:unbiased_linwf} of Algorithm \ref{alg:EG_quantization}. Then, we can use (\ref{eq:unbiased_temp8}), (\ref{eq:unbiased_temp9}), (\ref{eq:unbiased_temp10}), (\ref{eq:unbiased_temp11}) and get
\begin{align*}
    2\gamma \E{\text{gap}(\bar x^K, \bar z^K, \bar y^K)}
    \leq &
    \frac{1}{K}\bigg(5\max_{x\in \cX}\|\hat x^{0} - x \|^2 + \max_{x\in \cX}\| w^0 - x \|^2 + \max_{z\in \cX}\|z^{0} - z \|^2 
    \notag\\
    &
    + 5\max_{y\in \cY}\|\hat y^{0} - y \|^2 + \max_{y\in \cY}\| u^0 - y \|^2\bigg)
    \notag\\
    &+ 
    \frac{1}{4}\sum\limits_{k=0}^{K-1}  \E{ \tau (1 - \tau) \| w^{k} - x^k \|^2} + \frac{1}{4}\sum\limits_{k=0}^{K-1}  \E{ \tau (1 - \tau) \| u^{k} - y^k \|^2}
    \notag\\
    &
    - \left( \tau - \frac{1}{2} - 4\gamma^2 L_r^2\right) \cdot \frac{1}{K} \sum\limits_{k=0}^{K-1} \E{\|  x^{k+1/2} - x^k\|^2 }
    \notag\\
    &
    - (1 - \tau - 4\gamma^2 \lambda_{\max} (A^T A)) \cdot \frac{1}{K} \sum\limits_{k=0}^{K-1} \E{\| w^k -x^{k+1/2} \|^2 }
    \notag\\
    &
    - (1 - 4\gamma^2 - 2\gamma^2 L^2_{\ell})\frac{1}{K} \sum\limits_{k=0}^{K-1} \E{\|  z^{k+1/2} - z^k\|^2}
    \notag\\
    &
    - \left( \tau - \frac{1}{2} - 2\gamma^2\right) \cdot \frac{1}{K} \sum\limits_{k=0}^{K-1} \E{\|  y^{k+1/2} - y^k\|^2 }
    \notag\\
    &
    - (1 - \tau - 4\gamma^2 \lambda_{\max} (A A^T)) \cdot \frac{1}{K} \sum\limits_{k=0}^{K-1} \E{\| u^k -y^{k+1/2} \|^2}
    \notag\\
    &
    + 3 \gamma^2 \lambda_{\max} (A A^T) \cdot \frac{1}{K} \sum\limits_{k=0}^{K-1} \E{\|  e^k \|^2} + 3 \gamma^2 n \cdot \frac{1}{K}\sum\limits_{i=1}^n \sum\limits_{k=0}^{K-1} \E{\|   e^k_i \|^2}.
\end{align*}
Next we work  with error feedback terms:
\begin{align*}
 \E{\| e^{k+1}\|^2} =& 
 \E{\|y^{k+1/2} - u^k + e^k - C(y^{k+1/2} - u^k + e^k) \|^2}
 \\
\leq&  
\left(1-\frac{1}{\delta}\right) \E{\|y^{k+1/2} - u^k + e^k\|^2}.
\end{align*}
With Cauchy Schwartz inequality in the form 
$\| a + b\|^2 \leq \left( 1 + \tfrac{1}{\eta} \right) \| a\|^2 + (1 + \eta)\| b\|^2$ with $a = e^k$, $b = y^{k+1/2} - u^k$ and $\eta = 2 \delta$, we get
\begin{align*}
 \E{\| e^{k+1}\|^2}
\leq&  
\left(1-\frac{1}{\delta}\right) \left(1+\frac{1}{2\delta}\right) \E{\|e^k\|^2} + (2\delta + 1)\left(1-\frac{1}{\delta}\right)\E{\|y^{k+1/2} - u^k\|^2}
\\
\leq&  
\left(1-\frac{1}{2\delta}\right) \E{\|e^k\|^2} + 3\delta\E{\|y^{k+1/2} - u^k\|^2}.
\end{align*}
Running the recursion and using that $e_0 = 0$, we have
\begin{align*}
 \E{\| e^{k+1}\|^2}
\leq&  
3\delta\sum\limits_{j=0}^k \left(1-\frac{1}{2\delta}\right)^{k-j} \E{\|y^{j+1/2} - u^j\|^2}.
\end{align*}
Then we sum all over $k$ from $0$ to $K-1$, divide by $K$.
\begin{align}
\label{eq:biased_temp6}
 \frac{1}{K} \sum\limits_{k=0}^{K-1}\E{\| e^{k}\|^2}
\leq&  
3\delta \cdot \frac{1}{K} \sum\limits_{k=0}^{K-1} \sum\limits_{j=0}^{k-1} \left(1-\frac{1}{2\delta}\right)^{k-1-j} \E{\|y^{j+1/2} - u^j\|^2}
\notag\\
\leq&  
3\delta \cdot \frac{1}{K} \sum\limits_{k=0}^{K-1} \E{\|y^{k+1/2} - u^k\|^2} \sum\limits_{j=0}^{\infty} \left(1-\frac{1}{2\delta}\right)^{j} 
\notag\\
\leq&  
6\delta^2 \cdot \frac{1}{K} \sum\limits_{k=0}^{K-1} \E{\|y^{k+1/2} - u^k\|^2}.
\end{align}
The same way we can make the following estimate:
\begin{align}
\label{eq:biased_temp7}
 \frac{1}{K} \sum\limits_{k=0}^{K-1}\E{\| e^{k}_i\|^2}
\leq&  
6\delta^2 \cdot \frac{1}{K} \sum\limits_{k=0}^{K-1} \E{\|A_i x^{k+1/2}_i - A_i w^{k}_i\|^2}.
\end{align}
Combining (\ref{eq:biased_temp5}) with (\ref{eq:biased_temp6}) and (\ref{eq:biased_temp7}), we have
\begin{align*}
    2\gamma \E{\text{gap}(\bar x^K, \bar z^K, \bar y^K)}
    \leq&
    \frac{1}{K}\bigg(5\max_{x\in \cX}\|\hat x^{0} - x \|^2 + \max_{x\in \cX}\| w^0 - x \|^2 + \max_{z\in \cX}\|z^{0} - z \|^2 
    \notag\\
    &
    + 5\max_{y\in \cY}\|\hat y^{0} - y \|^2 + \max_{y\in \cY}\| u^0 - y \|^2\bigg)
    \notag\\
    &+ 
    \frac{1}{4}\sum\limits_{k=0}^{K-1}  \E{ \tau (1 - \tau) \| w^{k} - x^k \|^2} + \frac{1}{4}\sum\limits_{k=0}^{K-1}  \E{ \tau (1 - \tau) \| u^{k} - y^k \|^2}
    \notag\\
    &
    - \left( \tau - \frac{1}{2} - 4\gamma^2 L_r^2\right) \cdot \frac{1}{K} \sum\limits_{k=0}^{K-1} \E{\|  x^{k+1/2} - x^k\|^2 }
    \notag\\
    &
    - (1 - \tau - 4\gamma^2 \lambda_{\max} (A^T A)) \cdot \frac{1}{K} \sum\limits_{k=0}^{K-1} \E{\| w^k -x^{k+1/2} \|^2 }
    \notag\\
    &
    - (1 - 4\gamma^2 - 2\gamma^2 L^2_{\ell})\frac{1}{K} \sum\limits_{k=0}^{K-1} \E{\|  z^{k+1/2} - z^k\|^2}
    \notag\\
    &
    - \left( \tau - \frac{1}{2} - 2\gamma^2\right) \cdot \frac{1}{K} \sum\limits_{k=0}^{K-1} \E{\|  y^{k+1/2} - y^k\|^2 }
    \notag\\
    &
    - (1 - \tau - 4\gamma^2 \lambda_{\max} (A A^T)) \cdot \frac{1}{K} \sum\limits_{k=0}^{K-1} \E{\| u^k -y^{k+1/2} \|^2}
    \notag\\
    &
    + 18 \gamma^2 \delta^2 \lambda_{\max} (A A^T) \cdot \frac{1}{K} \sum\limits_{k=0}^{K-1} \E{\|y^{k+1/2} - u^k\|^2} 
    \notag\\
    &
    + 18 \gamma^2 \delta^2 n \cdot \frac{1}{K}\sum\limits_{k=0}^{K-1} \sum\limits_{i=1}^n \E{\|A_i x^{k+1/2}_i - A_i w^{k}_i\|^2}.
\end{align*}
For $\sum\limits_{i=1}^n \E{\|A_i x^{k+1/2}_i - A_i w^{k}_i\|^2}$ we get
\begin{align*}
    \sum\limits_{i=1}^n \E{\|A_i x^{k+1/2}_i - A_i w^{k}_i\|^2} 
    \leq&
    \sum\limits_{i=1}^n \lambda_{\max} (A_i^T A_i) \E{\|x^{k+1/2}_i - w^{k}_i\|^2}
    \\
    \leq&
    \max_i \lambda_{\max} (A_i^T A_i) \sum\limits_{i=1}^n  \E{\|x^{k+1/2}_i - w^{k}_i\|^2}
    \\
    \leq&
    \max_i [\lambda_{\max} (A_i^T A_i)] \E{\|x^{k+1/2} - w^{k}\|^2}.
\end{align*}
Then one can deduce
\begin{align*}
    2\gamma \E{\text{gap}(\bar x^K, \bar z^K, \bar y^K)}
    \leq&
    \frac{1}{K}\bigg(5\max_{x\in \cX}\|\hat x^{0} - x \|^2 + \max_{x\in \cX}\| w^0 - x \|^2 + \max_{z\in \cX}\|z^{0} - z \|^2 
    \notag\\
    &
    + 5\max_{y\in \cY}\|\hat y^{0} - y \|^2 + \max_{y\in \cY}\| u^0 - y \|^2\bigg)
    \notag\\
    &+ 
    \frac{1}{4}\sum\limits_{k=0}^{K-1}  \E{ \tau (1 - \tau) \| w^{k} - x^k \|^2} + \frac{1}{4}\sum\limits_{k=0}^{K-1}  \E{ \tau (1 - \tau) \| u^{k} - y^k \|^2}
    \notag\\
    &
    - \left( \tau - \frac{1}{2} - 4\gamma^2 L_r^2\right) \cdot \frac{1}{K} \sum\limits_{k=0}^{K-1} \E{\|  x^{k+1/2} - x^k\|^2 }
    \notag\\
    &
    - (1 - \tau - 4\gamma^2 \lambda_{\max} (A^T A) - 18 \gamma^2 \delta^2 n \max_i [\lambda_{\max} (A_i^T A_i)]) \cdot \frac{1}{K} \sum\limits_{k=0}^{K-1} \E{\| w^k -x^{k+1/2} \|^2 }
    \notag\\
    &
    - (1 - 4\gamma^2 - 2\gamma^2 L^2_{\ell})\frac{1}{K} \sum\limits_{k=0}^{K-1} \E{\|  z^{k+1/2} - z^k\|^2}
    \notag\\
    &
    - \left( \tau - \frac{1}{2} - 2\gamma^2\right) \cdot \frac{1}{K} \sum\limits_{k=0}^{K-1} \E{\|  y^{k+1/2} - y^k\|^2 }
    \notag\\
    &
    - (1 - \tau - 4\gamma^2 \lambda_{\max} (A A^T) (1 + 5\delta^2)) \cdot \frac{1}{K} \sum\limits_{k=0}^{K-1} \E{\| u^k -y^{k+1/2} \|^2}.
\end{align*}
With $\tau \leq 1$ and Cauchy Schwartz inequality, we have
\begin{align*}
    2\gamma \E{\text{gap}(\bar x^K, \bar z^K, \bar y^K)}
    \leq&
    \frac{1}{K}\bigg(5\max_{x\in \cX}\|\hat x^{0} - x \|^2 + \max_{x\in \cX}\| w^0 - x \|^2 + \max_{z\in \cX}\|z^{0} - z \|^2 
    \notag\\
    &
    + 5\max_{y\in \cY}\|\hat y^{0} - y \|^2 + \max_{y\in \cY}\| u^0 - y \|^2\bigg)
    \notag\\
    &\
    - \left( \frac{3\tau - 2}{2} - 4\gamma^2 L_r^2\right) \cdot \frac{1}{K} \sum\limits_{k=0}^{K-1} \E{\|  x^{k+1/2} - x^k\|^2 }
    \notag\\
    &
    - \left(\frac{1 - \tau}{2} - 4\gamma^2 \lambda_{\max} (A^T A) - 18 \gamma^2 \delta^2 n \max_i [\lambda_{\max} (A_i^T A_i)]\right) \cdot \frac{1}{K} \sum\limits_{k=0}^{K-1} \E{\| w^k -x^{k+1/2} \|^2 }
    \notag\\
    &
    - (1 - 4\gamma^2 - 2\gamma^2 L^2_{\ell})\frac{1}{K} \sum\limits_{k=0}^{K-1} \E{\|  z^{k+1/2} - z^k\|^2}
    \notag\\
    &
    - \left(  \frac{3\tau - 2}{2} - 2\gamma^2\right) \cdot \frac{1}{K} \sum\limits_{k=0}^{K-1} \E{\|  y^{k+1/2} - y^k\|^2 }
    \notag\\
    &
    - \left(\frac{1 - \tau}{2} - 4\gamma^2 \lambda_{\max} (A A^T) (1 + 5\delta^2)\right) \cdot \frac{1}{K} \sum\limits_{k=0}^{K-1} \E{\| u^k -y^{k+1/2} \|^2}.
\end{align*}
If we choose $\tau \geq \tfrac{1}{2}$ and $\gamma$ as follows
$$
\gamma \leq \frac{1}{4}\min\left\{ 1;\frac{1}{L_r}; \frac{1}{L_{\ell}};\sqrt{\frac{1-\tau}{5\delta^2 n \max_i [\lambda_{\max} (A_i^T A_i)]}}; \sqrt{\frac{1-\tau}{3\delta^2\lambda_{\max}(A A^T)}}; \right\},
$$
then one can obtain
\begin{align*}
    &\E{\text{gap}(\bar x^K, \bar z^K, \bar y^K)}
    \leq
    \frac{1}{2\gamma K}\bigg(5\max_{x \in \cX}\|x^{0} - x \|^2 + \max_{x \in \cX}\| w^0 - x \|^2 + \max_{z \in \cZ}\|z^{0} - z \|^2
    \notag\\
    &\hspace{7.5cm}
    + 5\max_{y \in \cY}\|y^{0} - y \|^2 + \max_{y \in \cY}\| u^0 - y \|^2 \bigg).
\end{align*}
With $\gamma = \frac{1}{4}\min\left\{ 1;\frac{1}{L_r}; \frac{1}{L_{\ell}};\sqrt{\frac{1-\tau}{5\delta^2 n \max_i [\lambda_{\max} (A_i^T A_i)]}}; \sqrt{\frac{1-\tau}{3\delta^2\lambda_{\max}(A A^T)}}; \right\}$, we finish the proof.
\end{proof}

\subsection{Proof of Theorem \ref{th:pp}} \label{app:th:pp}

{
\begin{theorem}[Theorem \ref{th:pp}]
Let Assumption \ref{as:convexity_smothness} hold. Let the problem (\ref{eq:vfl_lin_spp_1})
be solved by Algorithm~\ref{alg:EG_pp}. Then for $\tau = 1 - p$ and
$$
\gamma = \frac{1}{4}\min\left\{ 1;\frac{1}{L_r}; \frac{1}{L_{\ell}};\sqrt{\frac{1-\tau}{\lambda_{\max}(A A^T) + n \cdot \max_{i}\{\lambda_{\max}(A_i A^T_i)\}}} \right\},
$$
it holds that
\begin{align*}
&\E{\text{gap}(\bar x^K, \bar z^K, \bar y^K)} 
\\
&= 
    \mathcal{O}\left(
    \left[ 1+  \frac{1}{\sqrt{p}}\left(\sqrt{\lambda_{\max}(A A^T) } +  n \cdot \max\limits_{i = 1, \ldots, n}\left\{\sqrt{\lambda_{\max}(A_i A^T_i) }\right\} \right) + L_{\ell} + L_r\right] \cdot \frac{D^2}{K}\right),  
\end{align*}
where $\bar x^K := \tfrac{1}{K}\sum_{k=0}^{K-1} x^{k+1/2}$, $\bar z^K := \tfrac{1}{K}\sum_{k=0}^{K-1} z^{k+1/2}$, $\bar y^K := \tfrac{1}{K}\sum_{k=0}^{K-1} y^{k+1/2}$ and \\ $D^2 := \max_{x, z, y \in \cX, \cZ, \cY} \left[ \|x^0 - x \|^2 + \|z^0 - z \|^2 + \|y^0 - y \|^2 \right]$.
\end{theorem}
}
\begin{proof}
The proof repeats almost the same steps as the proof of Theorem \ref{th:EG_compressed_unbiased} (Section \ref{sec:proof_EG_compressed_unbiased}). In particular, in the proof of Theorem \ref{th:EG_compressed_unbiased} we need to replace $Q(y^{k+1/2} - u^k)$ by $y^{k+1/2} - u^k$ and $\sum_{i=1}^n Q(A_i [x^{k+1/2}_i - w^{k}_i])$ by $n \cdot A_{i_k} [x^{k+1/2}_{i_k} - w^{k}_{i_k}]$, and use. In the end, we arrive at the analogue of (\ref{eq:unbiased_temp14}).
\begin{align*}
    2\gamma \E{\text{gap}(\bar x^K, \bar z^K, \bar y^K)}
    \leq&
    \frac{1}{K}\bigg(6\max_{x \in \cX}\|x^{0} - x \|^2 + \max_{x \in \cX}\| w^0 - x \|^2 + \max_{z \in \cZ}\|z^{0} - z \|^2
    \notag\\
    &
    + 6\max_{y \in \cY}\|y^{0} - y \|^2 + \max_{y \in \cY}\| u^0 - y \|^2 \bigg)
    \notag\\
    &
    -  \left( \frac{3\tau - 1}{2} - 2\gamma^2 L_r^2 \right) \cdot \frac{1}{K} \sum\limits_{k=0}^{K-1} \E{\|  x^{k+1/2} - x^k\|^2} -  \frac{1 - \tau}{2K} \sum\limits_{k=0}^{K-1} \E{\| w^k -x^{k+1/2} \|^2} 
    \notag\\
    &
    - \left( 1 - 2\gamma^2 (1 + L_{\ell}^2) \right)\frac{1}{K} \sum\limits_{k=0}^{K-1} \E{\|  z^{k+1/2} - z^k\|^2}
    \notag\\
    &
    - \left( \frac{3\tau - 1}{2} - 2 \gamma^2 \right)\frac{1}{K} \sum\limits_{k=0}^{K-1} \E{\|  y^{k+1/2} - y^k\|^2} - \frac{1 - \tau}{2K} \sum\limits_{k=0}^{K-1} \E{\| u^k -y^{k+1/2} \|^2}
    \notag\\
    &
    + \frac{3\gamma^2}{K} \sum\limits_{k=0}^{K-1} \E{\| A^T(y^{k+1/2} - u^k) \|^2}
    + \frac{3\gamma^2 n^2}{K} \sum\limits_{k=0}^{K-1} \E{\| A_{i_k} (x^{k+1/2}_{i_k}  - w^{k}_{i_k} ) \|^2}.
\end{align*}
Using the notation of $\lambda_{\max}(\cdot)$ as a maximum eigenvalue and the random choice of $i_k$, we get
\begin{align*}
\E{\| A_{i_k} (x^{k+1/2}_{i_k} - w^k_{i_k}) \|^2}
=& 
\E{\mathbb{E}_{i_k}\left[\| A_{i_k} (x^{k+1/2}_{i_k} - w^k_{i_k}) \|^2 \right]}
\\
=& 
\frac{1}{n} \sum\limits_{i=1}^n \E{\| A_{i} (x^{k+1/2}_{i} - w^k_{i}) \|^2}
\\
\leq& \frac{1}{n} \sum\limits_{i=1}^n \lambda_{\max}(A_i^T A_i) \E{\| x^{k+1/2}_i - w^{k}_i \|^2}
\\
\leq& \frac{\max_i \{\lambda_{\max}(A_i^T A_i)\} }{n} \E{\| x^{k+1/2} - w^{k}\|^2}.
\end{align*}
Therefore, we obtain
\begin{align*}
    2\gamma \E{\text{gap}(\bar x^K, \bar z^K, \bar y^K)}
    \leq&
    \frac{1}{K}\bigg(6\max_{x \in \cX}\|x^{0} - x \|^2 + \max_{x \in \cX}\| w^0 - x \|^2 + \max_{z \in \cZ}\|z^{0} - z \|^2
    \notag\\
    &
    + 6\max_{y \in \cY}\|y^{0} - y \|^2 + \max_{y \in \cY}\| u^0 - y \|^2 \bigg)
    \notag\\
    &
    -  \left( \frac{3\tau - 1}{2} - 2\gamma^2 L_r^2 \right) \cdot \frac{1}{K} \sum\limits_{k=0}^{K-1} \E{\|  x^{k+1/2} - x^k\|^2} 
    \notag\\
    &
    -  \left( \frac{1 - \tau}{2} - 3 \gamma^2 n \max_i \{\lambda_{\max}(A_i^T A_i)\}\right) \frac{1}{K} \sum\limits_{k=0}^{K-1} \E{\| w^k -x^{k+1/2} \|^2} 
    \notag\\
    &
    - \left( 1 - 2\gamma^2 (1 + L_{\ell}^2) \right)\frac{1}{K} \sum\limits_{k=0}^{K-1} \E{\|  z^{k+1/2} - z^k\|^2}
    \notag\\
    &
    - \left( \frac{3\tau - 1}{2} - 2 \gamma^2 \right)\frac{1}{K} \sum\limits_{k=0}^{K-1} \E{\|  y^{k+1/2} - y^k\|^2}
    \notag\\
    &
    - \left( \frac{1 - \tau}{2} - 3\lambda_{\max}(A A^T) \gamma^2 \right)\frac{1}{K} \sum\limits_{k=0}^{K-1} \E{\| u^k -y^{k+1/2} \|^2}.
\end{align*}
If we choose $\tau \geq \tfrac{1}{2}$ and $\gamma$ as follows
$$
\gamma \leq \frac{1}{4}\min\left\{ 1;\frac{1}{L_r}; \frac{1}{L_{\ell}};\sqrt{\frac{1-\tau}{ n \max_i \{\lambda_{\max}(A_i^T A_i)\}}}; \sqrt{\frac{1-\tau}{\lambda_{\max}(A A^T)}}; \right\},
$$
then one can obtain
\begin{align*}
    &\E{\text{gap}(\bar x^K, \bar z^K, \bar y^K)}
    \leq
    \frac{1}{2\gamma K}\bigg(6\max_{x \in \cX}\|x^{0} - x \|^2 + \max_{x \in \cX}\| w^0 - x \|^2 + \max_{z \in \cZ}\|z^{0} - z \|^2
    \notag\\
    &\hspace{7.5cm}
    + 6\max_{y \in \cY}\|y^{0} - y \|^2 + \max_{y \in \cY}\| u^0 - y \|^2 \bigg).
\end{align*}
With $\gamma = \frac{1}{4}\min\left\{ 1;\frac{1}{L_r}; \frac{1}{L_{\ell}};\sqrt{\frac{1-\tau}{ n \max_i \{\lambda_{\max}(A_i^T A_i)\}}}; \sqrt{\frac{1-\tau}{\lambda_{\max}(A A^T)}}; \right\}$, we finish the proof.
\end{proof}


\subsection{Proof of Theorem \ref{th:EG_coord}} \label{app:th:EG_coord}

{
\begin{theorem}[Theorem \ref{th:EG_coord}]
Let Assumption \ref{as:convexity_smothness} hold. Let the problem (\ref{eq:vfl_lin_spp_1})
be solved by Algorithm~\ref{alg:EG_coord}. Then for $\tau = 1 - p$  and
$$
\textstyle{
\gamma\!=\!\tfrac{1}{4}\!\min
 \left\{\!1;\!\frac{1}{L_r};\!\frac{1}{L_{\ell}};\sqrt{\frac{1-\tau}{s ( \lambda_{\max} (A^T A) + \mathbb{I}(\text{diff. seed}) \max_{i} \left\{\lambda_{\max} (A_i^T A_i) \right\} )}};\! \sqrt{\frac{1-\tau}{d \max_{i} \{\lambda_{\max} (A_i^T A_i)\}}} \right\},
}
$$
it holds that
\begin{align*}
&\E{\text{gap}(\bar x^K, \bar z^K, \bar y^K)}
\\
& =  
    \mathcal{O}\Bigg(
    \left[ \frac{1}{\sqrt{p}}\left(s \sqrt{\lambda_{\max} (A^T A) + \mathbb{I}(\text{diff. seed}) \max\limits_{i = 1, \ldots, n} \left\{\lambda_{\max} (A_i^T A_i) \right\} }+    \right) + L_{\ell} + L_r\right] \cdot \frac{D^2}{K}
\\
& \qquad \qquad
     +
    \left[ \frac{1}{\sqrt{p}}\left(  d \cdot \max\limits_{i = 1, \ldots, n}\left\{\sqrt{\lambda_{\max}(A_i A^T_i) }\right\} \right)  \right] \cdot \frac{D^2}{K}\Bigg),
\end{align*}

where $\bar x^K := \tfrac{1}{K}\sum_{k=0}^{K-1} x^{k+1/2}$, $\bar z^K := \tfrac{1}{K}\sum_{k=0}^{K-1} z^{k+1/2}$, $\bar y^K := \tfrac{1}{K}\sum_{k=0}^{K-1} y^{k+1/2}$ and \\ $D^2 := \max_{x, z, y \in \cX, \cZ, \cY} \left[ \|x^0 - x \|^2 + \|z^0 - z \|^2 + \|y^0 - y \|^2 \right]$.
\end{theorem}
}
\begin{proof}
Most of the proof is the same as that of Theorem \ref{th:EG_compressed_unbiased}. We note only some main steps of the proof and changes regarding Section \ref{sec:proof_EG_compressed_biased} with the proof of Theorem \ref{th:EG_compressed_unbiased}.
We start with an analogue of (\ref{eq:unbiased_temp22}) and get
\begin{align*}
    \|x^{k+1}_i - x_i \|^2
    =&
    \|x^{k}_i - x_i \|^2 - \|  x^{k+1/2}_i - x^k_i\|^2 - \|x^{k+1/2}_i - x^{k+1}_i \|^2
    \notag\\
    &+
    (1 - \tau) ( \| w^k_i - x_i \|^2 - \| w^k_i -x^{k+1/2}_i \|^2 - \| x^{k+1/2}_i - x_i \|^2 ) 
    \notag\\
    &+
    (1 - \tau) ( \| x^{k+1/2}_i -x^k_i \|^2 + \| x^{k+1/2}_i - x_i \|^2 - \| x^k_i - x_i\|^2 )
    \notag\\
    &
    -
    2\gamma \langle  d_i \cdot \langle A_i^T (y^{k+1/2} - u^k), e_{j^k_i} \rangle e_{j^k_i} +  A_i^T u^k + \nabla r_i(x^{k+1/2}_i),  x^{k+1/2}_i - x_i\rangle
    \notag\\
    &
    + 2 \gamma \langle d_i \cdot \langle A_i^T (y^{k+1/2} - u^k), e_{j^k_i} \rangle e_{j^k_i} + \nabla r_i(x^{k+1/2}_i) - \nabla r_i (x^k_i) , x^{k+1/2}_i - x^{k+1}_i\rangle
    \\
    =&
    \tau \|x^{k}_i - x_i \|^2 + (1 - \tau) \| w^k_i - x_i \|^2 
    \\
    &
    - \tau\|  x^{k+1/2}_i - x^k_i\|^2 -
    (1 - \tau) \| w^k_i -x^{k+1/2}_i \|^2   - \|x^{k+1/2}_i - x^{k+1}_i \|^2
    \\
    &
    -
    2\gamma \langle  d_i \cdot \langle A_i^T (y^{k+1/2} - u^k), e_{j^k_i} \rangle e_{j^k_i}  +  A_i^T u^k + \nabla r_i(x^{k+1/2}_i),  x^{k+1/2}_i - x_i\rangle
    \notag\\
    &
    + 2 \gamma \langle d_i \cdot \langle A_i^T (y^{k+1/2} - u^k), e_{j^k_i} \rangle e_{j^k_i} + \nabla r_i(x^{k+1/2}_i) - \nabla r_i (x^k_i) , x^{k+1/2}_i - x^{k+1}_i\rangle.
\end{align*}
Summing over all $i$ from $1$ to $n$ and using the notation of $A = [A_1, \ldots, A_i, \ldots, A_n] $, $x = [x_1^T, \ldots,x_i^T, \ldots,x_n^T]^T$, $w = [w_1^T, \ldots,w_i^T, \ldots,w_n^T]^T$, we deduce
\begin{align*}
    \|x^{k+1} - x \|^2
    =&
    \tau \|x^{k} - x \|^2 + (1 - \tau) \| w^k - x \|^2 
    \\
    &
    - \tau \|  x^{k+1/2} - x^k\|^2 - (1 - \tau) \| w^k -x^{k+1/2} \|^2 - \|x^{k+1/2} - x^{k+1} \|^2
    \\
    &
    -
    2\gamma \langle  A  (x^{k+1/2} - x),  y^{k+1/2}\rangle
    -
    2\gamma \sum\limits_{i=1}^n\langle \nabla r_i(x^{k+1/2}_i),  x^{k+1/2}_i - x_i\rangle
    \\
    &
    -
    2\gamma \sum\limits_{i=1}^n \langle  d_i \cdot \langle A_i^T (y^{k+1/2} - u^k), e_{j^k_i} \rangle e_{j^k_i} -  A_i^T (y^{k+1/2} - u^k),  x^{k+1/2}_i - x_i\rangle
    \\
    &
    + 2 \gamma \sum\limits_{i=1}^n \langle d_i \cdot \langle A_i^T (y^{k+1/2} - u^k), e_{j^k_i} \rangle e_{j^k_i}, x^{k+1/2}_i - x^{k+1}_i\rangle
    \\
    &
    + 2 \gamma \sum\limits_{i=1}^n \langle \nabla r_i(x^{k+1/2}_i) - \nabla r_i (x^k_i) , x^{k+1/2}_i - x^{k+1}_i\rangle.
\end{align*}
By simple fact: $2 \langle a, b \rangle \leq \eta \| a \|^2 + \tfrac{1}{\eta}\| b \|^2$ with $a = d_i \cdot [A_i^T (y^{k+1/2} - u^k)]_{(j^k_i)}$, $b = x^{k+1/2}_i - x^{k+1}_i$, $\eta = 2\gamma$ and $a = \nabla r_i(x^{k+1/2}_i) - \nabla r_i (x^k_i)$, $b = x^{k+1/2}_i - x^{k+1}_i$, $\eta = 2\gamma$, we get
\begin{align*}
    \|x^{k+1} - x \|^2
    =&
    \tau \|x^{k} - x \|^2 + (1 - \tau) \| w^k - x \|^2 
    \\
    &
    - \tau \|  x^{k+1/2} - x^k\|^2 - (1 - \tau) \| w^k -x^{k+1/2} \|^2 - \|x^{k+1/2} - x^{k+1} \|^2
    \\
    &
    -
    2\gamma \langle  A  (x^{k+1/2} - x),  y^{k+1/2}\rangle
    -
    2\gamma \sum\limits_{i=1}^n\langle \nabla r_i(x^{k+1/2}_i),  x^{k+1/2}_i - x_i\rangle
    \\
    &
    -
    2\gamma \sum\limits_{i=1}^n \langle  d_i \cdot \langle A_i^T (y^{k+1/2} - u^k), e_{j^k_i} \rangle e_{j^k_i} -  A_i^T (y^{k+1/2} - u^k),  x^{k+1/2}_i - x_i\rangle
    \\
    &
    + 2 \gamma^2 \sum\limits_{i=1}^n \| d_i \cdot \langle A_i^T (y^{k+1/2} - u^k), e_{j^k_i} \rangle e_{j^k_i} \|^2 + \frac{1}{2}\| x^{k+1/2} - x^{k+1} \|^2
    \\
    &
    + 2\gamma^2\sum\limits_{i=1}^n \|\nabla r_i(x^{k+1/2}_i) - \nabla r_i (x^k_i)\|^2 + \frac{1}{2}\| x^{k+1/2} - x^{k+1} \|^2
    \\
    =&
    \tau \|x^{k} - x \|^2 + (1 - \tau) \| w^k - x \|^2 
    \\
    &
    - \tau \|  x^{k+1/2} - x^k\|^2 - (1 - \tau) \| w^k -x^{k+1/2} \|^2 
    \\
    &
    -
    2\gamma \langle  A  (x^{k+1/2} - x),  y^{k+1/2}\rangle
    -
    2\gamma \sum\limits_{i=1}^n\langle \nabla r_i(x^{k+1/2}_i),  x^{k+1/2}_i - x_i\rangle
    \\
    &
    -
    2\gamma \sum\limits_{i=1}^n \langle  d_i \cdot \langle A_i^T (y^{k+1/2} - u^k), e_{j^k_i} \rangle e_{j^k_i} -  A_i^T (y^{k+1/2} - u^k),  x^{k+1/2}_i - x_i\rangle
    \\
    &
    + 2 \gamma^2 \sum\limits_{i=1}^n \| d_i \cdot \langle A_i^T (y^{k+1/2} - u^k), e_{j^k_i} \rangle e_{j^k_i}\|^2 
    \\
    &
    + 2\gamma^2\sum\limits_{i=1}^n \|\nabla r_i(x^{k+1/2}_i) - \nabla r_i (x^k_i)\|^2.
\end{align*}
The analogue of (\ref{eq:unbiased_proof_temp1}) is
\begin{align*}
    \|x^{k+1} - x \|^2 + \| w^{k+1}  - x \|^2
    \leq&
    \|x^{k} - x \|^2 + \| w^k - x \|^2 
    \notag\\
    &
    - \tau \|  x^{k+1/2} - x^k\|^2 - (1 - \tau) \| w^k -x^{k+1/2} \|^2
    \notag\\
    &
    -
    2\gamma \langle  A  (x^{k+1/2} - x),  y^{k+1/2}\rangle
    -
    2\gamma \sum\limits_{i=1}^n\langle \nabla r_i(x^{k+1/2}_i),  x^{k+1/2}_i - x_i\rangle
    \notag\\
    &- (1 - \tau) \|x^{k}\|^2 - \tau \| w^k \|^2  + \|w^{k+1} \|^2
    \notag\\
    &+ 2\langle (1 - \tau) x^k + \tau w^k - w^{k+1}, x\rangle 
    \notag\\
    &
    -
    2\gamma \sum\limits_{i=1}^n \langle  d_i \cdot \langle A_i^T (y^{k+1/2} - u^k), e_{j^k_i} \rangle e_{j^k_i} -  A_i^T (y^{k+1/2} - u^k),  x^{k+1/2}_i - x^0_i\rangle
    \notag\\
    &
    -
    2\gamma \sum\limits_{i=1}^n \langle  d_i \cdot \langle A_i^T (y^{k+1/2} - u^k), e_{j^k_i} \rangle e_{j^k_i} -  A_i^T (y^{k+1/2} - u^k),  x^{0}_i - x_i\rangle
    \notag\\
    &
    + 2 \gamma^2 \sum\limits_{i=1}^n \| d_i \cdot \langle A_i^T (y^{k+1/2} - u^k), e_{j^k_i} \rangle e_{j^k_i}\|^2 
    \notag\\
    &
    + 2 \gamma^2 \sum\limits_{i=1}^n \| \nabla r_i(x^{k+1/2}_i) - \nabla r_i (x^k_i) \|^2.
\end{align*}
(\ref{eq:unbiased_proof_temp2}) is absolutely the same. The analogue of (\ref{eq:unbiased_proof_temp3}) is 
\begin{align*}
    \| y^{k+1} - y \|^2 + \| u^{k+1} - y \|^2
    \leq&
    \|y^{k} - y \|^2 + \| u^k - y \|^2 
    \notag\\
    &
    - \tau \|  y^{k+1/2} - y^k\|^2 - (1 - \tau) \| u^k -y^{k+1/2} \|^2
    \notag\\
    &
    -
    2\gamma \langle  z^{k+1/2},  y^{k+1/2} - y\rangle
    +
    2\gamma \langle \sum_{i=1}^n A_i x^{k+1/2}_i,  y^{k+1/2} - y\rangle
    \notag\\
    &
    +
    2\gamma \langle \sum_{i=1}^n [s \cdot \langle A_i(x^{k+1/2}_i - w^{k}_i), e_{c^k_i} \rangle e_{c^k_i} + A_i w^{k}_i - A_i x^{k+1/2}_i ] ,  y^{k+1/2} - y^0\rangle
    \notag\\
    &
    +
    2\gamma \langle \sum_{i=1}^n [s \cdot \langle A_i(x^{k+1/2}_i - w^{k}_i), e_{c^k_i} \rangle e_{c^k_i} + A_i w^{k}_i - A_i x^{k+1/2}_i ] ,  y^{0} - y\rangle
    \notag\\
    &- (1 - \tau) \|y^{k}\|^2 - \tau \| u^k \|^2  + \|y^{k+1} \|^2
    \notag\\
    &+ 2\langle (1 - \tau) y^k + \tau u^k - u^{k+1}, y\rangle 
    \notag\\
    &
    + 2 \gamma^2 \| \sum_{i=1}^n s \cdot \langle A_i(x^{k+1/2}_i - w^{k}_i), e_{c^k_i} \rangle e_{c^k_i} \|^2
    + 2 \gamma^2 \| z^{k+1/2} - z^k \|^2.
\end{align*}
The analogue of (\ref{eq:unbiased_temp5}) is
\begin{align*}
    2\gamma \E{\text{gap}(\bar x^K, \bar z^K, \bar y^K)}
    \leq&
    \frac{1}{K}\bigg(\max_{x \in \cX}\|x^{0} - x \|^2 + \max_{x \in \cX}\| w^0 - x \|^2 + \max_{z \in \cZ}\|z^{0} - z \|^2
    \notag\\
    &
    + \max_{y \in \cY}\|y^{0} - y \|^2 + \max_{y \in \cY}\| u^0 - y \|^2 \bigg)
    \notag\\
    &
    -  \frac{\tau}{K} \sum\limits_{k=0}^{K-1} \E{\|  x^{k+1/2} - x^k\|^2} -  \frac{1 - \tau}{K} \sum\limits_{k=0}^{K-1} \E{\| w^k -x^{k+1/2} \|^2} 
    \notag\\
    &
    - \frac{1}{K} \sum\limits_{k=0}^{K-1} \E{\|  z^{k+1/2} - z^k\|^2}
    \notag\\
    &
    - \frac{\tau}{K} \sum\limits_{k=0}^{K-1} \E{\|  y^{k+1/2} - y^k\|^2} - \frac{1 - \tau}{K} \sum\limits_{k=0}^{K-1} \E{\| u^k -y^{k+1/2} \|^2}
    \notag\\
    &
    +\frac{1}{K} \sum\limits_{k=0}^{K-1} \E{ \|w^{k+1} \|^2 - (1 - \tau) \|x^{k}\|^2 - \tau \| w^k \|^2} 
    \notag\\
    &
    + \frac{2}{K} \E{\max_{x\in \cX} \sum\limits_{k=0}^{K-1} \langle (1 - \tau) x^k + \tau w^k - w^{k+1}, x\rangle} 
    \notag\\
    &
    +\frac{1}{K} \sum\limits_{k=0}^{K-1} \E{  \|y^{k+1} \|^2 - (1 - \tau) \|y^{k}\|^2 - \tau \| u^k \|^2 }
    \notag\\
    &
    + \frac{2}{K} \E{\max_{y \in \cY}\sum\limits_{k=0}^{K-1} \langle (1 - \tau) y^k + \tau u^k - u^{k+1}, y\rangle }
    \notag\\
    &
    -
    \frac{2\gamma}{K} \sum\limits_{k=0}^{K-1} \sum\limits_{i=1}^n \E{\langle  d_i \cdot [A_i^T (y^{k+1/2} - u^k)]_{(j^k_i)} -  A_i^T (y^{k+1/2} - u^k),  x^{k+1/2}_i - x^0_i\rangle}
    \notag\\
    &
    +
    \frac{2\gamma}{K} \cdot \E{ \max_{x \in \cX}\sum\limits_{k=0}^{K-1} \sum\limits_{i=1}^n \langle  d_i \cdot [A_i^T (y^{k+1/2} - u^k)]_{(j^k_i)} -  A_i^T (y^{k+1/2} - u^k),  x_i - x^0_i\rangle}
    \notag\\
    &
    +
    \frac{2\gamma}{K} \sum\limits_{k=0}^{K-1} \E{\langle \sum_{i=1}^n [s \cdot \langle A_i(x^{k+1/2}_i - w^{k}_i), e_{c^k_i} \rangle e_{c^k_i} + A_i w^{k}_i - A_i x^{k+1/2}_i ] ,  y^{k+1/2} - y^0\rangle}
    \notag\\
    &
    +
    \frac{2\gamma}{K} \cdot \E{ \max_{y \in \cY}\sum\limits_{k=0}^{K-1} \langle \sum_{i=1}^n [s \cdot \langle A_i(x^{k+1/2}_i - w^{k}_i), e_{c^k_i} \rangle e_{c^k_i} + A_i w^{k}_i - A_i x^{k+1/2}_i ] ,  y^{0} - y\rangle}
    \notag\\
    &
    + \frac{2\gamma^2}{K} \sum\limits_{k=0}^{K-1} \sum\limits_{i=1}^n \E{\| d_i \cdot [A_i^T (y^{k+1/2} - u^k)]_{(j^k_i)}\|^2}
    \notag\\
    &
    + \frac{2\gamma^2}{K} \sum\limits_{k=0}^{K-1} \E{\| \sum_{i=1}^n s \cdot \langle A_i(x^{k+1/2}_i - w^{k}_i), e_{c^k_i} \rangle e_{c^k_i} \|^2}
    \notag\\
    &
    + \frac{2\gamma^2 L_r^2}{K} \sum\limits_{k=0}^{K-1} \E{\|x^{k+1/2} - x^k \|^2}
    + \frac{2\gamma^2}{K} \sum\limits_{k=0}^{K-1} \E{\|y^{k+1/2} - y^k\|^2}
    \notag\\
    &
    +
    \frac{2\gamma^2}{K} \sum\limits_{k=0}^{K-1} \E{\| z^{k+1/2} - z^k \|^2 }
    +
    \frac{2\gamma^2 L_{\ell}^2}{K} \sum\limits_{k=0}^{K-1} \E{\| z^{k+1/2} - z^k \|^2}.
\end{align*}
(\ref{eq:unbiased_temp6}), (\ref{eq:unbiased_temp7}), (\ref{eq:unbiased_temp8}), (\ref{eq:unbiased_temp9}) are absolutely the same. The analogue of (\ref{eq:unbiased_temp10}) is
\begin{align*}
&\E{\langle   d_i \cdot \langle A_i^T (y^{k+1/2} - u^k), e_{j^k_i} \rangle e_{j^k_i} -  A_i^T (y^{k+1/2} - u^k),  x^{k+1/2}_i - x^0_i\rangle}
\notag\\
&\hspace{3.5cm}=
\E{\langle  \mathbb{E}_{j^k_i} \big[ d_i \cdot \langle A_i^T (y^{k+1/2} - u^k), e_{j^k_i} \rangle e_{j^k_i}\big] -  A_i^T (y^{k+1/2} - u^k),  x^{k+1/2}_i - x^0_i\rangle}
= 0.
\end{align*}
The analogue of (\ref{eq:unbiased_temp11}) is
\begin{align*}
&\E{\langle \sum_{i=1}^n [s \cdot \langle A_i(x^{k+1/2}_i - w^{k}_i), e_{c^k_i} \rangle e_{c^k_i} + A_i w^{k}_i - A_i x^{k+1/2}_i ] ,  y^{k+1/2} - y^0\rangle}
\notag\\
&\hspace{3.5cm}=
\E{\langle \sum_{i=1}^n \mathbb{E}_{c^k_i}\big[s \cdot \langle A_i(x^{k+1/2}_i - w^{k}_i), e_{c^k_i} \rangle e_{c^k_i}\big] + A_i w^{k}_i - A_i x^{k+1/2}_i ,  y^{k+1/2} - y^0\rangle} = 0.
\end{align*}
The analogue of (\ref{eq:unbiased_temp12}) is
\begin{align*}
&\E{ \max_{x \in \cX}\sum\limits_{i=1}^n \sum\limits_{k=0}^{K-1} \langle   d_i \cdot \langle A_i^T (y^{k+1/2} - u^k), e_{j^k_i} \rangle e_{j^k_i} -  A_i^T (y^{k+1/2} - u^k),  x_i - x^0_i\rangle}
\notag\\
&\hspace{2.5cm}\leq \E{\max_{x \in \cX} \frac{1}{2 \gamma}  \sum\limits_{i=1}^n\| x^0_i - x_i \|^2}
\notag\\
&\hspace{2.9cm}
+ \E{\frac{\gamma}{2} \sum\limits_{i=1}^n \| \sum\limits_{k=0}^{K-1}  d_i \cdot \langle A_i^T (y^{k+1/2} - u^k), e_{j^k_i} \rangle e_{j^k_i} -  A_i^T (y^{k+1/2} - u^k)\|^2}
\notag\\
&\hspace{2.5cm}= \E{\max_{x \in \cX} \frac{1}{2\gamma} \| x^0 - x \|^2} 
\notag\\
&\hspace{2.9cm}
+ \E{\frac{\gamma}{2} \sum\limits_{i=1}^n \sum\limits_{k=0}^{K-1}  \|  d_i \cdot \langle A_i^T (y^{k+1/2} - u^k), e_{j^k_i} \rangle e_{j^k_i} -  A_i^T (y^{k+1/2} - u^k) \|^2}
\notag\\
&\hspace{2.9cm}+ \mathbb{E}\bigg[\gamma \sum\limits_{i=1}^n \sum\limits_{k_1 < k_2}  \langle  d_i \cdot \langle A_i^T (y^{k+1/2} - u^k), e_{j^k_i} \rangle e_{j^k_i} -  A_i^T (y^{k_1+1/2} - u^{k_1}), 
\notag\\
&\hspace{4.9cm}
d_i \cdot [A_i^T (y^{k_2+1/2} - u^{k_2})]_{(j^{k_2}_i)} -  A_i^T (y^{k_2+1/2} - u^k) \rangle \bigg]
\notag\\
&\hspace{2.5cm}= \E{\max_{x \in \cX} \frac{1}{2\gamma} \| x^0 - x \|^2} 
\notag\\
&\hspace{2.9cm}
+ \E{\frac{\gamma}{2} \sum\limits_{i=1}^n \sum\limits_{k=0}^{K-1}  \|  d_i \cdot \langle A_i^T (y^{k+1/2} - u^k), e_{j^k_i} \rangle e_{j^k_i} -  A_i^T (y^{k+1/2} - u^k) \|^2}
\notag\\
&\hspace{2.9cm}+ \mathbb{E}\bigg[\gamma \sum\limits_{i=1}^n \sum\limits_{k_1 < k_2}  \langle  d_i \cdot \langle A_i^T (y^{k+1/2} - u^k), e_{j^k_i} \rangle e_{j^k_i} -  A_i^T (y^{k_1+1/2} - u^{k_1}), 
\notag\\
&\hspace{4.9cm}
\mathbb{E}_{j^{k_2}_i} \big[d_i \cdot [A_i^T (y^{k_2+1/2} - u^{k_2})]_{(j^{k_2}_i)}\big] -  A_i^T (y^{k_2+1/2} - u^k) \rangle \bigg]
\notag\\
&\hspace{2.5cm}= \E{\max_{x \in \cX} \frac{1}{2\gamma} \| x^0 - x \|^2}
\notag\\
&\hspace{2.9cm} + \E{\frac{\gamma}{2} \sum\limits_{i=1}^n \sum\limits_{k=0}^{K-1}  \|  d_i \cdot \langle A_i^T (y^{k+1/2} - u^k), e_{j^k_i} \rangle e_{j^k_i} -  A_i^T (y^{k+1/2} - u^k) \|^2}
\notag\\
&\hspace{2.5cm}= \E{\max_{x \in \cX} \frac{1}{2\gamma} \| x^0 - x \|^2} 
\notag\\
&\hspace{2.9cm}
+ \E{\frac{\gamma}{2} \sum\limits_{i=1}^n \sum\limits_{k=0}^{K-1}  \| d_i \cdot [A_i^T (y^{k+1/2} - u^k)]_{(j^k_i)} -  \mathbb{E}_{j^k_i} \big[ d_i \cdot \langle A_i^T (y^{k+1/2} - u^k), e_{j^k_i} \rangle e_{j^k_i} \big]\|^2}
\notag\\
&\hspace{2.5cm}\leq \E{\max_{x \in \cX} \frac{1}{2\gamma} \| x^0 - x \|^2} + \E{\frac{\gamma}{2} \sum\limits_{i=1}^n \sum\limits_{k=0}^{K-1}  \|  d_i \cdot \langle A_i^T (y^{k+1/2} - u^k), e_{j^k_i} \rangle e_{j^k_i}\|^2}.
\end{align*}
The analogue of (\ref{eq:unbiased_temp13}) is
\begin{align*}
&\E{ \max_{y \in \cY}\sum\limits_{k=0}^{K-1} \langle \sum_{i=1}^n [s \cdot \langle A_i(x^{k+1/2}_i - w^{k}_i), e_{c^k_i} \rangle e_{c^k_i} + A_i w^{k}_i - A_i x^{k+1/2}_i ]  ,  y^{0} - y\rangle}
\notag\\
&\hspace{3.5cm}\leq \E{\max_{y \in \cY} \frac{1}{2\gamma} \| y^0 - y\|^2} + \E{\frac{\gamma}{2} \sum\limits_{k=0}^{K-1}  \| \sum_{i=1}^n s \cdot \langle A_i(x^{k+1/2}_i - w^{k}_i), e_{c^k_i} \rangle e_{c^k_i} \|^2}.
\end{align*}
The analogue of (\ref{eq:unbiased_temp14}) is
\begin{align}
    \label{eq:coord_temp1}
    2\gamma \E{\text{gap}(\bar x^K, \bar z^K, \bar y^K)}
    &\leq
    \frac{1}{K}\bigg(6\max_{x \in \cX}\|x^{0} - x \|^2 + \max_{x \in \cX}\| w^0 - x \|^2 + \max_{z \in \cZ}\|z^{0} - z \|^2
    \notag\\
    &\hspace{.4cm}
    + 6\max_{y \in \cY}\|y^{0} - y \|^2 + \max_{y \in \cY}\| u^0 - y \|^2 \bigg)
    \notag\\
    &\hspace{.4cm}
    -  \left( \frac{3\tau - 1}{2} - 2\gamma^2 L_r^2 \right) \cdot \frac{1}{K} \sum\limits_{k=0}^{K-1} \E{\|  x^{k+1/2} - x^k\|^2} 
    \notag\\
    &\hspace{.4cm}
    -  \frac{1 - \tau}{2K} \sum\limits_{k=0}^{K-1} \E{\| w^k -x^{k+1/2} \|^2} 
    \notag\\
    &\hspace{.4cm}
    - \left( 1 - 2\gamma^2 (1 + L_{\ell}^2) \right)\frac{1}{K} \sum\limits_{k=0}^{K-1} \E{\|  z^{k+1/2} - z^k\|^2}
    \notag\\
    &\hspace{.4cm}
    - \left( \frac{3\tau - 1}{2} - 2 \gamma^2 \right)\frac{1}{K} \sum\limits_{k=0}^{K-1} \E{\|  y^{k+1/2} - y^k\|^2} 
    \notag\\
    &\hspace{.4cm}
    - \frac{1 - \tau}{2K} \sum\limits_{k=0}^{K-1} \E{\| u^k -y^{k+1/2} \|^2}
    \notag\\
    &\hspace{.4cm}
    + \frac{3\gamma^2}{K} \sum\limits_{k=0}^{K-1} \sum\limits_{i=1}^n \E{\|  d_i \cdot \langle A_i^T (y^{k+1/2} - u^k), e_{j^k_i} \rangle e_{j^k_i}\|^2}
    \notag\\
    &\hspace{.4cm}
    + \frac{3\gamma^2}{K} \sum\limits_{k=0}^{K-1} \E{\| \sum_{i=1}^n s \cdot \langle A_i(x^{k+1/2}_i - w^{k}_i), e_{c^k_i} \rangle e_{c^k_i} \|^2}.
\end{align}
Let us estimate two last lines. Here we use that coordinates  $j_i$ and $c_i$ are chosen uniformly and independently. 
\begin{align*}
    \E{\|  d_i \cdot \langle A_i^T (y^{k+1/2} - u^k), e_{j^k_i} \rangle e_{j^k_i}\|^2} 
    =&
    d_i^2 \E{ \mathbb{E}_{e_{j^k_i}}\left[\| \langle A_i^T (y^{k+1/2} - u^k), e_{j^k_i} \rangle e_{j^k_i}\|^2\right]} 
    \\
    =&
    d_i^2 \E{ \frac{1}{d_i} \sum\limits_{r=1}^{d_i}\left[\| \langle A_i^T (y^{k+1/2} - u^k), e_{r} \rangle e_{r}\|^2\right]} 
    \\
    =&
    d_i \E{\| A_i^T (y^{k+1/2} - u^k)\|^2} 
    \\
    \leq&
    d_i \lambda_{\max}(A_i A_i^T) \E{\| y^{k+1/2} - u^k\|^2}.
\end{align*}
For $\E{\| \sum_{i=1}^n s \cdot \langle A_i(x^{k+1/2}_i - w^{k}_i), e_{c^k_i} \rangle e_{c^k_i}\|^2}$ we have two options. If $c^k_i = c^k$ for all $i$, then 
$\sum_{i=1}^n \sum_{i=1}^n s \cdot \langle A_i(x^{k+1/2}_i - w^{k}_i), e_{c^k_i} \rangle e_{c^k_i} = s\langle\sum_{i=1}^n [A_i (x^{k+1/2}_i-  w^{k}_i)], e_{c^k} \rangle e_{c^k} = s\langle A (x^{k+1/2} -  w^{k}), e_{c^k} \rangle e_{c^k}$, then
\begin{align*}
\E{\| \sum_{i=1}^n s \cdot \langle A_i(x^{k+1/2}_i - w^{k}_i), e_{c^k_i} \rangle e_{c^k_i} \|^2}
=& 
\E{\| s\langle A (x^{k+1/2}  - w^{k}), e_{c^k} \rangle e_{c^k} \|^2}
\\
=& 
s^2 \E{ \mathbb{E}_{c^k} \left[\| \langle A (x^{k+1/2}  - w^{k}), e_{c^k} \rangle e_{c^k} \|^2\right]}
\\
=& s^2 \E{ \frac{1}{s} \sum\limits_{r=1}^s \| \langle A (x^{k+1/2}  - w^{k}), e_{r} \rangle e_{r} \|^2}
\\
=& s \E{ \left[\| A (x^{k+1/2}  - w^{k}) \|^2\right]}
\\
\leq& s \lambda_{\max}(A^T A) \E{ \| x^{k+1/2}  - w^{k}\|^2}.
\end{align*}
If $c^k_i$ are chosen independently (i.e. $c^k_i \neq c^k_j$), then
\begin{align*}
&\E{\|\sum_{i=1}^n s \cdot \langle A_i(x^{k+1/2}_i - w^{k}_i), e_{c^k_i} \rangle e_{c^k_i} \|^2}
\\
&\hspace{3.4cm}= 
\sum_{i=1}^n \E{\| s \cdot \langle A_i(x^{k+1/2}_i - w^{k}_i), e_{c^k_i} \rangle e_{c^k_i} \|^2}
\\
&\hspace{3.8cm}+\sum_{i \neq j} \E{\langle s\cdot \langle A_i(x^{k+1/2}_i - w^{k}_i), e_{c^k_i} \rangle e_{c^k_i} , s \cdot \langle A_j(x^{k+1/2}_j - w^{k}_j), e_{c^k_j} \rangle e_{c^k_j} \rangle}
\\
&\hspace{3.4cm}= 
\sum_{i=1}^n \E{\| s \cdot \langle A_i(x^{k+1/2}_i - w^{k}_i), e_{c^k_i} \rangle e_{c^k_i} \|^2}
\\
&\hspace{3.8cm}+\sum_{i \neq j} \mathbb{E} \bigg[\langle \mathbb{E}_{c^k_i} \left[s\cdot \langle A_i(x^{k+1/2}_i - w^{k}_i), e_{c^k_i} \rangle e_{c^k_i} \right], 
\\
&\hspace{5.8cm}
\mathbb{E}_{c^k_j} \left[s \cdot \langle A_j(x^{k+1/2}_j - w^{k}_j), e_{c^k_j} \rangle e_{c^k_j} \rangle \right] \bigg]
\\
&\hspace{3.4cm}=  
\sum_{i=1}^n \E{\| s \cdot \langle A_i(x^{k+1/2}_i - w^{k}_i), e_{c^k_i} \rangle e_{c^k_i} \|^2}
\\
&\hspace{3.8cm}+\sum_{i \neq j} \E{\langle A_i(x^{k+1/2}_i - w^{k}_i), A_j(x^{k+1/2}_j - w^{k}_j)}
\\
&\hspace{3.4cm}= s^2\sum_{i=1}^n \E{\| \langle A_i(x^{k+1/2}_i - w^{k}_i), e_{c^k_i} \rangle e_{c^k_i} \|^2}
\\
&\hspace{1.8cm}+ \E{\| \sum_{i=1}^n [A_i x^{k+1/2}_i - A_i w^{k}_i] \|^2}
- \sum_{i=1}^n \E{\| A_i x^{k+1/2}_i - A_i w^{k}_i \|^2}
\\
&\hspace{1.4cm}\leq s \sum_{i=1}^n \E{\|A_i x^{k+1/2}_i - A_i w^{k}_i \|^2}
 + \E{\| A (x^{k+1/2} - w^{k}) \|^2}
 \\
&\hspace{1.4cm}\leq s \sum_{i=1}^n \lambda_{\max} (A_i^T A_i) \E{\|x^{k+1/2}_i - w^{k}_i \|^2}
 + \lambda_{\max} (A^T A) \E{\| x^{k+1/2} - w^{k} \|^2}
 \\
&\hspace{1.4cm}\leq s \max_{i} \left\{\lambda_{\max} (A_i^T A_i) \right\}\sum_{i=1}^n \E{\|x^{k+1/2}_i - w^{k}_i \|^2}
 \\
&\hspace{1.8cm} 
 + \lambda_{\max} (A^T A) \E{\| x^{k+1/2} - w^{k} \|^2}
  \\
&\hspace{1.4cm}= \left(s \max_{i} \left\{\lambda_{\max} (A_i^T A_i) \right\}+ \lambda_{\max} (A^T A) \right) \E{\| x^{k+1/2} - w^{k} \|^2}.
\end{align*}
Let us introduce
$$
\chi_{\text{coord}} =
\begin{cases}
s \lambda_{\max} (A^T A),\\
s \max_{i} \left\{\lambda_{\max} (A_i^T A_i) \right\}+ \lambda_{\max} (A^T A),
\end{cases}
$$
depending on the case $c_i$ we consider. It remains to come back to (\ref{eq:coord_temp1}) and get
\begin{align*}
    2\gamma \E{\text{gap}(\bar x^K, \bar z^K, \bar y^K)}
    \leq&
    \frac{1}{K}\bigg(6\max_{x \in \cX}\|x^{0} - x \|^2 + \max_{x \in \cX}\| w^0 - x \|^2 + \max_{z \in \cZ}\|z^{0} - z \|^2
    \notag\\
    &
    + 6\max_{y \in \cY}\|y^{0} - y \|^2 + \max_{y \in \cY}\| u^0 - y \|^2 \bigg)
    \notag\\
    &
    -  \left( \frac{3\tau - 1}{2} - 2\gamma^2 L_r^2 \right) \cdot \frac{1}{K} \sum\limits_{k=0}^{K-1} \E{\|  x^{k+1/2} - x^k\|^2} 
    \notag\\
    &
    -  \frac{1 - \tau}{2K} \sum\limits_{k=0}^{K-1} \E{\| w^k -x^{k+1/2} \|^2} 
    \notag\\
    &
    - \left( 1 - 2\gamma^2 (1 + L_{\ell}^2) \right)\frac{1}{K} \sum\limits_{k=0}^{K-1} \E{\|  z^{k+1/2} - z^k\|^2}
    \notag\\
    &
    - \left( \frac{3\tau - 1}{2} - 2 \gamma^2 \right)\frac{1}{K} \sum\limits_{k=0}^{K-1} \E{\|  y^{k+1/2} - y^k\|^2} 
    \notag\\
    &
    - \frac{1 - \tau}{2K} \sum\limits_{k=0}^{K-1} \E{\| u^k -y^{k+1/2} \|^2}
    \notag\\
    &
    + \frac{3\gamma^2}{K} \sum\limits_{k=0}^{K-1} \sum\limits_{i=1}^n d_i \lambda_{\max}(A_i A_i^T) \E{\| y^{k+1/2} - u^k\|^2}
    \notag\\
    &
    + \frac{3\gamma^2 \chi_{\text{coord}} }{K} \sum\limits_{k=0}^{K-1} \E{\| x^{k+1/2} - w^{k} \|^2}
    \notag\\
    \leq&
    \frac{1}{K}\bigg(6\max_{x \in \cX}\|x^{0} - x \|^2 + \max_{x \in \cX}\| w^0 - x \|^2 + \max_{z \in \cZ}\|z^{0} - z \|^2
    \notag\\
    &
    + 6\max_{y \in \cY}\|y^{0} - y \|^2 + \max_{y \in \cY}\| u^0 - y \|^2 \bigg)
    \notag\\
    &
    -  \left( \frac{3\tau - 1}{2} - 2\gamma^2 L_r^2 \right) \cdot \frac{1}{K} \sum\limits_{k=0}^{K-1} \E{\|  x^{k+1/2} - x^k\|^2} 
    \notag\\
    &
    -  \frac{1 - \tau}{2K} \sum\limits_{k=0}^{K-1} \E{\| w^k -x^{k+1/2} \|^2} 
    \notag\\
    &
    - \left( 1 - 2\gamma^2 (1 + L_{\ell}^2) \right)\frac{1}{K} \sum\limits_{k=0}^{K-1} \E{\|  z^{k+1/2} - z^k\|^2}
    \notag\\
    &
    - \left( \frac{3\tau - 1}{2} - 2 \gamma^2 \right)\frac{1}{K} \sum\limits_{k=0}^{K-1} \E{\|  y^{k+1/2} - y^k\|^2} 
    \notag\\
    &
    - \frac{1 - \tau}{2K} \sum\limits_{k=0}^{K-1} \E{\| u^k -y^{k+1/2} \|^2}
    \notag\\
    &
    + \frac{3\gamma^2}{K} \cdot d \max_{i} \left\{\lambda_{\max} (A_i^T A_i) \right\} \sum\limits_{k=0}^{K-1} \E{\| y^{k+1/2} - u^k\|^2}
    \notag\\
    &
    + \frac{3\gamma^2 \chi_{\text{coord}} }{K} \sum\limits_{k=0}^{K-1} \E{\| x^{k+1/2} - w^{k} \|^2}.
\end{align*}
If we choose $\tau \geq \tfrac{1}{2}$ and $\gamma$ as follows
$$
\gamma \leq \frac{1}{4}\min\left\{ 1;\frac{1}{L_r}; \frac{1}{L_{\ell}};\sqrt{\frac{1-\tau}{\chi_{\text{coord}}}}; \sqrt{\frac{1-\tau}{d \max_{i} \{\lambda_{\max} (A_i^T A_i)\}}}; \right\},
$$
then one can obtain
\begin{align*}
    &\E{\text{gap}(\bar x^K, \bar z^K, \bar y^K)}
    \leq
    \frac{1}{2\gamma K}\bigg(6\max_{x \in \cX}\|x^{0} - x \|^2 + \max_{x \in \cX}\| w^0 - x \|^2 + \max_{z \in \cZ}\|z^{0} - z \|^2
    \notag\\
    &\hspace{4.5cm}
    + 6\max_{y \in \cY}\|y^{0} - y \|^2 + \max_{y \in \cY}\| u^0 - y \|^2 \bigg).
\end{align*}
With $\gamma = \frac{1}{4}\min\left\{ 1;\frac{1}{L_r}; \frac{1}{L_{\ell}};\sqrt{\frac{1-\tau}{\chi_{\text{coord}}}}; \sqrt{\frac{1-\tau}{d \max_{i} \{\lambda_{\max} (A_i^T A_i)\}}}; \right\}$, we finish the proof.
\end{proof}

\subsection{Proof of Theorem \ref{th:EG_basic_2}}
\label{app:th:EG_basic_2}

{
\begin{theorem}[Theorem \ref{th:EG_basic_2}]
Let Assumption \ref{as:convexity_smothness} hold. Let the problem (\ref{eq:vfl_lin_spp_2})
be solved by Algorithm~\ref{alg:EG_basic_2}. Then for 
$$
\gamma = \frac{1}{2} \cdot \min\left\{ 1; \frac{1}{\sqrt{\max_i\{ \lambda_{\max}(A_i^T A_i) \}}}; \frac{1}{L_r}; \frac{1}{n L_{\ell}}\right\},
$$ 
it holds that
$$
\text{gap}_1(\bar x^{K}, \bar z^K, \bar y^K) = \mathcal{O}\left(\frac{\left( 1+ \sqrt{\max_{i = 1, \ldots, n}\{ \lambda_{\max}(A_i^T A_i) \}} + nL_{\ell} + L_r\right)D^2}{K} \right),
$$
where $\text{gap}_1(x,y) := \max_{\tilde y_i \in \mathcal{ \tilde Y}} \tilde L(x,z, \tilde y) - \min_{\tilde x,z \in \cX, \mathcal{\tilde Z}} \tilde L(\tilde x, \tilde z, y)$ and $\bar x^K := \tfrac{1}{K}\sum_{k=0}^{K-1} x^{k+1/2}$, $\bar z^K := \tfrac{1}{K}\sum_{k=0}^{K-1} z^{k+1/2}$, $\bar y^K := \tfrac{1}{K}\sum_{k=0}^{K-1} y^{k+1/2}$ and $D^2 := \max_{x, z, y \in \cX, \cZ, \cY} [ \|x^0 - x \|^2 + \|z^0 - z \|^2 + \|y^0 - y \|^2 ]$.
\end{theorem}
}
\begin{proof}
We start the proof from (\ref{eq:basic_proof_temp0}), since the updates for $x_i$ variables in Algorithms \ref{alg:EG}, \ref{alg:EG_basic_2} are the same (with a slight modification $y$ to $y_i$):
\begin{align*}
    \|x^{k+1}_i - x_i \|^2
    =&
    \|x^{k}_i - x_i \|^2 - \|  x^{k+1/2}_i - x^k_i\|^2 - \|x^{k+1/2}_i - x^{k+1}_i \|^2
    \notag\\
    &-
    2\gamma \langle  A_i (x^{k+1/2}_i - x_i) ,  y^{k+1/2}_i\rangle
    -
    2\gamma \langle  \nabla r_i(x^{k+1/2}_i),  x^{k+1/2}_i - x_i\rangle
    \notag\\
    &-
    2\gamma \langle  x^{k+1}_i - x^{k+1/2}_i ,  A_i^T (y^{k+1/2}_i - y^k_i)\rangle
    \notag\\
    &-
    2\gamma \langle  \nabla r_i(x^{k+1/2}_i) - \nabla r_i (x^k_i),  x^{k+1}_i - x^{k+1/2}_i\rangle.
\end{align*}
By simple fact: $2 \langle a, b \rangle \leq \eta \| a \|^2 + \tfrac{1}{\eta}\| b \|^2$ with $a = A^T_i ( y^{k+1/2}_i - y^{k}_i)$, $b = x^{k+1/2}_i - x^{k+1}_i$, $\eta = 2\gamma$ and $a = \nabla r_i(x^{k+1/2}_i) - \nabla r_i (x^k_i)$, $b = x^{k+1/2}_i - x^{k+1}_i$, $\eta = 2\gamma$, we get
\begin{align*}
    \|x^{k+1}_i - x_i \|^2
    =&
    \|x^{k}_i - x_i \|^2 - \|  x^{k+1/2}_i - x^k_i\|^2
    \notag\\
    &-
    2\gamma \langle  A_i (x^{k+1/2}_i - x_i) ,  y^{k+1/2}_i\rangle
    -
    2\gamma \langle  \nabla r_i(x^{k+1/2}_i),  x^{k+1/2}_i - x_i\rangle
    \notag\\
    &+ 
    2\gamma^2 \| A_i^T (y^{k+1/2}_i - y^k_i)\|^2 +
    2\gamma^2 \|\nabla r_i(x^{k+1/2}_i) - \nabla r_i (x^k_i)\|^2.
\end{align*}

Summing over all $i$ from $1$ to $n$ and using the notation of $A = [A_1, \ldots, A_i, \ldots, A_n] $, $x = [x_1^T, \ldots,x_i^T, \ldots,x_n^T]^T$, we deduce
\begin{align}
    \label{eq:basic_2_proof_temp1}
    \|x^{k+1} - x \|^2
    =&
    \|x^{k} - x \|^2 - \|  x^{k+1/2} - x^k\|^2 - \|x^{k+1/2} - x^{k+1} \|^2
    \notag\\
    &-
    2\gamma \sum\limits_{i=1}^n \langle  A_i (x^{k+1/2}_i - x_i) ,  y^{k+1/2}_i\rangle
    -
    2\gamma \sum\limits_{i=1}^n \langle  \nabla r_i(x^{k+1/2}_i),  x^{k+1/2}_i - x_i\rangle
    \notag\\
    &-
    2\gamma \sum\limits_{i=1}^n \langle  x^{k+1}_i - x^{k+1/2}_i ,  A_i^T (y^{k+1/2}_i - y^k_i)\rangle
    \notag\\
    &-
    2\gamma \sum\limits_{i=1}^n \langle  \nabla r_i(x^{k+1/2}_i) - \nabla r_i (x^k_i),  x^{k+1}_i - x^{k+1/2}_i\rangle.
\end{align}
Using the same steps as for (\ref{eq:basic_2_proof_temp1}), one can obtain for the notation of $z = [z_1^T, \ldots,z_i^T, \ldots,z_n^T]^T$ and \\ $y = [y_1^T, \ldots,y_i^T, \ldots,y_n^T]^T$, 
\begin{align}
\label{eq:basic_2_proof_temp2}
    \|z^{k+1} - z \|^2
    \leq&
    \|z^{k} - z \|^2 - \|  z^{k+1/2} - z^k\|^2 - \|  z^{k+1/2} - z^{k+1}\|^2 
    \notag\\
    &+
    2\gamma \sum\limits_{i=1}^n \langle  y^{k+1/2}_i ,  z^{k+1/2}_i - z_i\rangle
    -
    2\gamma  \langle  \nabla \ell \left( \sum_{j=1}^n z^{k+1/2}_j, b \right),  \sum\limits_{j=1}^n z^{k+1/2}_j - \sum\limits_{j=1}^n z_j\rangle
    \notag\\
    &-
    2\gamma^2 n \|  \nabla \ell\left(\sum_{j=1}^n z^{k+1/2}_j, b \right) -\nabla \ell \left(\sum_{j=1}^n z^{k}_j, b\right) \|^2 +
    2\gamma^2 \|  y^{k+1/2} - y^k \|^2,
\end{align}
and,
\begin{align}
    \label{eq:basic_2_proof_temp3}
    \|y^{k+1} - y \|^2
    \leq&
    \|y^{k} - y \|^2 - \|  y^{k+1/2} - y^k\|^2 - \|  y^{k+1/2} - y^{k+1}\|^2
    \notag\\
    &-
    2\gamma \sum\limits_{i=1}^n\langle  z^{k+1/2}_i ,  y^{k+1/2}_i - y_i\rangle
    +
    2\gamma \sum\limits_{i=1}^n \langle  A_i x^{k+1/2}_i,  y^{k+1/2}_i - y_i\rangle
     \notag\\
    &+
    2 \gamma^2 \sum\limits_{i=1}^n \| A_i (x^{k+1/2}_i - x^k_i) \|^2
     +
    2 \gamma^2 \| z^{k+1/2} - z^k \|^2.
\end{align}
Summing up (\ref{eq:basic_2_proof_temp1}), (\ref{eq:basic_2_proof_temp2}) and (\ref{eq:basic_2_proof_temp3}), we obtain
\begin{align*}
    \|x^{k+1} - x  \|^2 + \|z^{k+1} - z \|^2 + &\|y^{k+1} - y \|^2
    \\
    \leq&
    \|x^{k} - x \|^2 + \|z^{k} - z \|^2 + \|y^{k} - y \|^2
    \notag\\
    &
    - \| x^{k+1/2} - x^k\|^2 - \|  z^{k+1/2} - z^k\|^2 - \|  y^{k+1/2} - y^k\|^2 
    \notag\\
    &-
    2\gamma \sum\limits_{i=1}^n\langle  A_i (x^{k+1/2}_i - x_i) ,  y^{k+1/2}_i\rangle
    +
    2\gamma \sum\limits_{i=1}^n\langle  y^{k+1/2}_i ,  z^{k+1/2}_i - z_i\rangle
    \notag\\
    &-
    2\gamma \sum\limits_{i=1}^n\langle  z^{k+1/2}_i ,  y^{k+1/2}_i - y_i\rangle
    +
    2\gamma \sum\limits_{i=1}^n\langle  A_i x^{k+1/2}_i,  y^{k+1/2}_i - y_i\rangle
    \notag\\
    &-
    2\gamma \sum\limits_{i=1}^n \langle  \nabla r_i(x^{k+1/2}_i),  x^{k+1/2}_i - x_i\rangle
    -
    2\gamma \langle  \nabla \ell \left( \sum_{j=1}^n z^{k+1/2}_j, b \right),  \sum\limits_{j=1}^n z^{k+1/2}_j - \sum\limits_{j=1}^n z_j\rangle
    \notag\\
    &+
    2\gamma^2 \sum\limits_{i=1}^n \|A^T_i (y^{k+1/2}_i - y^k_i)\|^2 +
    2\gamma^2 \sum\limits_{i=1}^n \| \nabla r_i(x^{k+1/2}_i) - \nabla r_i (x^k_i)\|^2
    \notag\\
    &+
    2\gamma^2 \|y^{k+1/2} - y^k\|^2 +
    2\gamma^2 n \| \nabla \ell \left( \sum_{j=1}^n z^{k+1/2}_j, b \right) -\nabla \ell \left( \sum_{j=1}^n z^{k}_j, b \right)\|^2
    \notag\\
    &+
    2\gamma^2 \|z^{k+1/2} - z^k\|^2 +
    2\gamma^2 \sum\limits_{i=1}^n \| A_i (x^{k+1/2}_i - x^k_i) \|^2.
\end{align*}
Using the definition of $\lambda_{\max}(\cdot)$ as a maximum eigenvalue, we get
\begin{align*}
    \|x^{k+1} - x  \|^2 + \|z^{k+1} - z \|^2 + &\|y^{k+1} - y \|^2
    \\
    \leq&
    \|x^{k} - x \|^2 + \|z^{k} - z \|^2 + \|y^{k} - y \|^2
    \notag\\
    &
    - \| x^{k+1/2} - x^k\|^2 - \|  z^{k+1/2} - z^k\|^2 - \|  y^{k+1/2} - y^k\|^2 
    \notag\\
    &+
    2\gamma \sum\limits_{i=1}^n\langle  A_i x_i - z_i,  y^{k+1/2}_i\rangle
    -
    2\gamma \sum\limits_{i=1}^n\langle  A_i x^{k+1/2}_i - z^{k+1/2}_i,  y_i\rangle
    \notag\\
    &-
    2\gamma \sum\limits_{i=1}^n \langle  \nabla r_i(x^{k+1/2}_i),  x^{k+1/2}_i - x_i\rangle
    -
    2\gamma \langle  \nabla \ell \left( \sum_{j=1}^n z^{k+1/2}_j, b \right),  \sum\limits_{j=1}^n z^{k+1/2}_j - \sum\limits_{j=1}^n z_j\rangle
    \notag\\
    &+
    2\gamma^2 \sum\limits_{i=1}^n \lambda_{\max}(A_i A^T_i) \|y^{k+1/2}_i - y^k_i\|^2 +
    2\gamma^2 \sum\limits_{i=1}^n \| \nabla r_i(x^{k+1/2}_i) - \nabla r_i (x^k_i)\|^2
    \notag\\
    &+
    2\gamma^2 \|y^{k+1/2} - y^k\|^2 +
    2\gamma^2 n \| \nabla \ell \left( \sum_{j=1}^n z^{k+1/2}_j, b \right) -\nabla \ell \left( \sum_{j=1}^n z^{k}_j, b \right)\|^2
    \notag\\
    &+
    2\gamma^2 \|z^{k+1/2} - z^k\|^2 +
    2\gamma^2 \sum\limits_{i=1}^n \lambda_{\max}(A_i^T A_i) \| x^{k+1/2}_i - x^k_i \|^2
    \\
    \leq&
    \|x^{k} - x \|^2 + \|z^{k} - z \|^2 + \|y^{k} - y \|^2
    \notag\\
    &
    - \| x^{k+1/2} - x^k\|^2 - \|  z^{k+1/2} - z^k\|^2 - \|  y^{k+1/2} - y^k\|^2 
    \notag\\
    &+
    2\gamma \sum\limits_{i=1}^n\langle  A_i x_i - z_i,  y^{k+1/2}_i\rangle
    -
    2\gamma \sum\limits_{i=1}^n\langle  A_i x^{k+1/2}_i - z^{k+1/2}_i,  y_i\rangle
    \notag\\
    &-
    2\gamma \sum\limits_{i=1}^n \langle  \nabla r_i(x^{k+1/2}_i),  x^{k+1/2}_i - x_i\rangle
    -
    2\gamma \langle  \nabla \ell \left( \sum_{j=1}^n z^{k+1/2}_j, b \right),  \sum\limits_{j=1}^n z^{k+1/2}_j - \sum\limits_{j=1}^n z_j\rangle
    \notag\\
    &+
    2\gamma^2 \max_i\{ \lambda_{\max}(A_i A^T_i) \} \|y^{k+1/2} - y^k\|^2 +
    2\gamma^2 \sum\limits_{i=1}^n \| \nabla r_i(x^{k+1/2}_i) - \nabla r_i (x^k_i)\|^2
    \notag\\
    &+
    2\gamma^2 \|y^{k+1/2} - y^k\|^2 +
    2\gamma^2 n \| \nabla \ell \left( \sum_{j=1}^n z^{k+1/2}_j, b \right) -\nabla \ell \left( \sum_{j=1}^n z^{k}_j, b \right)\|^2
    \notag\\
    &+
    2\gamma^2 \|z^{k+1/2} - z^k\|^2 +
    2\gamma^2 \max_i\{ \lambda_{\max}(A_i^T A_i) \}\| x^{k+1/2} - x^k \|^2.
\end{align*}
Using convexity and $L_{r}$-smoothness of the function $r_i$ with convexity and $L_{\ell}$-smoothness of the function $\ell$, we have
\begin{align*}
    \|x^{k+1} - x  \|^2 + \|z^{k+1} - z \|^2 +& \|y^{k+1} - y \|^2
    \\
    \leq&
    \|x^{k} - x \|^2 + \|z^{k} - z \|^2 + \|y^{k} - y \|^2
    \notag\\
    &
    - \| x^{k+1/2} - x^k\|^2 - \|  z^{k+1/2} - z^k\|^2 - \|  y^{k+1/2} - y^k\|^2 
    \notag\\
    &+
    2\gamma \sum\limits_{i=1}^n\langle  A_i x_i - z_i,  y^{k+1/2}_i\rangle
    -
    2\gamma \sum\limits_{i=1}^n\langle  A_i x^{k+1/2}_i - z^{k+1/2}_i,  y_i\rangle
    \notag\\
    &-
    2\gamma \sum\limits_{i=1}^n (r_i(x^{k+1/2}_i) - r_i(x_i))
    -
    2\gamma (\ell \left( \sum_{j=1}^n z^{k+1/2}_j, b \right) - \ell \left( \sum_{j=1}^n z_j, b \right))
    \notag\\
    &+
    2\gamma^2 \max_i\{ \lambda_{\max}(A_i A^T_i) \} \|y^{k+1/2} - y^k\|^2 +
    2\gamma^2 L_r^2 \| x^{k+1/2} - x^k\|^2
    \notag\\
    &+
    2\gamma^2 \|y^{k+1/2} - y^k\|^2 +
    2\gamma^2 n L^2_{\ell} \| \sum_{j=1}^n z^{k+1/2}_j  - \sum_{j=1}^n z^{k}_j\|^2
    \notag\\
    &+
    2\gamma^2 \|z^{k+1/2} - z^k\|^2 +
    2\gamma^2 \max_i\{ \lambda_{\max}(A_i^T A_i) \}\| x^{k+1/2} - x^k \|^2.
\end{align*}
Cauchy Schwartz inequality in the form: $\|\sum_{j=1}^n (z^{k+1/2}_j - z^{k}_j)\|^2 \leq n \sum_{j=1}^n \|z^{k+1/2}_j - z^{k}_j\|^2$, gives
\begin{align*}
    \|x^{k+1} - x  \|^2 + \|z^{k+1} - z \|^2 + \|y^{k+1} &- y \|^2
    \\
    \leq&
    \|x^{k} - x \|^2 + \|z^{k} - z \|^2 + \|y^{k} - y \|^2
    \notag\\
    &
    - \| x^{k+1/2} - x^k\|^2 - \|  z^{k+1/2} - z^k\|^2 - \|  y^{k+1/2} - y^k\|^2 
    \notag\\
    &+
    2\gamma \sum\limits_{i=1}^n\langle  A_i x_i - z_i,  y^{k+1/2}_i\rangle
    -
    2\gamma \sum\limits_{i=1}^n\langle  A_i x^{k+1/2}_i - z^{k+1/2}_i,  y_i\rangle
    \notag\\
    &-
    2\gamma \sum\limits_{i=1}^n (r_i(x^{k+1/2}_i) - r_i(x_i))
    -
    2\gamma (\ell \left( \sum_{j=1}^n z^{k+1/2}_j, b \right) - \ell \left( \sum_{j=1}^n z_j, b \right))
    \notag\\
    &+
    2\gamma^2 \max_i\{ \lambda_{\max}(A_i A^T_i) \} \|y^{k+1/2} - y^k\|^2 +
    2\gamma^2 L_r^2 \| x^{k+1/2} - x^k\|^2
    \notag\\
    &+
    2\gamma^2 \|y^{k+1/2} - y^k\|^2 +
    2\gamma^2 n^2 L^2_{\ell} \| z^{k+1/2}  - z^{k}\|^2
    \notag\\
    &+
    2\gamma^2 \|z^{k+1/2} - z^k\|^2 +
    2\gamma^2 \max_i\{ \lambda_{\max}(A_i^T A_i) \}\| x^{k+1/2} - x^k \|^2.
\end{align*}
With the choice of $\gamma \leq \tfrac{1}{2} \cdot \min\left\{ 1; \tfrac{1}{\sqrt{\max_i\{ \lambda_{\max}(A_i^T A_i) \}}}; \tfrac{1}{L_r}; \tfrac{1}{n L_{\ell}}\right\}$, we get
\begin{align*}
    \|x^{k+1} - x  \|^2 + \|z^{k+1} - z \|^2 +& \|y^{k+1} - y \|^2
    \\
    \leq&
    \|x^{k} - x \|^2 + \|z^{k} - z \|^2 + \|y^{k} - y \|^2
    \notag\\
    &+
    2\gamma \sum\limits_{i=1}^n\langle  A_i x_i - z_i,  y^{k+1/2}_i\rangle
    -
    2\gamma \sum\limits_{i=1}^n\langle  A_i x^{k+1/2}_i - z^{k+1/2}_i,  y_i\rangle
    \notag\\
    &-
    2\gamma \sum\limits_{i=1}^n (r_i(x^{k+1/2}_i) - r_i(x_i))
    -
    2\gamma (\ell \left( \sum_{j=1}^n z^{k+1/2}_j, b \right) - \ell \left( \sum_{j=1}^n z_j, b \right)).
\end{align*}
After small rearrangements, we obtain
\begin{align*}
    &(\ell \left( \sum_{j=1}^n z^{k+1/2}_j, b \right) - \ell \left( \sum_{j=1}^n z_j, b \right)) + \sum_{i=1}^n \left( r_i (x^{k+1/2}_i) - r_i (x_i)\right) 
    \\
    &+
    \sum\limits_{i=1}^n\langle  A_i x^{k+1/2}_i - z^{k+1/2}_i,  y_i\rangle
    - \sum\limits_{i=1}^n\langle  A_i x_i - z_i,  y^{k+1/2}_i\rangle
    \\
    &\hspace{7cm}\leq
    \frac{1}{2\gamma}\Big(\| x^k - x\|^2 + \| z^k - z\|^2 + \| y^k - y\|^2 
    \notag\\
    &\hspace{8.4cm} -\| x^{k+1} - x\|^2 - \| z^{k+1} - z \|^2 - \| y^{k+1} -  y \|^2\Big).
\end{align*}
Then we sum all over $k$ from $0$ to $K-1$, divide by $K$, and have
\begin{align*}
    &\frac{1}{K} \sum\limits_{k=0}^{K-1}(\ell \left( \sum_{j=1}^n z^{k+1/2}_j, b \right) - \ell \left( \sum_{j=1}^n z_j, b \right)) + \sum_{i=1}^n \frac{1}{K} \sum\limits_{k=0}^{K-1} \left( r_i (x^{k+1/2}_i) - r_i (x_i)\right) 
    \\
    &+
   \sum\limits_{i=1}^n\langle  A_i x^{k+1/2}_i - z^{k+1/2}_i,  y_i\rangle
    - \sum\limits_{i=1}^n\langle  A_i x_i - z_i,  y^{k+1/2}_i\rangle
    \\
    &\hspace{8cm}\leq
    \frac{1}{2\gamma K}\Big(\| x^0 - x\|^2 + \| z^0 - z\|^2 + \| y^0 - y\|^2 
    \notag\\
    &\hspace{8.4cm} -\| x^{K} - x\|^2 - \| z^{K} - z \|^2 - \| y^{K} -  y \|^2\Big)
    \\
    &\hspace{8cm}\leq
    \frac{1}{2\gamma K}(\| x^0 - x\|^2 + \| z^0 - z\|^2 + \| y^0 - y\|^2).
\end{align*}
With Jensen inequality for convex functions $\ell$ and $r_i$, one can note that
\begin{align*}
 \ell \left( \frac{1}{K} \sum\limits_{k=0}^{K-1} \sum_{j=1}^n z^{k+1/2}_j, b \right) \leq \frac{1}{K} \sum\limits_{k=0}^{K-1} \ell \left(\sum_{j=1}^n z^{k+1/2}_j, b\right), \\
 r_i \left( \frac{1}{K} \sum\limits_{k=0}^{K-1} x^{k+1/2}_i \right) \leq \frac{1}{K} \sum\limits_{k=0}^{K-1} r_i (x^{k+1/2}_i).
\end{align*}
Then, with notation $\bar x^K_i = \frac{1}{K} \sum\limits_{k=0}^{K-1} x^{k+1/2}_i$, $\bar z^K_i = \frac{1}{K} \sum\limits_{k=0}^{K-1} z^{k+1/2}_i$, $\bar y^K_i = \frac{1}{K} \sum\limits_{k=0}^{K-1} y^{k+1/2}_i$, we have
\begin{align*}
    &\ell \left( \sum_{i=1}^n \bar z^K_i, b\right) - \ell \left( \sum_{i=1}^n z_i, b \right) + \sum_{i=1}^n \left( r_i (\bar x^{K}_i) - r_i (x_i)\right)
    \\
    &+
   \sum\limits_{i=1}^n\langle  A_i x^{k+1/2}_i - z^{k+1/2}_i,  y_i\rangle
    - \sum\limits_{i=1}^n\langle  A_i x_i - z_i,  y^{k+1/2}_i\rangle
    \leq
    \frac{1}{2\gamma K}(\| x^0 - x\|^2 + \| z^0 - z\|^2 + \| y^0 - y\|^2).
\end{align*}
Following the definition of $\text{gap}_1$, we only need to take the maximum in the variable $y_i \in \cY$ and the minimum in $x \in \cX$ and $z_i \in \cZ$.
\begin{align*}
    \text{gap}_1(\bar x^{K}, \bar z^K, \bar y^K)
    &=
    \max_{\tilde y \in \mathcal{\tilde Y}}\tilde L(\bar x^K,\bar z^K,\tilde y) - \min_{\tilde x,z \in \cX, \mathcal{\tilde Z}} \tilde L(\tilde x, \tilde z, \bar y^K)
    \notag\\
    &=\max_{y \in \mathcal{\tilde Y}} \left[\ell \left( \sum_{i=1}^n \bar z^K_i, b\right) + \sum_{i=1}^n r_i (\bar x^{K}_i) +
    \sum\limits_{i=1}^n\langle  A_i x^{k+1/2}_i - z^{k+1/2}_i,  y_i\rangle \right] 
    \\
    &\hspace{0.4cm}- \min_{x,z \in \cX, \mathcal{\tilde Z}} \left[ \ell \left( \sum_{i=1}^n z_i, b \right) +  \sum_{i=1}^n r_i (x_i) + \sum\limits_{i=1}^n\langle  A_i x_i - z_i,  y^{k+1/2}_i\rangle\right]
    \notag\\
    &=\max_{y \in \mathcal{\tilde Y}} \max_{x,z \in \cX, \mathcal{\tilde Z}} \Bigg[\ell \left( \sum_{i=1}^n \bar z^K_i, b\right) - \ell \left( \sum_{i=1}^n z_i, b \right) + \sum_{i=1}^n \left( r_i (\bar x^{K}_i) - r_i (x_i)\right)
    \\
    &\hspace{0.4cm}+
   \sum\limits_{i=1}^n\langle  A_i x^{k+1/2}_i - z^{k+1/2}_i,  y_i\rangle
    - \sum\limits_{i=1}^n\langle  A_i x_i - z_i,  y^{k+1/2}_i\rangle \Bigg]
    \notag\\
    &\leq
    \frac{1}{2\gamma K}\left(\max_{x \in \cX} \| x^0 - x\|^2 + \max_{z \in \mathcal{\tilde Z}} \| z^0 - z\|^2 + \max_{y \in \mathcal{\tilde Y}} \| y^0 - y\|^2\right).
\end{align*}
\end{proof}

\subsection{Proof of Theorem \ref{th:EG_prox}} \label{app:th:EG_prox}

{
\begin{theorem}[Theorem \ref{th:EG_prox}]
Let $\ell$ and $r$ be proximal-friendly and convex functions. Let the problem (\ref{eq:vfl_lin_spp_1})
be solved by Algorithm \ref{alg:prox_EG}. Then for 
$$
\gamma = \frac{1}{\sqrt{2}} \cdot \min \left\{ 1; \frac{1}{\sqrt{\lambda_{\max}(A^T A)}} \right\},
$$
it holds that
$$
\text{gap}(\bar x^K, \bar z^K,\bar y^K) = \mathcal{O} \left( \frac{(1 + \sqrt{\lambda_{\max}(A^T A)} ) D^2}{K}   \right),
$$
where $\bar x^K := \tfrac{1}{K}\sum_{k=0}^{K-1} x^{k+1/2}$, $\bar z^K := \tfrac{1}{K}\sum_{k=0}^{K-1} z^{k+1/2}$, $\bar y^K := \tfrac{1}{K}\sum_{k=0}^{K-1} y^{k+1/2}$ and \\ $D^2 := \max_{x, z, y \in \cX, \cZ, \cY} \left[ \|x^0 - x \|^2 + \|z^0 - z \|^2 + \|y^0 - y \|^2 \right]$.
\end{theorem}
}

Before we start proving Theorem \ref{th:EG_prox}, we need a small lemma concerning the proximal operator.

\begin{lemma} \label{lem:tech_prox}
Let $h$ be convex and $z^+ = \prox_{\gamma h} (z)$ with some $\gamma > 0$. Then for all $x \in \R^d$ the following inequality holds
\begin{equation*}
    \langle z^+ - z, x - z^+\rangle \geq \gamma \left( h(z^+) - h(x) \right).
\end{equation*}
\end{lemma}

\begin{proof}[Proof of Lemma \ref{lem:tech_prox}]
We use convexity of the function $\gamma h$ and get for any $h'(z^+) \in  \partial h(z^+)$
\begin{equation*}
    \gamma (h(x) - h(z^+)) - \langle  h'(z^+), x - z^+ \rangle \geq 0.
\end{equation*}
With definition of the proximal operator and the optimality condition, one can note that $z - z^+ \in \gamma \partial h(z^+)$. The only thing left to do is to take $\gamma h'(z^+) = z - z^+ $ and finish the proof.
\end{proof}

\begin{proof}[Proof of Theorem \ref{th:EG_prox}]
By Lemma \ref{lem:tech_prox} for convex function $ h = r_i$, $z^+ = x^{k+1}_i$, $z = x^{k}_i - \gamma A_i^T y^{k+1/2}$ (see line \ref{lin_alg:EG_prox_linx2} of Algorithm \ref{alg:prox_EG}) and $x = x_i \in \R^{d_i}$, we get
\begin{align*}
    \langle x^{k+1}_i - x^{k}_i + \gamma A_i^T y^{k+1/2} , x_i - x^{k+1}_i  \rangle
    \geq 
    \gamma \left( r_i (x^{k+1}_i) - r_i (x_i)\right),
\end{align*}
and for $z^+ = x^{k+1/2}_i$, $z =  x^{k}_i - \gamma A_i^T y^k$ (see line \ref{lin_alg:EG_prox_linx2} of Algorithm \ref{alg:prox_EG}), $x = x^{k+1}_i$,
\begin{align*}
    \langle x^{k+1/2}_i - x^{k}_i + \gamma A_i^T y^k, & x^{k+1}_i - x^{k+1/2}_i  \rangle
    \geq 
    \gamma \left( r_i (x^{k+1/2}_i) - r_i (x^{k+1}_i)\right).
\end{align*}
Summing up two previous inequalities, we get
\begin{align*}
    &\langle x^{k+1}_i - x^{k}_i + \gamma A_i^T y^{k+1/2}, x_i - x^{k+1}_i  \rangle
    +\langle x^{k+1/2}_i - x^{k}_i + \gamma A_i^T y^k , x^{k+1}_i - x^{k+1/2}_i  \rangle
    \geq 
    \gamma \left( r_i (x^{k+1/2}_i) - r_i (x_i)\right).
\end{align*}
After small rearrangements and multiplying by 2, we have
\begin{align*}
    &2 \langle x^{k+1}_i - x^{k}_i, x_i - x^{k+1}_i  \rangle
    +2\langle x^{k+1/2}_i - x^{k}_i , x^{k+1}_i - x^{k+1/2}_i  \rangle
    \\
    &+
    2\gamma \langle A_i^T y^{k+1/2}, x_i - x^{k+1/2}_i  \rangle
    +
    2\gamma \langle A_i^T (y^{k+1/2} - y^{k}), x^{k+1/2}_i - x^{k+1}_i  \rangle
    \geq 
    2\gamma \left( r_i (x^{k+1/2}_i) - r_i (x_i)\right).
\end{align*}
For the first line we use identity $2\langle a, b \rangle = \| a  + b\|^2 - \| a\|^2 - \| b\|^2$, and get
\begin{align*}
    &\left( \| x^k_i - x_i\|^2 - \| x^{k+1}_i - x_i \|^2 - \|x^{k+1}_i - x^k_i \|^2 \right)
    \\
    &+ \left( \| x^{k+1}_i - x^{k}_i\|^2 - \| x^{k+1/2}_i - x^{k}_i \|^2 - \| x^{k+1}_i - x^{k+1/2}_i \|^2\right)
    \\
    &+
    2\gamma \langle A_i^T y^{k+1/2}  , x_i - x^{k+1/2}_i  \rangle
    +
    2\gamma \langle A_i^T (y^{k+1/2} - y^{k}), x^{k+1/2}_i - x^{k+1}_i  \rangle
    \geq 
    2\gamma \left( r_i (x^{k+1/2}_i) - r_i (x_i)\right).
\end{align*}
A small rearrangement gives
\begin{align*}
    \| x^{k+1}_i - x_i\|^2 
    \leq&
    \| x^k_i - x_i\|^2 - \| x^{k+1/2}_i - x^{k}_i \|^2 - \| x^{k+1}_i - x^{k+1/2}_i \|^2
    \\
    &+
    2\gamma \langle A_i^T y^{k+1/2} , x_i - x^{k+1/2}_i  \rangle
    \\
    &+
    2\gamma \langle A_i^T (y^{k+1/2} - y^{k}), x^{k+1/2}_i - x^{k+1}_i  \rangle
    - 2 \gamma \left( r_i (x^{k+1/2}_i) - r_i (x_i)\right)
    \\
    =&
    \| x^k_i - x_i\|^2 - \| x^{k+1/2}_i - x^{k}_i \|^2 - \| x^{k+1}_i - x^{k+1/2}_i \|^2
    \\
    &+
    2\gamma \langle A_i (x_i - x^{k+1/2}_i), y^{k+1/2} \rangle
    \\
    &+
    2\gamma \langle A_i (x^{k+1/2}_i - x^{k+1}_i),   y^{k+1/2} - y^{k}\rangle
    - 2 \gamma \left( r_i (x^{k+1/2}_i) - r_i (x_i)\right).
\end{align*}
Summing over all $i$ from $1$ to $n$, we deduce
\begin{align*}
    \sum\limits_{i=1}^n \| x^{k+1}_i - x_i\|^2 
    \leq&
    \sum\limits_{i=1}^n \| x^k_i - x_i\|^2 - \sum\limits_{i=1}^n \| x^{k+1/2}_i - x^{k}_i \|^2 - \sum\limits_{i=1}^n \| x^{k+1}_i - x^{k+1/2}_i \|^2
    \\
    &+
    2\gamma\langle \sum\limits_{i=1}^n  A_i (x_i - x^{k+1/2}_i), y^{k+1/2} \rangle
    \\
    &+
    2 \gamma\langle \sum\limits_{i=1}^n A_i (x^{k+1/2}_i - x^{k+1}_i),   y^{k+1/2} - y^{k}\rangle
    \\
    &
    - 2 \gamma \sum\limits_{i=1}^n \left( r_i (x^{k+1/2}_i) - r_i (x_i)\right).
\end{align*}
With notation of $A = [A_1, \ldots, A_i, \ldots, A_n] $ and notation of $x = [x_1^T, \ldots,x_i^T, \ldots,x_n^T]^T$ from (\ref{eq:main_problem}) and (\ref{eq:main_problem_vfl}), one can obtain that $\sum_{i=1}^n A_i x_i = A x$:
\begin{align*}
    \| x^{k+1} - x\|^2 
    \leq&
    \| x^k - x\|^2 - \| x^{k+1/2} - x^{k}\|^2 - \| x^{k+1} - x^{k+1/2} \|^2
    \\
    &+
    2\gamma \langle A (x - x^{k+1/2}), y^{k+1/2} \rangle
    \\
    &+
    2 \gamma \langle A (x^{k+1/2} - x^{k+1}),   y^{k+1/2} - y^{k}\rangle
    - 2 \gamma \sum\limits_{i=1}^n \left( r_i (x^{k+1/2}_i) - r_i (x_i)\right)
    \\
    =&
    \| x^k - x\|^2 - \| x^{k+1/2} - x^{k}\|^2 - \| x^{k+1} - x^{k+1/2} \|^2
    \\
    &+
    2 \gamma \langle A (x - x^{k+1/2}), y^{k+1/2} \rangle
    \\
    &+
    2 \gamma \langle A^T ( y^{k+1/2} - y^{k}), x^{k+1/2} - x^{k+1}\rangle
    - 2 \gamma \sum\limits_{i=1}^n \left( r_i (x^{k+1/2}_i) - r_i (x_i)\right).
\end{align*}
By Cauchy Schwartz  inequality: $2 \langle a, b \rangle \leq \eta \| a \|^2 + \tfrac{1}{\eta}\| b \|^2$ with $a = A^T ( y^{k+1/2} - y^{k})$, $b = x^{k+1/2} - x^{k+1}$ and $\eta = \gamma$, we get
\begin{align}
    \label{eq:prox_proof_temp1}
    \| x^{k+1} - x\|^2
    \leq&
    \| x^k - x\|^2 - \| x^{k+1/2} - x^{k}\|^2
    \notag\\
    &+
    2\gamma \langle A (x - x^{k+1/2}), y^{k+1/2} \rangle
    \notag\\
    &+
    \gamma \| A^T ( y^{k+1/2} - y^{k}) \|^2
    - 2 \gamma \sum\limits_{i=1}^n \left( r_i (x^{k+1/2}_i) - r_i (x_i)\right).
\end{align}
Using the same steps, one can obtain for $z \in \R^s$,
\begin{align}
    \label{eq:prox_proof_temp2}
    \| z^{k+1} - z \|^2 
    \leq&
    \| z^k - z\|^2 - \| z^{k+1/2} - z^{k} \|^2
    \notag\\
    &-
    2\gamma_z \langle y^{k+1/2} , z - z^{k+1/2}  \rangle
    \notag\\
    &+
    \gamma^2_z \| y^{k+1/2} - y^{k}\|^2  - 2 \gamma \left( \ell (z^{k+1/2}, b) - \ell (z, b)\right),
\end{align}
and for all $y \in \R^s$,
\begin{align}
    \label{eq:prox_proof_temp3}
    \| y^{k+1} -  y \|^2 
    \leq&
    \| y^k - y\|^2 - \| y^{k+1/2} - y^{k} \|^2 
    \notag\\
    &-
    2\gamma \langle \sum_{i=1}^n A_i x^{k+1/2}_i  - z^{k+1/2} , y - y^{k+1/2}  \rangle
    \notag\\
    &+
    \gamma^2 \left\| \sum_{i=1}^n A_i (x^{k+1/2}_i - x^{k}_i) - (z^{k+1/2} - z^{k})\right\|^2
    \notag\\
    =&
    \| y^k - y\|^2 - \| y^{k+1/2} - y^{k} \|^2 
    \notag\\
    &-
    2 \gamma \langle A x^{k+1/2}  - z^{k+1/2} , y - y^{k+1/2}  \rangle
    \notag\\
    &+
    \gamma^2 \| A (x^{k+1/2} - x^{k}) - (z^{k+1/2} - z^{k})\|^2.
\end{align}
Here we also use notation of $A$ and $x$. 
Summing up (\ref{eq:prox_proof_temp1}), (\ref{eq:prox_proof_temp2}) and (\ref{eq:prox_proof_temp3}), we obtain
\begin{align*}
    \| x^{k+1} - x\|^2 + \| z^{k+1} - z \|^2& + \| y^{k+1} -  y \|^2
    \\
    \leq&
    \| x^k - x\|^2 + \| z^k - z\|^2 + \| y^k - y\|^2 
    \\
    &
    - \| x^{k+1/2} - x^{k} \|^2 - \| z^{k+1/2} - z^{k} \|^2 - \| y^{k+1/2} - y^{k} \|^2 
    \notag\\
    &+
    2 \gamma\langle A (x - x^{k+1/2} ), y^{k+1/2} \rangle
    -
    2 \gamma \langle y^{k+1/2} , z - z^{k+1/2}  \rangle
    \notag\\
    &-
    2 \gamma \langle A x^{k+1/2}  - z^{k+1/2} , y - y^{k+1/2}  \rangle
    \notag\\
    &+
    \gamma^2 \| A^T ( y^{k+1/2} - y^{k}) \|^2 +
    \gamma^2 \| y^{k+1/2} - y^{k}\|^2 
    \notag\\
    &+
    \gamma^2 \left\| A (x^{k+1/2} - x^{k}) - (z^{k+1/2} - z^{k})\right\|^2 
    \notag\\
    &
    - 2 \gamma \left( \ell (z^{k+1/2}, b) - \ell (z, b)\right) - 2 \gamma \sum_{i=1}^n \left( r_i (x^{k+1/2}_i) - r_i (x_i)\right).
\end{align*}
Again by Cauchy Schwartz  inequality: $\| a - b \|^2 \leq 2\| a \|^2 + 2\| b \|^2$ with $a = A (x^{k+1/2} - x^{k})$, $b = \gamma (z^{k+1/2} - z^{k})$, we get
\begin{align*}
    \| x^{k+1} - x\|^2 + \| z^{k+1} - z \|^2 + \| y^{k+1} &-  y \|^2
    \\
    \leq&
    \| x^k - x\|^2 + \| z^k - z\|^2 + \| y^k - y\|^2 
    \\
    &
    - \| x^{k+1/2} - x^{k} \|^2 - \| z^{k+1/2} - z^{k} \|^2 - \| y^{k+1/2} - y^{k} \|^2 
    \notag\\
    &+
    2 \gamma \langle A (x - x^{k+1/2}), y^{k+1/2} \rangle
    -
    2 \gamma \langle y^{k+1/2} , z - z^{k+1/2}  \rangle
    \notag\\
    &-
    2 \gamma \langle A x^{k+1/2}  - z^{k+1/2} , y - y^{k+1/2}  \rangle
    \notag\\
    &+
    \gamma^2 \| A^T ( y^{k+1/2} - y^{k}) \|^2 +
    \gamma^2 \| y^{k+1/2} - y^{k}\|^2 
    \notag\\
    &+
    2 \gamma^2 \left\| A (x^{k+1/2} - x^{k})\right\|^2
    + 2\gamma^2\left\| z^{k+1/2} - z^{k}\right\|^2
    \notag\\
    &
    - 2 \gamma \left( \ell (z^{k+1/2}, b) - \ell (z, b)\right) - 2 \gamma \sum_{i=1}^n \left( r_i (x^{k+1/2}_i) - r_i (x_i)\right).
\end{align*}
Using the definition of $\lambda_{\max}(\cdot)$ as a maximum eigenvalue, we get
\begin{align*}
    \| x^{k+1} - x\|^2 + \| z^{k+1} - z \|^2 + \| y^{k+1} &-  y \|^2
    \\
    \leq&
    \| x^k - x\|^2 + \| z^k - z\|^2 + \| y^k - y\|^2 
    \\
    &
    - \| x^{k+1/2} - x^{k} \|^2 - \| z^{k+1/2} - z^{k} \|^2 - \| y^{k+1/2} - y^{k} \|^2 
    \notag\\
    &+
    2 \langle \tilde  A (x - x^{k+1/2}), y^{k+1/2} \rangle
    -
    2 \gamma \langle  y^{k+1/2} , z - z^{k+1/2}  \rangle
    \notag\\
    &-
    2\gamma \langle A x^{k+1/2}  - z^{k+1/2} , y - y^{k+1/2}  \rangle
    \notag\\
    &+ 
    \lambda_{\max}(A A^T)\| y^{k+1/2} - y^{k} \|^2 +
    \gamma^2 \| y^{k+1/2} - y^{k}\|^2 
    \notag\\
    &+
    2 \lambda_{\max}(A^T A) \gamma^2 \left\| x^{k+1/2} - x^{k}\right\|^2
    + 2\gamma^2\left\| z^{k+1/2} - z^{k}\right\|^2
    \notag\\
    &
    - 2 \gamma \left( \ell (z^{k+1/2}, b) - \ell (z, b)\right) - 2 \gamma \sum_{i=1}^n \left( r_i (x^{k+1/2}_i) - r_i (x_i)\right).
\end{align*}
With the choice of $\gamma \leq \tfrac{1}{\sqrt{2}} \cdot \min\left\{ 1; \tfrac{1}{\sqrt{\lambda_{\max}(A^T A)}}\right\}$, we get
\begin{align*}
    \| x^{k+1} - x\|^2 + \| z^{k+1} - z \|^2 + \| y^{k+1} &-  y \|^2
    \\
    \leq&
    \| x^k - x\|^2 + \| z^k - z\|^2 + \| y^k - y\|^2 
    \notag\\
    &+
    2 \gamma \langle A (x - x^{k+1/2}), y^{k+1/2} \rangle
    -
    2 \gamma\langle  y^{k+1/2} , z - z^{k+1/2}  \rangle
    \notag\\
    &-
    2\gamma \langle A x^{k+1/2}  - z^{k+1/2} , y - y^{k+1/2}  \rangle
    \notag\\
    &
    - 2 \gamma \left( \ell (z^{k+1/2}, b) - \ell (z, b)\right) - 2 \gamma \sum_{i=1}^n \left( r_i (x^{k+1/2}_i) - r_i (x_i)\right)
    \\
    =&
    \| x^k - x\|^2 + \| z^k - z\|^2 + \| y^k - y\|^2 
    \notag\\
    &+
    2 \gamma \langle  A x - z, y^{k+1/2} \rangle
    -
    2\gamma \langle A x^{k+1/2} - z^{k+1/2}, y\rangle
    \notag\\
    &
    - 2 \gamma \left( \ell (z^{k+1/2}, b) - \ell (z, b)\right) - 2 \gamma \sum_{i=1}^n \left( r_i (x^{k+1/2}_i) - r_i (x_i)\right).
\end{align*}
After small rearrangements, we obtain
\begin{align*}
    &\left( \ell (z^{k+1/2}, b) - \ell (z, b)\right) + \sum_{i=1}^n \left( r_i (x^{k+1/2}_i) - r_i (x_i)\right) 
    \\
    &+
    \langle A x^{k+1/2} - z^{k+1/2}, y\rangle
    - \langle  A x - z, y^{k+1/2} \rangle
    \\
    &\hspace{8cm}\leq
    \frac{1}{2\gamma}\Big(\| x^k - x\|^2 + \| z^k - z\|^2 + \| y^k - y\|^2 
    \notag\\
    &\hspace{8.4cm} -\| x^{k+1} - x\|^2 - \| z^{k+1} - z \|^2 - \| y^{k+1} -  y \|^2\Big).
\end{align*}
Then we sum all over $k$ from $0$ to $K-1$, divide by $K$, and have
\begin{align*}
    &\frac{1}{K} \sum\limits_{k=0}^{K-1}\left( \ell (z^{k+1/2}, b) - \ell (z, b)\right) + \sum_{i=1}^n \frac{1}{K} \sum\limits_{k=0}^{K-1} \left( r_i (x^{k+1/2}_i) - r_i (x_i)\right) 
    \\
    &+
    \langle A \cdot \frac{1}{K} \sum\limits_{k=0}^{K-1} x^{k+1/2} - \frac{1}{K} \sum\limits_{k=0}^{K-1} z^{k+1/2}, y\rangle
    - \langle  A x - z, \frac{1}{K} \sum\limits_{k=0}^{K-1} y^{k+1/2} \rangle
    \\
    &\hspace{9cm}\leq
    \frac{1}{2\gamma K}\Big(\| x^0 - x\|^2 + \| z^0 - z\|^2 + \| y^0 - y\|^2 
    \notag\\
    &\hspace{9.4cm} -\| x^{K} - x\|^2 - \| z^{K} - z \|^2 - \| y^{K} -  y \|^2\Big)
    \\
    &\hspace{9cm}\leq
    \frac{1}{2\gamma K}(\| x^0 - x\|^2 + \| z^0 - z\|^2 + \| y^0 - y\|^2).
\end{align*}
With Jensen inequality for convex functions $\ell$ and $r_i$, one can note that
\begin{align*}
 \ell \left( \frac{1}{K} \sum\limits_{k=0}^{K-1} z^{k+1/2}, b \right) \leq \frac{1}{K} \sum\limits_{k=0}^{K-1} \ell (z^{k+1/2}, b), \\
 r_i \left( \frac{1}{K} \sum\limits_{k=0}^{K-1} x^{k+1/2}_i \right) \leq \frac{1}{K} \sum\limits_{k=0}^{K-1} r_i (x^{k+1/2}_i).
\end{align*}
Then, with notation $\bar x^K_i = \frac{1}{K} \sum\limits_{k=0}^{K-1} x^{k+1/2}_i$, $\bar z^K = \frac{1}{K} \sum\limits_{k=0}^{K-1} z^{k+1/2}$, $\bar y^K = \frac{1}{K} \sum\limits_{k=0}^{K-1} y^{k+1/2}$, we have
\begin{align*}
    &\ell (\bar z^K, b) - \ell (z, b) + \sum_{i=1}^n \left( r_i (\bar x^{K}_i) - r_i (x_i)\right) 
    +
    \langle A \bar x^{K} - \bar z^{K}, y\rangle
    - \langle  A x - z, \bar y^{K} \rangle
    \\
    &\hspace{8cm}\leq
    \frac{1}{2\gamma K}(\| x^0 - x\|^2 + \| z^0 - z\|^2 + \| y^0 - y\|^2).
\end{align*}
To complete the proof of the theorem, it is sufficient to do the same steps as when obtaining (\ref{eq:basic_crit1}).
Finally, we need to put $\gamma = \tfrac{1}{\sqrt{2}} \cdot \min\left\{ 1; \tfrac{1}{\sqrt{\lambda_{\max}(A^T A)}}\right\}$.
\end{proof}

\subsection{Proof of Theorem \ref{th:EG_add_noise}}
\label{app:th:EG_add_noise}

{
\begin{theorem}[Theorem \ref{th:EG_add_noise}]
Let Assumption \ref{as:convexity_smothness} hold. Let $\mathbb{E}[\xi^k] = 0$, $\mathbb{E}[\xi^k_i] = 0$, $\mathbb{E}[\|\xi^k\|^2] \leq \sigma^2$ and $\mathbb{E}[\|\xi^k_i\|^2] \leq \sigma^2$. Let the problem (\ref{eq:vfl_lin_spp_1})
be solved by Algorithm~\ref{alg:EG_noise}. Then for 
$$
\gamma = \min \left\{ \frac{1}{2}; \frac{1}{\sqrt{8\lambda_{\max}(A^T A)}}; \frac{1}{2L_r}; \frac{1}{2L_{\ell}}; \sqrt{\frac{D^2}{8(\lambda_{\max}(A A^T) + n)\sigma^2 K}}  \right\},
$$
it holds that
\begin{align*}
    \E{\text{gap}(\bar x^K, \bar z^K, \bar y^K)}
    =
    \cO\left(\frac{ (1 + \sqrt{\lambda_{\max}(A^T A)} + L_r + L_{\ell} ) D^2}{ K}
    +
    \sqrt{\frac{(\lambda_{\max}(A A^T) + n)\sigma^2 D^2}{K}} \right),
\end{align*}
where $\bar x^K := \tfrac{1}{K}\sum_{k=0}^{K-1} x^{k+1/2}$, $\bar z^K := \tfrac{1}{K}\sum_{k=0}^{K-1} z^{k+1/2}$, $\bar y^K := \tfrac{1}{K}\sum_{k=0}^{K-1} y^{k+1/2}$ and \\ $D^2 := \max_{x, z, y \in \cX, \cZ, \cY} \left[ \|x^0 - x \|^2 + \|z^0 - z \|^2 + \|y^0 - y \|^2 \right]$.
\end{theorem}
}
\begin{proof}
We start the proof with the following equations on the variables $x^{k+1}_i$, $x^{k+1/2}_i$, $x^k_i$ and any $x_i \in \R^{d_i}$:
\begin{align*}
    \|x^{k+1}_i - x_i \|^2
    &=
    \|x^{k}_i - x_i \|^2 + 2 \langle  x^{k+1}_i -  x^{k}_i,  x^{k+1}_i - x_i\rangle - \|  x^{k+1}_i -  x^{k}_i\|^2,
    \\
    \|x^{k+1/2}_i - x^{k+1}_i \|^2
    &=
    \| x^{k}_i - x^{k+1}_i \|^2 + 2 \langle x^{k+1/2}_i -   x^{k}_i, x^{k+1/2}_i - x^{k+1}_i\rangle - \|  x^{k+1/2}_i - x^k_i\|^2.
\end{align*}
Summing up two previous inequalities and making small rearrangements, we get
\begin{align*}
    \|x^{k+1}_i - x_i \|^2
    =&
    \|x^{k}_i - x_i \|^2 - \|  x^{k+1/2}_i - x^k_i\|^2 - \|x^{k+1/2}_i - x^{k+1}_i \|^2
    \\
    &+
    2 \langle  x^{k+1}_i -  x^{k}_i,  x^{k+1}_i - x_i\rangle + 2 \langle x^{k+1/2}_i -   x^{k}_i, x^{k+1/2}_i - x^{k+1}_i\rangle.
\end{align*}
Using that $x^{k+1}_i -  x^{k}_i = - \gamma (A_i^T (y^{k+1/2} +\xi^{k+1/2}) + \nabla r_i(x^{k+1/2}_i))$ and $x^{k+1/2}_i -   x^{k}_i = - \gamma (A_i^T (y^{k} +\xi^{k}) + \nabla r_i (x^k_i))$, we obtain
\begin{align*}
    \|x^{k+1}_i - x_i \|^2
    =&
    \|x^{k}_i - x_i \|^2 - \|  x^{k+1/2}_i - x^k_i\|^2 - \|x^{k+1/2}_i - x^{k+1}_i \|^2
    \\
    &-
    2\gamma \langle  A_i^T (y^{k+1/2} +\xi^{k+1/2}) + \nabla r_i(x^{k+1/2}_i),  x^{k+1}_i - x_i\rangle 
    \\
    &- 2 \gamma \langle A_i^T (y^{k} +\xi^{k}) + \nabla r_i (x^k_i), x^{k+1/2}_i - x^{k+1}_i\rangle
    \\
    =&
    \|x^{k}_i - x_i \|^2 - \|  x^{k+1/2}_i - x^k_i\|^2 - \|x^{k+1/2}_i - x^{k+1}_i \|^2
    \\
    &-
    2\gamma \langle  A_i^T (y^{k+1/2} +\xi^{k+1/2}) + \nabla r_i(x^{k+1/2}_i),  x^{k+1/2}_i - x_i\rangle
    \\
    &-
    2\gamma \langle  A_i^T (y^{k+1/2} +\xi^{k+1/2}  - y^k - \xi^{k})
    \\
    &-
    2\gamma \langle  \nabla r_i(x^{k+1/2}_i) - \nabla r_i (x^k_i),  x^{k+1}_i - x^{k+1/2}_i\rangle
    \\
    =&
    \|x^{k}_i - x_i \|^2 - \|  x^{k+1/2}_i - x^k_i\|^2 - \|x^{k+1/2}_i - x^{k+1}_i \|^2
    \\
    &-
    2\gamma \langle  A_i (x^{k+1/2}_i - x_i) ,  y^{k+1/2} +\xi^{k+1/2}\rangle
    -
    2\gamma \langle  \nabla r_i(x^{k+1/2}_i),  x^{k+1/2}_i - x_i\rangle
    \\
    &-
    2\gamma \langle  A_i (x^{k+1}_i - x^{k+1/2}_i) ,  y^{k+1/2} +\xi^{k+1/2} - y^k - \xi^{k}\rangle
    \\
    &-
    2\gamma \langle  \nabla r_i(x^{k+1/2}_i) - \nabla r_i (x^k_i),  x^{k+1}_i - x^{k+1/2}_i\rangle.
\end{align*}
Summing over all $i$ from $1$ to $n$, we deduce
\begin{align*}
    \sum\limits_{i=1}^n \|x^{k+1}_i - x_i \|^2
    =&
    \sum\limits_{i=1}^n \|x^{k}_i - x_i \|^2 - \sum\limits_{i=1}^n\|  x^{k+1/2}_i - x^k_i\|^2 - \sum\limits_{i=1}^n \|x^{k+1/2}_i - x^{k+1}_i \|^2
    \\
    &-
    2\gamma \langle  \sum\limits_{i=1}^n A_i (x^{k+1/2}_i - x_i) ,  y^{k+1/2} +\xi^{k+1/2}\rangle
    \\
    &
    -
    2\gamma \sum\limits_{i=1}^n \langle  \nabla r_i(x^{k+1/2}_i),  x^{k+1/2}_i - x_i\rangle
    \\
    &-
    2\gamma \langle  \sum\limits_{i=1}^n A_i (x^{k+1}_i - x^{k+1/2}_i) ,  y^{k+1/2} +\xi^{k+1/2} - y^k - \xi^{k}\rangle
    \\
    &-
    2\gamma \sum\limits_{i=1}^n \langle  \nabla r_i(x^{k+1/2}_i) - \nabla r_i (x^k_i),  x^{k+1}_i - x^{k+1/2}_i\rangle.
\end{align*}
With notation of $A = [A_1, \ldots, A_i, \ldots, A_n] $ and notation of $x = [x_1^T, \ldots,x_i^T, \ldots,x_n^T]^T$, one can obtain that $\sum_{i=1}^n A_i x_i = A x$:
\begin{align*}
    \|x^{k+1} - x \|^2
    =&
    \|x^{k} - x \|^2 - \|  x^{k+1/2} - x^k\|^2 - \|x^{k+1/2} - x^{k+1} \|^2
    \\
    &-
    2\gamma \langle  A (x^{k+1/2} - x) ,  y^{k+1/2} +\xi^{k+1/2}\rangle
    -
    2\gamma \sum\limits_{i=1}^n \langle  \nabla r_i(x^{k+1/2}_i),  x^{k+1/2}_i - x_i\rangle
    \\
    &-
    2\gamma \langle  A (x^{k+1} - x^{k+1/2}) ,  y^{k+1/2} +\xi^{k+1/2} - y^k -\xi^{k}\rangle
    \\
    &-
    2\gamma \sum\limits_{i=1}^n \langle  \nabla r_i(x^{k+1/2}_i) - \nabla r_i (x^k_i),  x^{k+1}_i - x^{k+1/2}_i\rangle
    \\
    =&
    \|x^{k} - x \|^2 - \|  x^{k+1/2} - x^k\|^2 - \|x^{k+1/2} - x^{k+1} \|^2
    \\
    &-
    2\gamma \langle  A (x^{k+1/2} - x) ,  y^{k+1/2} +\xi^{k+1/2}\rangle
    -
    2\gamma \sum\limits_{i=1}^n \langle  \nabla r_i(x^{k+1/2}_i),  x^{k+1/2}_i - x_i\rangle
    \\
    &-
    2\gamma \langle  A^T (y^{k+1/2} +\xi^{k+1/2} - y^k - \xi^{k}) , x^{k+1} - x^{k+1/2} \rangle
    \\
    &-
    2\gamma \sum\limits_{i=1}^n \langle  \nabla r_i(x^{k+1/2}_i) - \nabla r_i (x^k_i),  x^{k+1}_i - x^{k+1/2}_i\rangle.
\end{align*}
By simple fact: $2 \langle a, b \rangle \leq \eta \| a \|^2 + \tfrac{1}{\eta}\| b \|^2$ with $a = A^T ( y^{k+1/2} +\xi^{k+1/2} - y^{k} - \xi^{k})$, $b = x^{k+1/2} - x^{k+1}$, $\eta = 2\gamma$ and $a = \nabla r_i(x^{k+1/2}_i) - \nabla r_i (x^k_i)$, $b = x^{k+1/2}_i - x^{k+1}_i$, $\eta = 2\gamma$, we get
\begin{align}
    \label{eq:noise_proof_temp1}
    \|x^{k+1} - x \|^2
    \leq&
    \|x^{k} - x \|^2 - \|  x^{k+1/2} - x^k\|^2 - \|x^{k+1/2} - x^{k+1} \|^2
    \notag\\
    &-
    2\gamma \langle  A (x^{k+1/2} - x) ,  y^{k+1/2} +\xi^{k+1/2}\rangle
    -
    2\gamma \sum\limits_{i=1}^n \langle  \nabla r_i(x^{k+1/2}_i),  x^{k+1/2}_i - x_i\rangle
    \notag\\
    &+
    2\gamma^2 \|A^T (y^{k+1/2} - y^k) + A^T (\xi^{k+1/2} - \xi^{k})\|^2 + \frac{1}{2}\|x^{k+1} - x^{k+1/2} \|^2
    \notag\\
    &+
    2\gamma^2 \sum\limits_{i=1}^n \| \nabla r_i(x^{k+1/2}_i) - \nabla r_i (x^k_i)\|^2 + \frac{1}{2}  \sum\limits_{i=1}^n \| x^{k+1}_i - x^{k+1/2}_i\|^2
    \notag\\
    =&
    \|x^{k} - x \|^2 - \|  x^{k+1/2} - x^k\|^2 
    \notag\\
    &-
    2\gamma \langle  A (x^{k+1/2} - x) ,  y^{k+1/2} +\xi^{k+1/2}\rangle
    -
    2\gamma \sum\limits_{i=1}^n \langle  \nabla r_i(x^{k+1/2}_i),  x^{k+1/2}_i - x_i\rangle
    \notag\\
    &+
    2\gamma^2 \|A^T (y^{k+1/2} - y^k) + A^T (\xi^{k+1/2} - \xi^{k})\|^2
    \notag\\
    &
    +
    2\gamma^2 \sum\limits_{i=1}^n \| \nabla r_i(x^{k+1/2}_i) - \nabla r_i (x^k_i)\|^2
    \notag\\
    \leq&
    \|x^{k} - x \|^2 - \|  x^{k+1/2} - x^k\|^2 
    \notag\\
    &-
    2\gamma \langle  A (x^{k+1/2} - x) ,  y^{k+1/2} +\xi^{k+1/2}\rangle
    -
    2\gamma \sum\limits_{i=1}^n \langle  \nabla r_i(x^{k+1/2}_i),  x^{k+1/2}_i - x_i\rangle
    \notag\\
    &+
    4\gamma^2 \|A^T (y^{k+1/2} - y^k)\|^2
    +
    4\gamma^2 \|A^T (\xi^{k+1/2} - \xi^{k})\|^2
    \notag\\
    &
    +
    2\gamma^2 \sum\limits_{i=1}^n \| \nabla r_i(x^{k+1/2}_i) - \nabla r_i (x^k_i)\|^2.
\end{align}
Here we also take into account the simple fact (Cauchy Schwartz inequality): $\| a + b \|^2 \leq 2\| a \|^2 + 2\| b \|^2$. Using the same steps, one can obtain for $z \in \R^s$, $(\nabla \ell(z^k, b) - y^k)$
\begin{align}
    \label{eq:noise_proof_temp2}
    \|z^{k+1} - z \|^2
    \leq&
    \|z^{k} - z \|^2 - \|  z^{k+1/2} - z^k\|^2 
    \notag\\
    &+
    2\gamma \langle  y^{k+1/2} ,  z^{k+1/2} - z\rangle
    -
    2\gamma \langle  \nabla \ell(z^{k+1/2}, b),  z^{k+1/2} - z\rangle
    \notag\\
    &+
    2\gamma^2 \|y^{k+1/2} - y^k\|^2 +
    2\gamma^2 \| \nabla \ell(z^{k+1/2}, b) -\nabla \ell(z^k, b)\|^2.
\end{align}
and for all $y \in \R^s$,
\begin{align}
    \label{eq:noise_proof_temp3}
    \|y^{k+1} - y \|^2
    \leq&
    \|y^{k} - y \|^2 - \|  y^{k+1/2} - y^k\|^2 
    \notag\\
    &-
    2\gamma \langle  z^{k+1/2} ,  y^{k+1/2} - y\rangle
    +
    2\gamma \langle  \sum_{i=1}^n A_i x^{k+1/2}_i + \sum_{i=1}^n \xi^{k+1/2}_i,  y^{k+1/2} - y\rangle
    \notag\\
    &+
    2\gamma^2 \|z^{k+1/2} - z^k\|^2 +
    2\gamma^2 \left\| \sum_{i=1}^n A_i (x^{k+1/2}_i - x^k_i) + \sum_{i=1}^n (\xi^{k+1/2}_i - \xi^{k}_i)\right\|^2
    \notag\\
    \leq&
    \|y^{k} - y \|^2 - \|  y^{k+1/2} - y^k\|^2 
    \notag\\
    &-
    2\gamma \langle  z^{k+1/2} ,  y^{k+1/2} - y\rangle
    +
    2\gamma \langle  A x^{k+1/2} + \sum_{i=1}^n \xi^{k+1/2}_i,  y^{k+1/2} - y\rangle
    \notag\\
    &+
    2\gamma^2 \|z^{k+1/2} - z^k\|^2 +
    4\gamma^2 \| A (x^{k+1/2} - x^k) \|^2 
    \notag\\
    &
    + 4\gamma^2 \left\|\sum_{i=1}^n (\xi^{k+1/2}_i - \xi^{k}_i)\right\|^2.
\end{align}
Here we also use notation of $A$, $x$ and Cauchy Schwartz inequality. 
Summing up (\ref{eq:noise_proof_temp1}), (\ref{eq:noise_proof_temp2}) and (\ref{eq:noise_proof_temp3}), we obtain
\begin{align*}
    \|x^{k+1} - x \|^2 &+ \|z^{k+1} - z \|^2 + \|y^{k+1} - y \|^2
    \\
    \leq&
    \|x^{k} - x \|^2 + \|z^{k} - z \|^2 + \|y^{k} - y \|^2
    \notag\\
    &
    - \|  x^{k+1/2} - x^k\|^2 - \|  z^{k+1/2} - z^k\|^2 - \|  y^{k+1/2} - y^k\|^2 
    \notag\\
    &-
    2\gamma \langle  A (x^{k+1/2} - x) ,  y^{k+1/2} +\xi^{k+1/2}\rangle
    +
    2\gamma \langle  y^{k+1/2} ,  z^{k+1/2} - z\rangle
    \notag\\
    &-
    2\gamma \langle  z^{k+1/2} ,  y^{k+1/2} - y\rangle
    +
    2\gamma \langle  A x^{k+1/2} +\sum_{i=1}^n \xi^{k+1/2}_i,  y^{k+1/2} - y\rangle
    \notag\\
    &-
    2\gamma \sum\limits_{i=1}^n \langle  \nabla r_i(x^{k+1/2}_i),  x^{k+1/2}_i - x_i\rangle
    -
    2\gamma \langle  \nabla \ell(z^{k+1/2}, b),  z^{k+1/2} - z\rangle
    \notag\\
    &+
    4\gamma^2 \|A^T (y^{k+1/2} - y^k)\|^2
    +
    4\gamma^2 \|A^T (\xi^{k+1/2} - \xi^{k})\|^2
    \notag\\
    &+
    2\gamma^2 \sum\limits_{i=1}^n \| \nabla r_i(x^{k+1/2}_i) - \nabla r_i (x^k_i)\|^2
    +
    2\gamma^2 \|y^{k+1/2} - y^k\|^2
    \notag\\
    &+
    2\gamma^2 \| \nabla \ell(z^{k+1/2}, b) -\nabla \ell(z^k, b)\|^2
    +
    2\gamma^2 \|z^{k+1/2} - z^k\|^2 
    \notag\\
    &+
    4\gamma^2 \| A (x^{k+1/2} - x^k) \|^2 + 4\gamma^2 \left\|\sum_{i=1}^n (\xi^{k+1/2}_i - \xi^{k}_i)\right\|^2.
\end{align*}
Using convexity and $L_r$-smoothness of the function $r_i$ with convexity and $L_{\ell}$-smoothness of the function $\ell$, we have
\begin{align*}
    \|x^{k+1} - x \|^2 &+ \|z^{k+1} - z \|^2 + \|y^{k+1} - y \|^2
    \\
    \leq&
    \|x^{k} - x \|^2 + \|z^{k} - z \|^2 + \|y^{k} - y \|^2
    \notag\\
    &
    - \|  x^{k+1/2} - x^k\|^2 - \|  z^{k+1/2} - z^k\|^2 - \|  y^{k+1/2} - y^k\|^2 
    \notag\\
    &-
    2\gamma \langle  A (x^{k+1/2} - x) ,  y^{k+1/2} +\xi^{k+1/2}\rangle
    +
    2\gamma \langle  y^{k+1/2} ,  z^{k+1/2} - z\rangle
    \notag\\
    &-
    2\gamma \langle  z^{k+1/2} ,  y^{k+1/2} - y\rangle
    +
    2\gamma \langle  A x^{k+1/2} +\sum_{i=1}^n \xi^{k+1/2}_i,  y^{k+1/2} - y\rangle
    \notag\\
    &-
    2\gamma \sum\limits_{i=1}^n \langle  \nabla r_i(x^{k+1/2}_i),  x^{k+1/2}_i - x_i\rangle
    -
    2\gamma \langle  \nabla \ell(z^{k+1/2}, b),  z^{k+1/2} - z\rangle
    \notag\\
    &+
    4\gamma^2 \|A^T (y^{k+1/2} - y^k)\|^2
    +
    4\gamma^2 \|A^T (\xi^{k+1/2} - \xi^{k})\|^2
    \notag\\
    &+
    2\gamma^2 L_r^2 \| x^{k+1/2} - x^k\|^2
    +
    2\gamma^2 \|y^{k+1/2} - y^k\|^2
    \notag\\
    &+
    2\gamma^2 L_{\ell}^2 \| z^{k+1/2} - z^k \|^2
    +
    2\gamma^2 \|z^{k+1/2} - z^k\|^2 
    \notag\\
    &+
    4\gamma^2 \| A (x^{k+1/2} - x^k) \|^2 + 4\gamma^2 \left\|\sum_{i=1}^n (\xi^{k+1/2}_i - \xi^{k}_i)\right\|^2.
\end{align*}
Using the definition of $\lambda_{\max}(\cdot)$ as a maximum eigenvalue, we get
\begin{align*}
    \|x^{k+1} - x \|^2 + \|z^{k+1} - z \|^2 + \|y^{k+1} & - y \|^2
    \\
    \leq&
    \|x^{k} - x \|^2 + \|z^{k} - z \|^2 + \|y^{k} - y \|^2
    \notag\\
    &
    - \|  x^{k+1/2} - x^k\|^2 - \|  z^{k+1/2} - z^k\|^2 - \|  y^{k+1/2} - y^k\|^2 
    \notag\\
    &-
    2\gamma \langle  A (x^{k+1/2} - x) ,  y^{k+1/2} +\xi^{k+1/2}\rangle
    +
    2\gamma \langle  y^{k+1/2} ,  z^{k+1/2} - z\rangle
    \notag\\
    &-
    2\gamma \langle  z^{k+1/2} ,  y^{k+1/2} - y\rangle
    +
    2\gamma \langle  A x^{k+1/2} +\sum_{i=1}^n \xi^{k+1/2}_i,  y^{k+1/2} - y\rangle
    \notag\\
    &-
    2\gamma \sum\limits_{i=1}^n \langle  \nabla r_i(x^{k+1/2}_i),  x^{k+1/2}_i - x_i\rangle
    -
    2\gamma \langle  \nabla \ell(z^{k+1/2}, b),  z^{k+1/2} - z\rangle
    \notag\\
    &+
    4\gamma^2 \lambda_{\max}(A A^T) \|y^{k+1/2} - y^k\|^2
    +
    4\gamma^2 \lambda_{\max}(A A^T) \|\xi^{k+1/2} - \xi^{k}\|^2
    \notag\\
    &+
    2\gamma^2 L_r^2 \| x^{k+1/2} - x^k\|^2
    +
    2\gamma^2 \|y^{k+1/2} - y^k\|^2
    \notag\\
    &+
    2\gamma^2 L_{\ell}^2 \| z^{k+1/2} - z^k \|^2
    +
    2\gamma^2 \|z^{k+1/2} - z^k\|^2 
    \notag\\
    &+
    4\gamma^2 \lambda_{\max}(A^T A)\| x^{k+1/2} - x^k \|^2 + 4\gamma^2 \left\|\sum_{i=1}^n (\xi^{k+1/2}_i - \xi^{k}_i)\right\|^2.
\end{align*}
With the choice of $\gamma \leq \tfrac{1}{2} \cdot \min\left\{ 1; \tfrac{1}{\sqrt{2\lambda_{\max}(A^T A)}}; \tfrac{1}{L_r}; \tfrac{1}{L_{\ell}}\right\}$, we get
\begin{align*}
    \|x^{k+1} - x \|^2 + \|z^{k+1} - z \|^2 + \|y^{k+1} &- y \|^2
    \\
    \leq&
    \|x^{k} - x \|^2 + \|z^{k} - z \|^2 + \|y^{k} - y \|^2
    \notag\\
    &-
    2\gamma \langle  A (x^{k+1/2} - x) ,  y^{k+1/2} +\xi^{k+1/2}\rangle
    +
    2\gamma \langle  y^{k+1/2} ,  z^{k+1/2} - z\rangle
    \notag\\
    &-
    2\gamma \langle  z^{k+1/2} ,  y^{k+1/2} - y\rangle
    +
    2\gamma \langle  A x^{k+1/2} +\sum_{i=1}^n \xi^{k+1/2}_i,  y^{k+1/2} - y\rangle
    \notag\\
    &-
    2\gamma \sum\limits_{i=1}^n \left(r_i(x^{k+1/2}_i) - r_i(x_i) \right)
    -
    2\gamma \left(  l(z^{k+1/2}, b)  - l(z, b)\right)
    \\
    &+
    4\gamma^2 \lambda_{\max}(A A^T) \|\xi^{k+1/2} - \xi^{k}\|^2
    + 4\gamma^2 \left\|\sum_{i=1}^n (\xi^{k+1/2}_i - \xi^{k}_i)\right\|^2.
\end{align*}
After small rearrangements, we obtain
\begin{align*}
    &\left( \ell (z^{k+1/2}, b) - \ell (z, b)\right) + \sum_{i=1}^n \left( r_i (x^{k+1/2}_i) - r_i (x_i)\right) 
    \\
    &+
    \langle A x^{k+1/2} - z^{k+1/2}, y\rangle
    - \langle  A x - z, y^{k+1/2} \rangle
    \\
    &\hspace{5cm}\leq
    \frac{1}{2\gamma}\Big(\| x^k - x\|^2 + \| z^k - z\|^2 + \| y^k - y\|^2 
    \notag\\
    &\hspace{5.4cm} -\| x^{k+1} - x\|^2 - \| z^{k+1} - z \|^2 - \| y^{k+1} -  y \|^2\Big)
    \\
    &\hspace{5.4cm}
    -
    \langle  A (x^{k+1/2} - x) ,  \xi^{k+1/2}\rangle
    +
    \langle \sum_{i=1}^n \xi^{k+1/2}_i,  y^{k+1/2} - y\rangle
    \notag\\
    &\hspace{5.4cm}
    +
    2\gamma \lambda_{\max}(A A^T) \|\xi^{k+1/2} - \xi^{k}\|^2
    + 2\gamma \left\|\sum_{i=1}^n (\xi^{k+1/2}_i - \xi^{k}_i)\right\|^2
    \\
    &\hspace{5cm}\leq
    \frac{1}{2\gamma}\Big(\| x^k - x\|^2 + \| z^k - z\|^2 + \| y^k - y\|^2 
    \notag\\
    &\hspace{5.4cm} -\| x^{k+1} - x\|^2 - \| z^{k+1} - z \|^2 - \| y^{k+1} -  y \|^2\Big)
    \\
    &\hspace{5.4cm}
    -
    \langle  A (x^{k+1/2} - x^0) ,  \xi^{k+1/2}\rangle
    +
    \langle \sum_{i=1}^n \xi^{k+1/2}_i,  y^{k+1/2} - y^0\rangle
    \\
    &\hspace{5.4cm}
    -
    \langle  A (x^{0} - x) ,  \xi^{k+1/2}\rangle
    +
    \langle \sum_{i=1}^n \xi^{k+1/2}_i,  y^{0} - y\rangle
    \notag\\
    &\hspace{5.4cm}
    +
    2\gamma \lambda_{\max}(A A^T) \|\xi^{k+1/2} - \xi^{k}\|^2
    + 2\gamma \left\|\sum_{i=1}^n (\xi^{k+1/2}_i - \xi^{k}_i)\right\|^2.
\end{align*}
Then we sum all over $k$ from $0$ to $K-1$, divide by $K$, and have
\begin{align*}
    &\frac{1}{K} \sum\limits_{k=0}^{K-1}\left( \ell (z^{k+1/2}, b) - \ell (z, b)\right) + \sum_{i=1}^n \frac{1}{K} \sum\limits_{k=0}^{K-1} \left( r_i (x^{k+1/2}_i) - r_i (x_i)\right) 
    \\
    &+
    \langle A x^{k+1/2} - z^{k+1/2}, y\rangle
    - \langle  A x - z, y^{k+1/2} \rangle
    \\
    &\hspace{5cm}\leq
    \frac{1}{2\gamma K}\Big(\| x^0 - x\|^2 + \| z^0 - z\|^2 + \| y^0 - y\|^2 
    \notag\\
    &\hspace{5.4cm} -\| x^{K} - x\|^2 - \| z^{K} - z \|^2 - \| y^{K} -  y \|^2\Big)
    \\
    &\hspace{5.4cm}
    -
    \frac{1}{K} \sum\limits_{k=0}^{K-1}\langle  A (x^{k+1/2} - x^0) ,  \xi^{k+1/2}\rangle
    +
    \frac{1}{K} \sum\limits_{k=0}^{K-1} \langle \sum_{i=1}^n \xi^{k+1/2}_i,  y^{k+1/2} - y^0\rangle
    \\
    &\hspace{5.4cm}
    -
    \langle  A (x^{0} - x) ,  \frac{1}{K} \sum\limits_{k=0}^{K-1} \xi^{k+1/2}\rangle
    +
    \langle \frac{1}{K} \sum\limits_{k=0}^{K-1} \sum_{i=1}^n \xi^{k+1/2}_i,  y^{0} - y\rangle
    \notag\\
    &\hspace{5.4cm}
    +
    2\gamma \lambda_{\max}(A A^T) \cdot \frac{1}{K} \sum\limits_{k=0}^{K-1} \|\xi^{k+1/2} - \xi^{k}\|^2
    \\
    &\hspace{5.4cm}
    + 2\gamma \cdot \frac{1}{K} \sum\limits_{k=0}^{K-1} \left\|\sum_{i=1}^n (\xi^{k+1/2}_i - \xi^{k}_i)\right\|^2.
\end{align*}
With Jensen inequality for convex functions $\ell$ and $r_i$, one can note that
\begin{align*}
 \ell \left( \frac{1}{K} \sum\limits_{k=0}^{K-1} z^{k+1/2}, b \right) \leq \frac{1}{K} \sum\limits_{k=0}^{K-1} \ell (z^{k+1/2}, b), \\
 r_i \left( \frac{1}{K} \sum\limits_{k=0}^{K-1} x^{k+1/2}_i \right) \leq \frac{1}{K} \sum\limits_{k=0}^{K-1} r_i (x^{k+1/2}_i).
\end{align*}
Then, with notation $\bar x^K_i = \frac{1}{K} \sum\limits_{k=0}^{K-1} x^{k+1/2}_i$, $\bar z^K = \frac{1}{K} \sum\limits_{k=0}^{K-1} z^{k+1/2}$, $\bar y^K = \frac{1}{K} \sum\limits_{k=0}^{K-1} y^{k+1/2}$, we have
\begin{align*}
    &\ell (\bar z^K, b) - \ell (z, b) + \sum_{i=1}^n \left( r_i (\bar x^{K}_i) - r_i (x_i)\right) 
    +
    \langle A \bar x^{K} - \bar z^{K}, y\rangle
    - \langle  A x - z, \bar y^{K} \rangle
    \\
    &\hspace{3cm}\leq
    \frac{1}{2\gamma K}(\| x^0 - x\|^2 + \| z^0 - z\|^2 + \| y^0 - y\|^2)
    \\
    &\hspace{3.4cm}
    -
    \frac{1}{K} \sum\limits_{k=0}^{K-1}\langle  A (x^{k+1/2} - x^0) ,  \xi^{k+1/2}\rangle
    +
    \frac{1}{K} \sum\limits_{k=0}^{K-1} \langle \sum_{i=1}^n \xi^{k+1/2}_i,  y^{k+1/2} - y^0\rangle
    \\
    &\hspace{3.4cm}
    - \frac{1}{K} \sum\limits_{k=0}^{K-1}
    \langle  A (x^{0} - x) ,   \xi^{k+1/2}\rangle
    +\frac{1}{K} \sum\limits_{k=0}^{K-1}
    \langle  \sum_{i=1}^n \xi^{k+1/2}_i,  y^{0} - y\rangle
    \notag\\
    &\hspace{3.4cm}
    +
    2\gamma \lambda_{\max}(A A^T) \cdot \frac{1}{K} \sum\limits_{k=0}^{K-1} \|\xi^{k+1/2} - \xi^{k}\|^2
    \\
    &\hspace{3.4cm}
    + 2\gamma \cdot \frac{1}{K} \sum\limits_{k=0}^{K-1} \left\|\sum_{i=1}^n (\xi^{k+1/2}_i - \xi^{k}_i)\right\|^2 
    \\
    &\hspace{3cm}=
    \frac{1}{2\gamma K}(\| x^0 - x\|^2 + \| z^0 - z\|^2 + \| y^0 - y\|^2)
    \\
    &\hspace{3.4cm}
    -
    \frac{1}{K} \sum\limits_{k=0}^{K-1}\langle  A (x^{k+1/2} - x^0) ,  \xi^{k+1/2}\rangle
    +
    \frac{1}{K} \sum\limits_{k=0}^{K-1} \langle \sum_{i=1}^n \xi^{k+1/2}_i,  y^{k+1/2} - y^0\rangle
    \\
    &\hspace{3.4cm}
    - \frac{1}{K} 
    \langle  x^{0} - x ,  \sum\limits_{k=0}^{K-1} A^T \xi^{k+1/2}\rangle
    +\frac{1}{K} 
    \langle  \sum\limits_{k=0}^{K-1} \sum_{i=1}^n \xi^{k+1/2}_i,  y^{0} - y\rangle
    \notag\\
    &\hspace{3.4cm}
    +
    2\gamma \lambda_{\max}(A A^T) \cdot \frac{1}{K} \sum\limits_{k=0}^{K-1} \|\xi^{k+1/2} - \xi^{k}\|^2
    \\
    &\hspace{3.4cm}
    + 2\gamma \cdot \frac{1}{K} \sum\limits_{k=0}^{K-1} \left\|\sum_{i=1}^n (\xi^{k+1/2}_i - \xi^{k}_i)\right\|^2.
\end{align*}
With Cauchy Schwartz inequality $2 \langle a, b \rangle \leq \eta \| a \|^2 + \tfrac{1}{\eta}\| b \|^2$ with $a =  \sum_{k=0}^{K-1} A^T \xi^{k+1/2}$, $b = x - x^{0}$, $\eta = 2\gamma$ and $a = \sum\limits_{k=0}^{K-1} \sum_{i=1}^n \xi^{k+1/2}_i$, $b = y^{0} - y$, $\eta = 2\gamma$, we get
\begin{align*}
    &\ell (\bar z^K, b) - \ell (z, b) + \sum_{i=1}^n \left( r_i (\bar x^{K}_i) - r_i (x_i)\right) 
    +
    \langle A \bar x^{K} - \bar z^{K}, y\rangle
    - \langle  A x - z, \bar y^{K} \rangle
    \\
    &\hspace{7cm}\leq
    \frac{1}{2\gamma K}(\| x^0 - x\|^2 + \| z^0 - z\|^2 + \| y^0 - y\|^2)
    \\
    &\hspace{7.4cm}
    -
    \frac{1}{K} \sum\limits_{k=0}^{K-1}\langle  A (x^{k+1/2} - x^0) ,  \xi^{k+1/2}\rangle
    +
    \frac{1}{K} \sum\limits_{k=0}^{K-1} \langle \sum_{i=1}^n \xi^{k+1/2}_i,  y^{k+1/2} - y^0\rangle
    \\
    &\hspace{7.4cm}
    + \frac{2 \gamma}{ K} 
    \left\|\sum\limits_{k=0}^{K-1} A^T \xi^{k+1/2} \right\|^2 + \frac{1}{2 \gamma K} \| x^{0} - x \|^2
    \notag\\
    &\hspace{7.4cm}
    +\frac{2 \gamma}{K} 
     \left\|\sum\limits_{k=0}^{K-1} \sum_{i=1}^n \xi^{k+1/2}_i \right\|^2  + \frac{1}{2\gamma K} \| y^{0} - y\|^2
    \notag\\
    &\hspace{7.4cm}
    +
    2\gamma \lambda_{\max}(A A^T) \cdot \frac{1}{K} \sum\limits_{k=0}^{K-1} \|\xi^{k+1/2} - \xi^{k}\|^2
    \\
    &\hspace{7.4cm}
    + 2\gamma \cdot \frac{1}{K} \sum\limits_{k=0}^{K-1} \left\|\sum_{i=1}^n (\xi^{k+1/2}_i - \xi^{k}_i)\right\|^2.
\end{align*}
As in (\ref{eq:basic_crit1}) we pass to the gap criterion by taking the maximum in $y \in \cY$ and the minimum in $x\in \cX$ and $z \in \cZ$. Additionally, we also take the mathematical expectation
\begin{align*}
    \E{\text{gap}(\bar x^K, \bar z^K, \bar y^K)}
    &\leq
    \frac{1}{\gamma K}(\max_{x \in \cX}\| x^0 - x\|^2 + \max_{z \in \cZ}\| z^0 - z\|^2 + \max_{y \in \cY}\| y^0 - y\|^2)
    \\
    &\hspace{.4cm}
    -
    \frac{1}{K} \sum\limits_{k=0}^{K-1} \E{\langle  A (x^{k+1/2} - x^0) ,  \xi^{k+1/2}\rangle}
    +
    \frac{1}{K} \sum\limits_{k=0}^{K-1} \E{\langle \sum_{i=1}^n \xi^{k+1/2}_i,  y^{k+1/2} - y^0\rangle}
    \\
    &\hspace{.4cm}
    + \frac{2 \gamma}{ K} 
   \E{\left\|\sum\limits_{k=0}^{K-1} A^T \xi^{k+1/2} \right\|^2}
    +\frac{2 \gamma}{K} 
     \E{\left\|\sum\limits_{k=0}^{K-1} \sum_{i=1}^n \xi^{k+1/2}_i \right\|^2}
    \notag\\
    &\hspace{.4cm}
    +
    2\gamma \lambda_{\max}(A A^T) \cdot \frac{1}{K} \sum\limits_{k=0}^{K-1} \E{\|\xi^{k+1/2} - \xi^{k}\|^2}
    \\
    &\hspace{.4cm}
    + 2\gamma \cdot \frac{1}{K} \sum\limits_{k=0}^{K-1} \E{\left\|\sum_{i=1}^n (\xi^{k+1/2}_i - \xi^{k}_i)\right\|^2}.
\end{align*}
Using the fact that all $\xi$ and $\xi_i$ are independent and have $0$ in expectation, we have
\begin{align*}
    \E{\text{gap}(\bar x^K, \bar z^K, \bar y^K)}
    \leq&
    \frac{1}{\gamma K}(\max_{x \in \cX}\| x^0 - x\|^2 + \max_{z \in \cZ}\| z^0 - z\|^2 + \max_{y \in \cY}\| y^0 - y\|^2)
    \\
    &
    -
    \frac{1}{K} \sum\limits_{k=0}^{K-1} \E{\langle  A (x^{k+1/2} - x^0) ,  \mathbb{E}_{\xi^{k+1/2}}[\xi^{k+1/2}]\rangle}
    \\
    &
    +
    \frac{1}{K} \sum\limits_{k=0}^{K-1} \E{\langle \sum_{i=1}^n \mathbb{E}_{\xi^{k+1/2}_i}[\xi^{k+1/2}_i],  y^{k+1/2} - y^0\rangle}
    \\
    &
    + \frac{2 \gamma}{ K} \sum\limits_{k=0}^{K-1}
   \E{\|A^T \xi^{k+1/2}\|^2}
    +\frac{2 \gamma}{K} 
     \sum\limits_{k=0}^{K-1} \sum_{i=1}^n \E{\| \xi^{k+1/2}_i\|^2}
    \notag\\
    &
    +
    2\gamma \lambda_{\max}(A A^T) \cdot \frac{1}{K} \sum\limits_{k=0}^{K-1} \E{\|\xi^{k+1/2} \|^2} + \E{ \|\xi^{k}\|^2}
    \\
    &
    + 2\gamma \cdot \frac{1}{K} \sum\limits_{k=0}^{K-1} \sum_{i=1}^n \left[\E{\|\xi^{k+1/2}_i \|^2} + \E{ \|\xi^{k}_i\|^2} \right]
    \\
    \leq&
    \frac{1}{\gamma K}(\max_{x \in \cX}\| x^0 - x\|^2 + \max_{z \in \cZ}\| z^0 - z\|^2 + \max_{y \in \cY}\| y^0 - y\|^2)
    \\
    &
    +
    4\gamma \lambda_{\max}(A A^T) \cdot \frac{1}{K} \sum\limits_{k=0}^{K-1} \E{\|\xi^{k+1/2} \|^2} + \E{ \|\xi^{k}\|^2}
    \\
    &
    + 4\gamma \cdot \frac{1}{K} \sum\limits_{k=0}^{K-1} \sum_{i=1}^n \left[\E{\|\xi^{k+1/2}_i \|^2} + \E{ \|\xi^{k}_i\|^2}\right].
\end{align*}
With the bounded variance assumption of $\xi$ and $\xi_i$, we have
\begin{align*}
    \E{\text{gap}(\bar x^K, \bar z^K, \bar y^K)}
    \leq&
    \frac{1}{\gamma K} \left(\max_{x \in \cX}\| x^0 - x\|^2 + \max_{z \in \cZ}\| z^0 - z\|^2 + \max_{y \in \cY}\| y^0 - y\|^2 \right)
    \\
    &
    +
    8\gamma (\lambda_{\max}(A A^T) + n)\sigma^2.
\end{align*}
It remains to take 
$$\gamma = \min\left\{ \tfrac{1}{2}; \tfrac{1}{\sqrt{8\lambda_{\max}(A^T A)}}; \tfrac{1}{2L_r}; \tfrac{1}{2L_{\ell}}; \sqrt{\tfrac{\max_{x \in \cX}\| x^0 - x\|^2 + \max_{z \in \cZ}\| z^0 - z\|^2 + \max_{y \in \cY}\| y^0 - y\|^2}{8(\lambda_{\max}(A A^T) + n)\sigma^2 K}} \right\}$$
and get
\begin{align*}
    &\E{\text{gap}(\bar x^K, \bar z^K, \bar y^K)}
    \\
    &=
    \cO\Bigg(\frac{1 + L_r + L_{\ell} + \sqrt{\lambda_{\max}(A^T A)}}{ K} \left(\max_{x \in \cX}\| x^0 - x\|^2 + \max_{z \in \cZ}\| z^0 - z\|^2 + \max_{y \in \cY}\| y^0 - y\|^2 \right)
    \\
    &
    +
    \sqrt{\frac{8(\lambda_{\max}(A A^T) + n)\sigma^2\left(\max_{x \in \cX}\| x^0 - x\|^2 + \max_{z \in \cZ}\| z^0 - z\|^2 + \max_{y \in \cY}\| y^0 - y\|^2\right)}{K}} \Bigg).
\end{align*}
\end{proof}

\subsection{Proof of Theorem \ref{th:EG_aug}}
\label{app:th:EG_aug}

{
\begin{theorem}[Theorem \ref{th:EG_aug}]
Let Assumption \ref{as:convexity_smothness} hold. Let problem (\ref{eq:vfl_lin_spp_aug}) be solved by Algorithm \ref{alg:EG_aug}. Then for 
$$
\gamma = \frac{1}{4} \cdot \min\left\{ 1; \frac{1}{\rho}; \frac{1}{\sqrt{\lambda_{\max}(A^T A)}}; \frac{1}{\sqrt{\rho\lambda_{\max}(A^T A)}}; \frac{1}{\rho\lambda_{\max}(A^T A)}; \frac{1}{L_r}; \frac{1}{L_{\ell}}\right\},
$$
it holds that
\begin{align*}
    \text{gap}_{\text{aug}}(\bar x^{K}, \bar z^K, \bar y^K)
    =
    \mathcal{O} \left(\frac{ \left( 1 + \rho + \sqrt{(1 + \rho)\lambda_{\max}(A^T A)} + \rho \lambda_{\max}(A^T A) + L_{\ell} + L_r\right) D^2}{K} \right),
\end{align*}
where $\text{gap}_{\text{aug}}(x,z,y) := \max_{\tilde y \in \cY} L_{\text{aug}}(x, z, \tilde y) - \min_{\tilde x, \tilde z \in \cX, \cZ} L_{\text{aug}}(\tilde x, \tilde z, y)$ and \\
$\bar x^K := \tfrac{1}{K}\sum_{k=0}^{K-1} x^{k+1/2}$, $\bar z^K := \tfrac{1}{K}\sum_{k=0}^{K-1} z^{k+1/2}$, $\bar y^K := \tfrac{1}{K}\sum_{k=0}^{K-1} y^{k+1/2}$ and $D^2 := \max_{x, z, y \in \cX, \cZ, \cY} [ \|x^0 - x \|^2 +  \|z^0 - z \|^2 + \|y^0 - y \|^2 ]$.
\end{theorem}
}
To prove the convergence it is sufficient to show that the problem is convex--concave (Lemma \ref{lem:tech2}), to estimate the Lipschitz constant of gradients and use the general results from \cite{nemirovski2004prox}. But for completeness, we give the proof of our special case here. 
\begin{proof}
We start the proof with the following equations on the variables $x^{k+1}_i$, $x^{k+1/2}_i$, $x^k_i$ and any $x_i \in \R^{d_i}$:
\begin{align*}
    \|x^{k+1}_i - x_i \|^2
    &=
    \|x^{k}_i - x_i \|^2 + 2 \langle  x^{k+1}_i -  x^{k}_i,  x^{k+1}_i - x_i\rangle - \|  x^{k+1}_i -  x^{k}_i\|^2,
    \\
    \|x^{k+1/2}_i - x^{k+1}_i \|^2
    &=
    \| x^{k}_i - x^{k+1}_i \|^2 + 2 \langle x^{k+1/2}_i -   x^{k}_i, x^{k+1/2}_i - x^{k+1}_i\rangle - \|  x^{k+1/2}_i - x^k_i\|^2.
\end{align*}
Summing up two previous equations and making small rearrangements, we get
\begin{align*}
    \|x^{k+1}_i - x_i \|^2
    =&
    \|x^{k}_i - x_i \|^2 - \|  x^{k+1/2}_i - x^k_i\|^2 - \|x^{k+1/2}_i - x^{k+1}_i \|^2
    \\
    &+
    2 \langle  x^{k+1}_i -  x^{k}_i,  x^{k+1}_i - x_i\rangle + 2 \langle x^{k+1/2}_i -   x^{k}_i, x^{k+1/2}_i - x^{k+1}_i\rangle.
\end{align*}
Using that $x^{k+1}_i -  x^{k}_i = - \gamma (A_i^T y^{k+1/2} + \nabla r_i(x^{k+1/2}_i) + \rho A_i^T (\sum_{i=1}^n A_i x^{k+1/2}_i - z^{k+1/2}))$ and $x^{k+1/2}_i -   x^{k}_i = - \gamma (A_i^T y^k + \nabla r_i (x^k_i) + \rho A_i^T (\sum_{i=1}^n A_i x^{k}_i - z^k) )$ (see lines \ref{lin_alg:EG_aug_linx1} and \ref{lin_alg:EG_aug_linx2} of Algorithm \ref{alg:EG_aug}), we obtain
\begin{align}
    \label{eq:aug_proof_temp0}
    \|x^{k+1}_i - x_i \|^2
    =&
    \|x^{k}_i - x_i \|^2 - \|  x^{k+1/2}_i - x^k_i\|^2 - \|x^{k+1/2}_i - x^{k+1}_i \|^2
    \notag\\
    &-
    2\gamma \langle  A_i^T y^{k+1/2} + \nabla r_i(x^{k+1/2}_i) + \rho A_i^T (\sum_{i=1}^n A_i x^{k+1/2}_i - z^{k+1/2}),  x^{k+1}_i - x_i\rangle 
    \notag\\
    &- 2 \gamma \langle A_i^T y^k + \nabla r_i (x^k_i) + \rho A_i^T (\sum_{i=1}^n A_i x^{k}_i - z^k), x^{k+1/2}_i - x^{k+1}_i\rangle
    \notag\\
    =&
    \|x^{k}_i - x_i \|^2 - \|  x^{k+1/2}_i - x^k_i\|^2 - \|x^{k+1/2}_i - x^{k+1}_i \|^2
    \notag\\
    &-
    2\gamma \langle  A_i^T y^{k+1/2} + \nabla r_i(x^{k+1/2}_i) + \rho A_i^T (\sum_{i=1}^n A_i x^{k+1/2}_i - z^{k+1/2}),  x^{k+1/2}_i - x_i\rangle
    \notag\\
    &-
    2\gamma \langle  A_i^T (y^{k+1/2} - y^k) + \nabla r_i(x^{k+1/2}_i) - \nabla r_i (x^k_i),  x^{k+1}_i - x^{k+1/2}_i\rangle
    \notag\\
    &-
    2\gamma \rho \langle   A_i^T (\sum_{i=1}^n A_i (x^{k+1/2}_i - x^k_i) + z^k - z^{k+1/2}),  x^{k+1}_i - x^{k+1/2}_i\rangle
    \notag\\
    =&
    \|x^{k}_i - x_i \|^2 - \|  x^{k+1/2}_i - x^k_i\|^2 - \|x^{k+1/2}_i - x^{k+1}_i \|^2
    \notag\\
    &-
    2\gamma \langle  A_i (x^{k+1/2}_i - x_i) ,  y^{k+1/2}\rangle
    -
    2\gamma \langle  \nabla r_i(x^{k+1/2}_i),  x^{k+1/2}_i - x_i\rangle
    \notag\\
    &-
    2 \rho \gamma \langle   \sum_{i=1}^n A_i x^{k+1/2}_i - z^{k+1/2} ,  A_i (x^{k+1/2}_i - x_i)\rangle
    \notag\\
    &-
    2\gamma \langle  A_i (x^{k+1}_i - x^{k+1/2}_i) ,  y^{k+1/2} - y^k\rangle
    \notag\\
    &-
    2\gamma \langle  \nabla r_i(x^{k+1/2}_i) - \nabla r_i (x^k_i),  x^{k+1}_i - x^{k+1/2}_i\rangle
    \notag\\
    &-
    2\gamma \rho \langle  A_i (x^{k+1}_i - x^{k+1/2}_i) ,  \sum_{i=1}^n A_i (x^{k+1/2}_i - x^k_i) \rangle
    \notag\\
    &-
    2\gamma \rho \langle  A_i (x^{k+1}_i - x^{k+1/2}_i) ,  z^k - z^{k+1/2} \rangle.
\end{align}
Summing over all $i$ from $1$ to $n$, we deduce
\begin{align*}
    \sum\limits_{i=1}^n \|x^{k+1}_i - x_i \|^2
    =&
    \sum\limits_{i=1}^n \|x^{k}_i - x_i \|^2 - \sum\limits_{i=1}^n\|  x^{k+1/2}_i - x^k_i\|^2 - \sum\limits_{i=1}^n \|x^{k+1/2}_i - x^{k+1}_i \|^2
    \\
    &-
    2\gamma \langle  \sum\limits_{i=1}^n A_i (x^{k+1/2}_i - x_i) ,  y^{k+1/2}\rangle
    -
    2\gamma \sum\limits_{i=1}^n \langle  \nabla r_i(x^{k+1/2}_i),  x^{k+1/2}_i - x_i\rangle
    \notag\\
    &-
    2 \rho \gamma \sum\limits_{i=1}^n \langle   \sum_{i=1}^n A_i x^{k+1/2}_i - z^{k+1/2} ,  A_i (x^{k+1/2}_i - x_i)\rangle
    \\
    &-
    2\gamma \sum\limits_{i=1}^n \langle  \nabla r_i(x^{k+1/2}_i) - \nabla r_i (x^k_i),  x^{k+1}_i - x^{k+1/2}_i\rangle
    \\
    &-
    2\gamma \langle  \sum\limits_{i=1}^n A_i (x^{k+1}_i - x^{k+1/2}_i) ,  y^{k+1/2} - y^k\rangle
    \notag\\
    &-
    2\gamma \rho \sum\limits_{i=1}^n \langle  A_i (x^{k+1}_i - x^{k+1/2}_i) ,  \sum_{i=1}^n A_i (x^{k+1/2}_i - x^k_i) \rangle
    \notag\\
    &-
    2\gamma \rho \sum\limits_{i=1}^n \langle  A_i (x^{k+1}_i - x^{k+1/2}_i) ,  z^k - z^{k+1/2} \rangle.
\end{align*}
With notation of $A = [A_1, \ldots, A_i, \ldots, A_n] $ and notation of $x = [x_1^T, \ldots,x_i^T, \ldots,x_n^T]^T$ from \eqref{eq:main_problem} and \eqref{eq:main_problem_vfl}, one can obtain that $\sum_{i=1}^n A_i x_i = A x$:
\begin{align*}
    \|x^{k+1} - x \|^2
    =&
    \|x^{k} - x \|^2 - \|  x^{k+1/2} - x^k\|^2 - \|x^{k+1/2} - x^{k+1} \|^2
    \\
    &-
    2\gamma \langle  A (x^{k+1/2} - x) ,  y^{k+1/2}\rangle
    -
    2\gamma \sum\limits_{i=1}^n \langle  \nabla r_i(x^{k+1/2}_i),  x^{k+1/2}_i - x_i\rangle
    \notag\\
    &-
    2 \rho \gamma \langle   A x^{k+1/2}  - z^{k+1/2} ,  A (x^{k+1/2} - x)\rangle
    \\
    &-
    2\gamma \langle  A (x^{k+1} - x^{k+1/2}) ,  y^{k+1/2} - y^k\rangle
    \\
    &-
    2\gamma \sum\limits_{i=1}^n \langle  \nabla r_i(x^{k+1/2}_i) - \nabla r_i (x^k_i),  x^{k+1}_i - x^{k+1/2}_i\rangle
    \notag\\
    &-
    2\gamma \rho \langle  A (x^{k+1} - x^{k+1/2}) ,  A (x^{k+1/2} - x^k) \rangle
    \notag\\
    &-
    2\gamma \rho \langle  A (x^{k+1} - x^{k+1/2}) ,  z^k - z^{k+1/2} \rangle
    \\
    =&
    \|x^{k} - x \|^2 - \|  x^{k+1/2} - x^k\|^2 - \|x^{k+1/2} - x^{k+1} \|^2
    \\
    &-
    2\gamma \langle  A (x^{k+1/2} - x) ,  y^{k+1/2}\rangle
    -
    2\gamma \sum\limits_{i=1}^n \langle  \nabla r_i(x^{k+1/2}_i),  x^{k+1/2}_i - x_i\rangle
    \notag\\
    &-
    2 \rho \gamma \langle   A x^{k+1/2}  - z^{k+1/2} ,  A (x^{k+1/2} - x)\rangle
    \\
    &-
    2\gamma \langle  A^T (y^{k+1/2} - y^k) , x^{k+1} - x^{k+1/2} \rangle
    \\
    &-
    2\gamma \sum\limits_{i=1}^n \langle  \nabla r_i(x^{k+1/2}_i) - \nabla r_i (x^k_i),  x^{k+1}_i - x^{k+1/2}_i\rangle
    \notag\\
    &-
    2\gamma \rho \langle  A (x^{k+1} - x^{k+1/2}) ,  A (x^{k+1/2} - x^k) \rangle
    \notag\\
    &-
    2\gamma \rho \langle  A (x^{k+1} - x^{k+1/2}) ,  z^k - z^{k+1/2} \rangle.
\end{align*}
By Cauchy Schwartz inequality, we get
\begin{align}
    \label{eq:aug_proof_temp1}
    \|x^{k+1} - x \|^2
    \leq&
    \|x^{k} - x \|^2 - \|  x^{k+1/2} - x^k\|^2 - \|x^{k+1/2} - x^{k+1} \|^2
    \notag\\
    &-
    2\gamma \langle  A (x^{k+1/2} - x) ,  y^{k+1/2}\rangle
    -
    2\gamma \sum\limits_{i=1}^n \langle  \nabla r_i(x^{k+1/2}_i),  x^{k+1/2}_i - x_i\rangle
    \notag\\
    &-
    2 \rho \gamma \langle   A x^{k+1/2}  - z^{k+1/2} ,  A (x^{k+1/2} - x)\rangle
    \notag\\
    &+
    4\gamma^2 \|A^T (y^{k+1/2} - y^k)\|^2 + \frac{1}{4}\|x^{k+1} - x^{k+1/2} \|^2
    \notag\\
    &+
    4\gamma^2 \sum\limits_{i=1}^n \| \nabla r_i(x^{k+1/2}_i) - \nabla r_i (x^k_i)\|^2 + \frac{1}{4}  \sum\limits_{i=1}^n \| x^{k+1}_i - x^{k+1/2}_i\|^2
    \notag\\
    &+
    4\gamma^2 \rho^2 \|A^T (z^{k+1/2} - z^k)\|^2 + \frac{1}{4}\|x^{k+1} - x^{k+1/2} \|^2
    \notag\\
    &+
    4\gamma^2 \rho^2 \|A^T A(x^{k+1/2} - x^k)\|^2 + \frac{1}{4}\|x^{k+1} - x^{k+1/2} \|^2
    \notag\\
    =&
    \|x^{k} - x \|^2 - \|  x^{k+1/2} - x^k\|^2 
    \notag\\
    &-
    2\gamma \langle  A (x^{k+1/2} - x) ,  y^{k+1/2}\rangle
    -
    2\gamma \sum\limits_{i=1}^n \langle  \nabla r_i(x^{k+1/2}_i),  x^{k+1/2}_i - x_i\rangle
    \notag\\
    &-
    2 \rho \gamma \langle   A x^{k+1/2}  - z^{k+1/2} ,  A (x^{k+1/2} - x)\rangle
    \notag\\
    &+
    4\gamma^2 \|A^T (y^{k+1/2} - y^k)\|^2 +
    4\gamma^2 \sum\limits_{i=1}^n \| \nabla r_i(x^{k+1/2}_i) - \nabla r_i (x^k_i)\|^2
    \notag\\
    &+ 4\gamma^2 \rho^2 \|A^T (z^{k+1/2} - z^k)\|^2 + 4\gamma^2 \rho^2 \|A^T A(x^{k+1/2} - x^k)\|^2.
\end{align}
Using the same steps, one can obtain for $z \in \R^s$, 
\begin{align}
    \label{eq:aug_proof_temp2}
    \|z^{k+1} - z \|^2
    \leq&
    \|z^{k} - z \|^2 - \|  z^{k+1/2} - z^k\|^2 
    \notag\\
    &+
    2\gamma \langle  y^{k+1/2} ,  z^{k+1/2} - z\rangle
    -
    2\gamma \langle  \nabla \ell(z^{k+1/2}, b),  z^{k+1/2} - z\rangle
    \notag\\
    &-
    2\gamma \rho \langle z^{k+1/2} - A x^{k+1/2}, z^{k+1/2} - z\rangle
    \notag\\
    &+
    4\gamma^2 \|y^{k+1/2} - y^k\|^2 +
    4\gamma^2 \| \nabla \ell(z^{k+1/2}, b) -\nabla \ell(z^k, b)\|^2
    \notag\\
    &+
    4\gamma^2 \rho^2 \|z^{k+1/2} - z^k\|^2
    +
    4\gamma^2  \rho^2 \|A (x^{k+1/2} - x^k)\|^2.
\end{align}
and for all $y \in \R^s$,
\begin{align}
    \label{eq:aug_proof_temp3}
    \|y^{k+1} - y \|^2
    \leq&
    \|y^{k} - y \|^2 - \|  y^{k+1/2} - y^k\|^2 
    \notag\\
    &-
    2\gamma \langle  z^{k+1/2} ,  y^{k+1/2} - y\rangle
    +
    2\gamma \langle  A x^{k+1/2},  y^{k+1/2} - y\rangle
    \notag\\
    &+
    2\gamma^2 \|z^{k+1/2} - z^k\|^2 +
    2\gamma^2 \| A (x^{k+1/2} - x^k) \|^2.
\end{align}
Summing up \eqref{eq:aug_proof_temp1}, \eqref{eq:aug_proof_temp2} and \eqref{eq:aug_proof_temp3}, we obtain
\begin{align*}
    \|x^{k+1} - x \|^2 + \|z^{k+1} - z \|^2 + \|y^{k+1} &- y \|^2
    \\
    \leq&
    \|x^{k} - x \|^2 + \|z^{k} - z \|^2 + \|y^{k} - y \|^2
    \notag\\
    &
    - \|  x^{k+1/2} - x^k\|^2 - \|  z^{k+1/2} - z^k\|^2 - \|  y^{k+1/2} - y^k\|^2 
    \notag\\
    &-
    2\gamma \langle  A (x^{k+1/2} - x) ,  y^{k+1/2}\rangle
    +
    2\gamma \langle  y^{k+1/2} ,  z^{k+1/2} - z\rangle
    \notag\\
    &-
    2\gamma \langle  z^{k+1/2} ,  y^{k+1/2} - y\rangle
    +
    2\gamma \langle  A x^{k+1/2},  y^{k+1/2} - y\rangle
    \notag\\
    &-
    2\gamma \sum\limits_{i=1}^n \langle  \nabla r_i(x^{k+1/2}_i),  x^{k+1/2}_i - x_i\rangle
    -
    2\gamma \langle  \nabla \ell(z^{k+1/2}, b),  z^{k+1/2} - z\rangle
    \notag\\
    &-
    2 \rho \gamma \langle   A x^{k+1/2}  - z^{k+1/2} ,  A (x^{k+1/2} - x) - (z^{k+1/2} - z)\rangle
    \notag\\
    &-
    2\gamma \rho \langle z^{k+1/2} - A x^{k+1/2}, z^{k+1/2} - z\rangle
    \notag\\
    &+
    2\gamma^2 \|A^T (y^{k+1/2} - y^k)\|^2 +
    2\gamma^2 \sum\limits_{i=1}^n \| \nabla r_i(x^{k+1/2}_i) - \nabla r_i (x^k_i)\|^2
    \notag\\
    &+
    4\gamma^2 \|y^{k+1/2} - y^k\|^2 +
    4\gamma^2 \| \nabla \ell(z^{k+1/2}, b) -\nabla \ell(z^k, b)\|^2
    \notag\\
    &+
    4\gamma^2 \|z^{k+1/2} - z^k\|^2 +
    4\gamma^2 \| A (x^{k+1/2} - x^k) \|^2
    \notag\\
    &+ 4\gamma^2 \rho^2 \|A^T (z^{k+1/2} - z^k)\|^2 + 4\gamma^2 \rho^2 \|A^T A(x^{k+1/2} - x^k)\|^2
    \notag\\
    &+
    4\gamma^2 \rho^2 \|z^{k+1/2} - z^k\|^2
    +
    4\gamma^2  \rho^2 \|A (x^{k+1/2} - x^k)\|^2.
\end{align*}
Using convexity and $L_r$-smoothness of the function $r_i$ with convexity and $L_{\ell}$-smoothness of the function $\ell$ (Assumption \ref{as:convexity_smothness}), we have
\begin{align*}
    \|x^{k+1} - x \|^2 + \|z^{k+1} - z \|^2 + \|y^{k+1} & - y \|^2
    \\
    \leq&
    \|x^{k} - x \|^2 + \|z^{k} - z \|^2 + \|y^{k} - y \|^2
    \notag\\
    &
    - \|  x^{k+1/2} - x^k\|^2 - \|  z^{k+1/2} - z^k\|^2 - \|  y^{k+1/2} - y^k\|^2 
    \notag\\
    &-
    2\gamma \langle  A (x^{k+1/2} - x) ,  y^{k+1/2}\rangle
    +
    2\gamma \langle  y^{k+1/2} ,  z^{k+1/2} - z\rangle
    \notag\\
    &-
    2\gamma \langle  z^{k+1/2} ,  y^{k+1/2} - y\rangle
    +
    2\gamma \langle  A x^{k+1/2},  y^{k+1/2} - y\rangle
    \notag\\
    &-
    2\gamma \sum\limits_{i=1}^n \left(r_i(x^{k+1/2}_i) - r_i(x_i) \right)
    -
    2\gamma \left(  l(z^{k+1/2}, b)  - l(z, b)\right)
    \notag\\
    &-
    2 \rho \gamma \langle   A x^{k+1/2}  - z^{k+1/2} ,  A (x^{k+1/2} - x) - (z^{k+1/2} - z)\rangle
    \notag\\
    &+
    4\gamma^2 \|A^T (y^{k+1/2} - y^k)\| +
    4\gamma^2 L_r^2 \sum\limits_{i=1}^n  \| x^{k+1/2}_i - x^k_i\|^2
    \notag\\
    &+
    4\gamma^2 \|y^{k+1/2} - y^k\|^2 +
    4\gamma^2 L_{\ell}^2\| z^{k+1/2} - z^k\|^2
    \notag\\
    &+
    4\gamma^2 \|z^{k+1/2} - z^k\|^2 +
    4\gamma^2 \| A (x^{k+1/2} - x^k) \|^2
    \notag\\
    &+ 4\gamma^2 \rho^2 \|A^T (z^{k+1/2} - z^k)\|^2 + 4\gamma^2 \rho^2 \|A^T A(x^{k+1/2} - x^k)\|^2
    \notag\\
    &+
    4\gamma^2 \rho^2 \|z^{k+1/2} - z^k\|^2
    +
    4\gamma^2  \rho^2 \|A (x^{k+1/2} - x^k)\|^2.
\end{align*}
Using the definition of $\lambda_{\max}(\cdot)$ as a maximum eigenvalue, we get
\begin{align*}
    \|x^{k+1} - x \|^2 + \|z^{k+1} - z \|^2 + \|y^{k+1} & - y \|^2
    \\
    \leq&
    \|x^{k} - x \|^2 + \|z^{k} - z \|^2 + \|y^{k} - y \|^2
    \notag\\
    &
    - \|  x^{k+1/2} - x^k\|^2 - \|  z^{k+1/2} - z^k\|^2 - \|  y^{k+1/2} - y^k\|^2 
    \notag\\
    &-
    2\gamma \langle  A (x^{k+1/2} - x) ,  y^{k+1/2}\rangle
    +
    2\gamma \langle  y^{k+1/2} ,  z^{k+1/2} - z\rangle
    \notag\\
    &-
    2\gamma \langle  z^{k+1/2} ,  y^{k+1/2} - y\rangle
    +
    2\gamma \langle  A x^{k+1/2},  y^{k+1/2} - y\rangle
    \notag\\
    &-
    2\gamma \sum\limits_{i=1}^n \left(r_i(x^{k+1/2}_i) - r_i(x_i) \right)
    -
    2\gamma \left(  l(z^{k+1/2}, b)  - l(z, b)\right)
    \notag\\
    &-
    2 \rho \gamma \langle   A x^{k+1/2}  - z^{k+1/2} ,  A (x^{k+1/2} - x) - (z^{k+1/2} - z)\rangle
    \notag\\
    &+
    4\gamma^2 \lambda_{\max} (AA^T) \|y^{k+1/2} - y^k\| +
    4\gamma^2 L_r^2 \| x^{k+1/2} - x^k\|^2
    \notag\\
    &+
    4\gamma^2 \|y^{k+1/2} - y^k\|^2 +
    4\gamma^2 L_{\ell}^2\| z^{k+1/2} - z^k\|^2
    \notag\\
    &+
    4\gamma^2 \|z^{k+1/2} - z^k\|^2 +
    4\gamma^2 \lambda_{\max} (A^T A)\| x^{k+1/2} - x^k \|^2
    \notag\\
    &+ 4\gamma^2 \rho^2 \lambda_{\max} (AA^T) \|z^{k+1/2} - z^k\|^2 + 4\gamma^2 \rho^2 \lambda^2_{\max} (AA^T) \|x^{k+1/2} - x^k\|^2
    \notag\\
    &+
    4\gamma^2 \rho^2 \|z^{k+1/2} - z^k\|^2
    +
    4\gamma^2  \rho^2 \lambda_{\max} (AA^T) \|x^{k+1/2} - x^k\|^2.
\end{align*}
With the choice of $\gamma \leq \tfrac{1}{4} \cdot \min\left\{ 1; \tfrac{1}{\rho}; \tfrac{1}{\sqrt{\lambda_{\max}(A^T A)}}; \tfrac{1}{\sqrt{\rho\lambda_{\max}(A^T A)}}; \tfrac{1}{\rho\lambda_{\max}(A^T A)}; \tfrac{1}{L_r}; \tfrac{1}{L_{\ell}}\right\}$, we get
\begin{align*}
    \|x^{k+1} - x \|^2 + \|z^{k+1} - z \|^2 &+ \|y^{k+1} - y \|^2
    \\
    \leq&
    \|x^{k} - x \|^2 + \|z^{k} - z \|^2 + \|y^{k} - y \|^2
    \notag\\
    &-
    2\gamma \langle  A (x^{k+1/2} - x) ,  y^{k+1/2}\rangle
    +
    2\gamma \langle  y^{k+1/2} ,  z^{k+1/2} - z\rangle
    \notag\\
    &-
    2\gamma \langle  z^{k+1/2} ,  y^{k+1/2} - y\rangle
    +
    2\gamma \langle  A x^{k+1/2},  y^{k+1/2} - y\rangle
    \notag\\
    &-
    2\gamma \sum\limits_{i=1}^n \left(r_i(x^{k+1/2}_i) - r_i(x_i) \right)
    -
    2\gamma \left(  l(z^{k+1/2}, b)  - l(z, b)\right)
    \notag\\
    &-
    2 \rho \gamma \langle   A x^{k+1/2}  - z^{k+1/2} ,  A (x^{k+1/2} - x) - (z^{k+1/2} - z)\rangle
    \\
    =&
    \|x^{k} - x \|^2 + \|z^{k} - z \|^2 + \|y^{k} - y \|^2
    \notag\\
    &+
    2\gamma \langle  A x - z,  y^{k+1/2}\rangle
    -
    2\gamma \langle  A x^{k+1/2} - z^{k+1/2}, y\rangle
    \notag\\
    &-
    2\gamma \sum\limits_{i=1}^n \left(r_i(x^{k+1/2}_i) - r_i(x_i) \right)
    -
    2\gamma \left(  l(z^{k+1/2}, b)  - l(z, b)\right)
    \notag\\
    &-
    \rho \gamma \| A x^{k+1/2}  - z^{k+1/2} \|^2 
    +
    \rho \gamma \| A x - z\|^2 
    - \rho \gamma \| A (x^{k+1/2} - x) - (z^{k+1/2} - z) \|^2.
\end{align*}
After small rearrangements, we obtain
\begin{align*}
    &\left( \ell (z^{k+1/2}, b) - \ell (z, b)\right) + \sum_{i=1}^n \left( r_i (x^{k+1/2}_i) - r_i (x_i)\right) 
    \\
    &+
    \langle A x^{k+1/2} - z^{k+1/2}, y\rangle
    - \langle  A x - z, y^{k+1/2} \rangle
    \\
    &+
    \frac{\rho}{2} \| A x^{k+1/2}  - z^{k+1/2} \|^2 
    -
    \frac{\rho}{2} \| A x - z\|^2
    \\
    &\hspace{5cm}\leq
    \frac{1}{2\gamma}\Big(\| x^k - x\|^2 + \| z^k - z\|^2 + \| y^k - y\|^2 
    \notag\\
    &\hspace{5.4cm} -\| x^{k+1} - x\|^2 - \| z^{k+1} - z \|^2 - \| y^{k+1} -  y \|^2\Big).
\end{align*}
Then we sum all over $k$ from $0$ to $K-1$, divide by $K$, and have
\begin{align*}
    &\frac{1}{K} \sum\limits_{k=0}^{K-1}\left( \ell (z^{k+1/2}, b) - \ell (z, b)\right) + \sum_{i=1}^n \frac{1}{K} \sum\limits_{k=0}^{K-1} \left( r_i (x^{k+1/2}_i) - r_i (x_i)\right) 
    \\
    &+
    \langle A \cdot \frac{1}{K} \sum\limits_{k=0}^{K-1} x^{k+1/2} - \frac{1}{K} \sum\limits_{k=0}^{K-1} z^{k+1/2}, y\rangle
    - \langle  A x - z, \frac{1}{K} \sum\limits_{k=0}^{K-1} y^{k+1/2} \rangle
    \\
    &+
    \frac{\rho}{2} \frac{1}{K} \sum\limits_{k=0}^{K-1}\| A x^{k+1/2}  - z^{k+1/2} \|^2 
    -
    \frac{\rho}{2} \| A x - z\|^2
    \\
    &\hspace{7cm}\leq
    \frac{1}{2\gamma K}\Big(\| x^0 - x\|^2 + \| z^0 - z\|^2 + \| y^0 - y\|^2 
    \notag\\
    &\hspace{7.4cm} -\| x^{K} - x\|^2 - \| z^{K} - z \|^2 - \| y^{K} -  y \|^2\Big)
    \\
    &\hspace{7cm}\leq
    \frac{1}{2\gamma K}(\| x^0 - x\|^2 + \| z^0 - z\|^2 + \| y^0 - y\|^2).
\end{align*}
With Jensen inequality for convex functions $\ell$, $r_i$ and $\| \cdot \|^2$, one can note that
\begin{align*}
 &\ell \left( \frac{1}{K} \sum\limits_{k=0}^{K-1} z^{k+1/2}, b \right) \leq \frac{1}{K} \sum\limits_{k=0}^{K-1} \ell (z^{k+1/2}, b), \\
 &r_i \left( \frac{1}{K} \sum\limits_{k=0}^{K-1} x^{k+1/2}_i \right) \leq \frac{1}{K} \sum\limits_{k=0}^{K-1} r_i (x^{k+1/2}_i), \\
 &\left\| A  \frac{1}{K} \sum\limits_{k=0}^{K-1} x^{k+1/2}  -  \frac{1}{K} \sum\limits_{k=0}^{K-1} z^{k+1/2} \right\|^2 \leq  \frac{1}{K} \sum\limits_{k=0}^{K-1} \| A x^{k+1/2}  - z^{k+1/2} \|^2 .
\end{align*}
Then, with notation $\bar x^K_i = \frac{1}{K} \sum\limits_{k=0}^{K-1} x^{k+1/2}_i$, $\bar z^K = \frac{1}{K} \sum\limits_{k=0}^{K-1} z^{k+1/2}$, $\bar y^K = \frac{1}{K} \sum\limits_{k=0}^{K-1} y^{k+1/2}$, we have
\begin{align*}
    &\ell (\bar z^K, b) - \ell (z, b) + \sum_{i=1}^n \left( r_i (\bar x^{K}_i) - r_i (x_i)\right) 
    +
    \langle A \bar x^{K} - \bar z^{K}, y\rangle
    - \langle  A x - z, \bar y^{K} \rangle
    \\
    &+
    \frac{\rho}{2} \| A \bar x^{K}  - \bar z^{K} \|^2 
    -
    \frac{\rho}{2} \| A x - z\|^2\leq
    \frac{1}{2\gamma K}(\| x^0 - x\|^2 + \| z^0 - z\|^2 + \| y^0 - y\|^2).
\end{align*}
Following the definition $\text{gap}_{\text{aug}}$, we only need to take the maximum in the variable $y \in \cY$ and the minimum in $x \in \cX$ and $z \in \cZ$.
\begin{align}
\label{eq:basic_crit1_aug}
    \text{gap}_{\text{aug}}(\bar x^{K}, \bar z^K, \bar y^K)
    &=
    \max_{y \in \cY}L_{\text{aug}}(\bar x^{K}, \bar z^K, y) - \min_{x,z \in \cX, \cZ} L_{\text{aug}}( x, z, \bar y^K)
    \notag\\
    &=\max_{y \in \cY} \left[\ell (\bar z^K, b) + \sum_{i=1}^n r_i (\bar x^{K}_i) +
    \langle A \bar x^{K} - \bar z^{K}, y\rangle +
    \frac{\rho}{2} \| A \bar x^{K}  - \bar z^{K} \|^2  \right] 
    \notag\\
    &\hspace{0.4cm}
    - \min_{x,z \in \cX, \cZ} \left[ \ell (z, b) +  \sum_{i=1}^n r_i (x_i) + \langle  A x - z, \bar y^{K} \rangle + \frac{\rho}{2} \| A x - z\|^2 \right]
    \notag\\
    &\leq
    \frac{1}{2\gamma K}(\max_{x \in \cX} \| x^0 - x\|^2 + \max_{z \in \cZ} \| z^0 - z\|^2 + \max_{y \in \cY} \| y^0 - y\|^2).
\end{align}
To complete the proof in the cases \eqref{eq:basic_crit1_aug} 
, it remains to put \\$\gamma \leq \tfrac{1}{4} \min\left\{ 1; \tfrac{1}{\rho}; \tfrac{1}{\sqrt{\lambda_{\max}(A^T A)}}; \tfrac{1}{\sqrt{\rho\lambda_{\max}(A^T A)}}; \tfrac{1}{\rho\lambda_{\max}(A^T A)}; \tfrac{1}{L_r}; \tfrac{1}{L_{\ell}}\right\}$.
\end{proof}

\subsection{Proof of Theorem \ref{th:EG_basic_3}}
\label{app:th:EG_basic_3}

{
\begin{theorem}[Theorem \ref{th:EG_basic_3}]
Let $l^*$ be $L_{\ell^*}$-smooth and convex, $r$ be $L_r$-smooth and convex. Let the problem (\ref{eq:vfl_lin_spp_3}) be solved by Algorithm \ref{alg:EG_basic_3}. Then for 
$$
\gamma = \frac{1}{2} \min\left\{ 1; \frac{1}{\sqrt{\lambda_{\max}(A^T A)}}; \frac{1}{L_r}; \frac{1}{L_{\ell^*}}\right\},
$$
it holds that
\begin{align*}
    \text{gap}_2(\bar x^{K}, \bar y^K)
    =
    \mathcal{O} \left(\frac{ \left( 1 + \sqrt{\lambda_{\max}(A^T A)} + L_{\ell^*} + L_r\right) \hat D^2}{K} \right),
\end{align*}
where $\text{gap}_2(x,y) := \max_{\tilde y \in \cY} \hat L(x, \tilde y) - \min_{\tilde x \in \cX} \hat L(\tilde x, y)$ and $\bar x^K := \tfrac{1}{K}\sum_{k=0}^{K-1} x^{k+1/2}$, $\bar y^K := \tfrac{1}{K}\sum_{k=0}^{K-1} y^{k+1/2}$ and $\hat D^2 := \max_{x, y \in \cX, \cY} \left[ \|x^0 - x \|^2 + \|y^0 - y \|^2 \right]$.
\end{theorem}
}
\begin{proof}
We start the proof from (\ref{eq:basic_proof_temp1}), since the updates for $x$ variables in Algorithms \ref{alg:EG}, \ref{alg:EG_basic_3} are the same:
\begin{align}
    \label{eq:basic_3_proof_temp1}
    \|x^{k+1} - x \|^2
    \leq&
    \|x^{k} - x \|^2 - \|  x^{k+1/2} - x^k\|^2 
    \notag\\
    &-
    2\gamma \langle  A (x^{k+1/2} - x) ,  y^{k+1/2}\rangle
    -
    2\gamma \sum\limits_{i=1}^n \langle  \nabla r_i(x^{k+1/2}_i),  x^{k+1/2}_i - x_i\rangle
    \notag\\
    &+
    2\gamma^2 \|A^T (y^{k+1/2} - y^k)\|^2 +
    2\gamma^2 \sum\limits_{i=1}^n \| \nabla r_i(x^{k+1/2}_i) - \nabla r_i (x^k_i)\|^2.
\end{align}
Using the same steps as for (\ref{eq:basic_proof_temp1}), one can obtain for $y \in \R^s$ from Algorithms \ref{alg:EG_basic_3}, 
\begin{align}
    \label{eq:basic_3_proof_temp3}
    \|y^{k+1} - y\|^2
    \leq&
    \|y^{k} - y \|^2 - \|  y^{k+1/2} - y^k \|^2 
    \notag\\
    &-
    2\gamma \langle  \nabla \ell^* (y^{k+1/2}, b),  y^{k+1/2} - y\rangle
    +
    2\gamma \langle  \sum_{i=1}^n A_i x^{k+1/2}_i,  y^{k+1/2} - y\rangle
    \notag\\
    &+
    2\gamma^2 \|\nabla \ell^* (y^{k+1/2}, b) - \nabla \ell^* (y^{k}, b)\|^2 +
    2\gamma^2 \left\| \sum_{i=1}^n A_i (x^{k+1/2}_i - x^k_i) \right\|^2
    \notag\\
    =&
    \|y^{k} - y \|^2 - \|  y^{k+1/2} - y^k \|^2 
    \notag\\
    &-
    2\gamma \langle  \nabla \ell^* (y^{k+1/2}, b),  y^{k+1/2} - y\rangle
    +
    2\gamma \langle  Ax^{k+1/2},  y^{k+1/2} - y\rangle
    \notag\\
    &+
    2\gamma^2 \|\nabla \ell^* (y^{k+1/2}, b) - \nabla \ell^* (y^{k}, b)\|^2 +
    2\gamma^2 \| A (x^{k+1/2} - x^k) \|^2.
\end{align}
Here we also use notation of $A$ and $x$. Summing up (\ref{eq:basic_3_proof_temp1}) and (\ref{eq:basic_3_proof_temp3}), we obtain
\begin{align*}
    \|x^{k+1} - x \|^2 + \|y^{k+1} - y\|^2
    \leq&
    \|x^{k} - x \|^2 + \|y^{k} - y \|^2 - \|  x^{k+1/2} - x^k\|^2 - \|  y^{k+1/2} - y^k \|^2 
    \notag\\
    &-
    2\gamma \langle  A (x^{k+1/2} - x) ,  y^{k+1/2}\rangle
    -
    2\gamma \sum\limits_{i=1}^n \langle  \nabla r_i(x^{k+1/2}_i),  x^{k+1/2}_i - x_i\rangle
    \notag\\
    &-
    2\gamma \langle  \nabla \ell^* (y^{k+1/2}, b),  y^{k+1/2} - y\rangle
    +
    2\gamma \langle  Ax^{k+1/2},  y^{k+1/2} - y\rangle
    \notag\\
    &+
    2\gamma^2 \|A^T (y^{k+1/2} - y^k)\|^2 +
    2\gamma^2 \sum\limits_{i=1}^n \| \nabla r_i(x^{k+1/2}_i) - \nabla r_i (x^k_i)\|^2
    \notag\\
    &+
    2\gamma^2 \|\nabla \ell^* (y^{k+1/2}, b) - \nabla \ell^* (y^{k}, b)\|^2 +
    2\gamma^2 \| A (x^{k+1/2} - x^k) \|^2.
\end{align*}
Using convexity and $L_{r}$-smoothness of the function $r_i$ with convexity and $L_{\ell^*}$-smoothness of the function $\ell$ and with the definition of $\lambda_{\max}(\cdot)$ as a maximum eigenvalue, we have
\begin{align*}
    \|x^{k+1} - x \|^2 + \|y^{k+1} - y\|^2
    \leq&
    \|x^{k} - x \|^2 + \|y^{k} - y \|^2 - \|  x^{k+1/2} - x^k\|^2 - \|  y^{k+1/2} - y^k \|^2 
    \notag\\
    &+
    2\gamma \langle  A x ,  y^{k+1/2}\rangle
    -
    2\gamma \sum\limits_{i=1}^n  [r_i(x^{k+1/2}_i)  - r_i(x_i)]
    \notag\\
    &-
    2\gamma (l^* (y^{k+1/2}, b) - l^* (y, b))
    -
    2\gamma \langle  Ax^{k+1/2},  y \rangle
    \notag\\
    &+
    2\gamma^2 \lambda_{\max} (A A^T) \|y^{k+1/2} - y^k\|^2 +
    2\gamma^2 L_{r}^2 \sum\limits_{i=1}^n \| x^{k+1/2}_i - x^k_i\|^2
    \notag\\
    &+
    2\gamma^2 L_{\ell^*}^2 \|y^{k+1/2} - y^{k}\|^2 +
    2\gamma^2 \lambda_{\max} (A^T A) \| x^{k+1/2} - x^k \|^2.
\end{align*}
With the choice of $\gamma \leq \tfrac{1}{2} \cdot \min\left\{ 1; \tfrac{1}{\sqrt{\lambda_{\max}(A^T A)}}; \tfrac{1}{L_r}; \tfrac{1}{L_{\ell^*}}\right\}$, we get
\begin{align*}
    \|x^{k+1} - x \|^2 + \|y^{k+1} - y\|^2
    \leq&
    \|x^{k} - x \|^2 + \|y^{k} - y \|^2
    \notag\\
    &+
    2\gamma \langle  A x ,  y^{k+1/2}\rangle
    -
    2\gamma \sum\limits_{i=1}^n  [r_i(x^{k+1/2}_i)  - r_i(x_i)]
    \notag\\
    &-
    2\gamma (l^* (y^{k+1/2}, b) - l^* (y, b))
    -
    2\gamma \langle  Ax^{k+1/2},  y \rangle.
\end{align*}
After small rearrangements, we obtain
\begin{align*}
l^* (y^{k+1/2}, b) - l^* (y, b) &+ \sum\limits_{i=1}^n  [r_i(x^{k+1/2}_i)  - r_i(x_i)] + \langle  Ax^{k+1/2},  y \rangle - \langle  A x ,  y^{k+1/2}\rangle
    \\
    \leq&
    \frac{1}{2\gamma}\left(\|x^{k} - x \|^2 + \|y^{k} - y \|^2  - \|x^{k+1} - x \|^2 - \|y^{k+1} - y\|^2\right).
\end{align*}
Then we sum all over $k$ from $0$ to $K-1$, divide by $K$, and have
\begin{align*}
\frac{1}{K}\sum_{k=0}^{K-1} \big[l^* (y^{k+1/2}, b) - l^* (y, b) + \sum\limits_{i=1}^n  [r_i(x^{k+1/2}_i)  - r_i(x_i)] + \langle  Ax^{k+1/2},  & y \rangle - \langle  A x ,  y^{k+1/2}\rangle \big]
    \\
    \leq&
    \frac{1}{2\gamma K}\left(\|x^{0} - x \|^2 + \|y^{0} - y \|^2\right).
\end{align*}
With Jensen inequality for convex functions $\ell$ and $r_i$, one can note that
\begin{align*}
 \ell^* \left( \frac{1}{K} \sum\limits_{k=0}^{K-1} y^{k+1/2}, b \right) \leq \frac{1}{K} \sum\limits_{k=0}^{K-1} \ell^* (y^{k+1/2}, b), \\
 r_i \left( \frac{1}{K} \sum\limits_{k=0}^{K-1} x^{k+1/2}_i \right) \leq \frac{1}{K} \sum\limits_{k=0}^{K-1} r_i (x^{k+1/2}_i).
\end{align*}
Then, with notation $\bar x^K_i = \frac{1}{K} \sum\limits_{k=0}^{K-1} x^{k+1/2}_i$, $\bar y^K = \frac{1}{K} \sum\limits_{k=0}^{K-1} y^{k+1/2}$, we have
\begin{align*}
\ell^* (\bar y^{K}, b) - \ell^* (y, b) + \sum\limits_{i=1}^n  [r_i(\bar x^K_i)  - r_i(x_i)] + \langle  A \bar x^{K},  y \rangle - \langle  A x ,  \bar y^{K}\rangle
    \leq
    \frac{1}{2\gamma K}\left(\|x^{0} - x \|^2 + \|y^{0} - y \|^2\right).
\end{align*}
Following the definition of $\text{gap}_2$, we only need to take the maximum in the variable $y \in \cY$ and the minimum in $x \in \cX$.
\begin{align*}
    \text{gap}_2(\bar x^{K}, \bar y^K)
    &=
    \max_{y \in \cY}\hat L(\bar x^{K}, y) - \min_{x \in \cX} \hat L( x, \bar y^K)
    \notag\\
    &=\max_{y \in \cY} \left[-\ell^*\left(y,b\right) + \sum\limits_{i=1}^n r_i(\bar x^K_i) + y^T \left(\sum_{i=1}^n A_i \bar x^K_i \right) \right] 
    \\
    &\hspace{4.1cm}- \min_{x \in \cX} \left[ -\ell (\bar y^K, b) +  \sum_{i=1}^n r_i (x_i) +  \bar (y^K)^T \left(\sum_{i=1}^n A_i x_i \right) \right]
    \notag\\
    &=\max_{y \in \cY} \max_{x \in \cX} \left[\ell^* (\bar y^{K}, b) - \ell^* (y, b) + \sum\limits_{i=1}^n  [r_i(\bar x^K_i)  - r_i(x_i)] + \langle  A \bar x^{K},  y \rangle - \langle  A x ,  \bar y^{K}\rangle \right]
    \notag\\
    &\leq
    \frac{1}{2\gamma K}\left(\max_{x \in \cX} \| x^0 - x\|^2 + \max_{y \in \cY} \| y^0 - y\|^2 \right).
\end{align*}
To complete the proof, it remains to put $\gamma = \tfrac{1}{2} \cdot \min\left\{ 1; \tfrac{1}{\sqrt{\lambda_{\max}(A^T A)}}; \tfrac{1}{L_r}; \tfrac{1}{L_{\ell^*}}\right\}$.
\end{proof}


\subsection{Three lemmas}
\begin{lemma} \label{lem:tech1}
If $\ell$ and $r_i$ are convex, then $L (x,z,y)$ from (\ref{eq:vfl_lin_spp_1}) is convex-concave.
\end{lemma}
\begin{proof}
We start from checking of convexity.
\begin{equation*}
    \nabla_{(x,z)} L (x,z,y) = \left(\begin{array}{c}
       A_1^T y + \nabla r_1 (x_1)\\
       \ldots \\
       A_i^T y + \nabla r_i (x_i) \\
       \ldots \\
       A_n^T y + \nabla r_n (x_n)\\
       \nabla \ell (z, b) - y
    \end{array}\right).
\end{equation*}
Then, we need to check the condition of Theorem 2.1.3 from \cite{nesterov2003introductory}:
\begin{align*}
    &\langle \nabla_{(x,z)} L (x_1,z_1,y) - \nabla_{(x,z)} L(x_2,z_2,y), (x_1, z_1) - (x_2, z_2)\rangle 
    \\
    &\hspace{2cm}= \langle \left(\begin{array}{c}
       \nabla r_1 (x_{1,1}) - \nabla r_1 (x_{1,2}) \\
       \ldots \\
       \nabla r_i (x_{i,1}) - \nabla r_i (x_{i,2}) \\
       \ldots \\
       \nabla r_n (x_{n,1}) - \nabla r_n (x_{n,2}) \\
       \nabla \ell (z_1, b) - \nabla \ell (z_2, b) 
    \end{array}\right) , \left(\begin{array}{c}
       x_{1,1}- x_{1,2} \\
       \ldots \\
       x_{i,1} - x_{i,2} \\
       \ldots \\
       x_{n,1} - x_{n,2} \\
       z_1 - z_2 
    \end{array}\right)\rangle \geq 0.
\end{align*}
Here we also use that $\ell$ and $r_i$ are convex. It means that the problem (\ref{eq:vfl_lin_spp_1}) is convex on $(x,z)$. Next, we move to check concavity.
\begin{equation*}
    \nabla_{y} L(x,z,y) = \left(
       \sum_{i=1}^n A_i x_i - z \right).
\end{equation*}
Then, again with Theorem 2.1.3 from \cite{nesterov2003introductory}:
\begin{equation*}
    \langle \nabla_{y} L(x,z,y_1) - \nabla_{y} L(x,z,y_2), y_1 - y_2\rangle = 0 \leq 0,
\end{equation*}
we get that the problem (\ref{eq:vfl_lin_spp_1}) is concave on $y$.
\end{proof}

\begin{lemma} \label{lem:tech2}
If $\ell$ and $r_i$ are convex, then $L_{\text{aug}} (x,z,y)$ from (\ref{eq:vfl_lin_spp_aug}) is convex-concave.
\end{lemma}

\begin{proof}
We start from checking of convexity.
\begin{equation*}
    \nabla_{(x,z)} L (x,z,y) = \left(\begin{array}{c}
       A_1^T y + \nabla r_1 (x_1) + \rho A_1^T (Ax - z)\\
       \ldots \\
       A_i^T y + \nabla r_i (x_i) + \rho A_i^T (Ax - z)\\
       \ldots \\
       A_n^T y + \nabla r_n (x_n) + \rho A_n^T (Ax - z)\\
       \nabla \ell (z, b) - y + \rho (z - Ax)
    \end{array}\right).
\end{equation*}
Then, we need to check the condition of Theorem 2.1.3 from \cite{nesterov2003introductory}:
\begin{align*}
    &\langle \nabla_{(x,z)} L (x_1,z_1,y) - \nabla_{(x,z)} L(x_2,z_2,y), (x_1, z_1) - (x_2, z_2)\rangle 
    \\
    &\hspace{1cm}= \langle \left(\begin{array}{c}
       \nabla r_1 (x_{1,1}) - \nabla r_1 (x_{1,2}) + \rho A_1^T [A(x_1 - x_2) - (z_1 - z_2)]\\
       \ldots \\
       \nabla r_i (x_{i,1}) - \nabla r_i (x_{i,2}) + \rho A_i^T [A(x_1 - x_2) - (z_1 - z_2)]\\
       \ldots \\
       \nabla r_n (x_{n,1}) - \nabla r_n (x_{n,2}) + \rho A_n^T [A(x_1 - x_2) - (z_1 - z_2)]\\
       \nabla \ell (z_1, b) - \nabla \ell (z_2, b) + \rho [z_1 - z_2 - A(x_1 - x_2)]
    \end{array}\right) , \left(\begin{array}{c}
       x_{1,1}- x_{1,2} \\
       \ldots \\
       x_{i,1} - x_{i,2} \\
       \ldots \\
       x_{n,1} - x_{n,2} \\
       z_1 - z_2 
    \end{array}\right)\rangle
    \\
    &\hspace{1cm}=
    \langle \left(\begin{array}{c}
       \nabla r_1 (x_{1,1}) - \nabla r_1 (x_{1,2}) \\
       \ldots \\
       \nabla r_i (x_{i,1}) - \nabla r_i (x_{i,2}) \\
       \ldots \\
       \nabla r_n (x_{n,1}) - \nabla r_n (x_{n,2}) \\
       \nabla \ell (z_1, b) - \nabla \ell (z_2, b)
    \end{array}\right) , \left(\begin{array}{c}
       x_{1,1}- x_{1,2} \\
       \ldots \\
       x_{i,1} - x_{i,2} \\
       \ldots \\
       x_{n,1} - x_{n,2} \\
       z_1 - z_2 
    \end{array}\right)\rangle
    \\
    &\hspace{1.4cm} + \rho \left( \| z_1 - z_2 \|^2 - 2 (z_1 - z_2)^T A (x_1 - x_2) + \| A (x_1 - x_2) \|^2\right)
    \\
    &\hspace{1cm}=
    \langle \left(\begin{array}{c}
       \nabla r_1 (x_{1,1}) - \nabla r_1 (x_{1,2}) \\
       \ldots \\
       \nabla r_i (x_{i,1}) - \nabla r_i (x_{i,2}) \\
       \ldots \\
       \nabla r_n (x_{n,1}) - \nabla r_n (x_{n,2}) \\
       \nabla \ell (z_1, b) - \nabla \ell (z_2, b)
    \end{array}\right) , \left(\begin{array}{c}
       x_{1,1}- x_{1,2} \\
       \ldots \\
       x_{i,1} - x_{i,2} \\
       \ldots \\
       x_{n,1} - x_{n,2} \\
       z_1 - z_2 
    \end{array}\right)\rangle
    \\
    &\hspace{1.4cm} + \rho \| z_1 - z_2  - A (x_1 - x_2)\|^2
    \\&\hspace{1cm}\geq 0.
\end{align*}
Here we also use that $\ell$ and $r_i$ are convex. It means that the problem (\ref{eq:vfl_lin_spp_aug}) is convex on $(x,z)$. Next, we move to check concavity.
\begin{equation*}
    \nabla_{y} L(x,z,y) = \left(
       \sum_{i=1}^n A_i x_i - z \right).
\end{equation*}
Then, again with Theorem 2.1.3 from \cite{nesterov2003introductory}:
\begin{equation*}
    \langle \nabla_{y} L(x,z,y_1) - \nabla_{y} L(x,z,y_2), y_1 - y_2\rangle = 0 \leq 0,
\end{equation*}
we get that the problem (\ref{eq:vfl_lin_spp_aug}) is concave on $y$.
\end{proof}

\begin{lemma} \label{lem:matrix}
{
For any matrix $A = [A_1 \ldots A_n]$ it holds that $\| A \| \leq \sqrt{ \sum_{i=1}^n \| A_i \|^2}.$
}
\end{lemma}

{
\begin{proof}
Let us consider $A = [A_1 A_2]$. Then, we have 
\begin{align*}
    \| A \| 
    &= \sup_{\|x\|^2 = 1} \left[ \| A x \| \right] = \sup_{ \|x_1 \|^2 + \|x_2 \|^2 = 1} \left[ \| A_1 x_1 + A_2 x_2\| \right] 
    \\
    &\leq \sup_{ \|x_1 \|^2 + \|x_2 \|^2 = 1} \left[ \| A_1 x_1 \| + \| A_2 x_2 \| \right]
    \\
    &= \sup_{ \alpha \in [0;1]} \left[ \sup_{ \|x_1 \|^2 = \alpha } \| A_1 x_1 \| +\sup_{ \|x_1 \|^2 = 1 - \alpha } \| A_2 x_2 \| \right] 
    \\
    &= 
    \sup_{ \alpha \in [0;1]} \left[ \sqrt{\alpha} \cdot \sup_{ \|x_1 \|^2 = 1 } \| A_1 x_1 \| + \sqrt{1 - \alpha} \cdot \sup_{ \|x_1 \|^2 = 1} \| A_2 x_2 \| \right]
    \\
    &=
    \sup_{ \alpha \in [0;1]} \left[ \sqrt{\alpha} \| A_1 \| + \sqrt{1 - \alpha} \| A_2 \| \right].
\end{align*}
 Optimizing $\alpha \in [0;1]$, we get that $\alpha^* = \frac{\| A_1\|^2}{\| A_1\|^2 + \| A_2\|^2}$ and 
 $$ 
 \| A \| \leq \sqrt{\| A_1 \|^2 + \| A_2 \|^2}.
 $$ 
 This result can be extended to any $n$ by induction. In more details, if $A = [\tilde A_{n-1} A_n]$ with $\tilde A_{n-1} = [A_1 \ldots A_{n-1}]$ and we assume that $ \| \tilde A_{n-1} \| \leq \sqrt{ \sum_{i=1}^{n-1} \| A_i\|^2}$, then we have
 $$ 
 \| A \| \leq \sqrt{\| \tilde A_{n-1} \|^2 + \| A_n \|^2} \leq \sqrt{ \sum_{i=1}^{n-1} \| A_i\|^2 + \| A_n \|^2} = \sqrt{ \sum_{i=1}^{n} \| A_i\|^2}.
 $$
\end{proof}
}

\subsection{On convergence gap} \label{sec:gap}

{
In our theoretical analysis, we use the criterion: $\text{gap}(x, z, y) := \max_{\tilde y \in \cY} L(x, z, \tilde y) -$  \\ $ \min_{\tilde x, \tilde z \in \cX, \cZ} L(\tilde x, \tilde z, y)$, where $L(x, z, y) = \ell (z,b) +r(x) + y^T (Ax - z )$. Since \\ $\min_{\tilde x, \tilde z \in \cX, \cZ} L(\tilde x, \tilde z, y) \leq L(x^*, z^*, y)$, we get 
$$ \max_{\tilde y \in \cY} L(x, z, \tilde y) - L(x^*, z^*, y) \leq \text{gap}(x, z, y). 
$$ 
We note that $Ax^* = z^*$, then 
\begin{align*}
    \max_{\tilde y \in \cY} L(x, z, \tilde y) - L(x^*, z^*, y) 
    =& [ \ell(z, b) + r(x) + \max_{\tilde y \in \cY} \tilde y^T(Ax -z) ] 
    \\
    &- [ \ell (z^*, b) + r(x^*) + (y)^T(Ax^* - z^*) ] 
    \\
    =& [ \ell ( z, b) + r(x) + \max_{\tilde y \in \cY} \tilde y^T(Ax-z) ] - [ \ell (Ax^*, b) + r(x^*)]. 
\end{align*}
When taking maximum for $y \in \cY$ we can define $\cY$ as we need. In particular, we can choose $\cY = [ y \in R^s ~| ~|| y ||_{\infty} \leq C ]$ for some $C > 0$. Then 
$$ 
\max_{\tilde y \in \cY} \tilde y^T(Ax-z) = C \| A x - z^k \|_1 \geq C \| A x - z \|. 
$$ 
Finally, we get 
$$ 
\text{gap}(x, z, y) \geq [ \ell(z, b) + r(x) - \ell (Ax^*, b) + r(x^*)] + C \| A x^k - z^k \| = \text{newgap}(x, z). 
$$
If it holds that $\text{gap}(x, z, y) \leq \varepsilon$, we guarantee that $\text{newgap}(x, z, y) \leq \varepsilon$. The question that arises is whether $\text{newgap}(x, z, y) \leq \varepsilon$ implies that $[ \ell(z, b) + r(x) - \ell (Ax^*, b) + r(x^*) ]$ as well as $\| A x - z\|$ are also “small” in the sense that they are smaller than $\varepsilon$ (up to constants). In general, the answer is no: $[ \ell(z, b) + r(x) - \ell (Ax^*, b) + r(x^*)]$ might be very small (and negative), and $\| A x - z \|_2$ can be very large. But Theorem 3.60 from \cite{beck2017first} shows that if $C$ is large enough such a conclusion can be drawn. In particular, if $\text{newgap}(x^k, z^k, y^k) \leq \varepsilon$ then $C \| A x^k - z^k \|_2 \leq \varepsilon$ and we have $A x^k \to z^k$. 
}

\subsection{On tuning of stepsize} \label{sec:beta}
{
We can rewrite the original problem (\ref{eq:main_problem}) in the following way:
\begin{equation*}
        \min_{x \in \R^d} ~~ \left[\ell\left(Ax,b\right) + r(x)\right] = \left[ \ell\left( \frac{1}{\beta} \cdot \beta Ax,b\right) + r(x)\right] = \left[ \tilde \ell\left(\tilde Ax,b\right) + r(x)\right],
\end{equation*}
where $\tilde \ell \left(y,b\right) = \ell \left( \frac{y}{\beta}, b\right)$ and $\tilde A = \beta A$. Next, we can estimate $L_{\tilde \ell}$ and $\lambda_{\max} (\tilde A^T \tilde A)$:
\begin{align*}
    \| \nabla \tilde \ell (y_1 , b) -  \nabla \tilde \ell (y_2, b)  \| &= \| \nabla_y \ell \left(\frac{y_1}{\beta} , b \right) -  \nabla_y \ell \left(\frac{y_2}{\beta}, b \right)  \| 
    \\
    &= \frac{1}{\beta}\| \nabla \ell \left(\frac{y_1}{\beta} , b \right) -  \nabla \ell \left(\frac{y_2}{\beta}, b \right)  \| 
    \\
    &\leq \frac{L_{\ell}}{\beta^2} \| y_1 - y_2\|,
    \\
    \lambda_{\max} (\tilde A^T \tilde A) &= \lambda_{\max} (\beta^2 A^T A) = \beta^2 \lambda_{\max} (A^T A).
\end{align*}
We get that $L_{\tilde \ell} = \frac{L_{\ell}}{\beta^2}$ and $\lambda_{\max} (\tilde A^T \tilde A) = \beta^2 \lambda_{\max} (A^T A)$. 

Our goal is to equivalize $L_{\tilde \ell}$ and $\sqrt{\lambda_{\max} (\tilde A^T \tilde A)}$ in Theorem \ref{th:EG_basic_1} to make stepsize bigger for free. Then
$$
\frac{L_{\ell}}{\beta^2}= L_{\tilde \ell} = \sqrt{\lambda_{\max} (\tilde A^T \tilde A)} = \beta \sqrt{\lambda_{\max} (A^T A)}
$$
$$
\Rightarrow \quad \beta = \frac{L^{1/3}_{\ell}}{\lambda^{1/6}_{\max} (A^T A)} \quad \Rightarrow \quad L_{\tilde \ell} = L^{1/3}_{\ell} \lambda^{1/3}_{\max} (A^T A).
$$
Hence, the bound on the stepsize in Theorem \ref{th:EG_basic_1} become
$$
\gamma = \frac{1}{2} \cdot \min \left\{ 1; \frac{1}{\sqrt[3]{L_{\ell} \lambda_{\max} (A^T A)}}; \frac{1}{L_r}\right\}.
$$
This, in turn, modifies the convergence result of the theorem as follows:
$$ 
\text{gap}(\bar x^K, \bar z^K,\bar y^K) = \mathcal{O} \left(  \frac{ ( 1 + \sqrt[3]{L_{\ell} \lambda_{\max} (A^T A)} + L_r  ) D^2}{K}  \right).
$$
}

\newpage

\section{Additional Experiments} \label{app:exp}

In the main part (Figure \ref{fig:comparison1} of Section \ref{sec:exp}) we shown that the concept of the saddle point reformulation and Algorithm \ref{alg:EG} for its solution is competitive in the deterministic case. Here we present additional experiments. 

As in the main part, we conduct experiments on the linear regression problem:
$$
\min_{x \in \R^d} f(x) = \tfrac{1}{2}\|Ax - b \|^2 + \lambda \| x \|^2_2.
$$

We take \texttt{mushrooms}, \texttt{a9a}, \texttt{w8a} and \texttt{MNIST} datasets from LibSVM library \cite{chang2011libsvm}. We vertically (by features) uniformly divide the whole dataset between 5 devices. 

First we repeat the same experiments as in the main part, but now for each method we tune the parameters using a grid search. The results are shown in Figure \ref{fig:comparison1_tuned}. If we compare Figure \ref{fig:comparison1} and Figure \ref{fig:comparison1_tuned}, the one method that accelerates the most is Algorithm \ref{alg:EG}. 

\begin{figure*}[h]
\centering
\begin{minipage}{0.49\textwidth}
  \centering
\includegraphics[width =  \textwidth ]{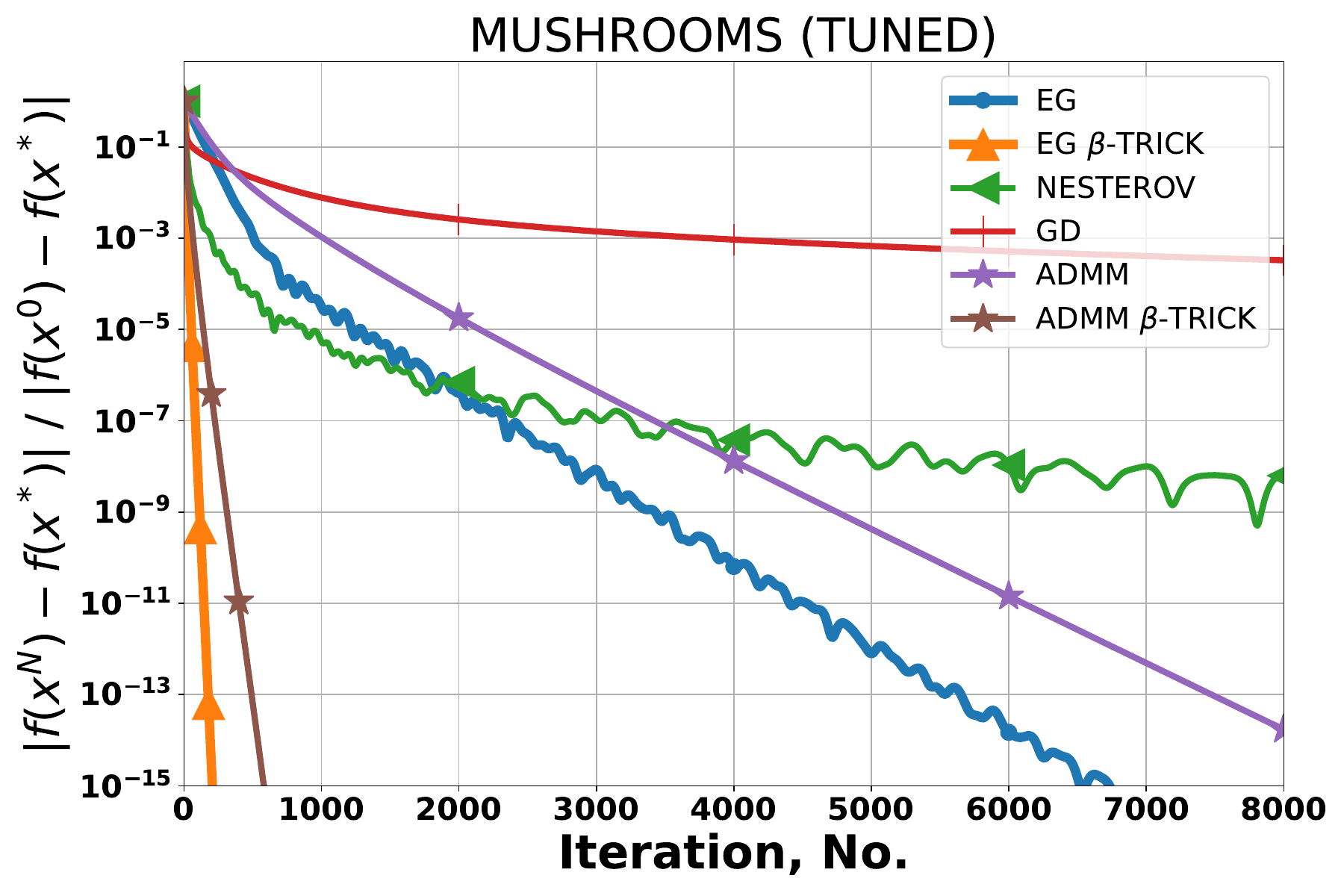}
\end{minipage}%
\begin{minipage}{0.49\textwidth}
  \centering
\includegraphics[width =  \textwidth ]{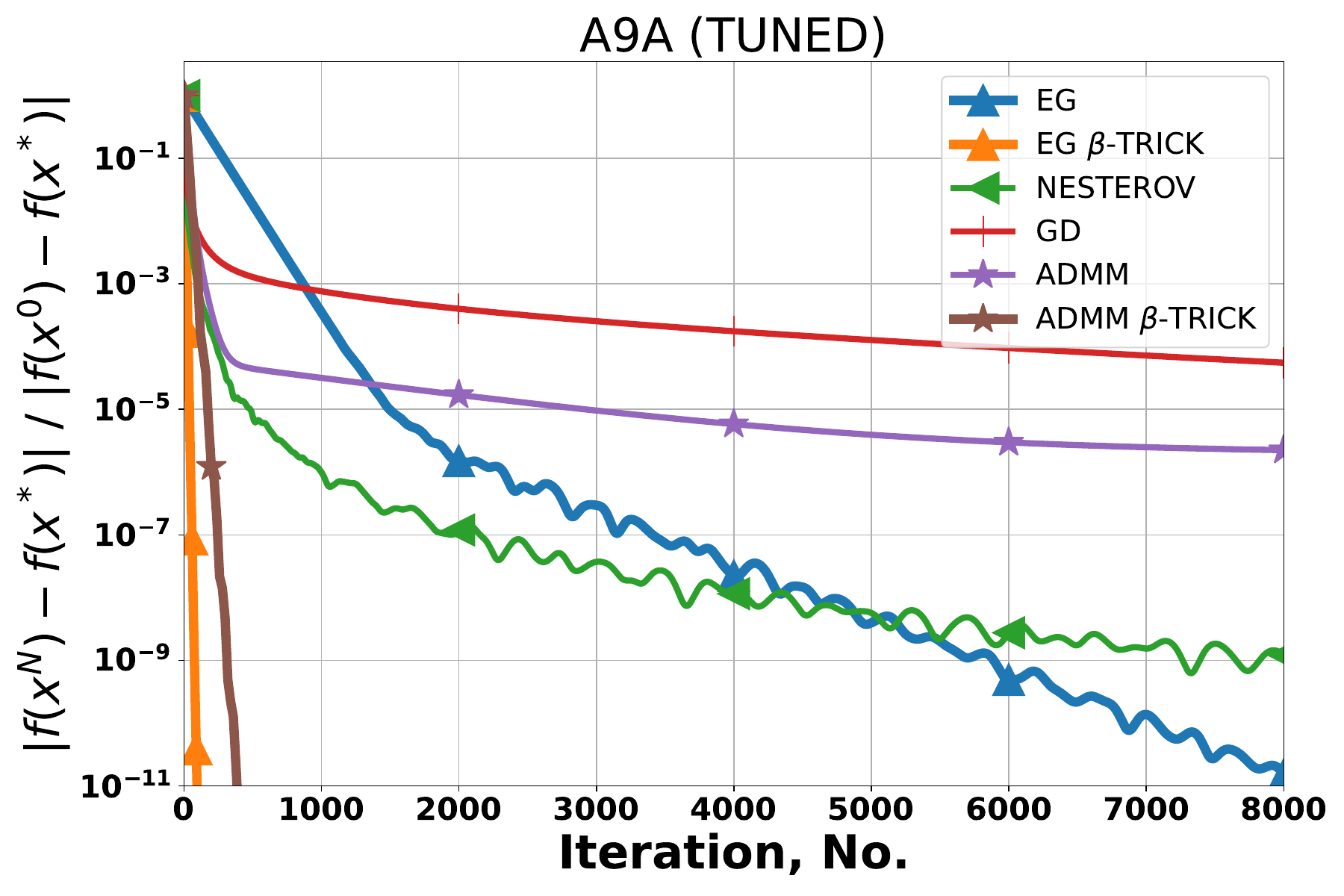}
\end{minipage}%
\\
\begin{minipage}{0.01\textwidth}
\quad
\end{minipage}%
\begin{minipage}{0.49\textwidth}
  \centering
(a) \texttt{mushrooms}
\end{minipage}%
\begin{minipage}{0.49\textwidth}
\centering
 (b) \texttt{a9a}
\end{minipage}%
\\
\begin{minipage}{0.49\textwidth}
  \centering
\includegraphics[width =  \textwidth ]{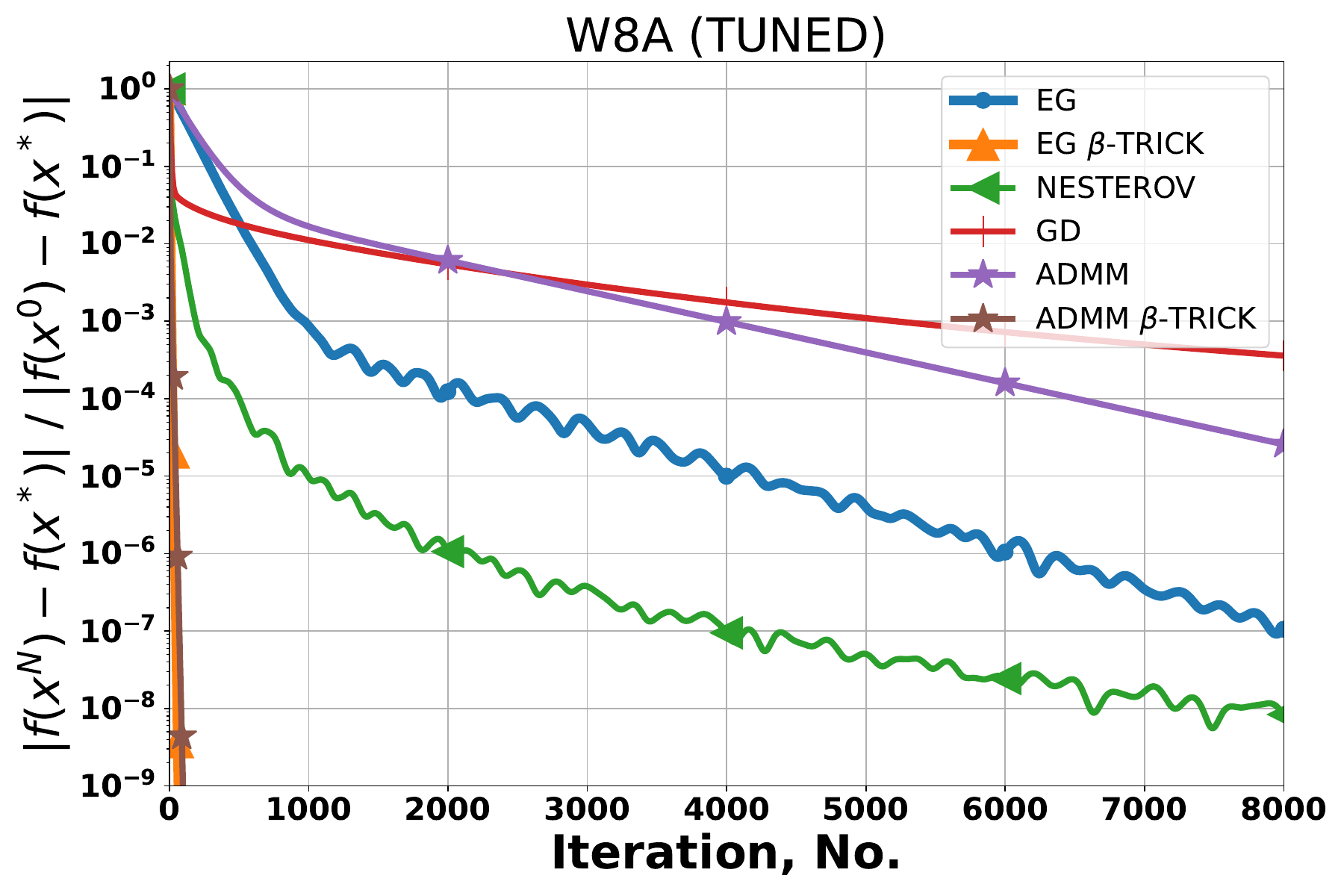}
\end{minipage}%
\begin{minipage}{0.49\textwidth}
  \centering
\includegraphics[width =  \textwidth ]{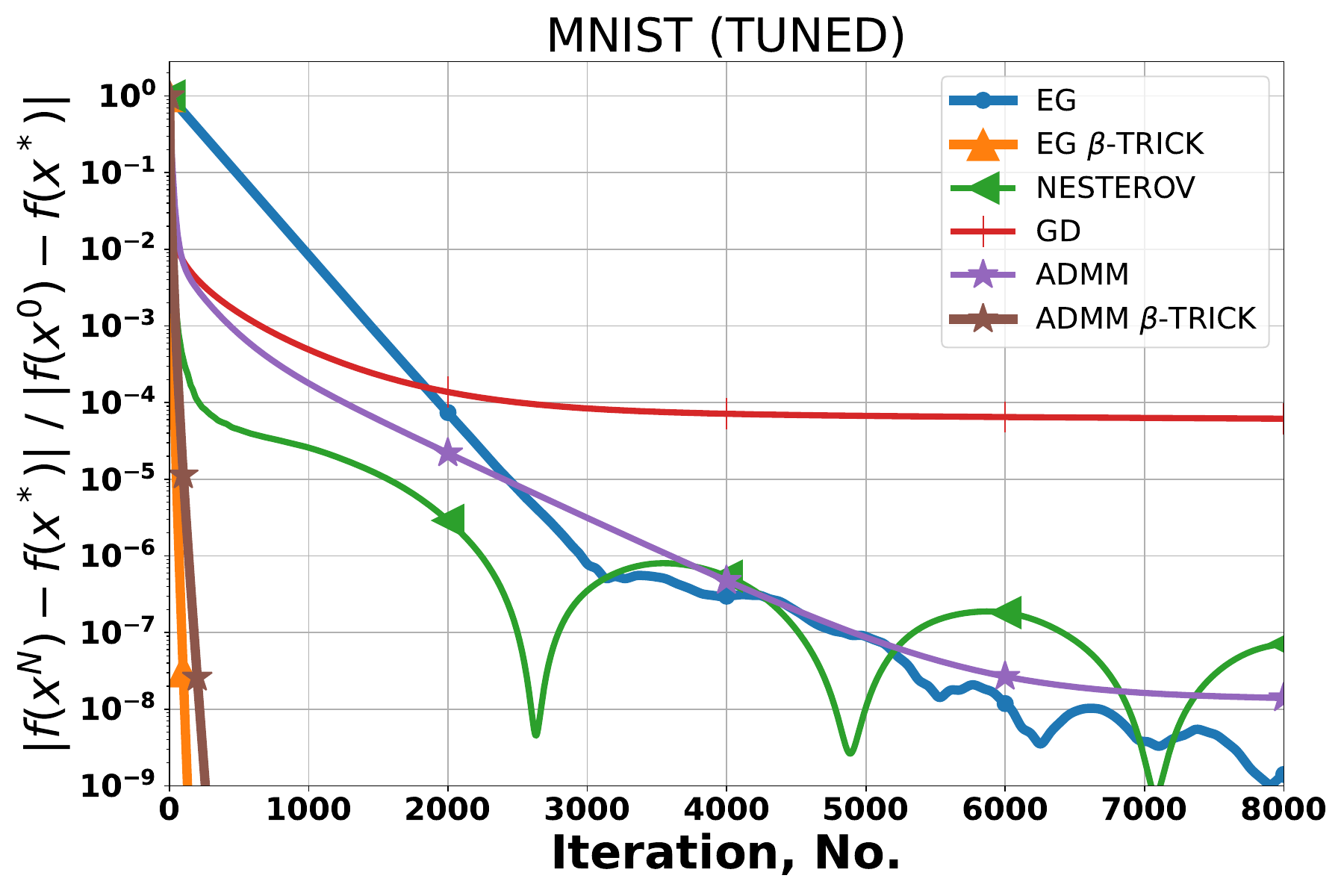}
\end{minipage}%
\\
\begin{minipage}{0.49\textwidth}
\centering
  (c) \texttt{w8a}
\end{minipage}%
\begin{minipage}{0.49\textwidth}
\centering
  (d) \texttt{MNIST}
\end{minipage}%
\caption{
Comparison of tuned methods for solving the VFL problem in different formulations: minimization (\texttt{GD}, \texttt{Nesterov}) and saddle point (\texttt{ADMM}, \texttt{ExtraGradient}/Algorithm \ref{alg:EG}). The comparison is made on LibSVM datasets \texttt{mushrooms}, \texttt{a9a}, \texttt{w8a} and \texttt{MNIST}.}
    \label{fig:comparison1_tuned}
\end{figure*}

{If we continue to talk about step selection, we also suggest considering the option of choosing a step without knowing $\lambda_{\max} (A^T A)$, as suggested by Theorem \ref{th:EG_basic_1}. In particular, Lemma \ref{lem:matrix} gives that the step of the form $\gamma = \tfrac{1}{2} \min \left\{ 1; \tfrac{1}{\sqrt{n \sum_{i=1}^n \lambda_{\max}(A^T_i A_i)}}; \tfrac{1}{L_r}; \tfrac{1}{L_{\ell}} \right\}$ is satisfy the conditions of Theorem \ref{th:EG_basic_1}: $\gamma \leq \tfrac{1}{2} \min \left\{ 1; \tfrac{1}{\sqrt{\lambda_{\max}(A^T A)}}; \tfrac{1}{L_r}; \tfrac{1}{L_{\ell}} \right\}$ (see the proof in Section \ref{app:th:EG_basic_1}). In Figure \ref{fig:comparison1_d4}, we compare Algorithm \ref{alg:EG} with the steps: $\tfrac{1}{2} \min \left\{ 1; \tfrac{1}{\sqrt{\lambda_{\max}(A^T A)}}; \tfrac{1}{L_r}; \tfrac{1}{L_{\ell}} \right\}$ and $\tfrac{1}{2} \min \left\{ 1; \tfrac{1}{\sqrt{n \sum_{i=1}^n \lambda_{\max}(A^T_i A_i)}}; \tfrac{1}{L_r}; \tfrac{1}{L_{\ell}} \right\}$. It can be seen that the second step selection option does not significantly worsen convergence, and sometimes improves it.}

\begin{figure*}[h]
\centering
\begin{minipage}{0.49\textwidth}
  \centering
\includegraphics[width =  \textwidth ]{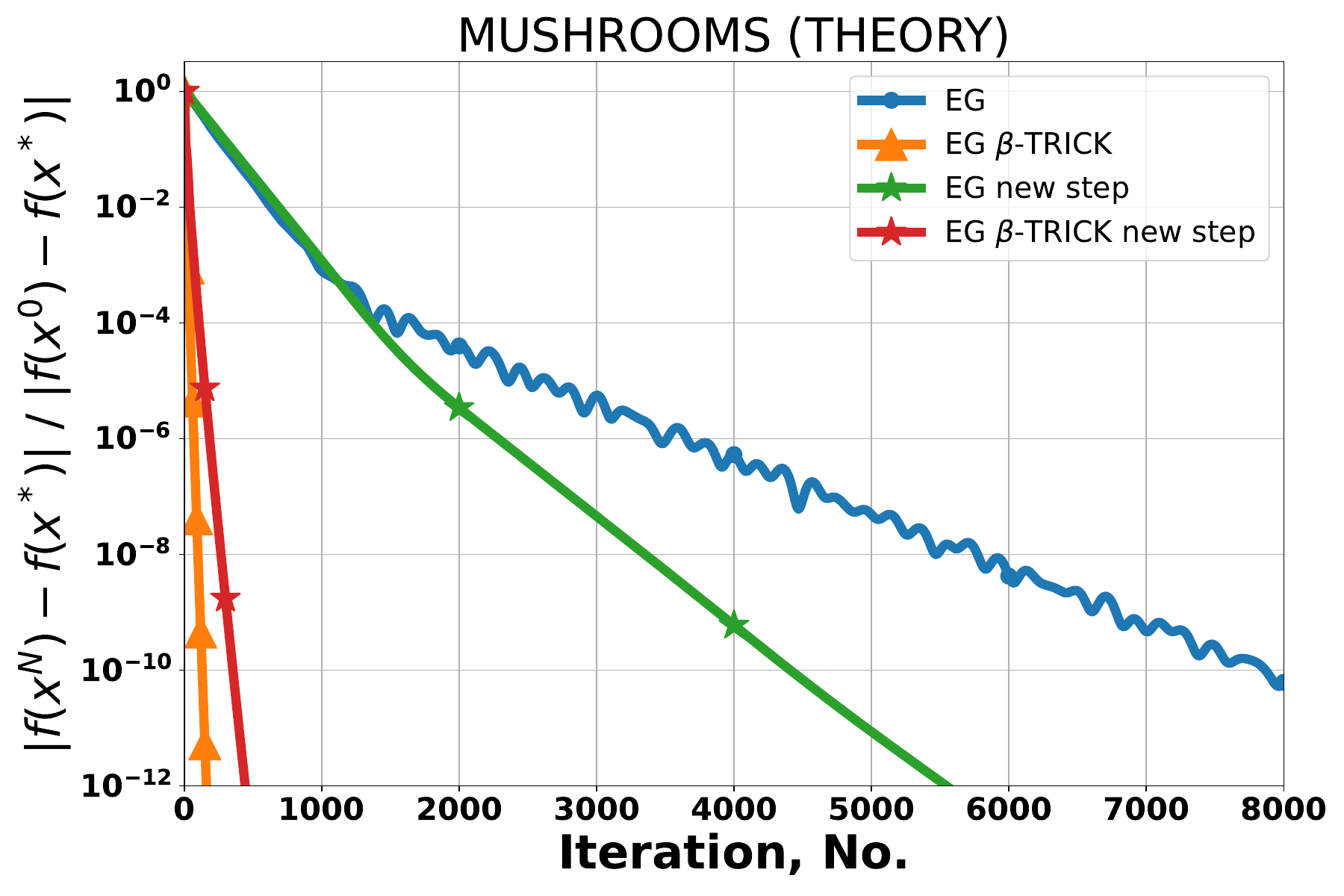}
\end{minipage}%
\begin{minipage}{0.49\textwidth}
  \centering
\includegraphics[width =  \textwidth ]{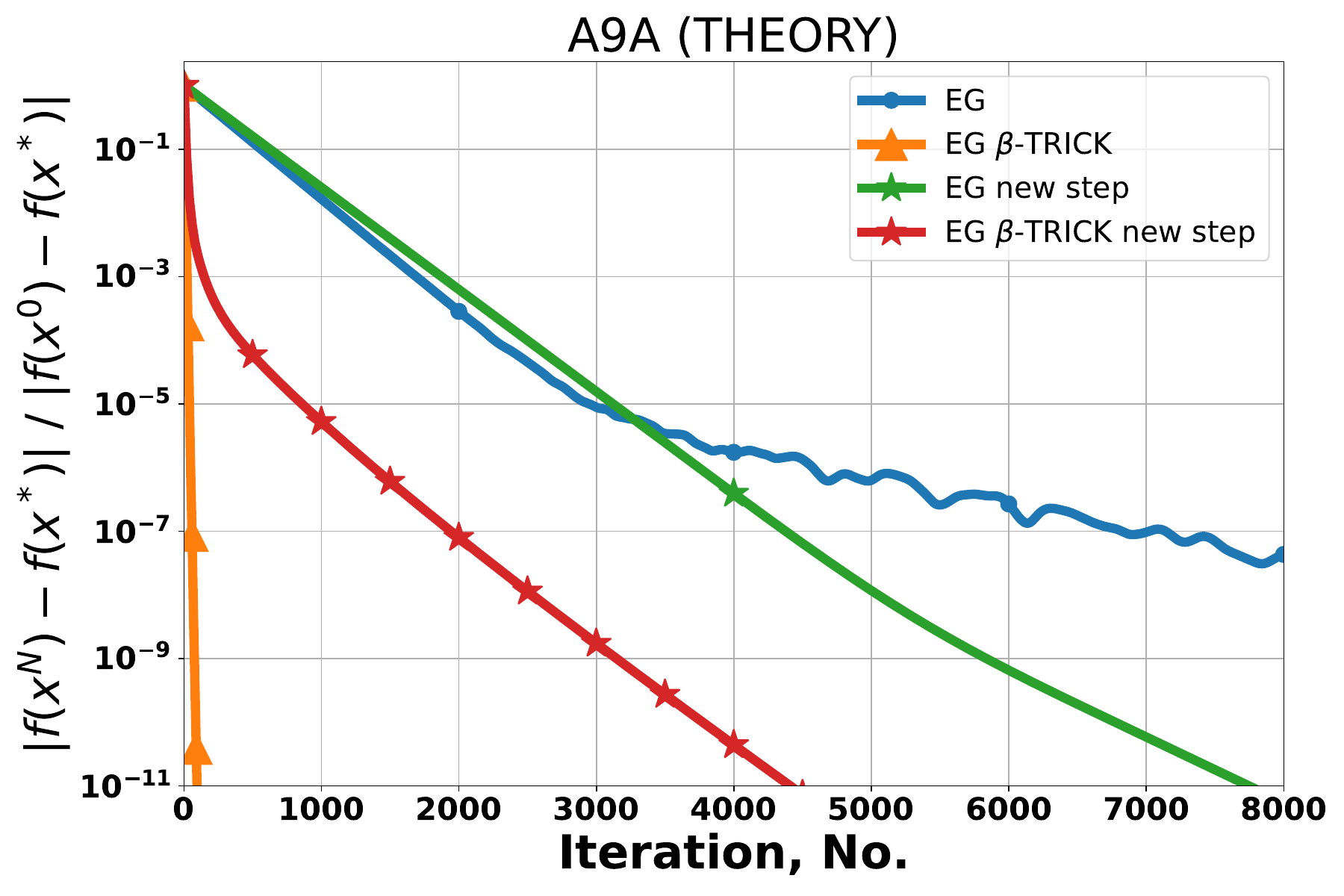}
\end{minipage}%
\\
\begin{minipage}{0.01\textwidth}
\quad
\end{minipage}%
\begin{minipage}{0.49\textwidth}
  \centering
(a) \texttt{mushrooms}
\end{minipage}%
\begin{minipage}{0.49\textwidth}
\centering
 (b) \texttt{a9a}
\end{minipage}%
\\
\begin{minipage}{0.49\textwidth}
  \centering
\includegraphics[width =  \textwidth ]{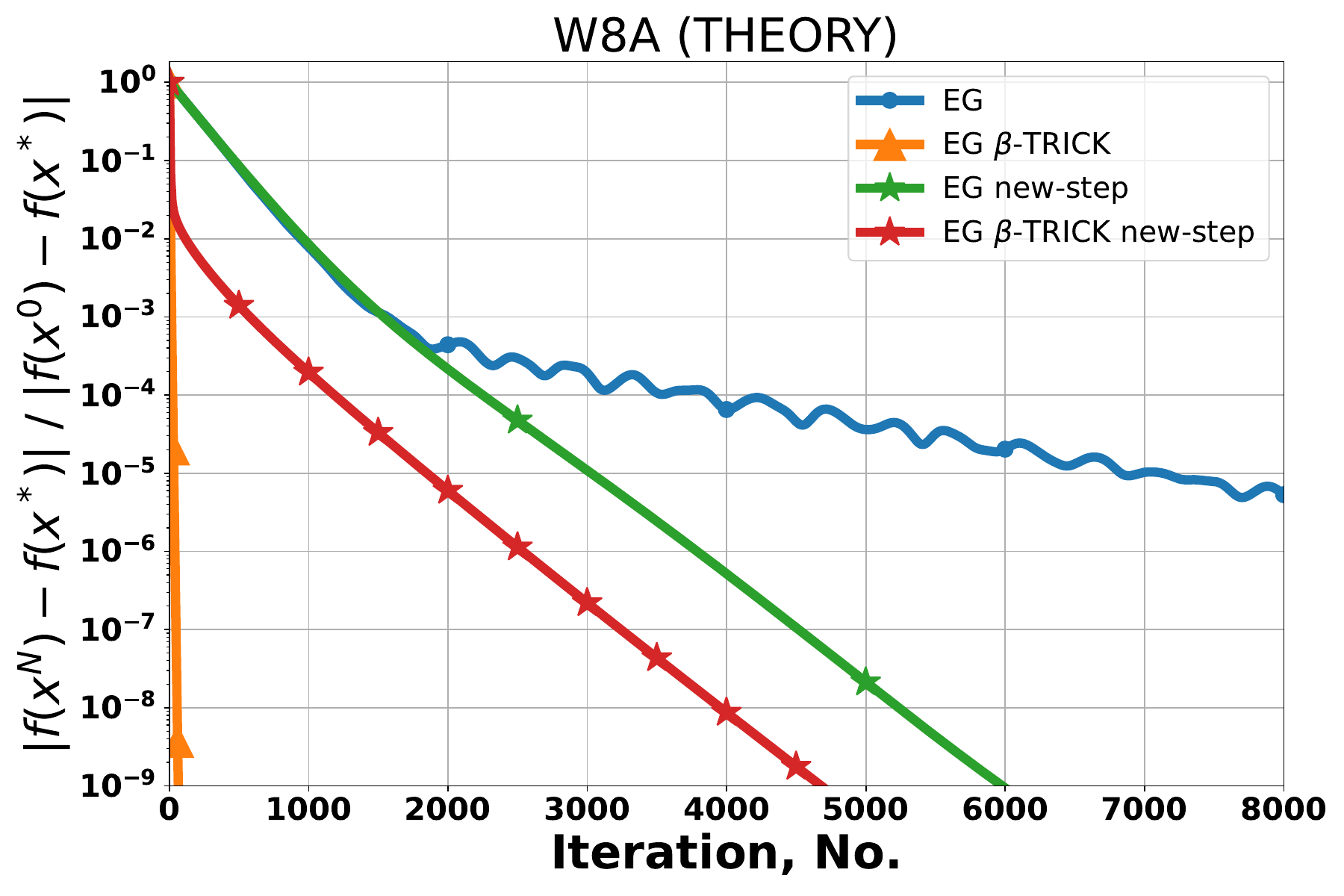}
\end{minipage}%
\begin{minipage}{0.49\textwidth}
  \centering
\includegraphics[width =  \textwidth ]{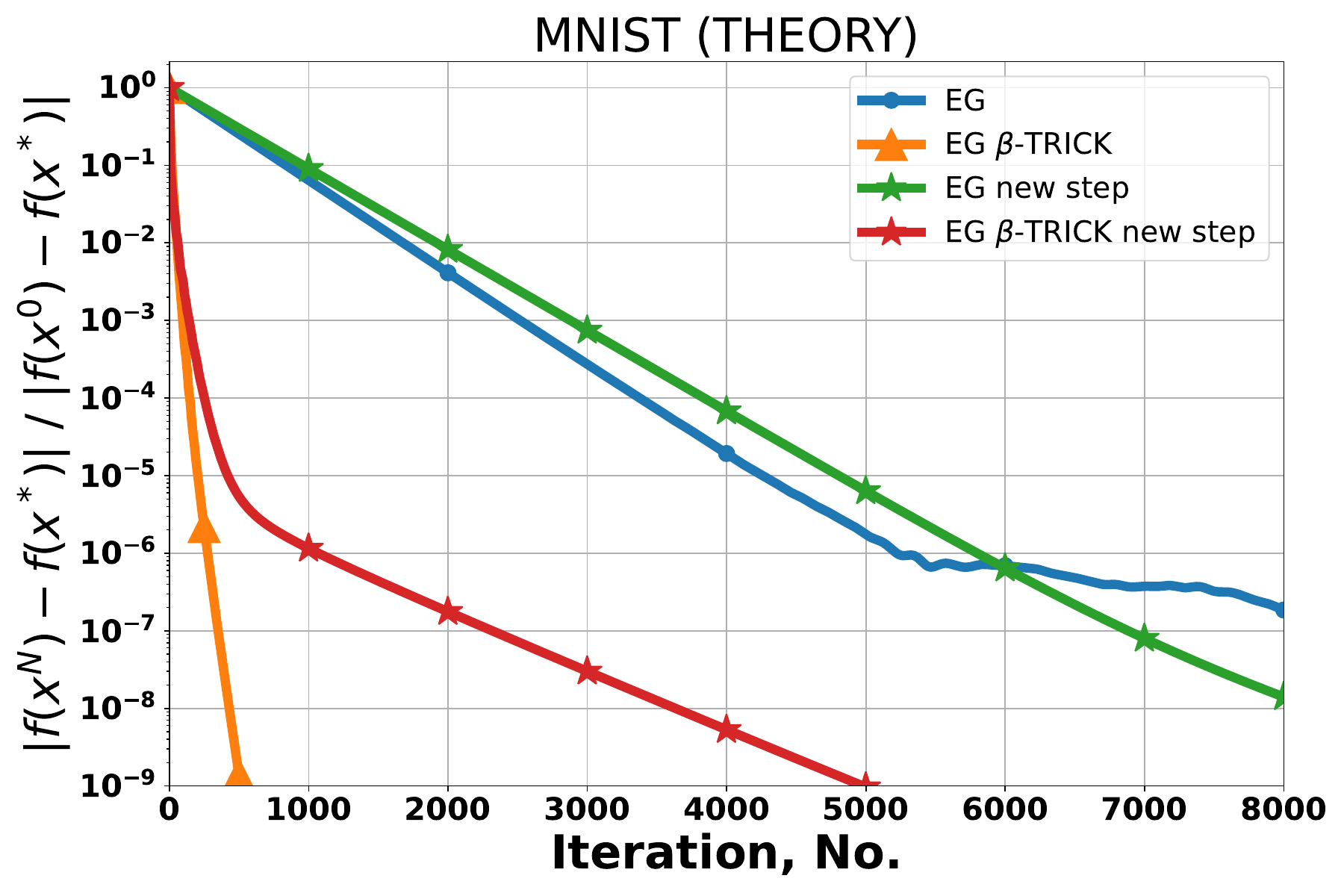}
\end{minipage}%
\\
\begin{minipage}{0.49\textwidth}
\centering
  (c) \texttt{w8a}
\end{minipage}%
\begin{minipage}{0.49\textwidth}
\centering
  (d) \texttt{MNIST}
\end{minipage}%
\caption{
Comparison of Algorithm \ref{alg:EG} for solving the VFL problem with different stepsize tunings: according to Theorem \ref{th:EG_basic_1} and Lemma \ref{lem:matrix}. The comparison is made on LibSVM datasets \texttt{mushrooms}, \texttt{a9a}, \texttt{w8a} and \texttt{MNIST}.}
    \label{fig:comparison1_d4}
\end{figure*}

Next, we want to consider modifications of Algorithm \ref{alg:EG} and show that they can speed Algorithm \ref{alg:EG} up from different points of view.

In the first group of experiments with modifications (Figure \ref{fig:comparison2_unbiased}), we test the performance of Algorithm \ref{alg:EG_quantization}. We use the compression operator $Q = \text{RandK}\%$, which is a random selection coordinates: 100\% (Algorithm \ref{alg:EG}), 50\%, 25\%, 10\%. An important detail is that we set the same random generator and seed on each of the devices. Therefore, at each iteration we send random coordinates, but they are the same for all devices.  The comparison is made in terms of the number of full vectors transmitted. In contrast to the main part, here we tune stepsizes, since with the theoretical step it is not possible to achieve the best acceleration compared to Algorithm \ref{alg:EG}. The comparison is done in two settings: the basic one and using the $\beta$-trick (see disscusion after Corollary \ref{cor:EG_basic_1}). The results show that compression can indeed speed up the communication process.

\begin{figure}[h]
\centering
\begin{minipage}{0.4\textwidth}
  \centering
\includegraphics[width =  \textwidth ]{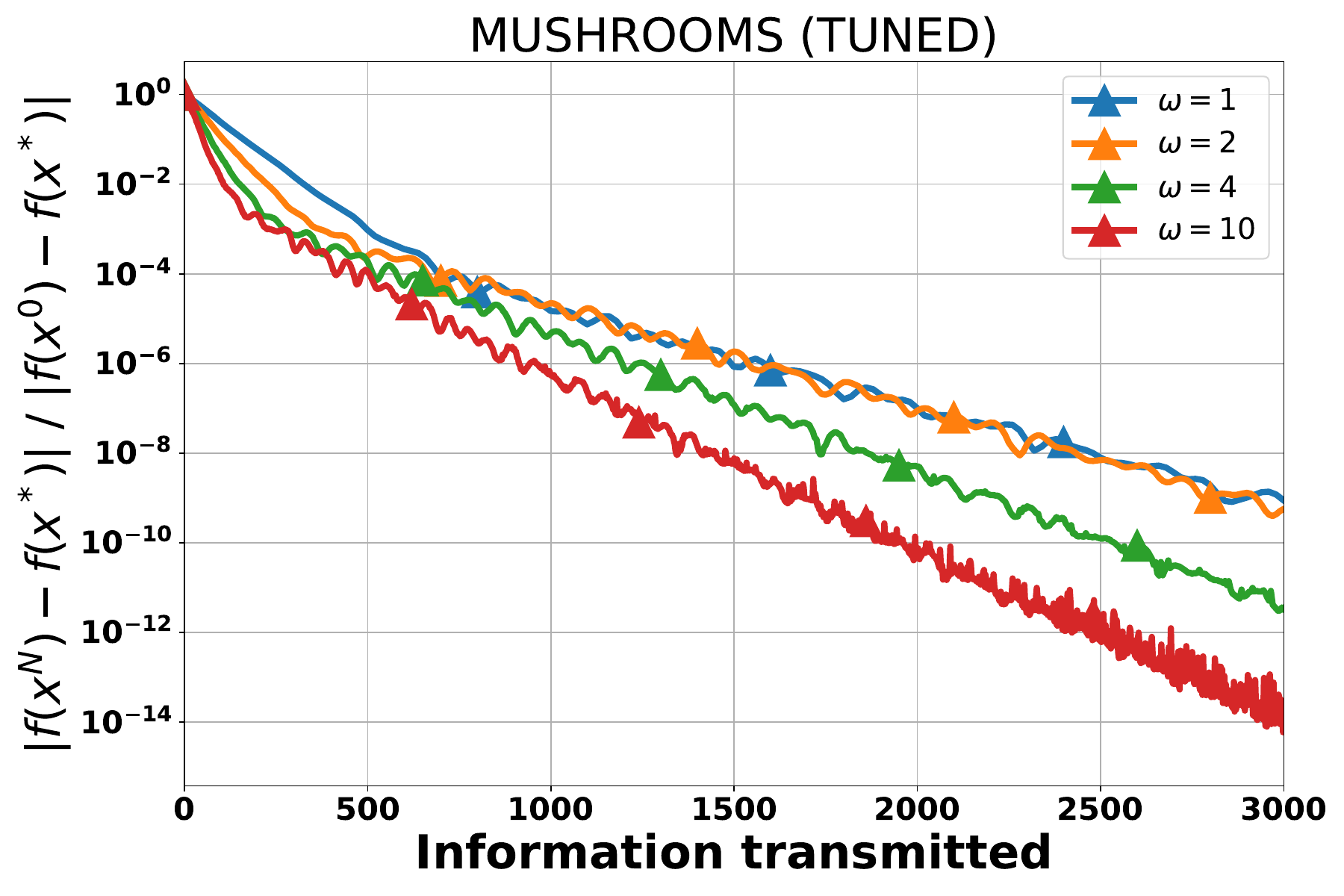}
\end{minipage}%
\begin{minipage}{0.4\textwidth}
  \centering
\includegraphics[width =  \textwidth ]{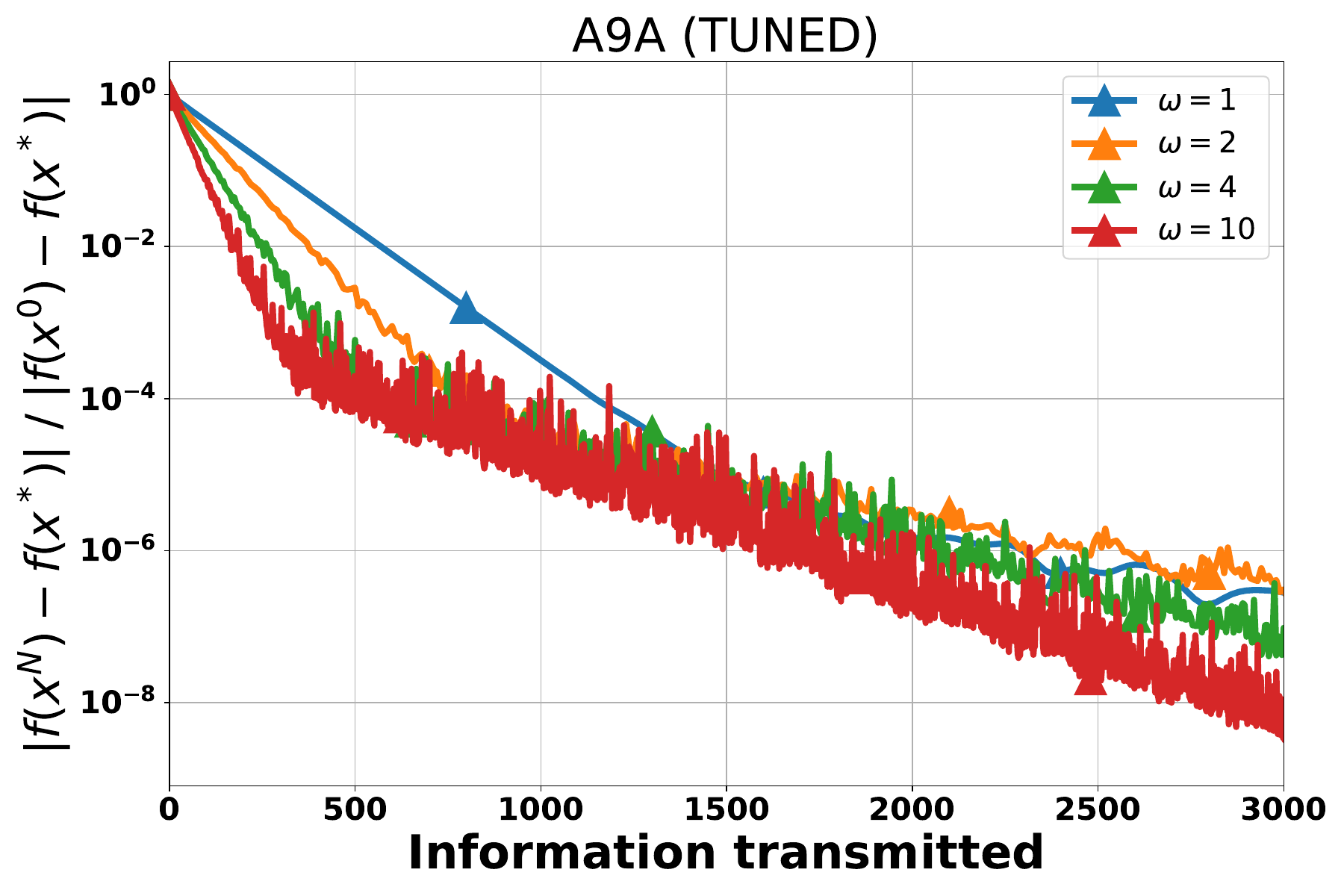}
\end{minipage}%
\\
\begin{minipage}{0.4\textwidth}
  \centering
\includegraphics[width =  \textwidth ]{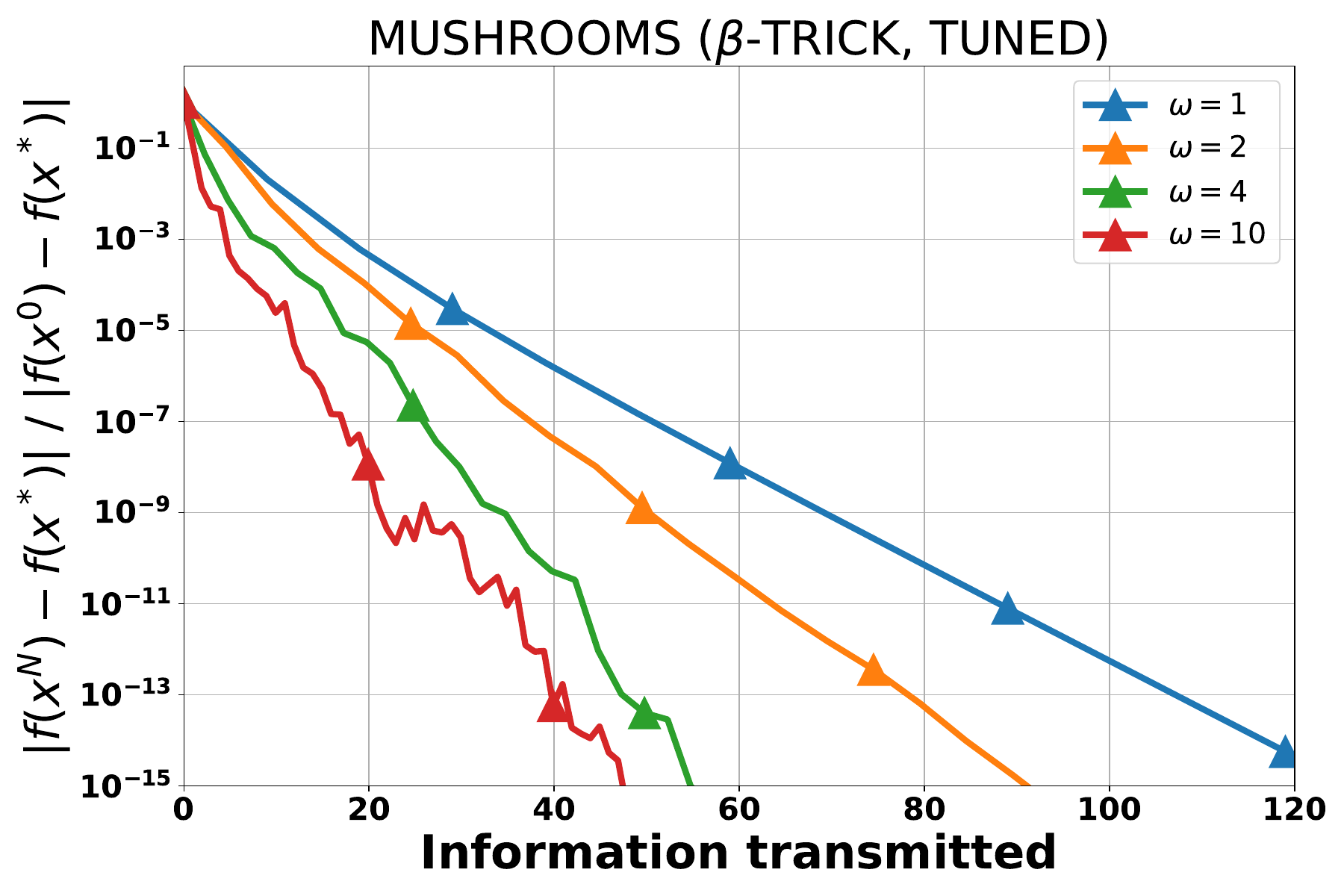}
\end{minipage}%
\begin{minipage}{0.4\textwidth}
  \centering
\includegraphics[width =  \textwidth ]{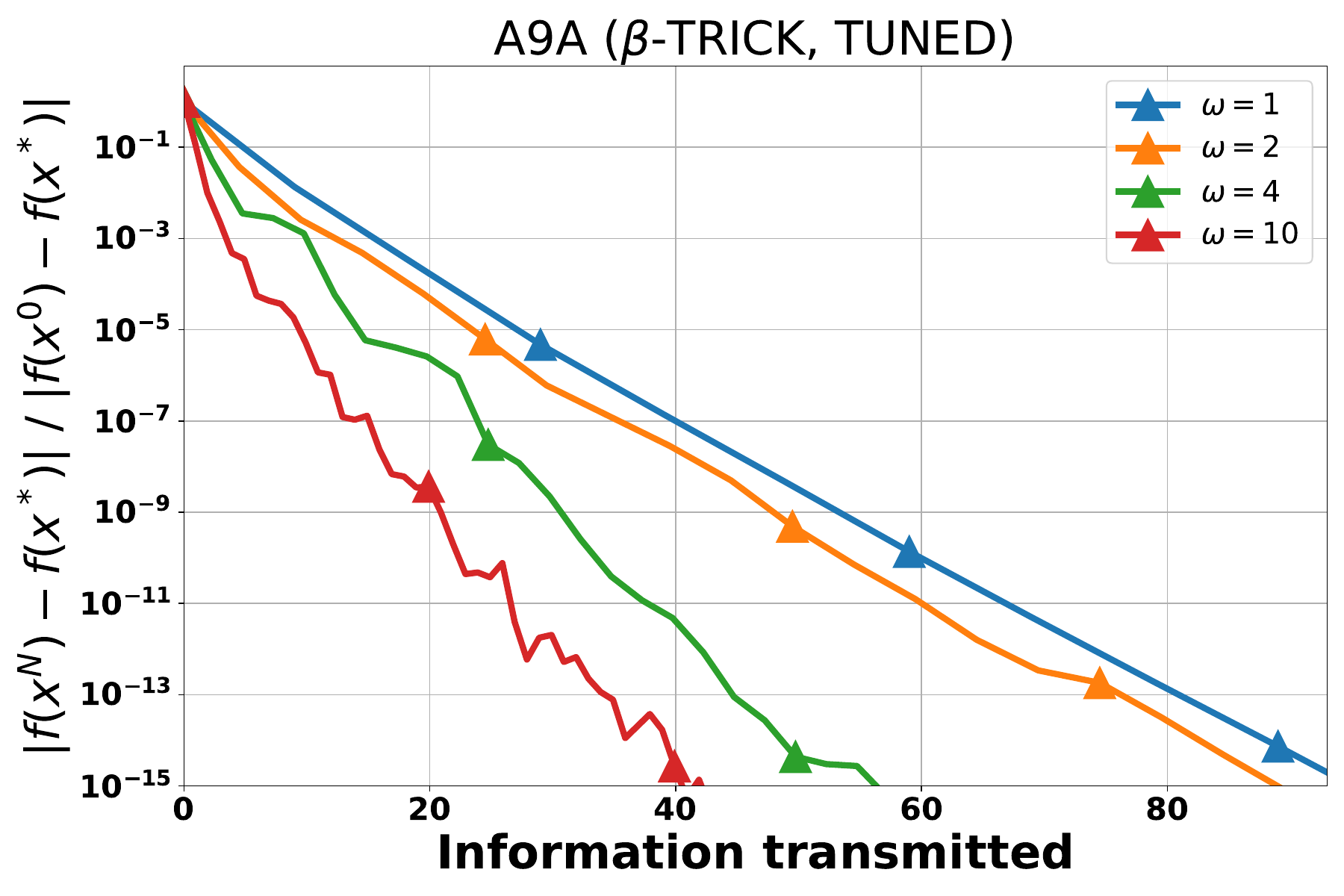}
\end{minipage}%
\\
\begin{minipage}{0.01\textwidth}
\quad
\end{minipage}%
\begin{minipage}{0.4\textwidth}
  \centering
(a) \texttt{mushrooms}
\end{minipage}%
\begin{minipage}{0.4\textwidth}
\centering
 (b) \texttt{a9a}
\end{minipage}%
\\
\begin{minipage}{0.4\textwidth}
  \centering
\includegraphics[width =  \textwidth ]{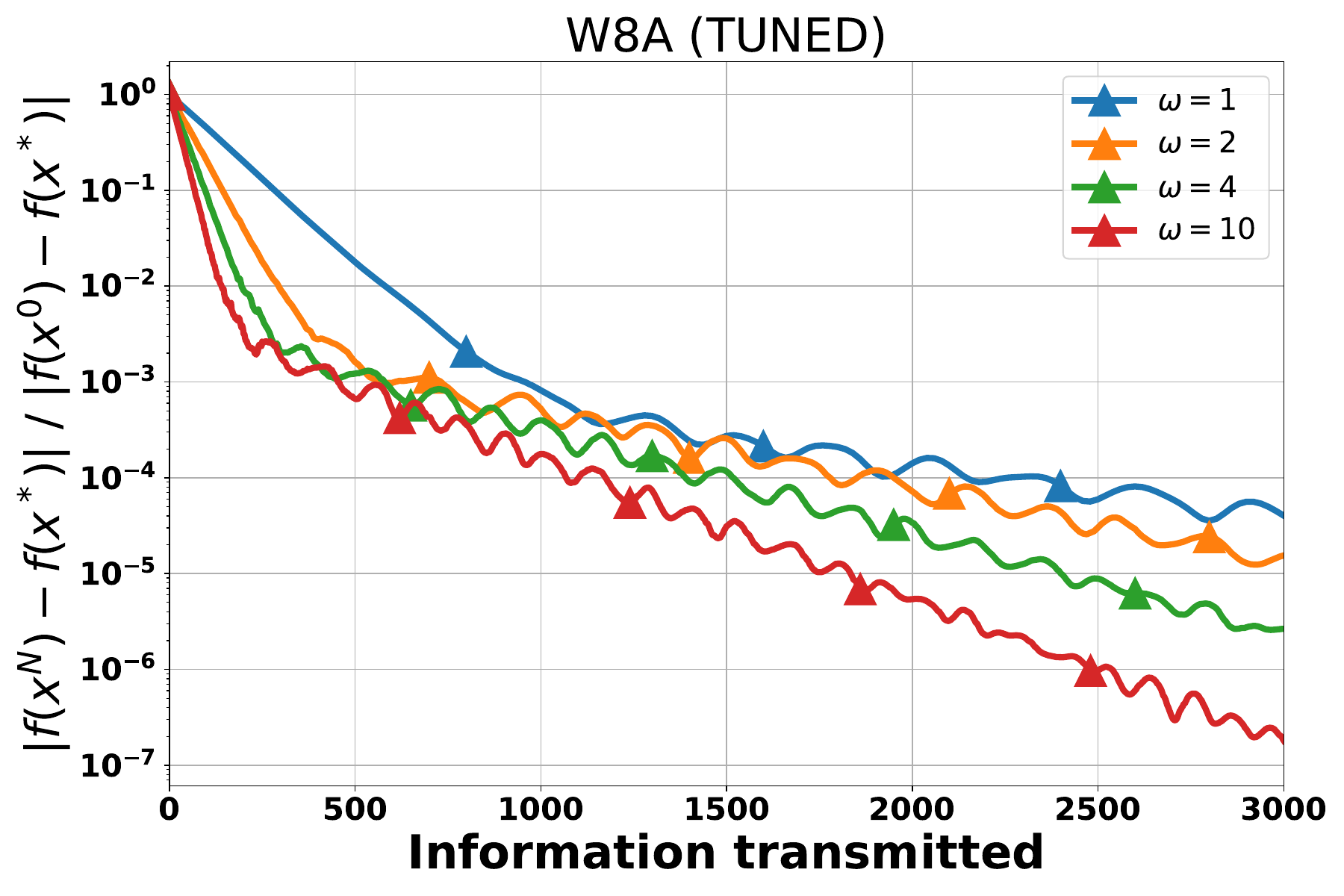}
\end{minipage}%
\begin{minipage}{0.4\textwidth}
  \centering
\includegraphics[width =  \textwidth ]{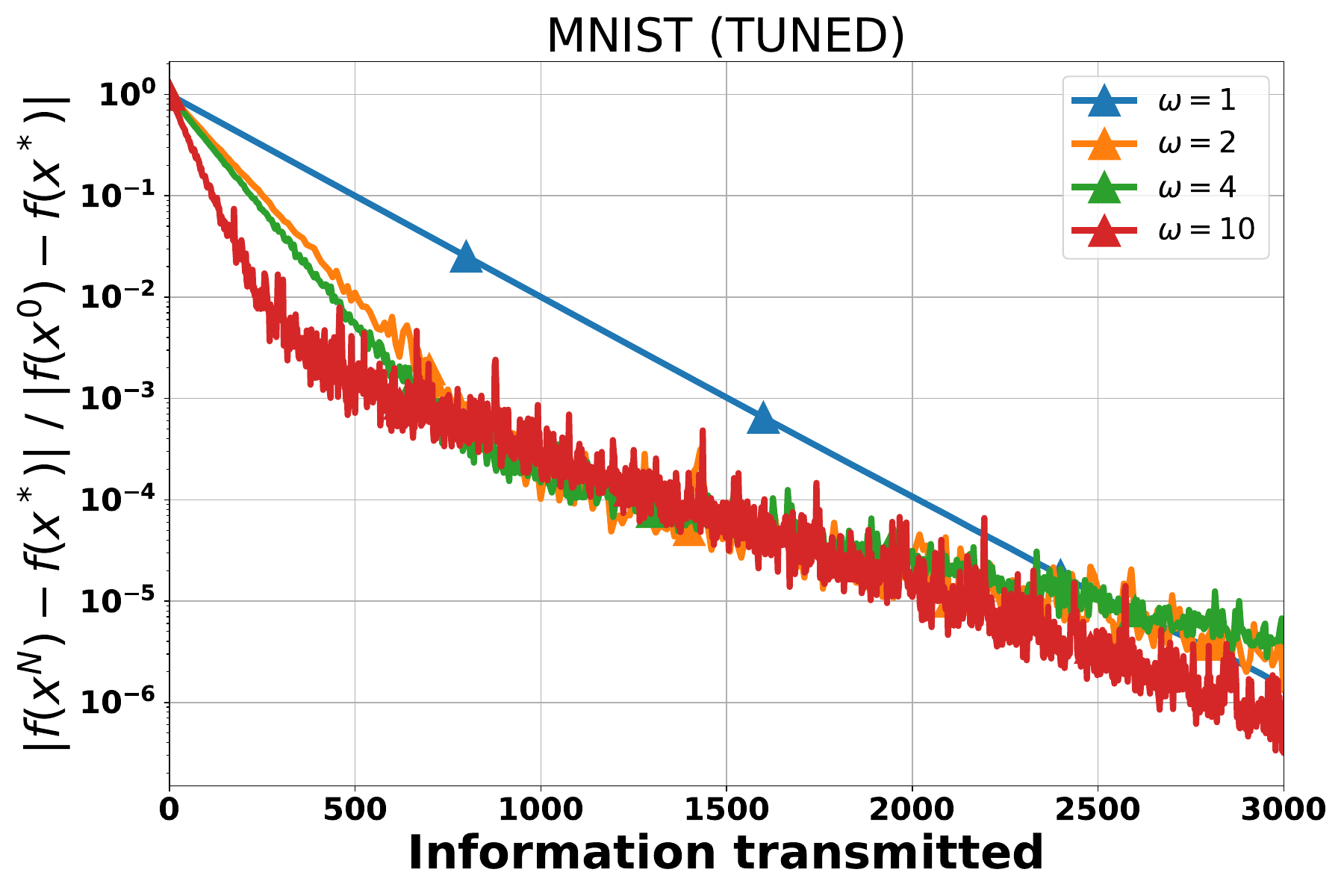}
\end{minipage}%
\\
\begin{minipage}{0.4\textwidth}
  \centering
\includegraphics[width =  \textwidth ]{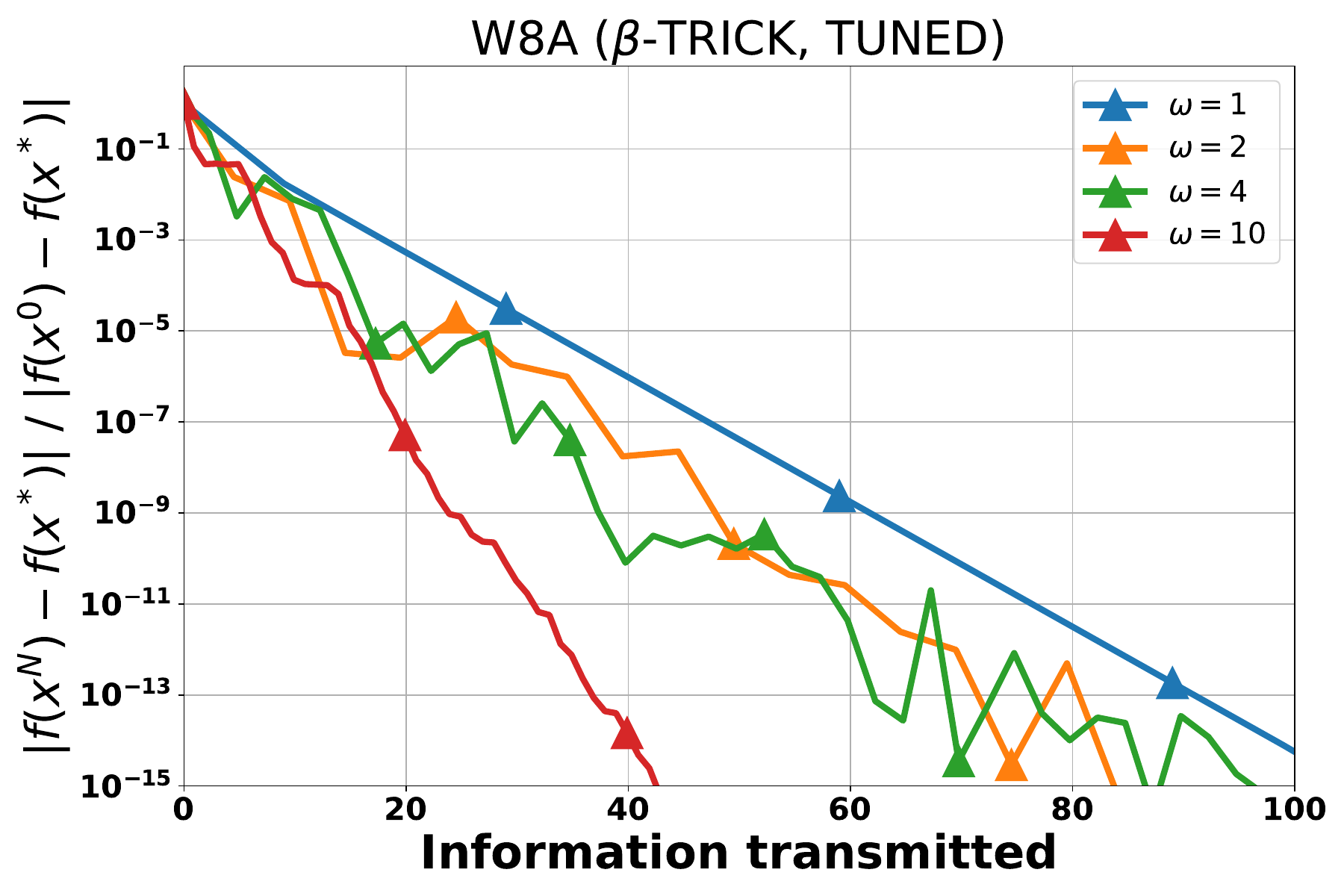}
\end{minipage}%
\begin{minipage}{0.4\textwidth}
  \centering
\includegraphics[width =  \textwidth ]{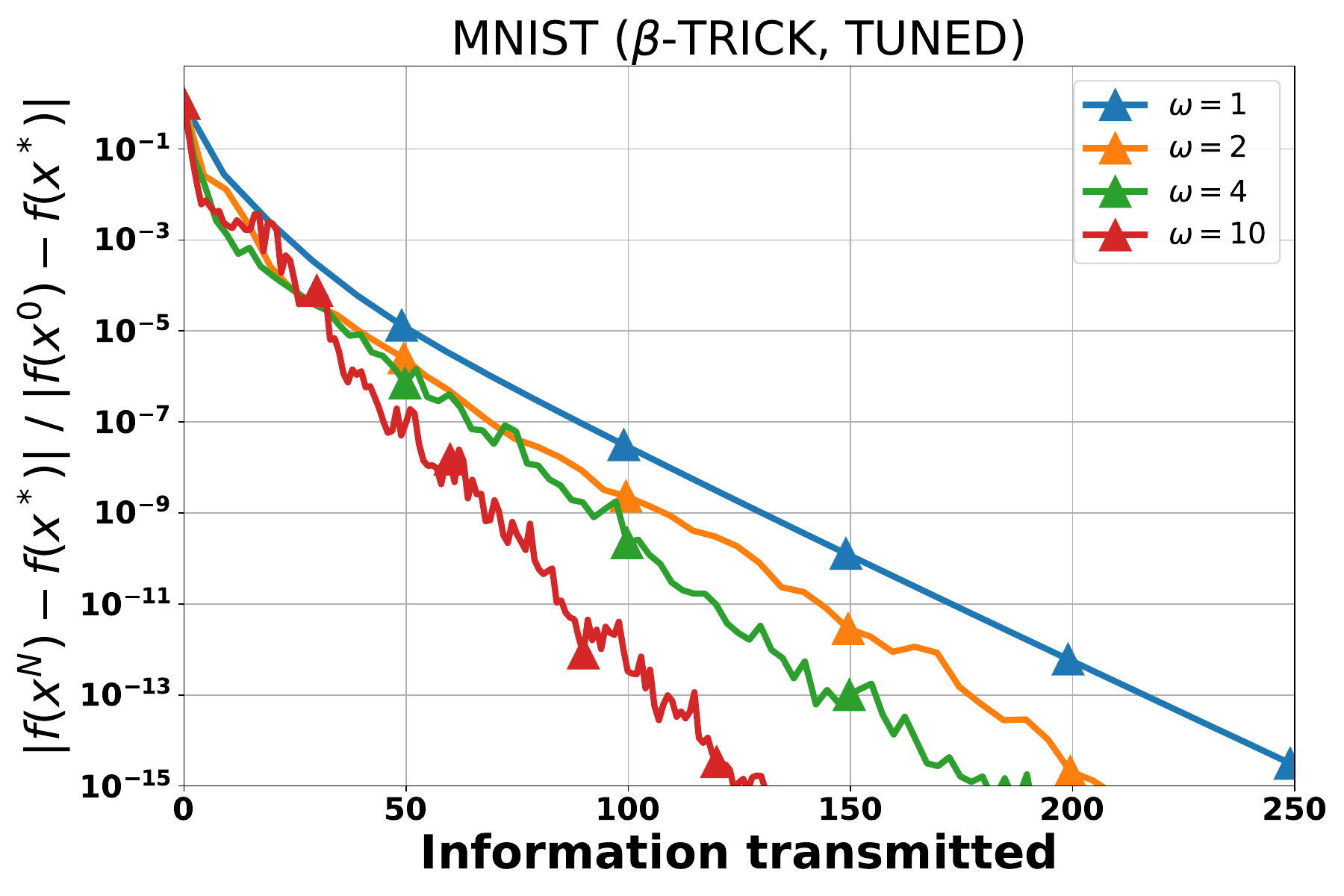}
\end{minipage}%
\\
\begin{minipage}{0.4\textwidth}
\centering
 (c) \texttt{w8a}
\end{minipage}%
\begin{minipage}{0.4\textwidth}
\centering
  (d) \texttt{MNIST}
\end{minipage}%
\vskip-5pt
\caption{
Comparison of Algorithm \ref{alg:EG_quantization} for solving the VFL problem (\ref{eq:vfl_lin_spp_1}). The comparison is made on LibSVM datasets \texttt{mushrooms}, \texttt{a9a}, \texttt{w8a} and \texttt{MNIST}. The compression operator $Q = \text{RandK}\%$. The criterion for comparison is the number of full vectors transmitted. The top line reflects the work of methods on the basic problem, the bottom line solves the problem with the $\beta$-trick (see disscusion after Corollary \ref{cor:EG_basic_1}).}
    \label{fig:comparison2_unbiased}
\end{figure}


In the second group of experiments with modifications (Figure \ref{fig:comparison3}), we test the performance of Algorithm \ref{alg:EG_biased} in comparison with Algorithm \ref{alg:EG_quantization}. We use the compression operators $C = \text{TopK}\%$ (for Algorithm \ref{alg:EG_biased}), which is a greedy selection coordinates, and $Q = \text{RandK}\%$ (for Algorithm \ref{alg:EG_quantization}) with K = 25\% and 10\%. The comparison is made in terms of the number of full vectors transmitted. As in the previous experiment, we tune stepsizes. In experiments, we see that unbiased compression outperforms biased compression almost always. In the horizontal case, the opposite is usually true \cite{beznosikov2020biased}. We attribute this effect to the fact that in the case of RandK\% compression we set the same random generator and seed on different devices and therefore they send the same random coordinates at each iteration. In the case of using TopK\% operator we cannot do this, therefore convergence is worse.

\begin{figure}[h]
\centering
\begin{minipage}{0.4\textwidth}
  \centering
\includegraphics[width =  \textwidth ]{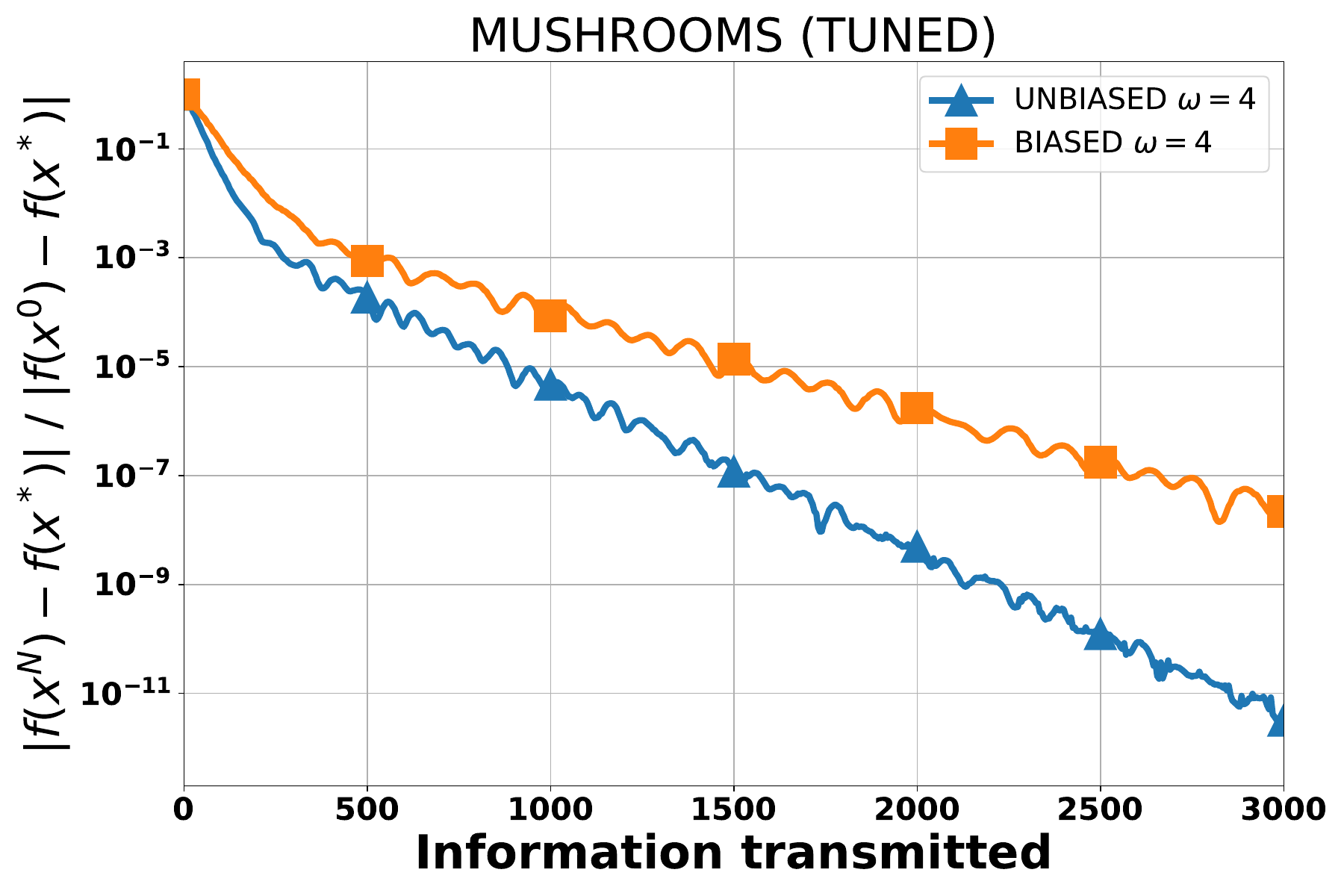}
\end{minipage}%
\begin{minipage}{0.4\textwidth}
  \centering
\includegraphics[width =  \textwidth ]{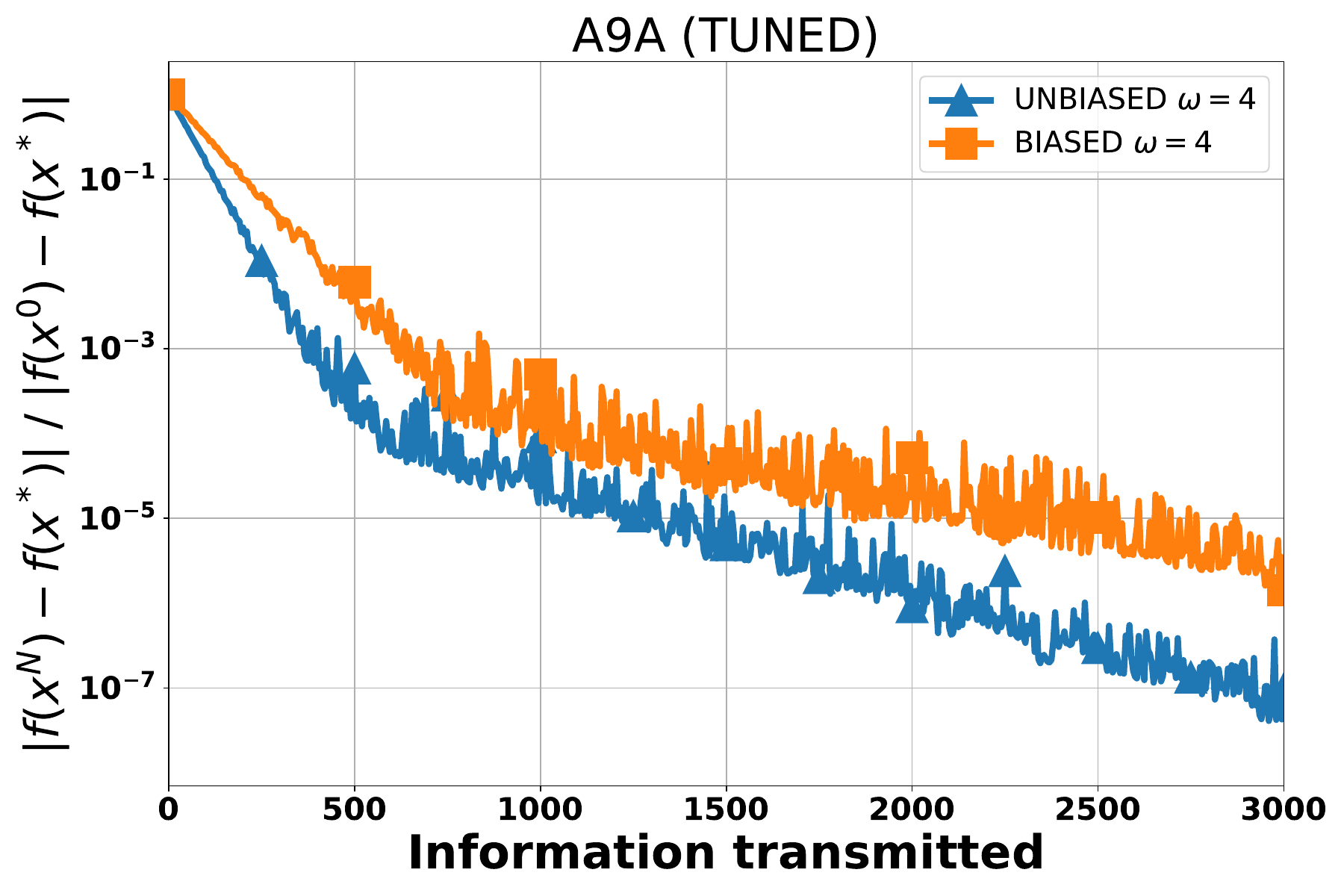}
\end{minipage}%
\\
\begin{minipage}{0.4\textwidth}
  \centering
\includegraphics[width =  \textwidth ]{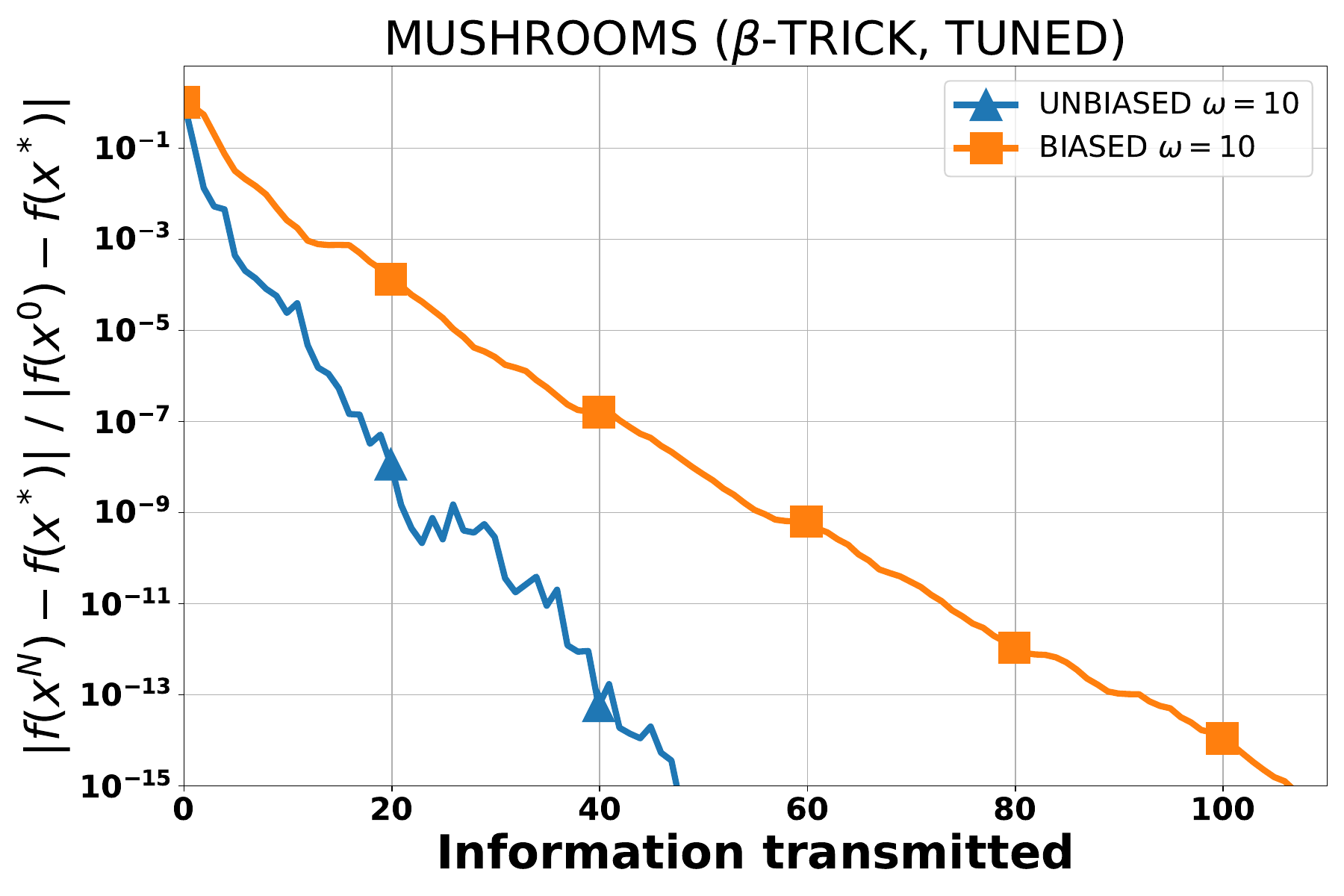}
\end{minipage}%
\begin{minipage}{0.4\textwidth}
  \centering
\includegraphics[width =  \textwidth ]{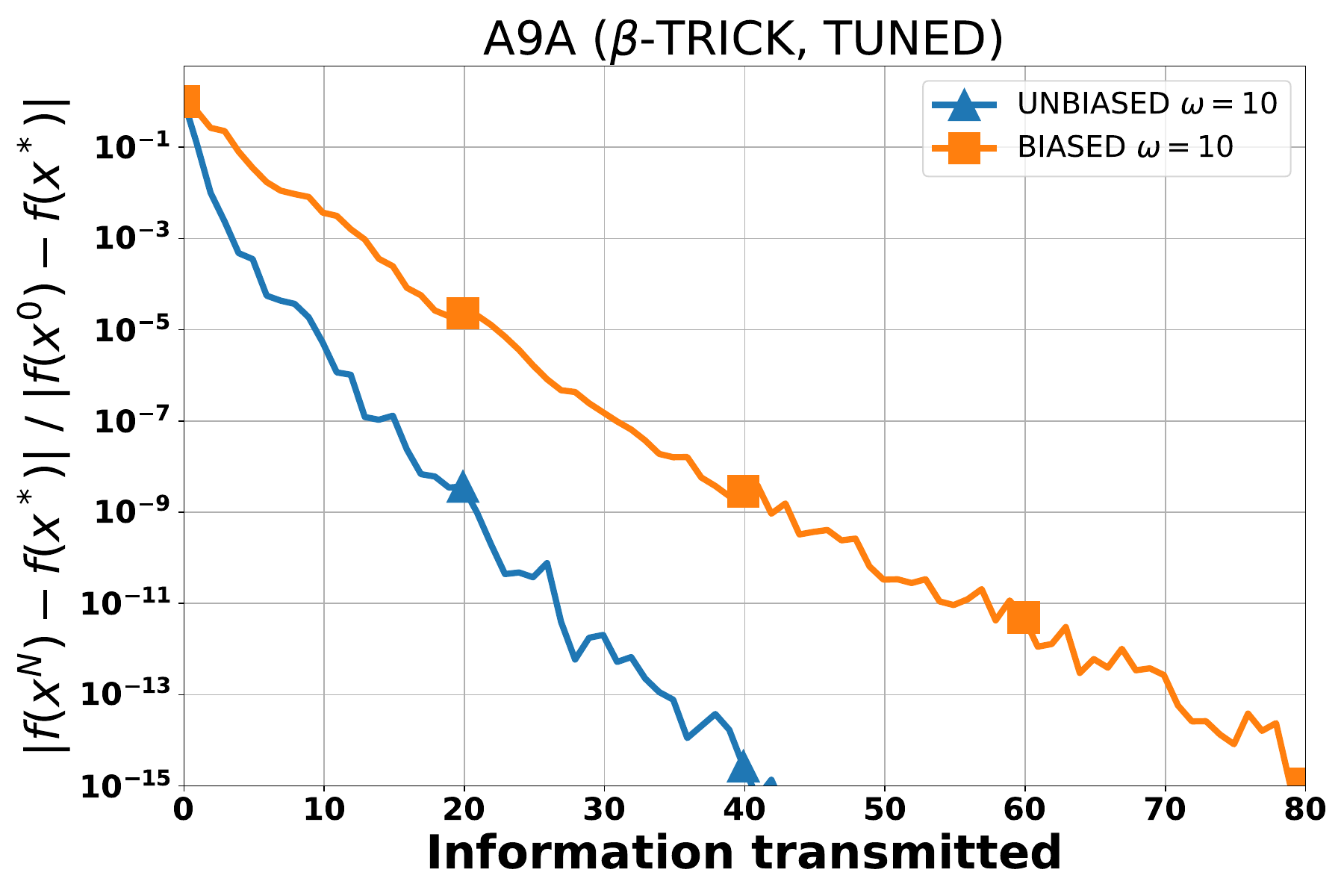}
\end{minipage}%
\\
\begin{minipage}{0.01\textwidth}
\quad
\end{minipage}%
\begin{minipage}{0.4\textwidth}
  \centering
(a) \texttt{mushrooms}
\end{minipage}%
\begin{minipage}{0.4\textwidth}
\centering
 (b) \texttt{a9a}
\end{minipage}%
\\
\begin{minipage}{0.4\textwidth}
  \centering
\includegraphics[width =  \textwidth ]{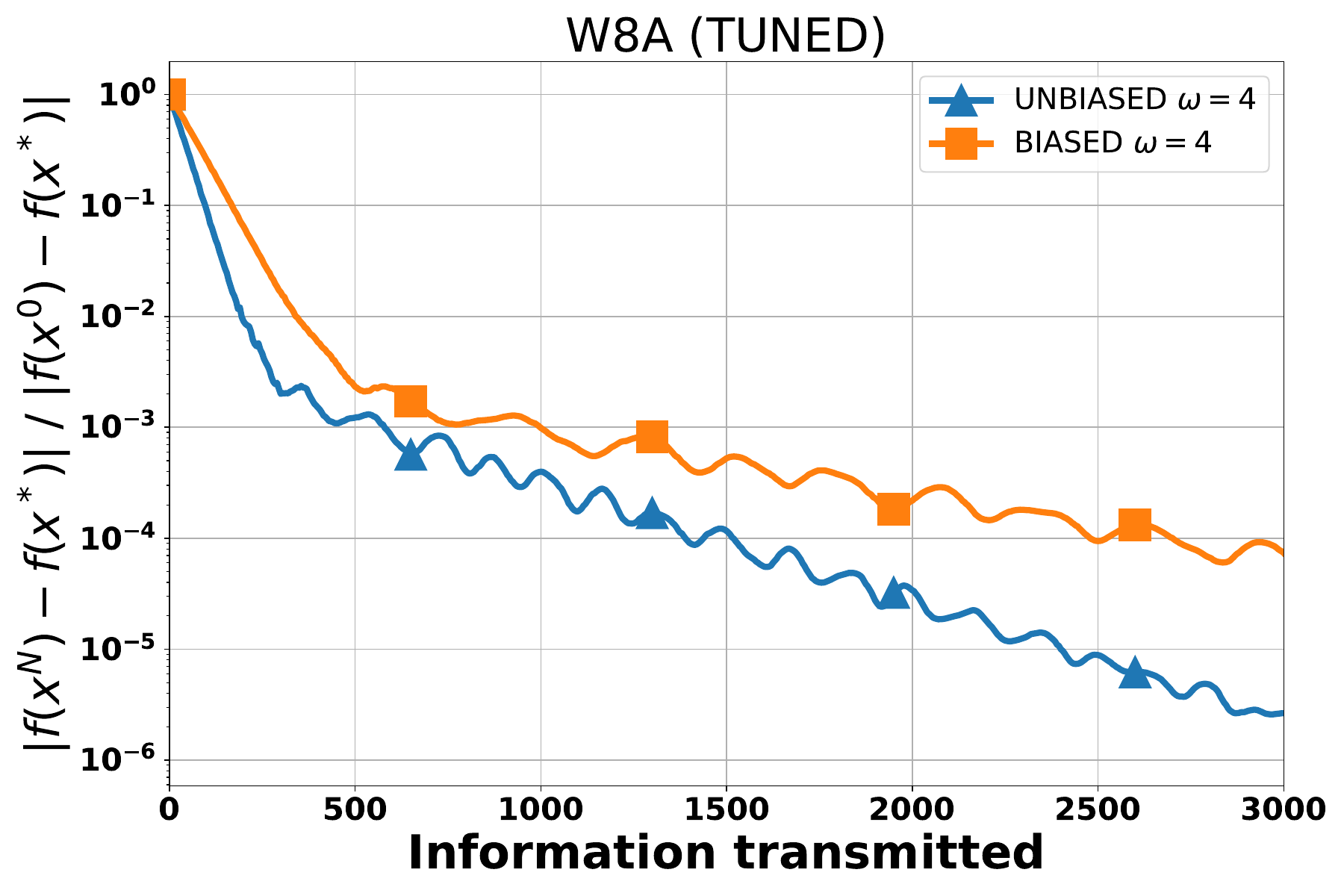}
\end{minipage}%
\begin{minipage}{0.4\textwidth}
  \centering
\includegraphics[width =  \textwidth ]{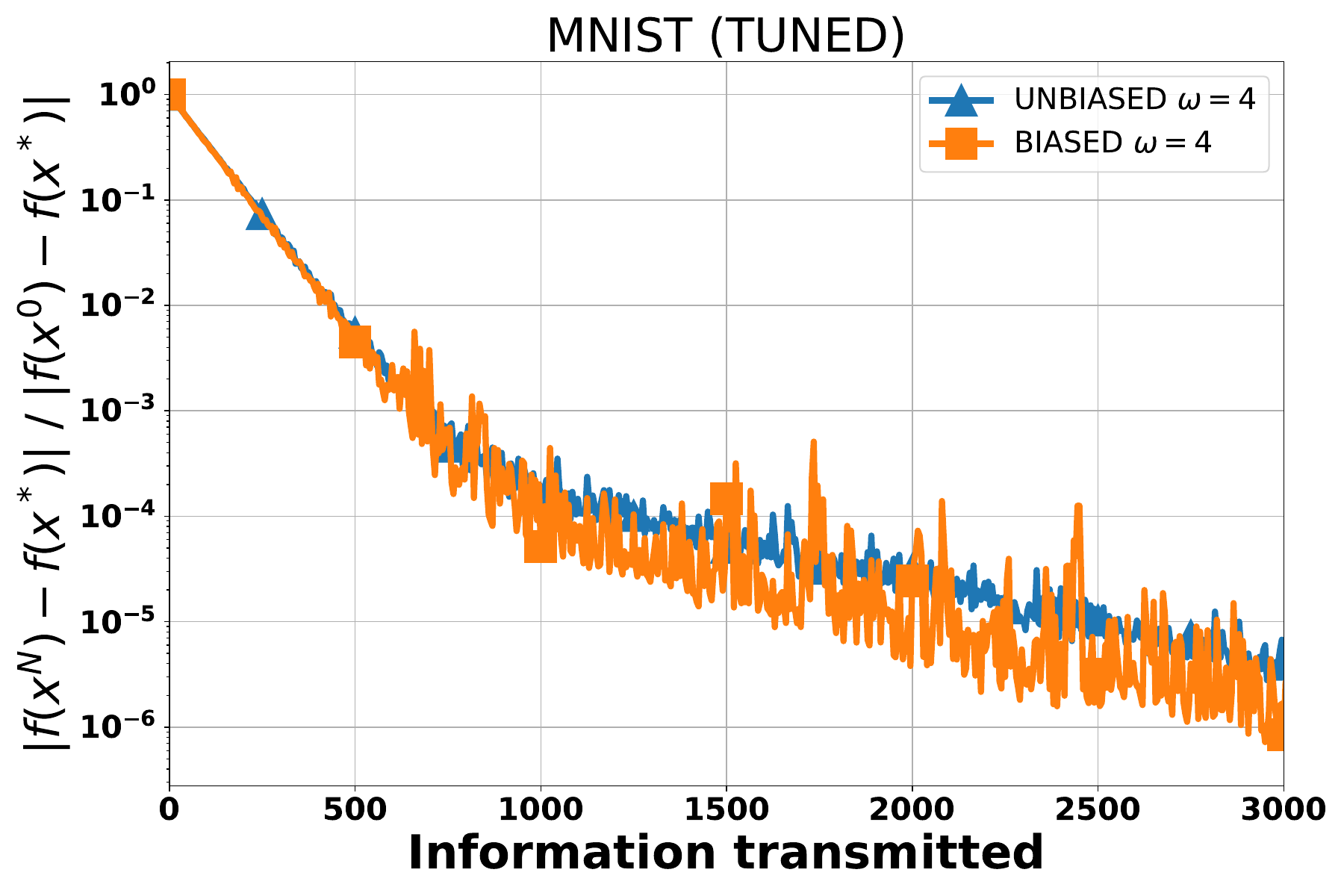}
\end{minipage}%
\\
\begin{minipage}{0.4\textwidth}
  \centering
\includegraphics[width =  \textwidth ]{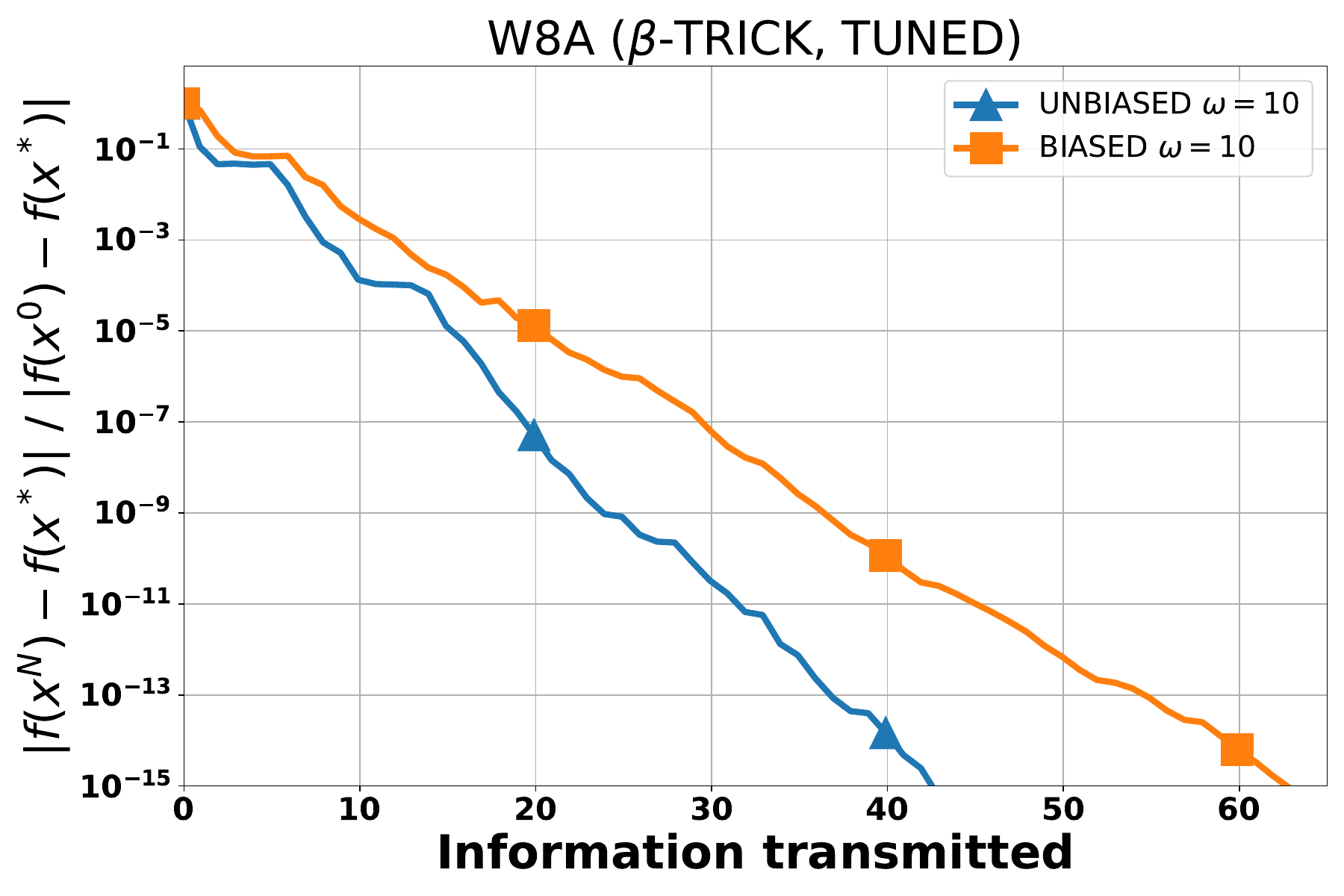}
\end{minipage}%
\begin{minipage}{0.4\textwidth}
  \centering
\includegraphics[width =  \textwidth ]{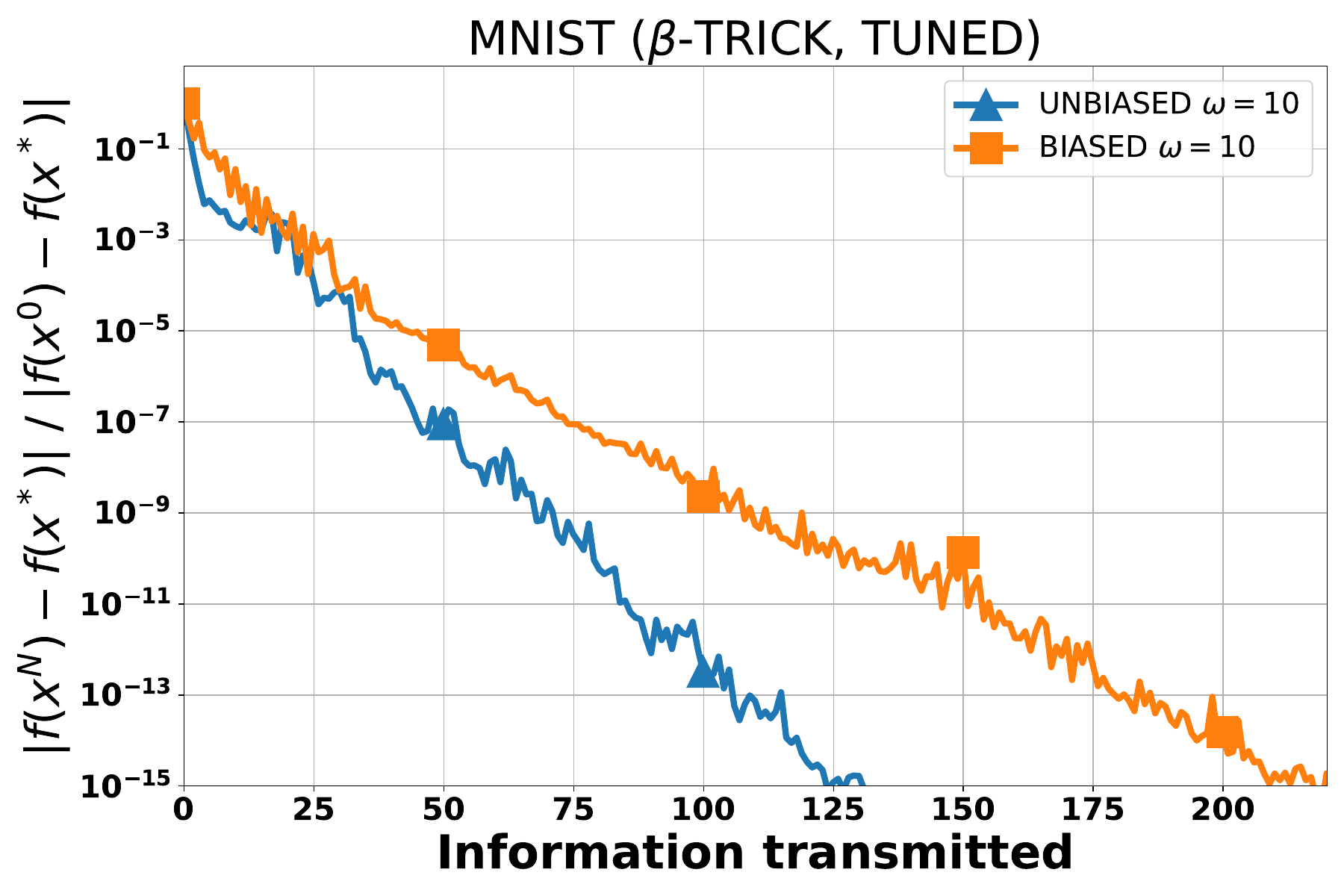}
\end{minipage}%
\\
\begin{minipage}{0.4\textwidth}
\centering
  (c) \texttt{w8a}
\end{minipage}%
\begin{minipage}{0.4\textwidth}
\centering
  (d) \texttt{MNIST}
\end{minipage}%
\vskip-5pt
\caption{
Comparison of Algorithm \ref{alg:EG_biased} and Algorithm \ref{alg:EG_quantization} for solving the VFL problem (\ref{eq:vfl_lin_spp_1}). The comparison is made on LibSVM datasets \texttt{mushrooms}, \texttt{a9a}, \texttt{w8a} and \texttt{MNIST}. The compression operators $C = \text{TopK}\%$ and $Q = \text{RandK}\%$. The criterion for comparison is the number of full vectors transmitted. The top line reflects the work of methods on the basic problem, the bottom line solves the problem with the $\beta$-trick (see disscusion after Corollary \ref{cor:EG_basic_1}).}
    \label{fig:comparison3}
\end{figure}


In the third group of experiments with modifications (Figure \ref{fig:comparison5}), we test the performance of Algorithm \ref{alg:EG_pp}. At each iteration we generate only 2 devices out of 5 that communicate. The comparison is made in terms of the number of full vectors transmitted from all devices. As in the previous experiments, we tune stepsizes. The results show that the partial participation technique can indeed speed up the communication process in terms of the number of devices communicated.

\begin{figure}[h]
\centering
\begin{minipage}{0.49\textwidth}
  \centering
\includegraphics[width =  \textwidth ]{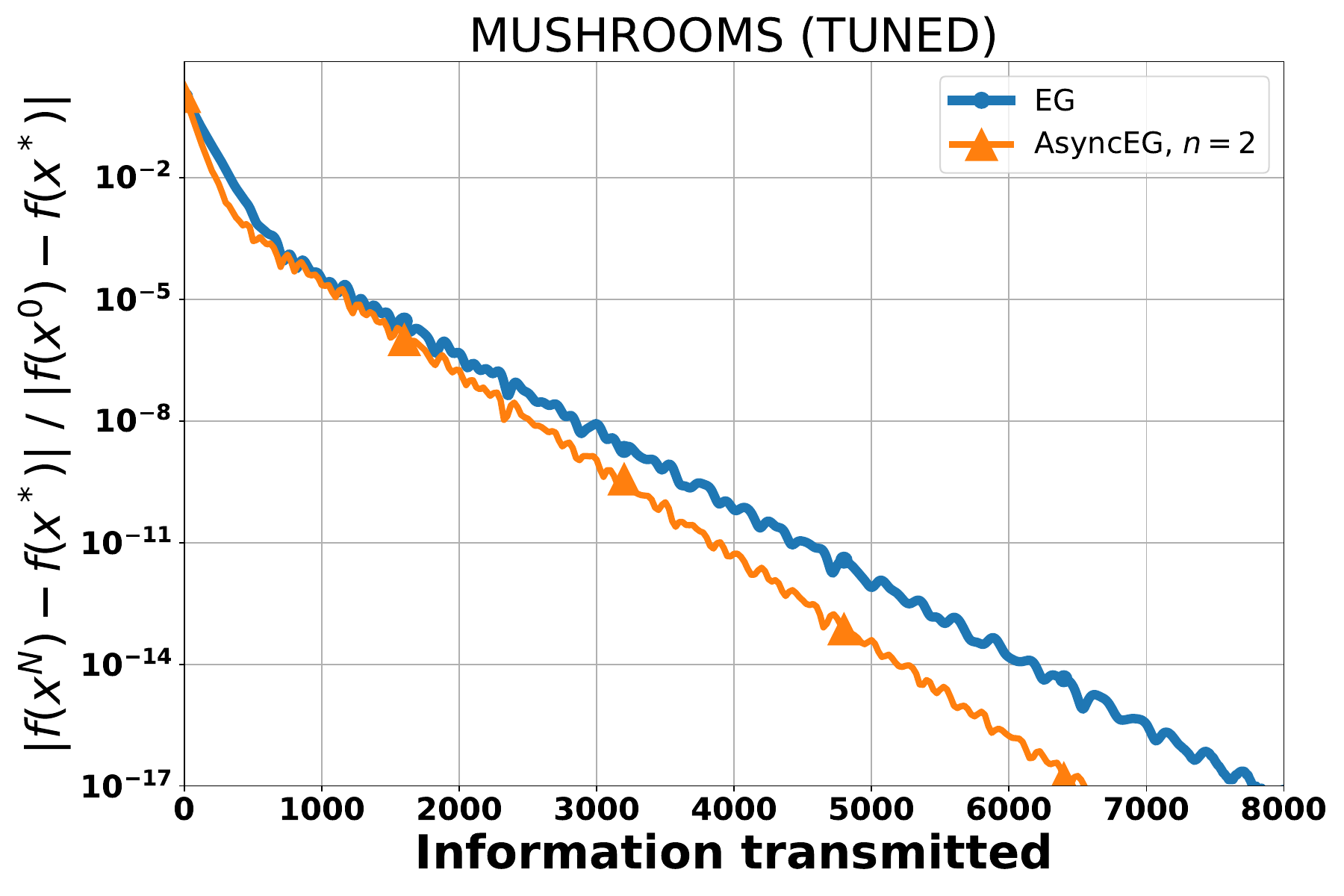}
\end{minipage}%
\begin{minipage}{0.49\textwidth}
  \centering
\includegraphics[width =  \textwidth ]{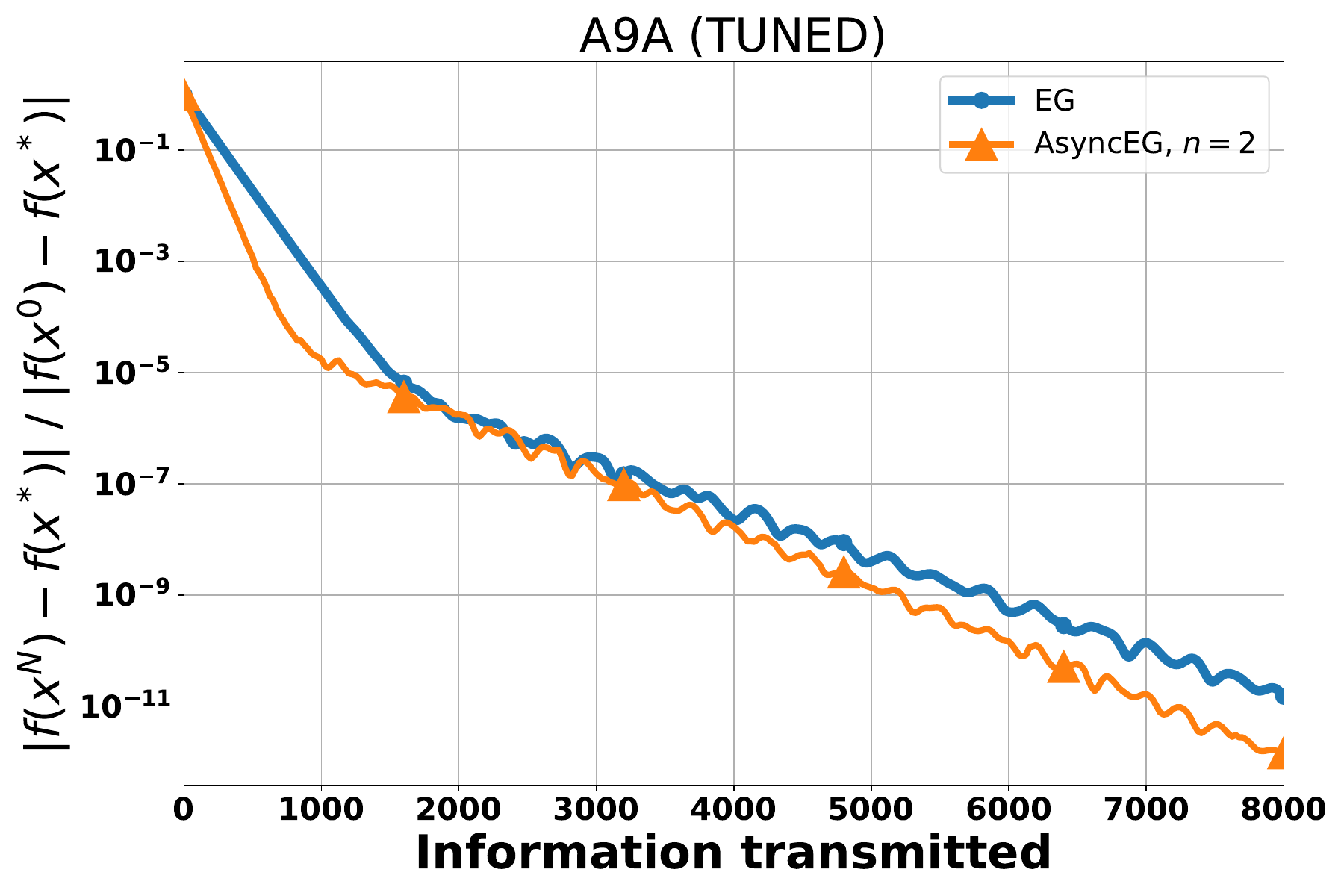}
\end{minipage}%
\\
\begin{minipage}{0.01\textwidth}
\quad
\end{minipage}%
\begin{minipage}{0.49\textwidth}
  \centering
(a) \texttt{mushrooms}
\end{minipage}%
\begin{minipage}{0.49\textwidth}
\centering
 (b) \texttt{a9a}
\end{minipage}%
\\
\begin{minipage}{0.49\textwidth}
  \centering
\includegraphics[width =  \textwidth ]{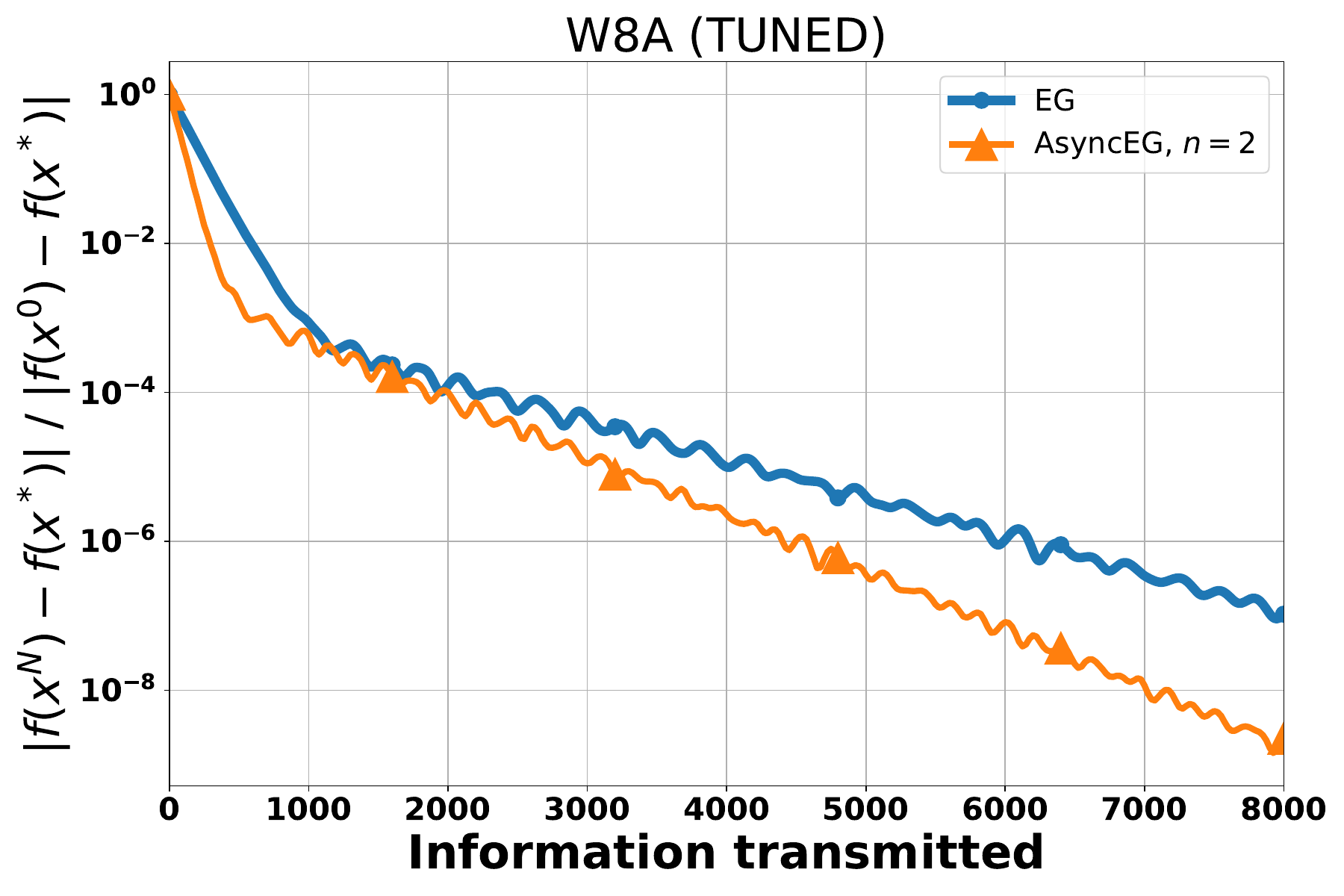}
\end{minipage}%
\begin{minipage}{0.49\textwidth}
  \centering
\includegraphics[width =  \textwidth ]{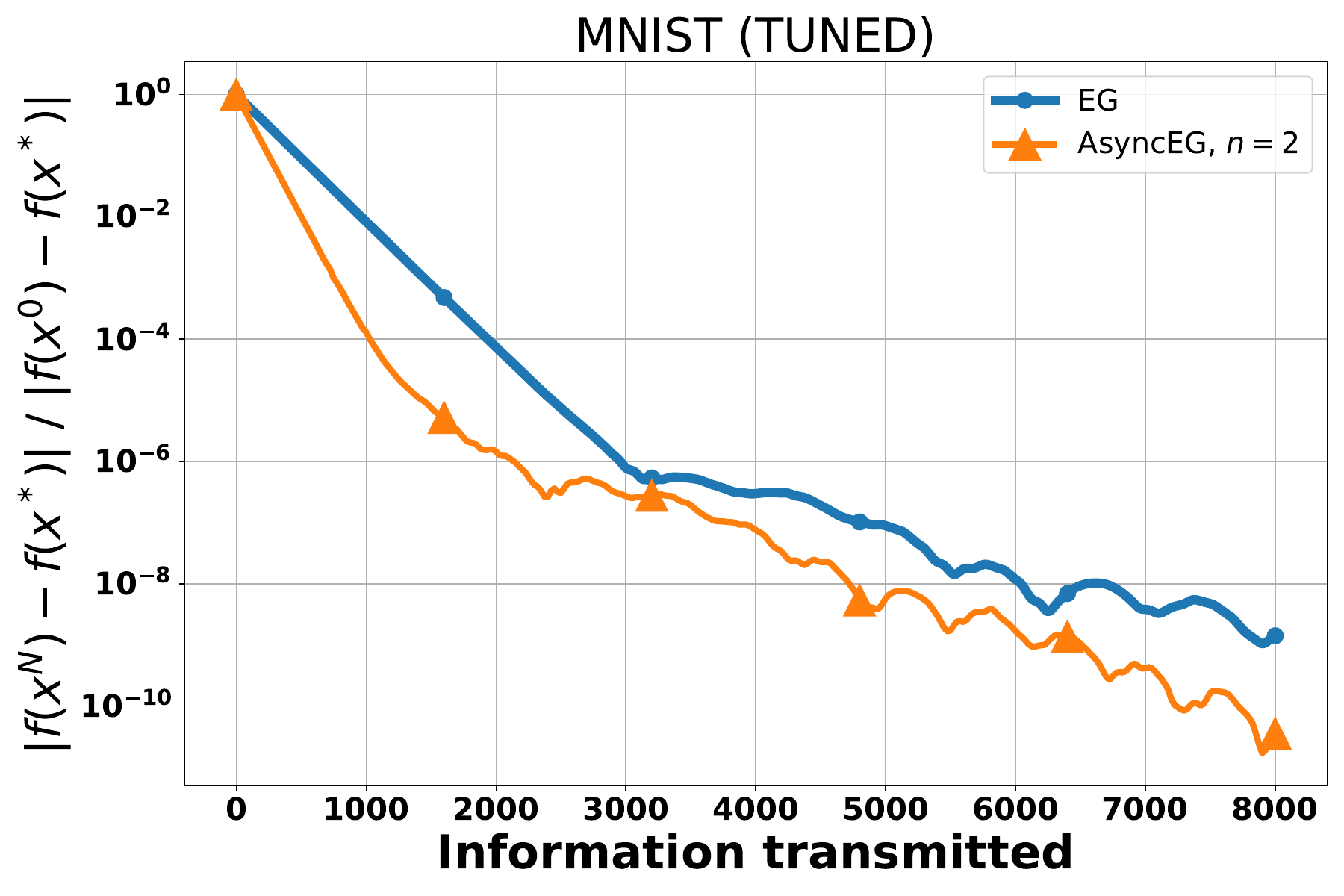}
\end{minipage}%
\\
\begin{minipage}{0.49\textwidth}
\centering
  (c) \texttt{w8a}
\end{minipage}%
\begin{minipage}{0.49\textwidth}
\centering
  (d) \texttt{MNIST}
\end{minipage}%
\caption{
Comparison of Algorithm \ref{alg:EG_pp} for solving the VFL problem (\ref{eq:vfl_lin_spp_1}). The comparison is made on LibSVM datasets \texttt{mushrooms}, \texttt{a9a}, \texttt{w8a} and \texttt{MNIST}. The criterion for comparison is the number of full vectors transmitted.}
    \label{fig:comparison5}
\end{figure}

In the fouth group of experiments with modifications (Figure \ref{fig:comparison6}), we test the performance of Algorithm \ref{alg:EG_coord}. We use a random selection coordinates: 100\% (Algorithm \ref{alg:EG}), 50\%, 25\%, 10\%. An important detail is that we set the same random generator and seed on each of the devices. The comparison is made in terms of the computational powers. Here we also tune stepsizes. The results show that the random coordinate selection can indeed speed up the computational process.

\begin{figure}[h]
\centering
\begin{minipage}{0.4\textwidth}
  \centering
\includegraphics[width =  \textwidth ]{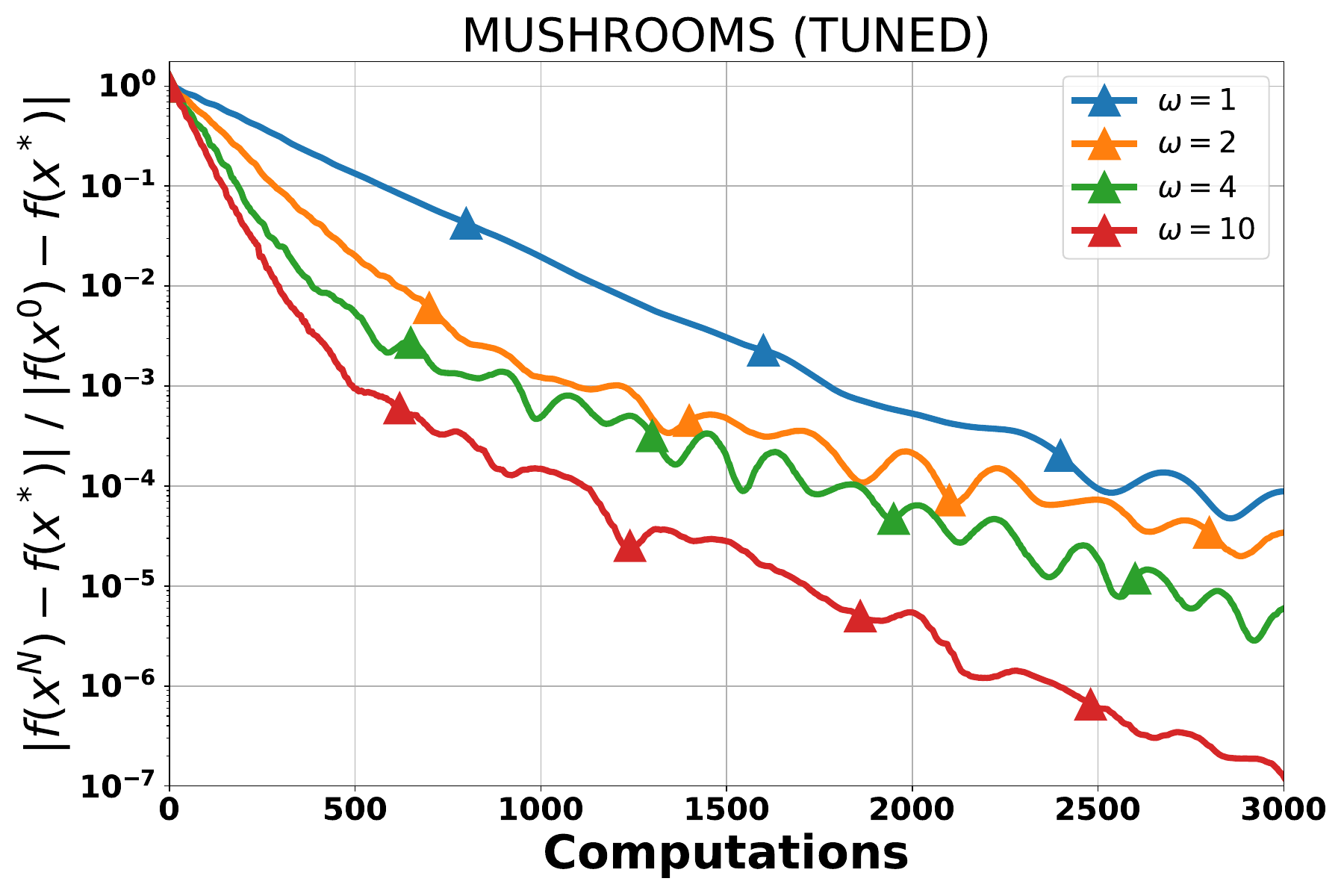}
\end{minipage}%
\begin{minipage}{0.4\textwidth}
  \centering
\includegraphics[width =  \textwidth ]{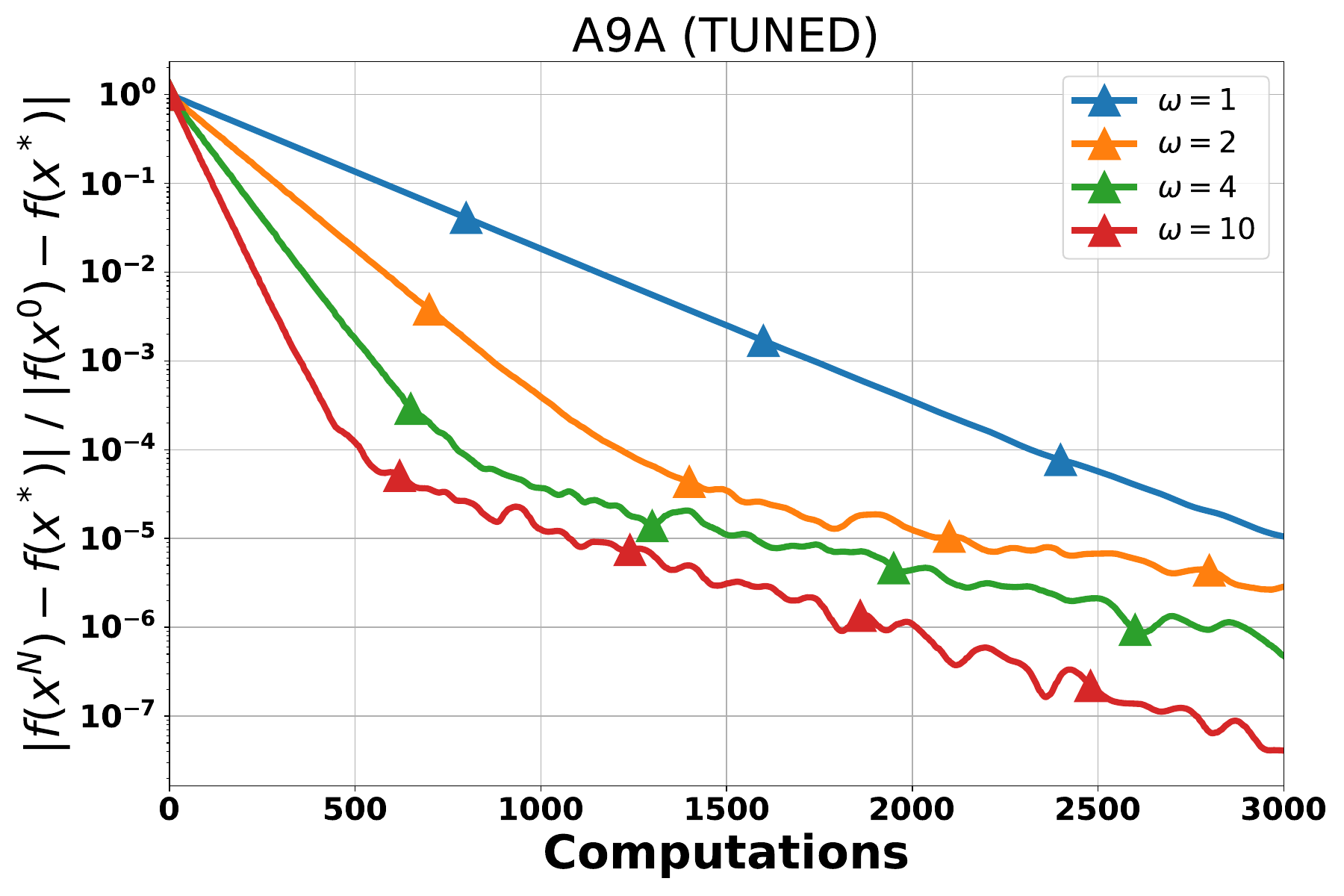}
\end{minipage}%
\\
\begin{minipage}{0.4\textwidth}
  \centering
\includegraphics[width =  \textwidth ]{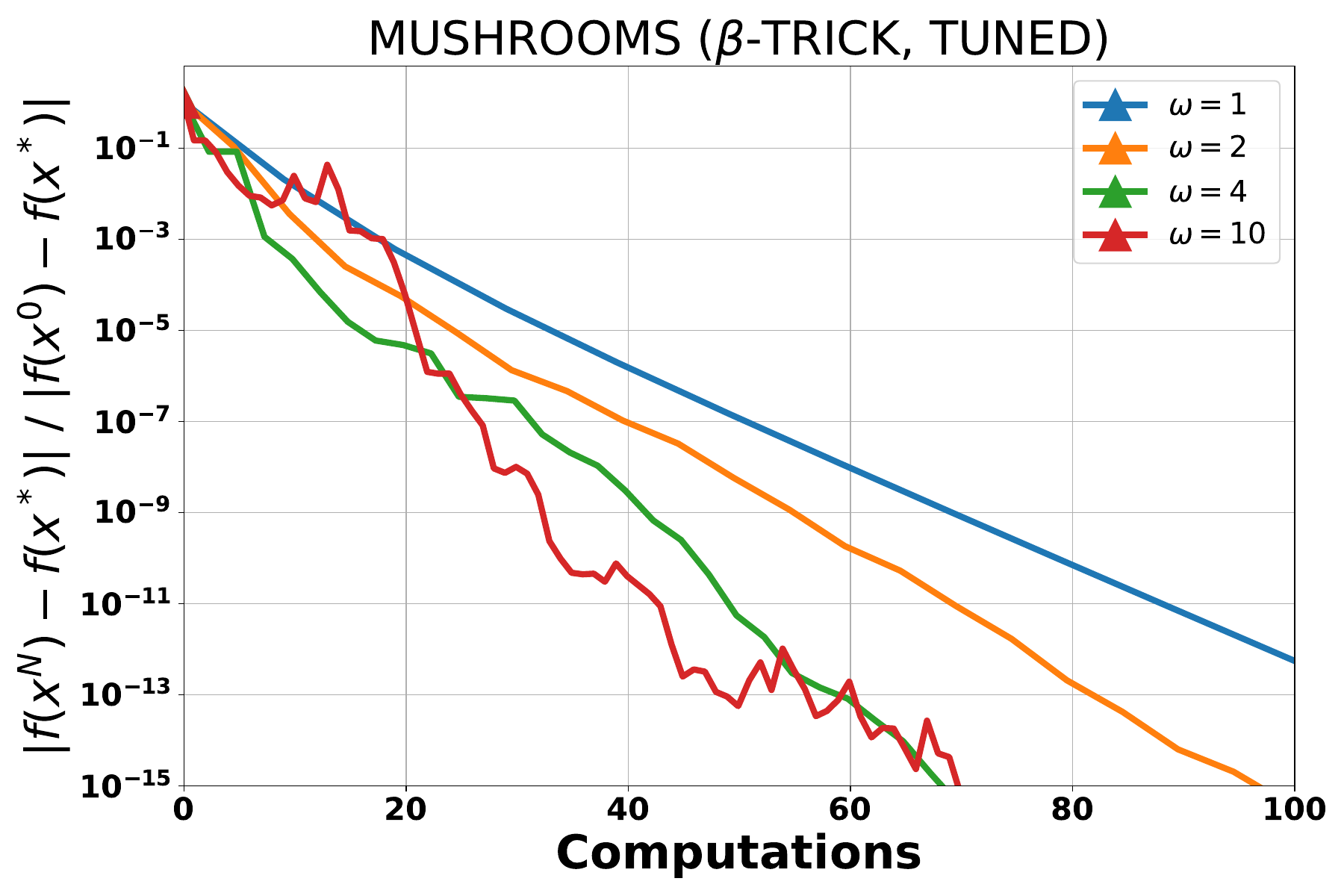}
\end{minipage}%
\begin{minipage}{0.4\textwidth}
  \centering
\includegraphics[width =  \textwidth ]{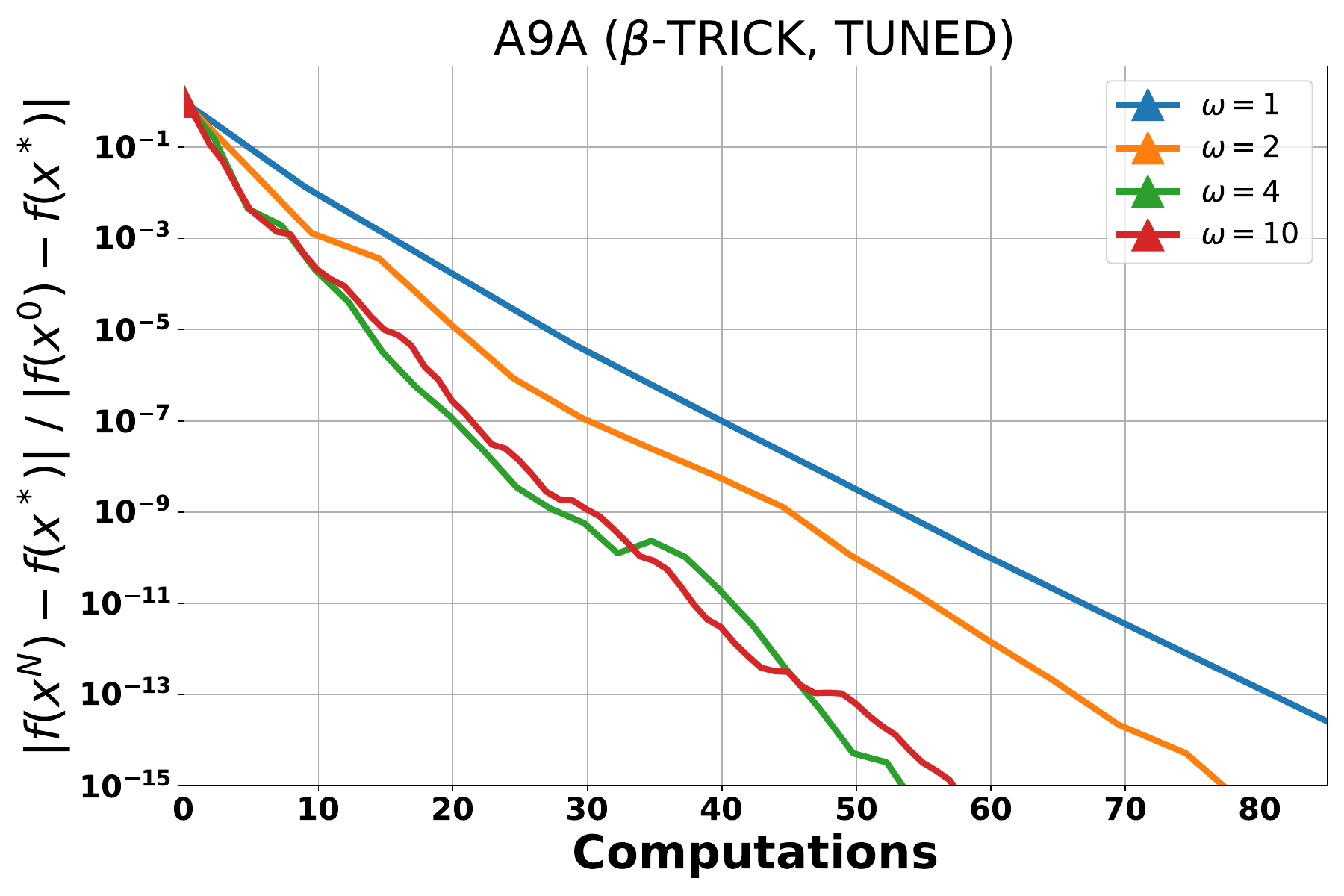}
\end{minipage}%
\\
\begin{minipage}{0.01\textwidth}
\quad
\end{minipage}%
\begin{minipage}{0.4\textwidth}
  \centering
(a) \texttt{mushrooms}
\end{minipage}%
\begin{minipage}{0.4\textwidth}
\centering
 (b) \texttt{a9a}
\end{minipage}%
\\
\begin{minipage}{0.4\textwidth}
  \centering
\includegraphics[width =  \textwidth ]{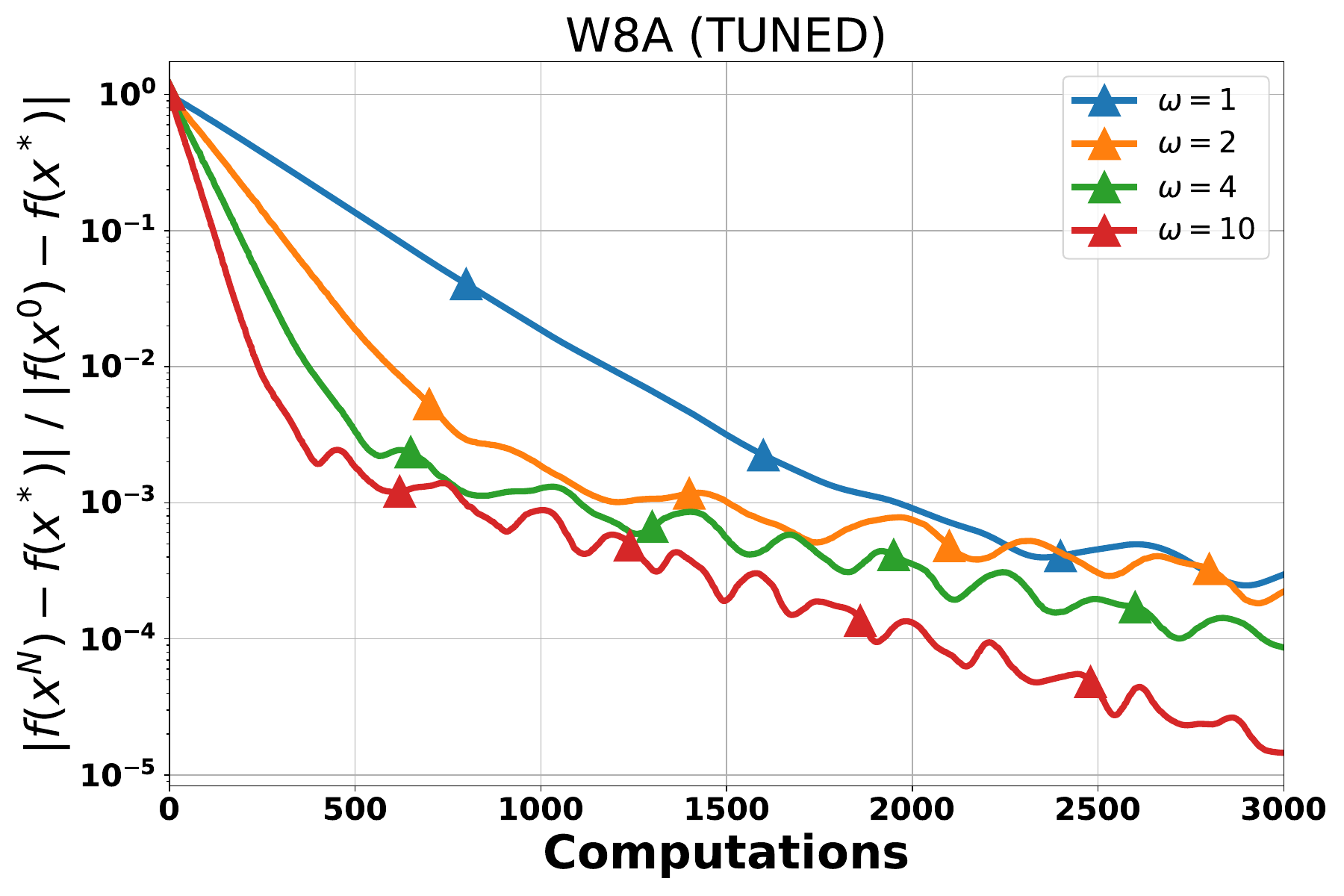}
\end{minipage}%
\begin{minipage}{0.4\textwidth}
  \centering
\includegraphics[width =  \textwidth ]{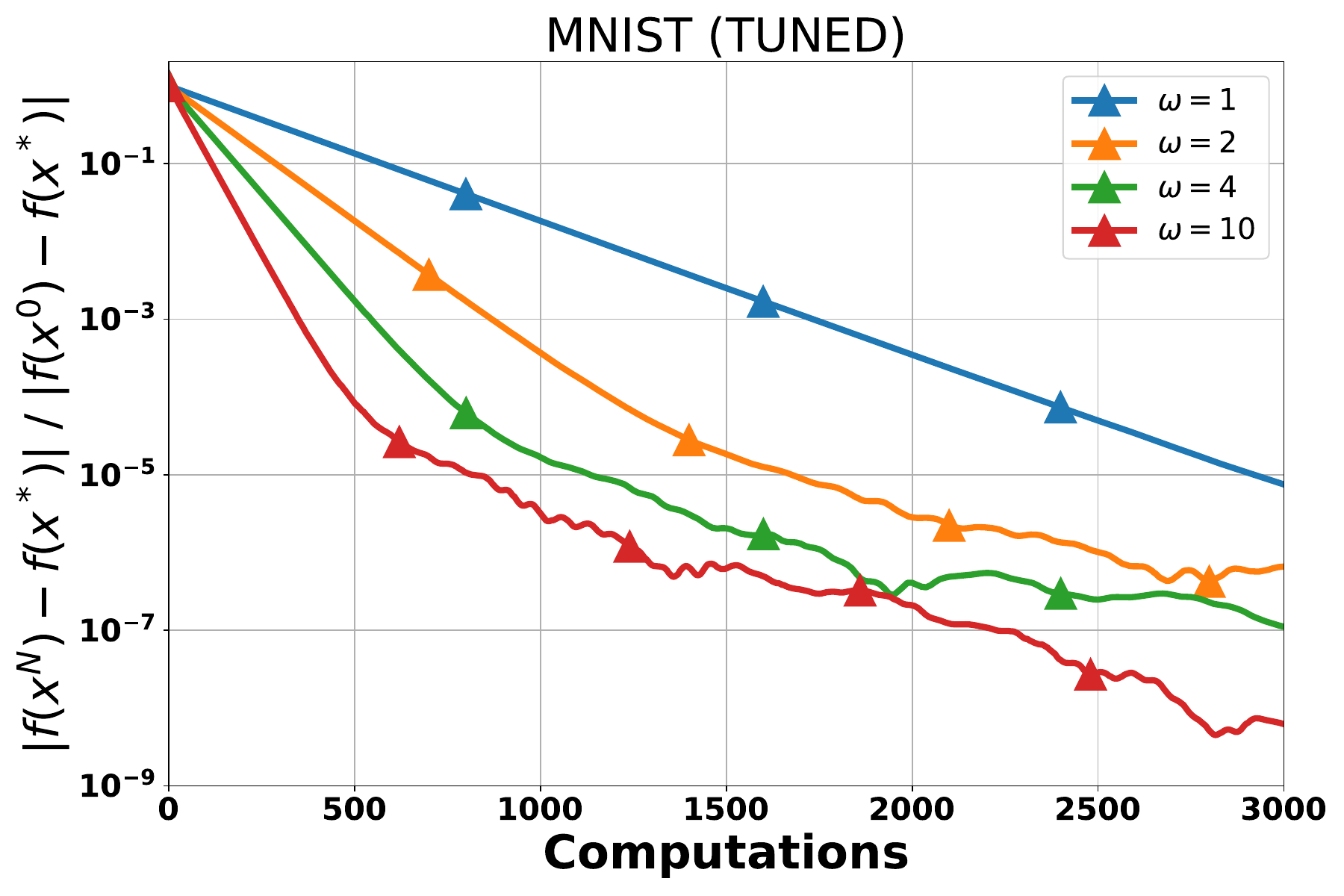}
\end{minipage}%
\\
\begin{minipage}{0.4\textwidth}
  \centering
\includegraphics[width =  \textwidth ]{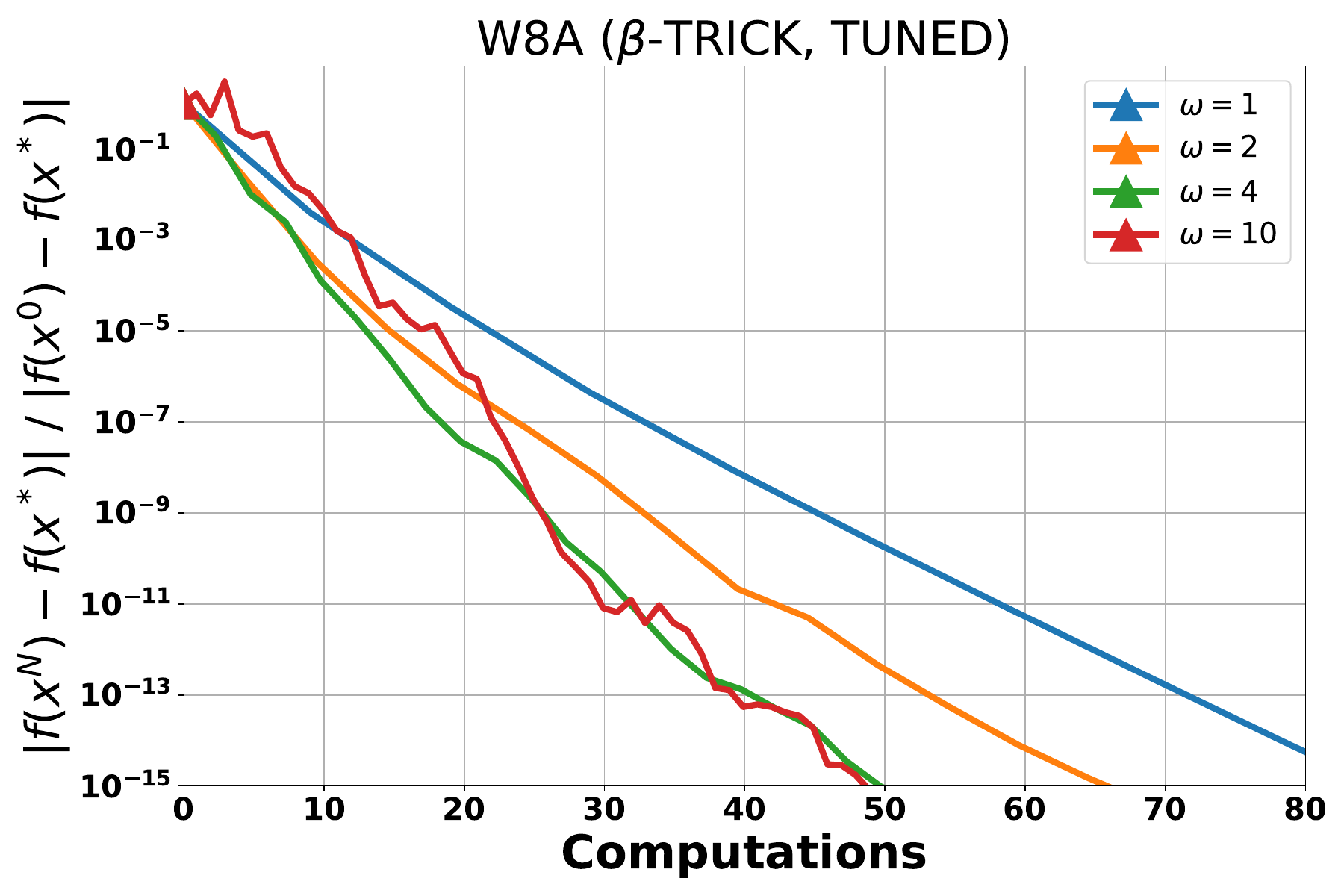}
\end{minipage}%
\begin{minipage}{0.4\textwidth}
  \centering
\includegraphics[width =  \textwidth ]{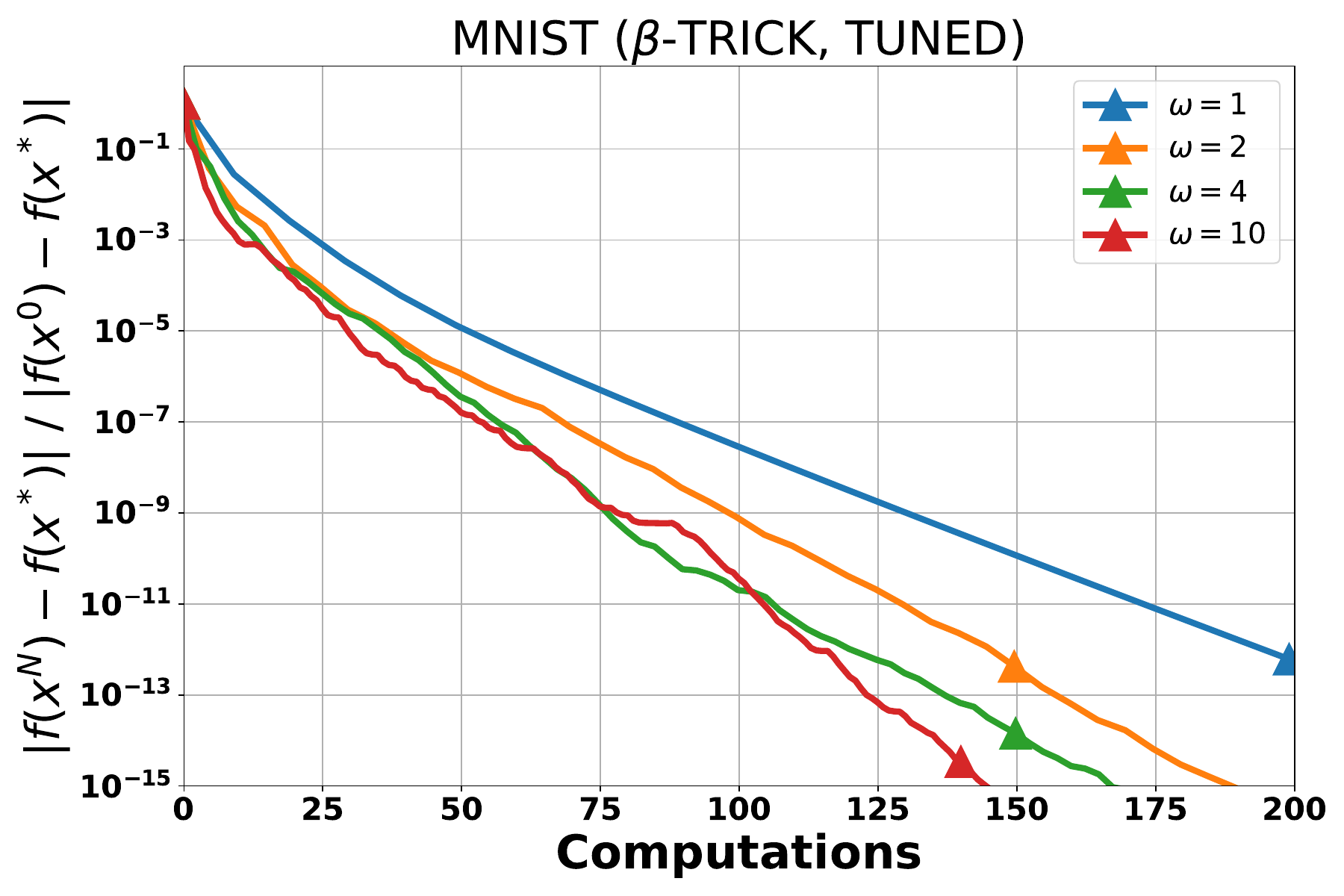}
\end{minipage}%
\\
\begin{minipage}{0.4\textwidth}
\centering
  (c) \texttt{w8a}
\end{minipage}%
\begin{minipage}{0.4\textwidth}
\centering
  (d) \texttt{MNIST}
\end{minipage}%
\caption{
Comparison of Algorithm \ref{alg:EG_coord} for solving the VFL problem (\ref{eq:vfl_lin_spp_1}) on LibSVM datasets \texttt{mushrooms}, \texttt{a9a}, \texttt{w8a} and \texttt{MNIST}. The criterion for comparison is the computational powers. The top line reflects the work of methods on the basic problem, the bottom line solves the problem with the $\beta$-trick (see disscusion after Corollary \ref{cor:EG_basic_1}).}
    \label{fig:comparison6}
\end{figure}

{In the fifth group of experiments with modifications (Figure \ref{fig:comparison4_noise}), we test the performance of Algorithm \ref{alg:EG_noise}. We use different levels of noise and observe how it affects convergence. The results show that noise leads to oscillations of convergence, the higher the noise level, the greater the degree of oscillations.}
\begin{figure}[h]
\begin{minipage}{0.49\textwidth}
  \centering
\includegraphics[width =  \textwidth ]{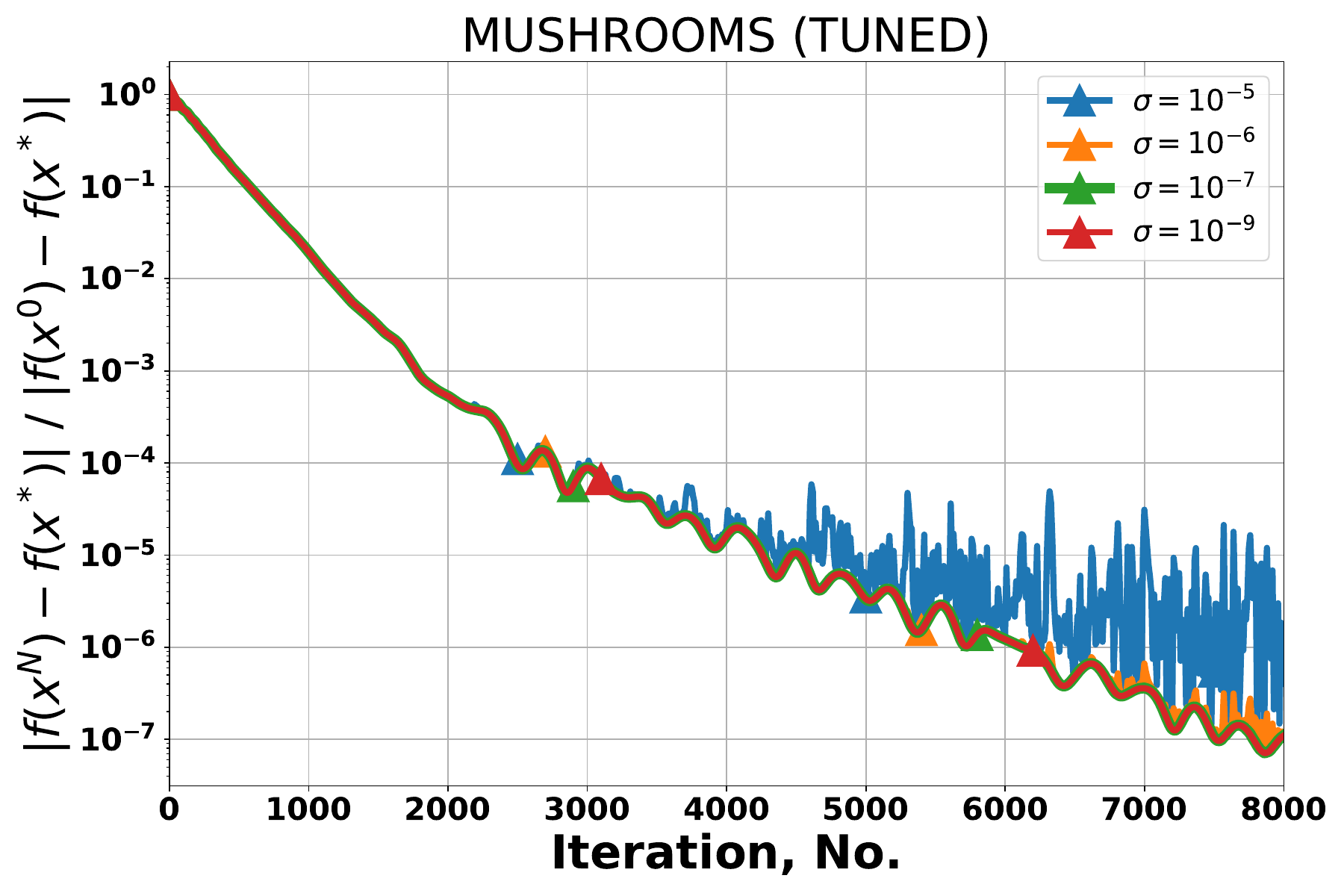}
\end{minipage}%
\begin{minipage}{0.49\textwidth}
  \centering
\includegraphics[width =  \textwidth ]{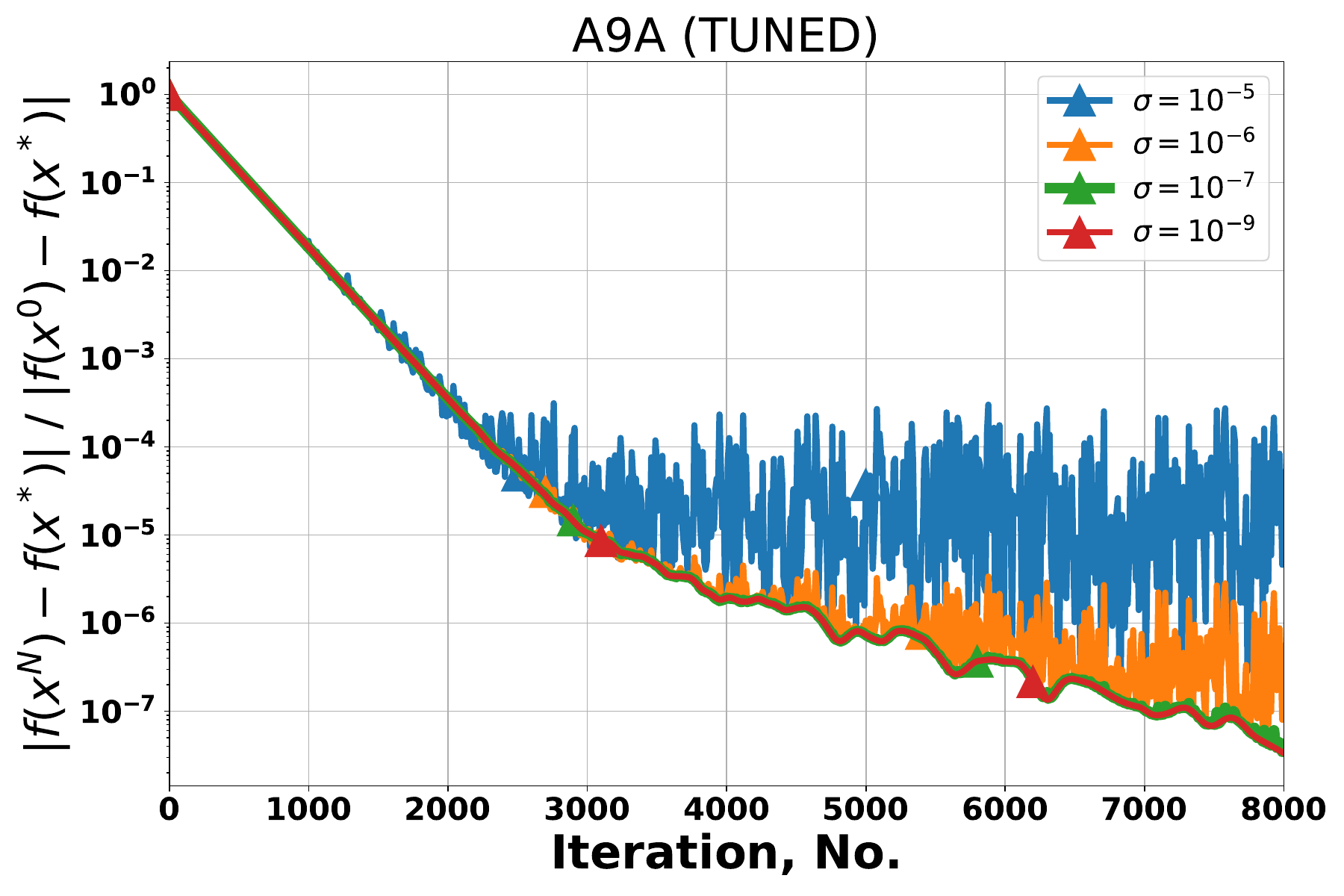}
\end{minipage}%
\\
\begin{minipage}{0.49\textwidth}
  \centering
(a) \texttt{mushrooms}
\end{minipage}%
\begin{minipage}{0.49\textwidth}
\centering
 (b) \texttt{a9a}
\end{minipage}%
\\
\begin{minipage}{0.49\textwidth}
  \centering
\includegraphics[width =  \textwidth ]{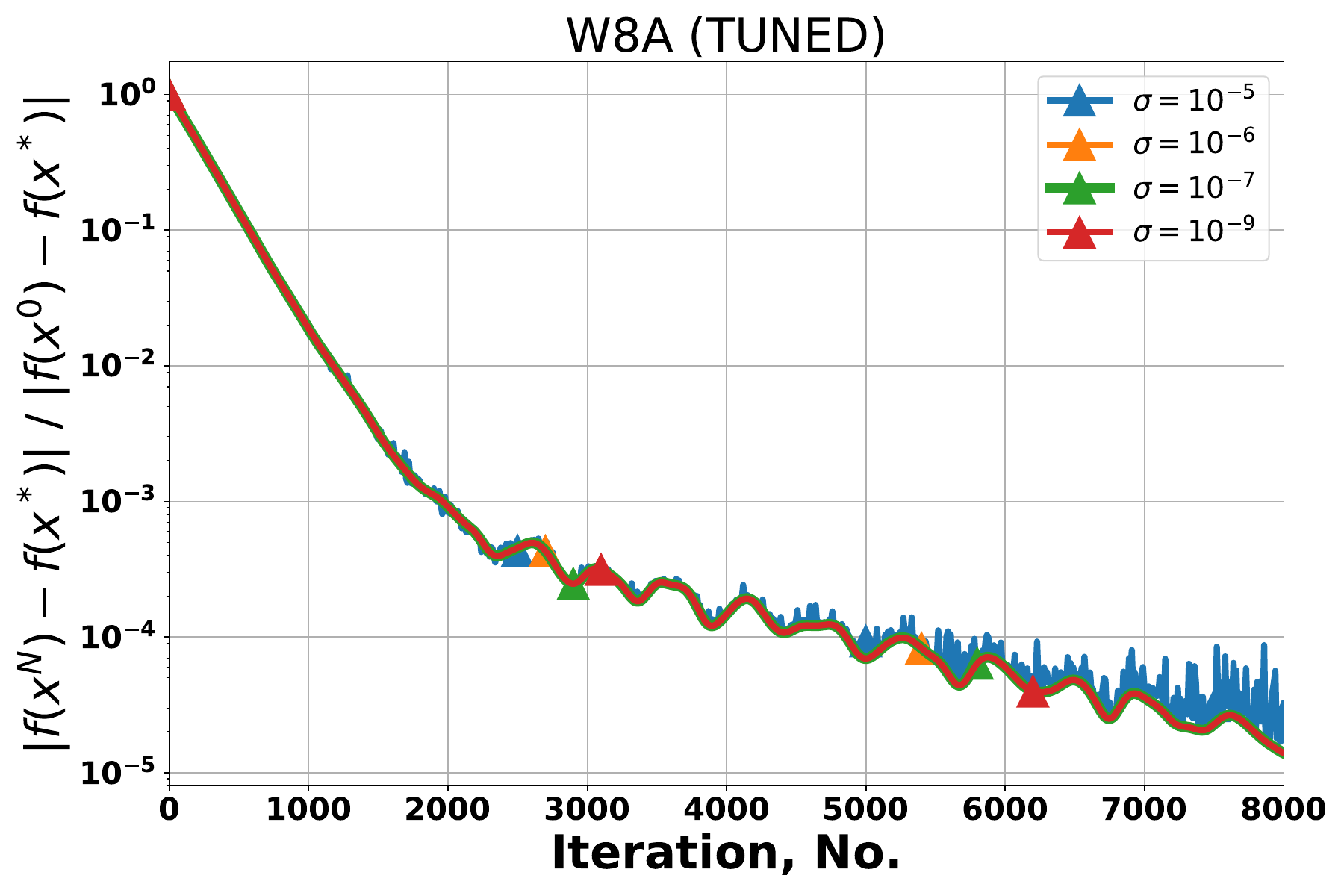}
\end{minipage}%
\begin{minipage}{0.49\textwidth}
  \centering
\includegraphics[width =  \textwidth ]{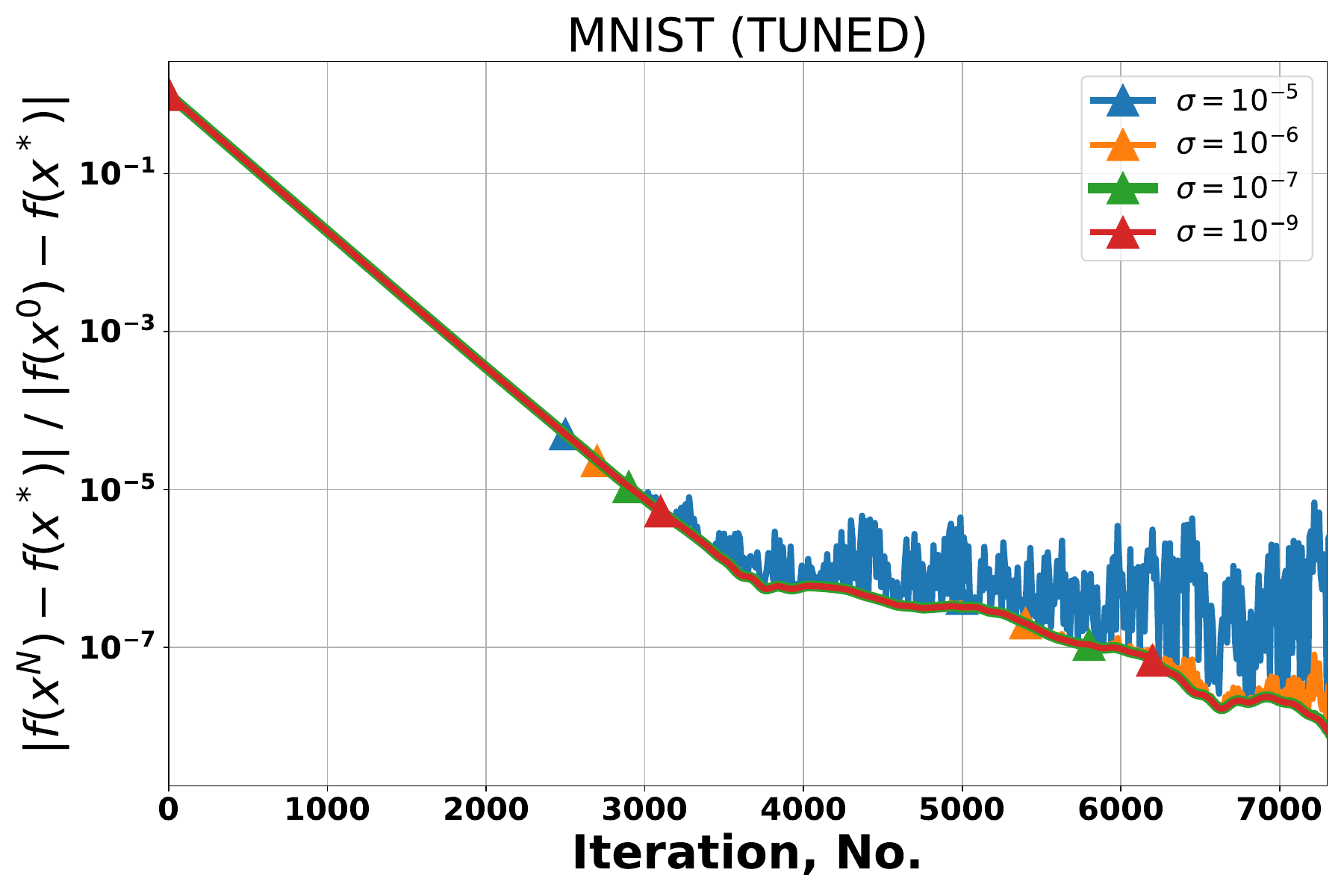}
\end{minipage}%
\\
\begin{minipage}{0.49\textwidth}
\centering
  (c) \texttt{w8a}
\end{minipage}%
\begin{minipage}{0.49\textwidth}
\centering
  (d) \texttt{MNIST}
\end{minipage}%
\caption{
Comparison of Algorithm \ref{alg:EG_noise} with different noise levels for solving the VFL problem \eqref{eq:vfl_lin_spp_1}. The comparison is made on LibSVM datasets \texttt{mushrooms}, \texttt{a9a}, \texttt{w8a} and \texttt{MNIST}. The criterion for comparison is the number of iterations.}
    \label{fig:comparison4_noise}
\end{figure}

\subsection{Technical details}\label{sec:tech_det}

Our algorithms is written in Python 3.10, with the use of PyTorch optimization library. We implement a simulation of distributed optimization system on a single server. Our server is AMD Ryzen Threadripper 2950X 16-Core Processor @ 2.2 GHz CPU and x2 NVIDIA GeForce GTX 1080 Ti GPU.

\end{appendixpart}
\end{document}